\documentclass[11pt,reqno]{amsart}
\PassOptionsToPackage{dvipsnames}{xcolor}
\usepackage[
  a4paper,
  hmargin=1in,  
  vmargin=1in
]{geometry}
\usepackage[T1]{fontenc}
\usepackage[english]{babel}

\usepackage{mathtools} 
\usepackage{amssymb, amsthm}
\numberwithin{equation}{section}
 
\usepackage[T1]{fontenc}
 
\usepackage[sb]{libertinus}   
\usepackage[amsthm,upint]{libertinust1math} 
\usepackage[cal=euler,scr=rsfso,bb=boondox]{mathalpha} 
\usepackage{comment}

\usepackage{graphicx}
\usepackage{array}
\usepackage{enumerate}          
\usepackage{tikz}
\usepackage[dvipsnames]{xcolor} 
\usepackage[
  breaklinks   = true,          
  pagebackref  = true,          
  colorlinks   = true,          
  citecolor    = Green,
  linkcolor    = OrangeRed,
  urlcolor     = black,
  pdfstartview = FitW           
]{hyperref}
\usepackage[normalem]{ulem}

\theoremstyle{plain} 
\newtheorem{thm}{Theorem}[section]
\newtheorem{corollary}[thm]{Corollary}
\newtheorem{lem}[thm]{Lemma}
\newtheorem{prop}[thm]{Proposition}

\theoremstyle{remark}
\newtheorem{rem}[thm]{Remark}

\newcommand{\bbE}{\mathbf{E}}
\newcommand{\E}{\mathbf{E}}

\newcommand{\Var}{\operatorname{Var}}
\newcommand{\Cov}{\operatorname{Cov}}
\newcommand{\bbP}{\mathbf{P}}
\renewcommand{\P}{\mathbf{P}}

\newcommand{\bbZ}{\mathbb{Z}}
\newcommand{\bbR}{\mathbb{R}}

\newcommand{\bbG}{\mathbb{G}}

\newcommand{\wtN}{\widetilde{N}}

\newcommand{\ol}{\overline}

\renewcommand{\bar}{\widebar}
\newcommand{\wh}{\widehat}

\DeclareMathOperator*{\argmax}{arg\,max}

\newcommand{\cI}{\mathcal I}

\newcommand{\cF}{\mathcal F}

\newcommand{\cS}{\mathcal S}

\newcommand{\rmc}{\mathrm{c}}

\newcommand{\rmd}{\mathrm{d}}
\newcommand{\rme}{\mathrm{e}}
\newcommand{\rmB}{\mathrm{B}}

\newcommand{\supp}{\mathrm{supp}}

\newcommand{\wt}{\widetilde}

\newcommand{\eqd}{\overset{\mathrm{d}}{=}}
\newcommand{\1}{{\bf 1}}

\newcommand{\ind}[1]{\mathbf{1}_{\{ #1 \}}}    

\newcommand*{\dif}{\ensuremath{\mathop{}\!\mathrm{d}}}

\usepackage{bbm}
\newcommand{\cM}{\mathcal{M}}

\title
{Aging in a spin glass with logarithmic correlations}

\makeatletter
\let\aff@setauthors\@setauthors
\def\@setauthors{\aff@setauthors
  \begingroup\centering\footnotesize\scshape
  Technion -- Israel Institute of Technology, Haifa 3200003, Israel\par\endgroup}
\makeatother

\author{Aser Cortines}
\address[Aser Cortines]{Technion -- Israel Institute of Technology, Haifa 3200003, Israel}

\author{Oren Louidor\textsuperscript{1}}
\address[Oren Louidor]{Technion -- Israel Institute of Technology, Haifa 3200003, Israel}
\email{oren.louidor@gmail.com}

\author{Heng Ma\textsuperscript{2}}
\address[Heng Ma]{Technion -- Israel Institute of Technology, Haifa 3200003, Israel}
\email{hengmamath@gmail.com}

\author{Adela Svejda}
\address[Adela Svejda]{Technion -- Israel Institute of Technology, Haifa 3200003, Israel}

\thanks{\textsuperscript{1}\,Email: \texttt{oren.louidor@gmail.com}}
\thanks{\textsuperscript{2}\,Email: \texttt{hengmamath@gmail.com}}

\makeatletter
\newif\ifnotoc@skip
\let\notoc@tocwrite\@tocwrite
\def\@tocwrite#1#2{%
  \ifnotoc@skip\global\notoc@skipfalse
  \else\notoc@tocwrite{#1}{#2}\fi}
\newcommand{\notoc}[1]{\global\notoc@skiptrue #1}
\makeatother

\begin{document}

\begin{abstract}
We consider a continuous-time random walk on the discrete two dimensional box, driven by the discrete Gaussian Free Field (DGFF) acting as potential: When at a vertex, the walk waits an exponentially distributed time with mean given by the exponential of the field times an inverse temperature parameter and then jumps to one of its neighbors uniformly at random. We prove that when the temperature is below the critical value the walk exhibits ``aging'' at a range of pre-equilibrium time scales: Observed at any such time and then again after an additional time of the same order, there is a positive probability that the walk is found within finite distance from where it was before, with this probability given asymptotically by the Generalized Arcsine Law with a temperature-dependent parameter. We show that this is a consequence of an intricate trapping mechanism which localizes the walk for periods of time which increase with the age of the system, and describe the complex structure of the underlying trapping landscape, which is intimately related to the geometry of the near-extreme level-sets of the DGFF. Altogether, this work demonstrates for the first time an Arcsine-Law aging in the case of a spin-glass-type system with a logarithmically correlated potential, throughout its glassy phase, as predicted in the physic literature~\cite{carpentier2001glass}.
\end{abstract}

\maketitle

\vspace{-0.1cm}
\setcounter{tocdepth}{2}
\begingroup\linespread{0.95}\selectfont
\tableofcontents
\endgroup

\newpage

\section{Introduction}
\subsection{Model and statement of aging}
For $N \geq 1$, let $\mathsf{V}_{N} \coloneq(-N,N)^2 \cap \bbZ^2 $ be the 2D box of side-length $2N-1$, viewed as a subgraph of $\bbZ^2$ containing all nearest-neighbor edges between its vertices. Define $\mathsf{T}_N$ as the torus obtained from $\mathsf{V}_{N}$ by adding an edge between any two vertices on the inner boundary of $\mathsf{V}_N$ which share a common coordinate. Consider then two processes on the same probability space. The first is the discrete Gaussian Free Field (DGFF) on $\mathsf{V}_N$ with zero boundary conditions outside, namely the Gaussian field $h = (h_x :\: x \in \mathsf{V}_N)$ with zero mean and covariance given by the discrete Green Function on $\mathsf{V}_N$ with 
Dirichlet boundary conditions (see Subsection~\ref{s:1.3} for more details).

The second is defined, conditionally on a realization of the first, as the continuous-time Markov Chain $X = (X_t :\: t \geq 0)$ on $\mathsf{T}_N$ with transition probabilities given by
\begin{equation*}
 \P \bigl( X_{t+\dif t} = y  \mid X_t = x \,,\, h \bigr) = 
\begin{cases}
\frac{1}{4} \, \rme^{-\beta  h_x }\, \dif t, & x \sim y \,,\\
0, & \text{otherwise.}
\end{cases}
\end{equation*}
Above, $\beta > 0$ will be referred to as {\em inverse-temperature}, and $x \sim y$ means that $x$ and $y$ share an edge in $\mathsf{T}_N$. Evidently, $X$ is a continuous-time simple random walk on the torus $\mathsf{T}_N$ with mean holding time $\rme^{\beta h_x}$ at $x$. It can be viewed as an instance of a random walk in random potential, with the potential given by $-h$.

Our subject of interest is the phenomenon of {\em aging}, whereby a diffusive motion gradually slows down as it moves about a medium. As it turns out, this phenomenon is already present in the above setup, provided the walk is observed at time scales which are properly tuned to the size of the underlying domain. To make a precise statement, recall that the distribution function of the {\em Generalized Arcsine Law} with parameter $\lambda \in (0,1)$ is given by
\begin{equation}
\label{e:1.2a}
{\rm Asl}_{\lambda}(v)
\coloneq \frac{\sin (\lambda \pi)}{\pi} \int_{0}^{v} u^{\lambda-1} (1-u)^{-\lambda} \dif u  
\quad ; \qquad v \in [0,1] \,,
\end{equation}
and set
\begin{equation}
\label{e:alpha}
	g\coloneq \frac{2}{\pi}
\qquad ; \qquad
	\alpha \coloneq \sqrt{2\pi} \,.
\end{equation}
The principal result of this paper is thus:
\begin{thm}[Aging]
\label{thm:1}
Fix any $\beta > \alpha$ and let $\epsilon > 0$ be arbitrarily small. Let $\{t_N\}$ be any sequence satisfying
\begin{equation}
(\log N)^{-\frac{5-\epsilon}{2\alpha} \beta}
\leq \frac{t_N}{N^{2\sqrt{g}\beta} \log N} \leq (\log N)^{-\frac{3+\epsilon}{2\alpha}\beta}.
\label{eq:gamma-N-range}
\end{equation}
Then, for any $\theta \in [0,\infty)$, 
\begin{equation} \label{equation:thm:1}
\lim_{R \to \infty}\limsup_{N \to \infty}  \Big| \P \Bigl( \big\|X_{t_N(1+\theta)} - X_{t_N}\big\| < R \bigr) - \mathrm{Asl}_{\frac{\alpha}{\beta}}\Bigl(\tfrac{1}{1+\theta}\Bigr) \Big|=0 \,,
\end{equation}
where $\mathrm{Asl}_\lambda$ is the Generalized Arcsine Law defined in~\eqref{e:1.2a}.   The convergence is uniform w.r.t. $\{t_N\}$.
\end{thm} 

In words, the theorem says that if the location of the walk is observed at time $t_N = t_N(\gamma)$ for $\gamma \in (\epsilon,1-\epsilon)$, where
\begin{equation}
\label{e:1.11}
t_N(\gamma) \coloneq N^{2\sqrt{g}\beta} (\log N)^{1-\frac{3+2\gamma}{2\alpha}\beta} \,,
\end{equation}
and then again $\theta t_N$ time later, then with positive probability the two locations are within $O(1)$ of each other and, moreover, this probability tends to an explicit value as $N \to \infty$. Since, for fixed $N$, the time $t_N$ can still span a large range of values in~\eqref{eq:gamma-N-range}, we see that the further the process has run on the domain, the longer it stays localized to where it has arrived. Metaphorically, it is said that the motion {\em ages} over time. See Subsection~\ref{s:1.3} for further discussion of the physical context and interpretation of this result. 

We remark that the convergence in~\eqref{equation:thm:1} is with respect to both the law of the field and the law of the walk. However, it takes an elementary argument to turn this {\em annealed}-type result to a {\em quenched}-type statement, whereby the convergence in~\eqref{equation:thm:1} holds in probability with respect to the law of the field, with the probability in the statement taken conditionally on $h$.

\subsection{The trapping landscape}
\label{s:1.2}
The underlying mechanism behind aging is (rather tautologically) that of {\em trapping}. Informally, traps are subsets of the domain to which the walk stays localized for a long period of time (a formal definition will be given below). Using the terminology of physical traps (as holes in the ground, or more physically, energy-wells), we shall refer to the {\em depth} of a trap, again informally for the moment, as a measure of how long it localizes the walk for. 

Aging occurs when deeper and deeper traps are discovered as the motion progresses and, for aging of the type shown here, when the time until a trap is discovered is proportional to the time for which it localizes the walk (as determined by its depth). This entails that a necessary input for a proof and characterization of aging is a detailed description of the underlying trapping landscape and its interplay with the motion. This is precisely the focus of this subsection.

Henceforth we implicitly fix $\beta > 0$, assume $\gamma \in (0,1)$ and abbreviate $t_N \equiv t_N(\gamma)$ as defined in~\eqref{e:1.11}. We remark that all results in this subsection hold uniformly in $\gamma$ provided that it is chosen from a fixed compact subinterval of $(0,1)$.

\subsubsection{Traps and their classification}
\label{s:1.2.1}
The mean holding time at vertex $x$ as determined by the field $h$ is 
\begin{equation*}
\tau_x \coloneq \rme^{\beta h_x} \,.
\end{equation*}
As it turns out (see Subsection~\ref{s:1.4a}), the relevant scale for this quantity to induce trapping at time $\Theta(t_N)$ is
\begin{equation}
\label{e:1.2}
\wh{t}_N \coloneq \frac{t_N}{\log N} = N^{2\sqrt{g}\beta} (\log N)^{-\frac{3+2\gamma}{2\alpha}\beta} \,.
\end{equation}
This motivates the definition of the normalized mean holding time
\begin{equation*}
\wh{\tau}_x \coloneq \frac{\tau_x}{\wh{t}_N} \,.
\end{equation*}
We shall treat {\em all} vertices in domain as {\em trap-vertices}, even if the slow-down they inflict upon the walk is just minor, and accoridngluy refer to $\wh{\tau}_x$ as the {\em depth} of the trap-vertex $x$.

Given $I \subset [0,\infty)$, we denote the set of trap-vertices whose depth lies in $I$ by
\begin{equation*}
\mathsf{T}(I) \coloneq \bigl\{ x \in \mathsf{V}_{N} :\: \wh{\tau}_x \in I \bigr\} \,.
\end{equation*}
For $\delta > 0$, we shall use the abbreviation 
\begin{equation*}
\mathsf{T}(\delta) \equiv \mathsf{T}\bigl([\delta, \infty)\bigr) \,,
\end{equation*}
and refer to the vertices in the above set as {\em $\delta$-trap-vertices}.
Due to the nature of the DGFF, $\delta$-trap-vertices tend to {\em cluster} (see Subsection~\ref{s:1.3}). This structural-geometric property is of utmost importance for the interplay between the landscape and the motion, and consequently for the way we define and classify traps.

To quantify this clustering structure, we take a sequence $\{r_N\}$ which grows faster than any power of $\log N$ but slower than any polynomial in $N$:
\begin{equation}\label{def_rn}
r_N \coloneq \rme^{(\log \log N)^{100(2+\gamma) }} \,,
\end{equation} 
and denote the set of {\em isolated} $\delta$-trap-vertices by
\begin{equation*}
\mathsf{IT}(\delta) \coloneq \Bigl\{ x \in \mathsf{T}(\delta) :\: x+y \notin \mathsf{T} ( \delta  ) \,,\,\,
\forall y \in \mathsf{B}_{N/r_N} \setminus \mathsf{B}_{r_N} 
\Bigr\} \,.
\end{equation*}
Hereafter we use $\mathsf{B}_r$ to denote the ball of radius $r$ in the supremum norm and occasionally use the notation $\mathsf{B}_r(x) \coloneq x+\mathsf{B}_r$.

It follows by construction that,
\begin{equation}
\label{e:1.16a}
\forall \  x,y \in \mathsf{IT}(\delta) \, , \quad  	\|x-y\|_\infty \in [0,r_N] \cup [N/r_N, N] \,.
\end{equation}
We call a set satisfying the above {\em $r_N$-clustered}. The vertices of $\mathsf{IT}(\delta)$ uniquely partition into subsets of diameter at most $r_N$, which are at least $N/r_N$ apart. These subsets will be referred to as the {\em $r_N$-clusters} of the set. In the case of $\mathsf{IT}(\delta)$, each such $r_N$-cluster can be naturally represented by the a.s. unique vertex carrying the $r_{N}$-local maximum of $h$ on it, so that altogether the collection of all $r_N$-clusters of $\mathsf{IT}(\delta)$ can be captured by
\begin{equation*}
\mathsf{IT}^*(\delta) \coloneq \Bigl\{ x \in \mathsf{IT}(\delta) :\: \wh{\tau}_x = \max_{y \in \mathsf{B}_{r_N}} \wh{\tau}_{x+y} \Bigr\} \,.
\end{equation*}
Using again the terminology of physical traps, we refer to $x \in \mathsf{IT}^*(\delta)$ as a {\em $\delta$-trap-bottom} (of the corresponding $r_N$-cluster of $\delta$-trap-vertices in $\mathsf{IT}(\delta)$). (Henceforth, all sets containing trap-bottoms will be marked by a $^*$ superscript.)

Next we distinguish between {\em $\delta$-normal} and {\em $\delta$-deep} trap-bottoms, according to whether their depth lies below or above $\delta^{-1}$. The corresponding subsets of $\mathsf{IT}^*(\delta)$ are then,
\begin{equation*}
\begin{split}
\mathsf{NT}^*(\delta) & \coloneq \Bigl\{ x \in \mathsf{IT}^*(\delta) :\: \wh{\tau}_x \in \bigl[\delta,\, \delta^{-1}\bigr] \Bigr\}  \,,\\
\mathsf{DT}^*(\delta) & \coloneq \mathsf{IT}^*(\delta) \setminus \mathsf{NT}^*(\delta) = \Bigl\{ x \in \mathsf{IT}^*(\delta) :\: \wh{\tau}_x \in \bigl(\delta^{-1},\,\infty\bigr) \Bigr\} \,.
\end{split}
\end{equation*}
The $\delta$-normal trap-bottoms are further divided according to whether the diameter of the corresponding cluster is smaller or larger than $r > 0$. The former are called {\em $(\delta,r)$-regular trap-bottoms}, and the latter {\em $(\delta, r)$-wide trap-bottoms}:
\begin{equation*}
\begin{split}
\mathsf{RT}^*(\delta, r) & \coloneq \Bigl\{ x \in \mathsf{NT}^*(\delta) :\: \max_{y \in \mathsf{B}_{r_N} \setminus \mathsf{B}_r} \wh{\tau}_{x+y} \in \bigl(0,\, \delta\bigr)\ \Bigr\}  \,,\\
\mathsf{WT}^*(\delta, r) & \coloneq  \mathsf{NT}^*(\delta) \setminus \mathsf{RT}^*(\delta, r) = \Bigl\{ x \in \mathsf{NT}^*(\delta) :\: \max_{y \in \mathsf{B}_{r_N} \setminus \mathsf{B}_r} \wh{\tau}_{x+y} \in \big[\delta,\, \infty\big)\ \Bigr\} \,.
\end{split}
\end{equation*}

The {\em traps} themselves are defined as the collections of vertices in balls of appropriate radii around the corresponding bottoms. Thus, 
\begin{equation}
\begin{split}
x+\mathsf{B}_{r} :\,  x \in \mathsf{NT}^*(\delta)
\ , \quad
& x+\mathsf{B}_{r} :\,  x \in \mathsf{RT}^* (\delta,r)
\ , \quad
x+\mathsf{B}_{2 r_N} :\,  x \in \mathsf{WT}^* (\delta,r)  
\ \ \text{and} \\
& x+\mathsf{B}_{2 r_N}  :\,  x \in \mathsf{DT}^*(\delta) \,,
\end{split}
\end{equation}
denote the {\em $(\delta,r)$-normal}, {\em $(\delta,r)$-regular}, {\em $(\delta,r)$-wide} and {\em $\delta$-deep} {\em trap} at (bottom) $x$, respectively. The union of all such traps are then given, respectively, by
\begin{equation}
\begin{split}
	\mathsf{NT}(\delta, r) \coloneq \mathsf{NT}^*(\delta) + \mathsf{B}_{r}  \ , \ \ 
& \mathsf{RT}(\delta, r) \coloneq \mathsf{RT}^*(\delta, r) + \mathsf{B}_{r}  \ , \ \ 
\mathsf{WT}(\delta, r) \coloneq \mathsf{WT}^*(\delta, r) + \mathsf{B}_{2 r_N}  \ \ \text{and} \\
	& \mathsf{DT}(\delta) \coloneq \mathsf{DT}^* (\delta)+ \mathsf{B}_{2 r_N} \,.
\end{split}
\end{equation}
Thus, a {\em trap} is a ball of certain radius around a trap-bottom, which, in turn, is a local-max of an $r_N$-cluster of the isolated set of (trap-)vertices $x \in \mathsf{V}_N$ whose depth $\wh{\tau}_x$ is above $\delta$. The type of the trap (normal, regular, deep, etc.) is determined by its bottom and the $r_N$-cluster around it. These, rather subtle, definitions are reflective of the intricate mechanism by which the trapping landscape inflicts a slowdown on the motion (see Subsection~\ref{s:1.4a}).

A $\delta$-trap-vertex is called a {\em non-isolated $\delta$-trap-vertex} if it lies within distance $2 r_N$ of a $\delta$-trap-vertex which is not in $\mathsf{IT}(\delta)$, namely
\begin{equation*}
\mathsf{NIT}(\delta) \coloneq  \bigl( \mathsf{T}(\delta) \setminus \mathsf{IT}(\delta) \bigr) + \mathsf{B}_{2 r_N}.
\end{equation*}
Lastly, a trap-vertex is called {\em $\delta$-shallow} if its depth is below $\delta$ and it is neither non-isolated nor contained in a deep trap:
\begin{equation*}
	\mathsf{ST}(\delta) \coloneq \mathsf{T}\bigl((0, \delta)\bigr) \setminus \bigl(\mathsf{NIT}(\delta) \cup \mathsf{DT}(\delta)\bigr) \,.
\end{equation*}

\bigskip
It is not difficult to verify that, by construction, we overall have
\begin{equation}
	\mathsf{V}_{N} = \mathsf{RT}(\delta,r) \cup \mathsf{WT}(\delta, r) \cup \mathsf{DT}(\delta) \cup \mathsf{NIT}(\delta)   \cup \mathsf{ST}(\delta) \,.
	 \label{eq:torus-decomp}
\end{equation}
We shall show below that the only traps which are ``relevant'' for the walk are the regular ones. Up to time $\Theta(t_N)$, the time spent by the walk at the remaining wide, deep, shallow and non-isolated trap-vertices is negligible, asymptotically with high probability (a.w.h.p.) as $N \to \infty$ followed by $(\delta, r) \to (0,\infty)$. The reason for distinguishing between the different types of non-regular traps is that their negligibility follows from different reasons, as shown in the next subsection.
\begin{figure}[tp]
	\centering
	\includegraphics[width=0.7\textwidth]{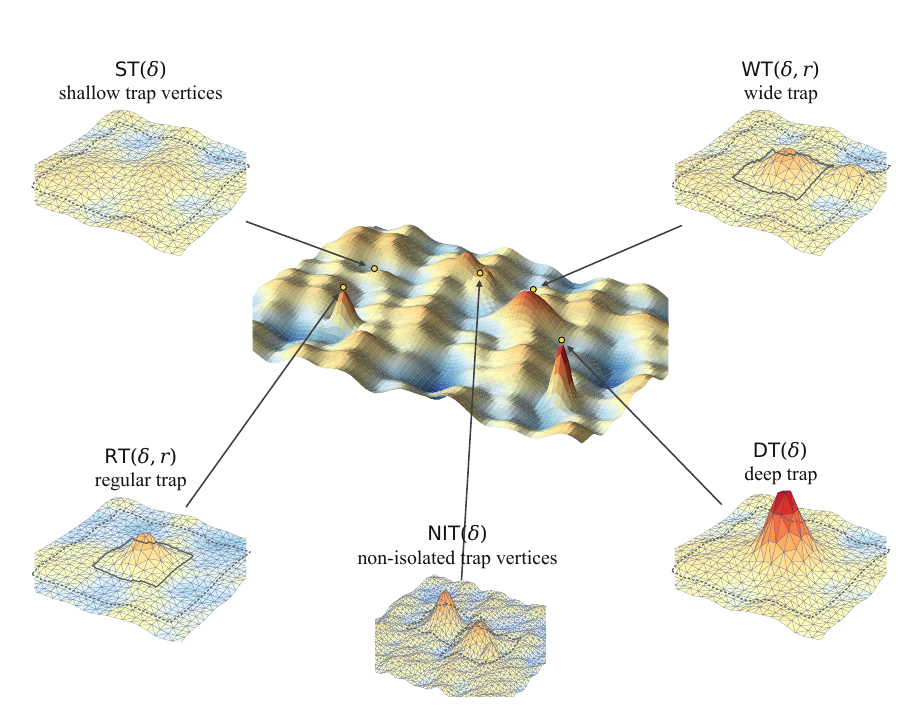}
	\caption{Illustration of different types of traps (not a real simulation). $\partial \mathsf{B}_{r_N}$ and $\partial \mathsf{B}_{r}$ are represented by dashed and solid lines, resp., and are not drawn to scale ($r \ll r_N \ll N/r_N$).
	}
	\label{fig:traps}
\end{figure}

\subsubsection{Trapping landscape results}
Let us introduce another scale parameter:
\begin{equation}
\label{e:1.12}
	\kappa_N \coloneq (\log N)^\gamma \,,
\end{equation}
with $\gamma$ the same as in the definition of $t_N(\gamma)$ in~\eqref{e:1.11}. As will be shown, $\kappa_N^{-1}$ captures the order of the fraction of vertices {\em and} $r_N$-clusters found by the walk during time $t_N(\gamma)$. Any set of trap-vertices of size $o(\kappa_N)$ or composed of that many $r_N$-clusters is completely avoided by the walk with high probability w.h.p. Our first statement shows that this is indeed the case for wide, deep and non-isolated traps. 
\begin{thm}[Non-reachable traps]
 \label{thm:non_reachable_traps} Fix any $\epsilon > 0$. Then the following assertions hold:
 \begin{align}  
& \lim_{\delta \downarrow 0}\,
\limsup_{N \to \infty}\, \P( 
| \mathsf{DT}^*(\delta) | >  \epsilon \kappa_N  ) =0  . \label{e:10} \\
& \lim_{N \to \infty}\, \P( 
|  \mathsf{T}(\delta) \setminus\mathsf{IT}(\delta)  | > \epsilon \kappa_N ) =0 , \, \text{ for all } \delta \in (0,1) . \label{e:11}  \\
&   \lim_{r \to \infty} \,
\limsup_{N \to \infty}\, \P( | \mathsf{WT}^*(\delta, r) | > \epsilon \kappa_N ) =0 ,\,\text{ for all } \delta \in (0,1) . \label{e:12}  
 \end{align}
\end{thm} 

We turn to the $\delta$-shallow trap-vertices. While their number is not $o(\kappa_N)$, it is their shallowness which ensures that their total contribution to the time accumulated by the walk is negligible. This is a consequence of the next theorem, which shows that the sum of the depths of all such vertices is $o(\kappa_N)$ w.h.p. Since only a fraction of order $\kappa_N^{-1}$ of them will be visited, their total contribution to the accumulated depth will be $o(1)$ and hence negligible as well.
\begin{thm}[Negligibility of shallow traps]
\label{thm:sum_shallow_traps}
For any $\epsilon > 0$,  we have 
\begin{equation*}
	\lim_{\delta \downarrow 0}\, \limsup_{N \to \infty}\, \P \Bigl( \sum_{x \in \mathsf{ST} (\delta)}  \wh{\tau}_x  >  \epsilon  \, \kappa_N  \Bigr) = 0 . 
\end{equation*}
\end{thm}

This leaves only the regular traps. The next, ``positive'', result shows that these are indeed abundant enough. Again the relevant scale is $\kappa_N$.
\begin{thm}[Abundance of regular traps]
\label{thm:regular traps1}
For  any $\epsilon > 0$,  there exist constants $0<\rho^{-}_{\epsilon}<\rho^{+}_{\epsilon} <\infty$ such that   
\begin{equation}
\label{e:14}  \inf_{\delta \in (0,1/2]} \, \inf_{r \ge 1} \  
\liminf_{N \to \infty} \, 
\P\Bigl( \rho^{-}_{\epsilon} \le \frac{ |\mathsf{RT}^*(\delta, r) | }{\delta^{-\alpha/\beta}  \kappa_N  } \le  \frac{ |\mathsf{NT}^*(\delta ) | }{\delta^{-\alpha/\beta}  \kappa_N  } \le \rho^{+}_{\epsilon}  \Bigr)  \ge 1- \epsilon  .
\end{equation}
\end{thm}
We remark that the number of regular traps should no doubt admit a scaling limit after normalization by $\kappa_N$ (see the discussion in Subsection~\ref{s:1.3} and also~\cite{biskup2024near}). However, since deriving such a limit is a non-trivial task, we circumvent it by proving tightness for the total number of regular traps (both from above and away from $0$, as in the previous theorem), together with a limit in probability for the empirical distribution of the joint relative depths of all vertices in such traps. The latter is the subject of the next theorem.

\begin{thm}[Asymptotic distribution of regular traps]
\label{thm:regular traps2} Let $r \ge 1$ and 
 $s \in (0,1]^{\bbZ^2}$ with $\{y: s_{y} < 1\} \subset \mathsf{B}_r \setminus \{0\}$.  
 Let also $I \subset [\delta,\delta^{-1}]$ be an interval. Then, for all $\epsilon > 0$, 
\begin{equation} \label{e:16.2} 
 \limsup_{N \to \infty}\P\biggl( \Big|     \sum_{x \in \mathsf{NT}^*(\delta)}  \frac{  \ind{\wh{\tau}_x \in I ,\,\, \wh{\tau}_{x+y}/\wh{\tau}_x \le s_y ,\,\forall\, y  \in \mathsf{B}_{r} }  }
 { |\mathsf{NT}^*(\delta) |}   \, 
- \, \frac{\int_{I} t^{-(\alpha/\beta + 1)} \rmd t}{\int_{\delta}^{\delta^{-1}} t^{-(\alpha/\beta + 1)} \rmd t} 
\times \sigma_\beta \Bigl(\prod_{y \in \bbZ^2} (0, s_y]\Bigr)  \Big| > \epsilon \biggr)=0 \,,
\end{equation} 
where $\sigma_\beta$ is a probability measure 
on $(0,1]^{\bbZ^2}$.
\end{thm} 

We shall also need the following $L^1$ summability of the relative depths under $\sigma_{\beta}$.
\begin{thm}[Summability of the cluster distribution]\label{thm:sumability_clusters}
Fix $\beta > \alpha$ and let $\sigma_\beta$ be as in Theorem~\ref{thm:regular traps2}. Then, we have 
\begin{equation*}
	\int \|s\|_1 \sigma_\beta(\rmd s) < \infty \,,
\end{equation*}
where $\|s\|_1$ denotes the $\ell_1$ norm of $s \in (0,1]^{\bbZ^2}$.
\end{thm}

Finally, we must ensure that, with overwhelming probability, the distance between any vertex $x \in \mathsf{V}_{N}$, the origin $\mathbf{0}\coloneq(0,0)$, in particular, and the set of $\delta$-trap-vertices is sufficiently large. This is shown in
\begin{thm}\label{thm:distance_origin} We have 
\[ 
\limsup_{N \to \infty}\, \sup_{x \in \mathsf{V}_N} \P\bigl( d  \bigl(x,  \mathsf{T}(\delta) \bigr) \le 2 N/r_N  \bigr) = 0.\] 
\end{thm}

The results above will be used to prove Theorem~\ref{thm:1} (see Section~\ref{s:2}).

\subsection{Interpretation and context}
\label{s:1.3}
Aging is a phenomenon in disordered systems whereby temporal correlations increase as the system gets older. A central class of models where aging is observed is that of the spin-glass type, of which our model can be seen as an instance\footnote{Henceforth, we use the term spin-glass to refer, loosely, to a thermodynamical system whose Boltzmann-Gibbs equilibrium law is governed by quenched disorder (random energy landscape), not necessarily derived from spin interactions. This is rather customary in this area, see, e.g.~\cite{derrida1981random, arous2008universality}. See also~\cite{carpentier2001glass} for a discussion on how the present setup arises from more ``physical'' models, including one of the ``pure'' spin-glass type.}
Indeed, it is not difficult to see that (conditional on $h$) the Markov process $X$ admits a limit in law given by its stationary-reversible measure:
\begin{equation}
\label{e:1.29}
	\mu_N^\beta(\{x\}) \coloneq \frac{\rme^{\beta h_x}}{\sum_{z \in \mathsf{V}_{N}} \rme^{\beta h_z}}
		\quad ; \qquad x \in \mathsf{V}_{N} \,.
\end{equation}
 Thus $X$ can be seen as Glauber Dynamics for a thermodynamical system with states $\mathsf{V}_{N}$ 
whose equilibrium law is the Boltzmann Distribution above, with $-h$ acting as the potential.

It is well known that whenever $\beta > \alpha$, with $\alpha$ as in~\eqref{e:alpha}, the measure~\eqref{e:1.29} is asymptotically carried by the extreme values of the field $h$ (the minimal energy states of the potential $-h$; cf.~\cite{BL3}). A common way to record the joint law of all extreme values of $h$ is via the so-called {\em (multi-scale) extremal process} of $h$,
\begin{equation*}
\eta_{N,\rho}\coloneq\sum_{x\in V_N}
\delta_{\,x/N}\otimes\delta_{\,h_x-m_N}\otimes\delta_{\,\{h_{x+y}-h_x\colon y\in\bbZ^2\}} 
\1_{\{h_x=\max_{z\in \rmB_\rho(x)} h_z\}}\,.
\end{equation*}
This is a point process on $[-1,1]^2 \times \bbR \times \bbR_{-}^{\bbZ^2}$ which records the {\em scaled} position, {\em centered} value and {\em relative heights} of all $\rho$-local maxima of $h$ on $\mathsf{V}_{N}$. The centering sequence $(m_N)_{N \geq 1}$, first derived in full in~\cite{BDZ16}, is given by
\begin{equation*}
m_N\coloneq 2\sqrt{g} \log N - \frac{3}{4}\sqrt{g}\,  \log \log N \,,
\end{equation*}
and captures the typical height of the maximum of the field. Here $g$ is as in~\eqref{e:alpha}.

It was shown in~\cite{BL1,BL3,BL2} that when
$\rho_N \wedge (N/\rho_N) \underset{N\to\infty}\longrightarrow \infty$, the above extremal process admits the weak limit 
\begin{equation}
\label{e:1.10i}
\eta_{N,\rho_N}
\underset{N \to \infty}\Longrightarrow \eta_\infty 
\quad; \qquad
	\eta_\infty \sim\  \text{\rm PPP}\bigl({\mathcal M}^\alpha_\infty(\rmd x)\otimes\rme^{-\alpha u}\rmd u\otimes\nu(\rmd \omega)\bigr) \,,
\end{equation}
as a (random) element of the space of Radon measures on $[-1,1]^2 \times \bbR \times \bbR_{-}^{\bbZ^2}$ equipped with the Vague Topology. 
Above, ${\rm PPP}$ stands for {\em Poisson Point Process}, $\alpha$ is as in~\eqref{e:alpha}, $\cM_\infty^\alpha$ is the {\em critical Liouville Quantum Gravity Measure} (LQGM) - a random measure on $[-1,1]^2$ (see, e.g.,~\cite{DRSV2,DRSV1}), and $\nu$ is a deterministic probability law on $\bbR_-^{\bbZ^2}$.

Writing the intensity measure in~\eqref{e:1.10i} as
\begin{equation*}
\wh{{\mathcal M}}^\alpha_\infty(\rmd x)\otimes\rme^{-\alpha (u-\frac{1}{\alpha} \log 
|{\mathcal M}^\alpha_\infty|)} \rmd u \otimes \nu(\rmd \omega) \,,
\end{equation*}
with $\wh{\cM}^\alpha_\infty$ the normalized version of $\cM^\alpha_\infty$ -- after division by its total mass $|\cM^\alpha_\infty| \equiv \cM^\alpha_\infty([-1,1]^2)$, the convergence in~\eqref{e:1.10i} can be interpreted as follows: W.h.p., the extreme values of $h$ on $\mathsf{V}_{N}$ form clusters of diameter $\Theta(1)$ and are $\Theta(N)$ apart. The local maxima of $h$ in each cluster, recentered by $m_N$, asymptotically form a PPP with an exponential intensity at rate $\alpha$ and a random global shift of size $\alpha^{-1}\log |\cM^\alpha_\infty|$. The locations of these maxima, rescaled by $N^{-1}$, are asymptotically i.i.d. and chosen from the random law $\wh{\cM}^\alpha_\infty$. Finally, the configuration of heights at vertices around and relative-to each local maximum, are also asymptotically i.i.d. and chosen according to $\nu$, which is thus henceforth referred to as  the {\em extremal cluster law}.

A consequence of the above complete asymptotic description is the convergence of the scaled equilibrium measure
$\mu_N^\beta (N\,\cdot)$ from~\eqref{e:1.29} to the law $\wh{\cM}^\beta_\infty$, which is the normalized version of the measure
\begin{equation}
\label{e:1.36}
\cM^\beta_\infty \coloneq \sum_{(x,u,\omega) \in \eta_\infty} \!\!\!\!\!\! \rme^{\beta u}\Bigl(
\sum_{y \in \bbZ^2} \rme^{\beta\omega_y}\Bigr) \delta_x  
\overset{\rmd} = \sum_{(x,t) \in \chi^\beta_\infty}  \!\!\!\! t \delta_x  
\quad ; \qquad \chi^\beta_\infty \sim \text{PPP} \bigl(\wh{\cM}^\alpha_\infty(\rmd z) \otimes \kappa_{\beta} |\cM^\alpha_\infty|\, t^{-1-\frac{\alpha}{\beta}}  \rmd t \bigr) \,.
\end{equation} 
Here $\kappa_\beta \in (0,\infty)$ is a constant and the outer sums range over all atoms of the relevant point process. In particular, the limiting statistics of energy levels (the collection of masses of $\wh{\cM}_\infty^\beta$) in equilibrium is {\em Poisson-Dirichlet} (PD). See, e.g.~\cite{BL3}.

The static picture painted above has dynamical consequences. Indeed, when viewed in equilibrium time scales, namely $\Theta(t_N(0))$ (\eqref{e:1.11} with $\gamma=0$), the dynamics $X$ admits a functional space-time scaling limit in law, given by {\em super-critical Liouville Brownian Motion} at parameter $\beta>\alpha$ (a spatial version of Kolmogorov's K-process driven by $\chi^\beta_{\infty}$). In particular, the dynamics spends (asymptotically) all of its time in the energy levels (mass points) of $\wh{\cM}_\infty^\beta$ with transitions between them occurring at (asymptotically) infinitesimal time. See in~\cite{cortines2018dynamical}.

Viewed in this context, Theorem~\ref{thm:1} thus focuses on the pre-equilibrium behavior of the dynamics, i.e. just before the system has relaxed to its stationary law~\eqref{e:1.29}. It thus demonstrates that on its way to relaxation the system ``gets stuck'' in {\em meta-stable} low energy states, from which it takes it a non-negligible time to escape. 

Aging has been studied in many other mean-field spin-glass and spin-glass-like models, with different underlying graphs and correlation structures for the random potential.
These include, e.g., Derrida's {\em Random Energy Model} (REM)~\cite{arous2003glauberI, arous2003glauberII} (i.i.d.\ energy levels on the hypercube) and the {\em $p$-spin Sherrington-Kirkpatrick Model} (SK) ($\ell^p$-type correlations on the hypercube)~\cite{Bovier13, arous2008universality}. In all of these models, the same Generalized Arcsine Law was exhibited as the limit of properly chosen two-point temporal correlation functions, at different temperature and size related time scales. 

A similar kind of aging was also shown in the 
so-called {\em Bouchaud Trap Model} (BTM)~\cite{arous2006aging} on various graphs. The latter can be seen as an effective model for aging, in which the mechanism giving rise to the Generalized Arcsine Law appears in its most distilled form. These results and others render the Generalized Arcsine Distribution a universal law for aging. An excellent survey on the subject, including a recipe for proving aging, which we partly follow here, is given in~\cite{arous2006course} and also~\cite{arous2008arcsine}.

Unlike in any of the models discussed above, the potential considered in this work has approximate logarithmic correlations. Indeed, the Green function on $\mathsf{V}_{N}$, which determines the covariances of the DGFF, obeys
\begin{equation*}
\Cov(h_x, h_y) = 
\bbG_N(x,y) = \E_x \sum_{n=0}^{\tau_{\mathsf{V}_{N}^\rmc}} \1_{y}(S_n) = -g \log \frac{\|x-y\|+1}{N} + O(1) \,,
\end{equation*}
Above, $\{S_n\}$ is a simple random walk on $\bbZ^2$, which under $\E_x$ starts from $x \in \bbZ^2$, and $\tau_{\mathsf{V}_{N}^\rmc}$ denotes its exit time from $\mathsf{V}_{N}$. These asymptotics are valid for $x,y$ in the bulk.

Logarithmically correlated fields, of which the DGFF is a canonical example, have been studied extensively over the past several decades. Here, too, a universal picture has emerged, for their extreme and large values, with features such as the clustering of extreme values, Poissonian statistics for their joint law, and the randomly shifted exponential intensity consistently appearing (see, e.g.,~\cite{biskuppims, arguin2016extrema}). 

This common extremal and large-value behavior finds expression also in properties of measures formed by taking the exponential of the field as a density w.r.t. a reference measure, in a manner similar to~\eqref{e:1.29} (in the discrete; in the continuum a limiting procedure is usually needed to make the definition mathematically sound).
Such measures naturally arise in various mathematical and physical contexts (not just as Boltzmann Distribution with respect to log-correlated potentials), and go by different names, depending on the underlying field. Prominent examples include the Gaussian Multiplicative Chaos (GMC; also known as Liouville Quantum Gravity; see~\cite{rhodes2014gaussian}) and the Mandelbrot Multiplicative Cascade (MMC; see~\cite{mandelbrot1989multifractal}). The former includes the case of the DGFF (in the limit) as treated in the present work.

For such measures, there is a critical temperature $\beta^{-1}$, below which, they are (asymptotically, in the discrete) supported on a discrete, clustered, set of the arg-extrema (or thickest points in the continuum) of the field, while above they are (asymptotically) carried by a fractal set of partial Hausdorff Dimension. In the former regime, the (limiting) measure takes the same form as in~\eqref{e:1.36}, with Poisson-Dirichlet Law for its masses (e.g.,~\cite{rhodes2014gaussian, mandelbrot1989multifractal}). When such a measure is viewed as the stationary law of a spin-glass, the sub-critical regime is often referred to as the {\em glassy} (sometimes {\em frozen}) phase, reflecting the existence of the system in one of a few dominant energy states at stationarity (see, e.g.~\cite{fyodorov2008freezing, castillo2001freezing}). This static picture was shown mathematically for various fields, including in the cases of the GMC and MMC mentioned above.

As in the case of the statics, universal dynamical features are expected for Glauber or Langevin-type dynamics of which these measures form the equilibrium law.
For in-equilibrium time scales, a {\em tunneling phenomenon} is expected, whereby the system transitions {\em instantaneously} from one meta-stable state to another. This was indeed shown in~\cite{cortines2018dynamical} in the case of the DGFF. For pre-equilibrium time scales, Meta-stable-type behavior, of the kind shown here, was predicted in~\cite{carpentier2001glass}, where the authors argue that logarithmic correlations are {\em the natural} type of correlations for the observation of such phenomena without any artificial scaling of the strength of the field or the time at which the dynamics is observed. To the best of our knowledge, Theorem~\ref{thm:1} is the first mathematical affirmation of this prediction. 

Thus, using Physics language, this work proves that for the particular choice of the DGFF as the underlying potential, Arcsine-type aging occurs for a range of pre-equilibrium time scales, throughout the glassy phase, placing spin-glasses with logarithmic potentials inside the universality class for this type of aging.

\subsection{Heuristics}
\label{s:1.4a}
Let us now give a broad explanation of the mechanism through which aging occurs in this model, elucidating on the scales which appear in Theorem~\ref{thm:1} and the role of the trapping landscape whose description was given in Subsection~\ref{s:1.2}. 

As shown in the proof, during $t_N(\gamma)$ time, the process $X$ makes 
\begin{equation}
\label{e:1.35}
\Theta(L_N) \quad ; \qquad L_N := N^2 (\log N)^{1-\gamma} \,,
\end{equation}
many jumps.
A study of the underlying jump chain, which performs a simple random walk on the torus, shows that for any fixed subset of $\mathsf{V}_{N}$ which is $r_N$-clustered in the sense of~\eqref{e:1.16a}, a fraction 
\begin{equation}
\label{e:1.37}
\Theta(\kappa_N^{-1})  
\quad ; \qquad \kappa_N := (\log N)^\gamma \,,
\end{equation}
of the clusters, essentially uniformly chosen, will be visited. Here $\kappa_N$ is as in~\eqref{e:1.12}.
 Moreover, if an $r_N$-cluster is visited, then each of the vertices in the cluster will be visited $\Theta(\log N)$ times. 

At the same time, features of the limiting statistics of extreme values, per ~\eqref{e:1.10i}, carry over qualitatively also to vertices at height $m_N - \Theta(\log \log N)$ (in fact, essentially all the way to $m_N - o(\sqrt{\log N})$. In particular, such values are still clustered, but only after removing a negligible fraction thereof and with the diameter of clusters $O(r_N)$ not $O(1)$ as before. Moreover, the number of clusters whose local maximum is (at least) $m_N-u$ is still 
\begin{equation}
\label{e:1.36a}
\Theta(\rme^{\alpha u})\,,
\end{equation}
and the law of the relative heights at vertices around them still follows the same cluster law $\nu$ as before.

This implies that the vertices which carry most of the time accumulated by the walk are those lying in clusters whose local maximum under $h$ is at
\begin{equation}
\label{e:1.38}
m_N - \frac{\gamma}{\alpha} \log \log N + \Theta(1) \,.
\end{equation}
Indeed, in view of~\eqref{e:1.36a}, the number of such clusters is 
\begin{equation}
\label{e:1.41}
\Theta(\kappa_N)
\end{equation}
and so, by~\eqref{e:1.37}, $\Theta(1)$ many such clusters will be visited. Vertices belonging to clusters with a higher maximum of the field, while potentially capable of holding the walk for a longer time, are entirely missed by the walk during the~\eqref{e:1.35} many steps it takes. Vertices belonging to clusters whose maximum is smaller than~\eqref{e:1.38}, while abundant, contribute to the accumulated time a total quantity which is negligible compared to~\eqref{e:1.11}.

Multiplying~\eqref{e:1.38} by $\beta$ and exponentiating, one gets precisely the value $\wh{t}_N$ from~\eqref{e:1.2}. Multiplying further by $\Theta(\log N)$ for the number of returns to each vertex in a cluster and observing that the sum of depths of all vertices in a cluster is finite (as shown by Theorem~\ref{thm:sumability_clusters}), we see that the total running time of the walk is indeed of the order $t_N(\gamma)$.

In relation to the classification of traps in Subsection~\ref{s:1.2.1}, we now also see that the local maxima at height~\eqref{e:1.38} correspond to trap-bottoms of normal depth. Higher local maxima and vertices of smaller height correspond to the deep trap-bottoms and (roughly) to the shallow trap-vertices, respectively. The negligible subset of vertices at level~\eqref{e:1.38} or higher which do not form clusters corresponds to the non-isolated trap-vertices. 

Now, while the relative-height configurations around the normal local maxima still follow $\nu$ (as shown by Theorem~\ref{thm:regular traps2}), the number~\eqref{e:1.41} of such maxima increases with $N$, and so some of the cluster configurations will exhibit atypical behavior. Thus, while around most of the normal local maxima, the set of vertices which are at height~\eqref{e:1.38} has a finite diameter, for a vanishing fraction of maxima this set is unusually wide. The two kinds of normal local maxima correspond to the regular and wide trap bottoms in the classification of Subsection~\ref{s:1.2.1}. Fortunately, it is the regular maxima which dominate, so that localization occurs on the $\Theta(1)$ scale.

We remark that while the number of $r_N$-local maxima at height $m_N - u$ grows exponentially as in~\eqref{e:1.36a}, the full level set at that height (without  restricting to local maxima), has size
\begin{equation*}
\Theta(u \rme^{\alpha u})\,.
\end{equation*}
The additional linear prefactor comes from the contribution to the level set of the clusters whose local maximum is at height $m_N -v$ for $v < u$. Nevertheless, due to the geometry of the range of the walk, it is not the total size of the target subset that matters, but rather the number of $r_N$-clusters it contains. Thus, most of the vertices in the $m_N-u$ level set, which come from vertices around ``deep'' local maxima, will end up being completely avoided by the walk. For more information, see~\cite{CHL1, CHL2}, where a sharp description of the extremal landscape (albeit for heights $m_N-\Theta(1)$) is derived in the analogous case of the tree.

\subsection{Open questions}
We conclude our discussion with several open questions which call for additional work. 
\subsubsection{Symmetric dynamics}
First, a more natural Glauber-type dynamics for the spin-glass law~\eqref{e:1.29} is given by 
\begin{equation}
\label{e:1.2222}
 \P \bigl( X_{t+\dif t} = y  \mid X_t = x \,,\, h \bigr) = 
\begin{cases}
\rme^{\beta (a h_y - (1-a) h_x)}  \dif t \quad & x \sim y\,, \\
0, & \text{otherwise}
\end{cases}\,,
\end{equation}
for $a \in [0,1]$. The choice $a=\frac12$ is perhaps the most natural from a physical point of view: the walk transitions from one state to another at a rate determined by the difference between the corresponding energy levels. 

The present work treats the most tractable case $a=0$. When $a \neq 0$, the underlying jump chain is no longer a simple random walk, and the mean holding time is correlated with the jump distribution. Nevertheless, using a method of Durrett and Resnick~\cite{durrett1978functional} for proving convergence of sums of dependent random variables to subordinators, Gayrard et al.~\cite{gayrard2012convergence, gayrard2013convergence} treat the case $a \neq 0$ for the BTM. This method could be applicable here as well.

\subsubsection{Shorter time scales}
The time scales $t_N(\gamma)$ for $\gamma \in (0,1)$ are, in fact, the longest possible pre-equilibrium time scales. As shown in~\cite{cortines2018dynamical}, when $\gamma = 0$ we are already in the equilibrium regime, where aging cannot occur. An interesting question therefore concerns the behavior of the system at much smaller time scales. In the case of the REM, for example, aging is exhibited over a much larger spectrum of times, with the parameter of the Generalized Arcsine Law varying according to the scale chosen. The difficulty here is that at time scales of order $o(N^2)$ the walk does not get the chance to explore the full torus, and thus the effective potential for the problem is the DGFF restricted to the range of the walk. This is a non-trivial object, which is not at all understood.

\subsubsection{Dynamics in infinite volume}
A more natural setting for the dynamics is when the underlying domain is the whole plane $\bbZ^2$. Although the DGFF cannot be defined on all $\bbZ^2$, one can take as the underlying potential the pinned DGFF on $\bbZ^2 \setminus \{0\}$, or alternatively, treat the case $a=1/2$ in~\eqref{e:1.2222} and work with the gradient field, which is well defined in infinite volume. This would be more in line with the model studied in~\cite{carpentier2001glass}, which is a more natural version of the problem. Indeed, as mentioned in that work, in this setting there is no artificial tuning between the size of the system and the time scales considered. The region explored by the walk up to a given time, and which therefore traps it, would be determined intrinsically by the walk, much as in the case of the BTM, but in the more physical context.

\subsection{Paper outline and general notation}
The remainder of the paper is organized as follows. In Section~\ref{s:2} we give a high-level proof of Theorem~\ref{thm:1}, based on the description of the trapping landscape of Subsection~\ref{s:1.2}. This proof relies on several key steps, which are first used and then proved in the rest of that section. Section~\ref{sec:3} contains the proofs of all the trapping-landscape results of Subsection~\ref{s:1.2}. These proofs rely on pointwise estimates for the near-extreme values of the DGFF, which are proved in Section~\ref{s:4}. Finally, the appendix contains auxiliary estimates for a simple random walk on the line and on the plane.

We use $C, C', c'$, etc. to denote constants, which may vary from line to line. Constants labeled with numerical subscripts, e.g.\ $C_0$, remain fixed.
 We say that the function $f(N,r,\delta)$ converges to $L$ as $N \to \infty$, then $r \to \infty$, and finally $\delta \downarrow 0$, if
 \[  \lim_{\delta \downarrow 0} \limsup_{r \to \infty} \limsup_{N \to \infty} \ \mathrm{dist}( f(N,r,\delta) , L ) = 0. \]
Here $\mathrm{dist}(\cdot,\cdot)$ denotes a suitable
metric on the space in which $f$,$L$ take values. Note that we do not require the existence of the corresponding intermediate limits.

\section{Aging} 
\label{s:2}
We start by proving the main result, assuming the trapping landscape description of Subsection~\ref{s:1.2}. The argument will use several key propositions which we state here and first use to prove the theorem, with the remainder of this section devoted to their proof. We continue to assume implicitly that $\beta > 0$ and $\gamma \in (0,1)$, and often use the abbreviation $t_N \equiv t_N(\gamma)$. As in Subsection~\ref{s:1.2} all results here hold uniformly in $\gamma$ on compact subsets of $(0,1)$.

\subsection{High-level proof}
\label{sec:high-level-proof}
In what follows, let
$Y=(Y_n)$ denote the discrete-time jump chain of $(X_t)$, i.e., the simple random walk on the torus $\mathsf{T}_N$. We write $\P_x$ for the law of the random walk started at $x$, and $\P$ for $\P_{\mathbf{0}}$. 
Given $r, \delta > 0$, let  $0= \sigma_0 < \sigma_1 < \sigma_2 < \ldots$ be stopping times defined inductively as follows
\begin{equation*}
\sigma_{i} \equiv \sigma_{i}(\delta, r) \coloneq  \min \bigl\{ j > \sigma_{i-1} :\: Y_j \in \mathsf{NT}(\delta, r) \setminus \mathsf{B}_{2r}(Y_{\sigma_{i-1}}) \bigr\} \quad , \qquad  i \ge 1.
\end{equation*} 
We shall always assume that $r \ll N/r_N$. Thus by \eqref{e:1.16a}, $(\sigma_n)_{n\ge 1}$ represents the sequence of successive visit times to new (as in different than the current one) $(\delta, r)$-normal traps, and
\begin{equation*}
\Delta \sigma_i\coloneq  \sigma_{i} - \sigma_{i-1} \,,
\end{equation*}
is the number of steps taken between these visits. The identities of visited traps are recorded using their local maxima, namely
\begin{equation}\label{eq:def-U_i}
	U_i = U_i(\delta,r)  \coloneq  \arg\!\max \{\tau_x ;\, x \in \mathsf{B}_{r}(Y_{\sigma_i}) \} \in \mathsf{NT}(\delta, r)
	\quad , \qquad i \geq 1 \,.
\end{equation}

Next we introduce a key and common player in an argument for aging, namely the \emph{Clock Process} associated with $X$. This is the process
$\mathcal{S} =  ( \mathcal{S} (n): n \geq 0)$ defined via
\begin{equation}\label{equa:def:modif:clock} 
\mathcal{S}(n) \equiv  \mathcal{S}_{N,\delta,r} (n)
    \coloneq  \sum_{ i =0}^{n}   \sum_{   \sigma_{i} \le j < \sigma_{i+1}}  \mathbb{e}_{j} \tau_{Y_{j}}  \ \text{ for }\ n \ge 0  ,
\end{equation}
where $(\mathbb{e}_{i} : i \in \mathbb{N})$ is a family of   i.i.d. standard exponential random variables, which are independent of everything else.
Thus, $\mathcal{S}(n)$ represents the total running time of $X$ up to its $(n+1)$-st visit to a $(\delta,r)$-normal trap.
The process $(\mathcal{S}(n))_{n \ge 0}$ depends on $\delta$ and $r$, but for simplicity we suppress this dependence from the notation unless it is needed.

Henceforth we shall consider the jump chain $Y$ up to the fixed discrete time:
\begin{equation*}
L_{N,m} \coloneq  m N^2 (\log N)^{1-\gamma} \,,
\end{equation*}
where $\gamma \in (0,1)$ is as in the main theorem and $m \geq 0$ is to be determined later. The total number of normal traps visited during $L_{N,m}$ steps is
\begin{equation*}
\zeta_{N,m} \equiv \zeta_{N,m} (\delta, r) \coloneq  \max\{j \ge 0 :   \sigma_j < L_{N,m} \} \,.
\end{equation*}
Then $\mathcal{S}(\zeta_{N,m}-1)$ represents the running time accumulated by $X$ up to the moment when the jump chain $Y$ reaches the last normal trap visited before $L_{N,m}$ steps.

The quantities above are ``tuned'' so that in $L_{N,m}$ steps, for any $\delta, r > 0$, the number of normal traps visited $\zeta_{N,m}$ is $O(1)$ with high probability for any $m > 0$. At the same time, for any $\theta > 0$, the total time accumulated by the continuous walk $\cS(\zeta_{N,m}-1)$ can be made to exceed $\theta t_N$ with arbitrarily high probability, provided $m$ is chosen large enough:
\begin{prop}\label{lem:cond4} 
Fix any $\theta \in (1,\infty)$. For every $\varepsilon \in (0,1)$, there exists  
$m \in \mathbb{N}$, depending only on $\varepsilon$ and $\theta$, such that 
\[
\liminf_{\delta \downarrow 0} 
\liminf_{r \to \infty} \ \liminf_{N\to \infty} \bbP_\mathbf{0} \bigl(  \zeta_{N,m} \ge \theta \delta^{-\alpha/\beta} ;\mathcal{S}(\zeta_{N,m}-1) > \theta t_N \bigr) > 1 - \varepsilon .
\]  
\end{prop}

By definition, the ``jump'' $(\cS(n-1), \cS(n)]$ of $\cS$ encodes the period of running time starting from $X$'s visit to the $n$-th normal trap and ending just before its visit to the $(n+1)$-st one. We wish to claim that the occurrence of a jump over $[t_N,t_N(1+\theta)]$ for $\cS$ is a.w.h.p. equivalent to $X$ being at two vertices that are $O(1)$ apart at times $t_N$ and $(1+\theta)t_N$. The first step in this direction is to show that when $X$ is observed at a deterministic time $\Theta (t_N)$ it is w.h.p. found in a normal trap.
\begin{prop}\label{lem:only:moderate:traps:visited} Fix  $0<a<b<\infty$. 
Let $\{t_N'\}$ be a deterministic sequence such that $a t_N \le t'_N \le b t_N$ for all $N \geq 1$.
Then 
\begin{equation}\label{eq:lem:only:moderate:traps:visited}
	\lim_{\delta \downarrow 0} \,
\limsup_{r\to \infty} \, \limsup_{N\to \infty} \, \bbP_{\mathbf{0}} \bigl( X_{t_N'} \notin \mathsf{NT}(\delta,r)   \bigr)  = 0.
\end{equation}
\end{prop}

By construction, normal traps are $N/r_N$ apart. It thus follows from the proposition that being $O(1)$ apart at times $t_N$ and $t_N(1+\theta)$ is a.w.h.p. equivalent to being in the {\em same} normal trap at these two times, and that the latter is a.w.h.p. implied by a jump over $[t_N,t_N(1+\theta)]$ by the clock process. Indeed, during a jump, no new traps are visited. To prove the opposite, namely that being in the same trap at two different times must imply a jump over the corresponding interval, we must rule out the case of a revisit to an already visited trap. This is the subject of
\begin{prop}
\label{p:2.3}
For any fixed $\delta \in (0,1)$, $r \ge 1$, and $m \ge 0$, we have 
\begin{equation*}
\lim_{N \to \infty} \, \P \bigl(\exists \, 1 \le i < j \le \zeta_{N,m} :\:  U_i =  U_{j} \bigr) = 0 .  
\end{equation*}
\end{prop}

With the last three propositions at hand we can replace the target event $\|X_{t_N(1+\theta)}- X_{t_N}\| < r$ in the theorem with the event that the clock process has a jump over $[t_N, t_N(1+\theta)]$. A key result here is that under the appropriate limiting procedure $\cS$ tends to a stable subordinator, for which the probability of such a jump is well understood and expressed in terms of the Generalized Arcsine law. We thus state the final and most substantial ingredient in the proof of the theorem.

\begin{thm}[Convergence of the clock process]
\label{p:2.4}
Let $\mathscr{Y}$ be a standard
$\alpha/\beta$-stable subordinator.  For $\delta \in (0,1)$, $r \ge 1$ and $N \ge 1$, define the rescaled clock process  
\[  {\mathscr{Y}}_{N,\delta,r}(u):= \frac{1}{t_N} \mathcal{S}_{N,\delta,r}(  \lfloor  u \delta^{-\frac{\alpha}{\beta}}   \rfloor  ) \quad \text{ for all } \quad u \ge 0. \] 
Then there exists a constant $c_\beta \in (0,\infty)$ such that,   
as $N \to \infty$,
followed by $r \to \infty$
and then $\delta \downarrow 0$,
the sequence
$(c_{\beta}\mathscr{Y}_{N,\delta, r})$
converges weakly to $\mathscr{Y}$ in the
Skorokhod space $D[0,\infty)$ equipped with the $J_1$ topology.
\end{thm}
  
We can now give the proof.
\begin{proof}[Proof of Theorem~\ref{thm:1}]
We will equivalently show the convergence in~\eqref{equation:thm:1} with $t_N = t_N(\gamma)$ uniformly in $\gamma$ on compacts. We recall that such uniformity holds for all other statements in this subsection, and note that this will also be the case, implicitly, for all the limits in the proof to follow.

Fix $\theta>0$.
Define the good event $G^{N}_{\star}(m;\delta,r)$
as the intersection of the following:
\begin{align*}
	G^{N}_{1}( \delta,r) & \coloneq  \{ X_{  t_N} \in   \mathsf{NT}(\delta,r) \, , \  X_{(1+\theta) t_N} \in   \mathsf{NT}(\delta,r) \} , \\
	G^{N}_{2}(m;\delta,r) & \coloneq \{ \mathcal{S}(\zeta_{N,m}) \ge 2(1+\theta) t_N \} \cap  \{ \text{ for all } 1 \le i < j \le \zeta_{N,m}, U_i \neq U_{j} \} .
\end{align*} 
Combining Propositions~\ref{lem:cond4},  \ref{lem:only:moderate:traps:visited} and  \ref{p:2.3} we have
 \begin{equation}\label{eq:Good-event-star}
	 \lim_{m \to \infty} \limsup_{\delta \downarrow 0} \,
	\limsup_{r\to \infty} \, \limsup_{N\to \infty} \P\bigl( G^{N}_{\star}(m;\delta,r) ^{c}\bigr) = 0 .  
 \end{equation}
 
 \noindent\underline{\emph{Step 1. }}
We first demonstrate that the event
$\{|X_{t_N} - X_{t_N(1+\theta)}| \le 2 r \}$
is asymptotically equivalent to
\begin{equation*}
	 \mathtt{ID}_{N }(\delta,r)  \coloneq  \{ \exists x \in \mathsf{NT}(\delta,r) \text{ s.t. }   d(X_{t_N} ,x ) \le r, d( X_{t_N(1+\theta)},   x ) \le r \} .
\end{equation*} 
Clearly we have $  \mathtt{ID}_{N }(\delta,r) \subset \{   |X_{t_N} - X_{t_N(1+\theta)}| \le 2 r \} $. Conversely, because the distance between any
two distinct points in $\mathsf{NT}(\delta,r)$
is at least $N/r_N \gg  r$ by \eqref{e:1.16a},
the intersection of $\{ |X_{t_N} - X_{t_N(1+\theta)}| \le 2 r \}$
and $G^{N}_{1}(\delta,r)$ implies $\mathtt{ID}_{N }(\delta,r)$.
Thus, we obtain:
\begin{equation*}
	\{   |X_{t_N} - X_{t_N(1+\theta)}| \le 2 r \} \setminus   \mathtt{ID}_{N }(\delta,r) \subset  G^{N}_{1}( \delta,r)^{c} .
\end{equation*}  
The probability of the latter goes to zero in the desired limits by~\eqref{eq:Good-event-star}.

\noindent\underline{\emph{Step 2. }}
Next, with $\mathrm{Im}(f)$ denoting the image of the function $f$, we wish to establish that 
\begin{equation}\label{eq:Id-VS-Clock-image}
	 \lim_{\delta\downarrow 0} \, \limsup_{r\to \infty} \, \limsup_{N\to \infty}  |
\P ( \mathtt{ID}_{N }(\delta,r)  ) -   \P \bigl( \mathrm{Im}( t_N^{-1} \mathcal{S}) \cap  [1,1+\theta] = \emptyset  \bigr) | = 0 \,.
\end{equation}  
To this end, we rewrite the target event $\{\mathrm{Im}( t_N^{-1} \mathcal{S}) \cap  [1,1+\theta] = \emptyset\}$ as  
\[ 
  \bigl\{  \exists k \ge 0\text{ s.t. }   \mathcal{S}(k) < t_N < (1+\theta) t_N < \mathcal{S}(k+1)   \bigr\} \cup \{ \mathcal{S}(0) > (1+\theta) t_N \} .
\]
By the definition of the clock process   $\mathcal{S}$ in \eqref{equa:def:modif:clock}, the first event means that, within the time window $[t_N,(1+\theta)t_N]$, the walk $X_t$ does not find a different normal trap from the last one it visited. Since $G^{N}_{1}( \delta,r) \cap \{\mathcal{S}(0) > (1+\theta) t_N\}= \emptyset $, we thus get 
\begin{equation}
		\label{eq:Clock-image-Id}
	  \{ \mathrm{Im}( t_N^{-1} \mathcal{S}) \cap  [1,1+\theta] = \emptyset \} \cap G^{N}_{1}( \delta,r) \subset  \mathtt{ID}_{N }(\delta,r). 
\end{equation}
Conversely, the occurrence of
$\mathtt{ID}_{N }(\delta,r) \cap G^{N}_{2}(m;\delta,r)$
implies that $X_{t_N}$ and $X_{(1+\theta) t_N}$
occupy the same normal trap
(say, centered at $U_i$).
Moreover, until time $2(1+\theta) t_N$,
the walk does not come back to $\mathsf{B}_{r}(U_{i})$ 
after discovering a different normal trap.
Thus:
\begin{equation}
	  \label{eq:Id-Clock-image}
	 \mathtt{ID}_{N }(\delta,r) \cap   G^{N}_{2}(m;\delta,r) \subset  \{ \mathrm{Im}( t_N^{-1} \mathcal{S}) \cap  [1,1+\theta] = \emptyset \}.
\end{equation} 
Combining \eqref{eq:Clock-image-Id} and \eqref{eq:Id-Clock-image} with \eqref{eq:Good-event-star}, the claim \eqref{eq:Id-VS-Clock-image} follows immediately.

\smallskip
\noindent\underline{\emph{Step 3. }}
Using the first two steps, 
it thus remains to show
\begin{equation}
\label{e:2.13}
 \lim_{\delta\downarrow 0} \, \limsup_{r\to \infty} \, \limsup_{N\to \infty}  \Big|
\P \bigl( \mathrm{Im}( t_N^{-1} \mathcal{S}) \cap  [1,1+\theta] = \emptyset \bigr)- \mathrm{Asl}_{\alpha/\beta}\bigl(1/(1+\theta)\bigr) \Big| = 0 .
\end{equation}
Observe that  the images of ${t_{N}^{-1}}\cS$ and $\mathscr{Y}_{N,\delta,r}$ are the same. It thus follows from  Theorem~\ref{p:2.4}   that the first probability in \eqref{e:2.13} tends in the stated limits to
\begin{equation*}
	  \P \bigl( \mathrm{Im}(\mathscr{Y} ) \cap  [c_{\beta},c_{\beta}(1+\theta)] = \emptyset  \bigr) \,,
\end{equation*}
where $\mathscr{Y}$ is the standard stable subordinator with index $\alpha/\beta$. Indeed, the discontinuity set of 
the functional 
$f \mapsto \ind{ \mathrm{Im}(f) 
\cap [a,b] = \emptyset}$ on 
$D[0,\infty)$ is contained 
in the set of trajectories 
that hit or asymptotically 
approach the boundary $\{a,b\}$. 
Since this occurs with probability 
zero for the subordinator $\mathscr{Y}$ 
(e.g., from the 
Arcsine Law; see below), 
the claim follows from the 
generalized continuous mapping theorem.

The classical Arcsine Law for L\'evy processes (see e.g.~\cite[Theorem 6 in Section III.3]{Ber96})  yields that the probability 
that a stable subordinator of index 
$\lambda \in (0,1)$ does not 
intersect the interval $[a,b]$ is 
given by
\begin{equation*}
    \mathrm{Asl}_{\lambda}(a/b).
\end{equation*}
This identifies the limiting probability 
with the right-hand side of~\eqref{e:2.13}, and 
completes the proof.
\end{proof}

\subsection{Key dynamical lemmas}
\label{s:2.2}
Our next task is to prove the high-level steps in the proof of the main theorem, namely Propositions~\ref{lem:cond4},  \ref{lem:only:moderate:traps:visited}, \ref{p:2.3} and,  most importantly, Theorem~\ref{p:2.4}. The proofs of all these statements rely on several key results concerning the dynamics of the process $X$ and its relation to the trapping landscape given by the DGFF.  We gather these essential properties below, beginning with a result showing that, with high probability, the walk $(Y_n)$ avoids deep, wide, or non-isolated traps during the first $L_{N,m}$ steps.

\begin{lem}\label{lem:nonreachable_traps} Let $H_{\mathsf{DT}(\delta)}$, $H_{\mathsf{WT}(\delta,r)}$ and $H_{\mathsf{NIT}(\delta)}$ denote the hitting times of the sets $ \mathsf{DT}(\delta)$, $\mathsf{WT}(\delta,r)$ and $\mathsf{NIT}(\delta)$ by the discrete-time random walk $(Y_n)$, respectively. Then  for all $m \in \mathbb{N}$, we have 
\begin{equation}\label{equa:DT:lem:nonreachable_traps}
\lim_{\delta \to 0} \ \limsup_{N\to \infty} \bbP_\mathbf{0} \left( H_{{\mathsf{DT}(\delta)}} \leq L_{N,m}  \right) =0 .
\end{equation}
Moreover, for wide and non-isolated traps, we have for any $m \in \mathbb{N}$ and $\delta > 0$, 
\begin{equation*}
 \lim_{r\to \infty} \ \limsup_{N\to \infty} \bbP_\mathbf{0} \left( H_{{\mathsf{WT}(\delta,r)}} \leq L_{N,m}  \right) =0 \, , 
 \, \text{ and } \,
 \lim_{N\to \infty} \bbP_\mathbf{0} \left( H_{\mathsf{NIT}(\delta)} \leq L_{N,m}  \right) =0 .
\end{equation*}
\end{lem}

The remaining   non-regular traps are shallow traps. While these are certainly visited within the designated number of steps, their nature ensures that the total time spent in them  by $X$ is negligible on the scale of $t_N$. Recall \eqref{equa:def:modif:clock}.

\begin{lem}\label{lem:sum_shallow_traps} For all $\varepsilon > 0$ and $m \in \mathbb{N}$, we have 
\[ \lim_{\delta \downarrow 0} \,
 \limsup_{N\to \infty} \bbP_{\mathbf{0}} \Bigl( \sum_{ j =0}^{ L_{N,m} } \mathbb{e}_{j} \tau_{Y_j} \ind{ Y_j \in \mathsf{ST}(\delta) } > \varepsilon t_N \Bigr) 
= 0.
\]  
\end{lem}

The last two lemmas imply that we can restrict our attention to $\delta$-normal traps. We shall show that, after properly scaling the number of steps by the quantity
\begin{equation*}
	\chi_N = \chi_N(\delta )
	 \coloneq \frac{\pi}{2} \frac{| \mathsf{NT}^*(\delta)|}  {  |\mathsf{T}_N| \log N } \,,
\end{equation*}   
such  $\delta$-normal traps are encountered by $Y$ in the manner of a Poisson process. Moreover, letting
\begin{equation*}
\bar{\tau}_{x}(y) \coloneq \frac{\tau_{x+y}}{\tau_x} \quad ; \qquad 
x \in \mathsf{NT}^*(\delta) \,,\,\,
y \in \mathbb{Z}^2\,,
\end{equation*}
represent the ``depth'' of the vertices in the normal-trap relative to its local maximum, the 
joint law of $(\wh{\tau}_x, \bar{\tau}_x) \in \bbR_+ \times \bbR_+^{\bbZ_2}$ as encountered by $Y$ will tend to a limit in law, which is i.i.d. across different traps.

To describe the limiting law,  let $w^*_{\delta}$ and $\bar{w} $   be independent random variables taking values in $[\delta,\delta^{-1}]$ and $(0,1]^{\mathbb{Z}^2}$, respectively, with laws given by
\begin{equation}
\label{e:2.21}
\P( w^*_{\delta}  > u ) = \frac{u^{-\alpha/\beta} - \delta^{\alpha/\beta}}{\delta^{-\alpha/\beta}-\delta^{\alpha/\beta}}   \,, \ \text{ and } \quad 
\P (\bar{w}  \in A) =  \sigma_{\beta}  (   A )\,,
\end{equation}
where $u \in [\delta, \delta^{-1}]$ and $A \subset (0,1]^{\mathbb{Z}^2}$ is a measurable set.
Above, $\sigma_{\beta}$ is the probability measure from Theorem~\ref{thm:regular traps2}. We shall thus show:
\begin{lem}\label{lem:approximation:sigma} Fix $\delta \in (0,1/2]$ and $r \ge 1$.
The following assertions hold: 
\begin{enumerate}[(i)]
    \item As $N \to \infty$, the sequence $(\chi_N \Delta \sigma_i)_{i \ge 1}$ 
    converges in distribution to a sequence of i.i.d. standard
	exponential random variables.
    \item As $N \to \infty$,  the sequence 
    $(\wh{\tau}_{U_{i}},   \bar{\tau}_{U_i} |_{ \mathsf{B}_r} )_{i \ge 1}$ 
    converges in distribution to a sequence of i.i.d. copies of $(w^*_\delta, \bar{w}|_{\mathsf{B}_r})$. 
\item  For   any $m \in \mathbb{N}$,  
$    \P   (\exists 1 \le i < j \le \zeta_{N,m} :\:  U_i =  U_{j}  )    $ converges to zero as $N \to \infty$
\end{enumerate}   
\end{lem}

Lastly, we claim that not only does the joint law of the depths at vertices of encountered normal trap converge weakly to a limit, but also the entire time spent by the walk in those traps. This amounts to showing that the joint number of visits to vertices in such traps obeys a limit in law, after scaling, and that these local time configurations at different traps are asymptotically independent. To this end, for $i \geq 0$ and $r$,$\delta$ as before, we let
\begin{equation*} 
\Delta \mathcal{S}^{\star}(i) \equiv
\Delta \mathcal{S}_{N,\delta,r}^{\star}(i) 
 \coloneq  \sum_{ \sigma_{i} \le j < \sigma_{i+1}}  \mathbb{e}_{j} \tau_{Y_j} \ind{Y_j \in \mathsf{NT}(\delta,r) } \,,
\end{equation*} 
denote the total time spent in $(\delta,r)$-normal traps between the $i$-th and $i+1$-st visit to such traps. For the limiting law, we set
\begin{equation*}
  \mathcal{E}_{\delta,r} \coloneq
 \mathbb{e} \times \Bigl( \frac{2}{\pi}w^*_{\delta} \,  \mathcal{W}_r \Bigr) 
  \,, \quad \text{and} \quad
 \mathcal{W}_{r} \coloneq \sum_{x \in \mathsf{B}_{r} } \ \bar{w}_x \,, \  \text{ for } r \in \mathbb{N} \cup \{ \infty \}\,,
\end{equation*}
with $w^*_\delta, \bar{w}$ as in~\eqref{e:2.21}, and $\mathbb{e}$ an independent standard exponential random variable.
We thus have the following lemma.
\begin{lem}\label{lem:convergence:time:intervals}  Fix $\delta \in (0,1)$ and $r \ge 1$. 
As $N \to \infty$,  the sequence 
\[ \Bigl(  \frac{1}{t_N} \Delta \mathcal{S}^{\star}_{N,\delta,r}  (j)    :j  \ge 1   \Bigr)   \]  converges in distribution to  a sequence $(\mathcal{E}^{(j)}_{\delta,r})_{j \ge 1}$ of i.i.d. copies of
$\mathcal{E}_{\delta,r}$.
\end{lem}

\subsection{Convergence of the clock process}
With the lemmas of the previous subsection at our disposal we can easily prove the convergence of the clock process to a standard stable subordinator. To this end, we first consider a version of this process where only the times spent in $(\delta,r)$-normal traps are accumulated,
\begin{equation*}
 \mathcal{S}^{\star} (n) \equiv 
 \mathcal{S}^{\star}_{N,\delta,r} (n) \coloneq \sum_{ i =0}^{ n }\Delta \mathcal{S}^{\star}(i)  = \sum_{ i =0}^{ n }  \sum_{ \sigma_{i} \le j < \sigma_{i+1}} \mathbb{e}_{j} \tau_{Y_j} \ind{ Y_j \in  \mathsf{NT}(\delta,r)} 
\quad , \qquad n \geq 0 \,.
\end{equation*} 
Trivially $ \mathcal{S}^{\star} (n) \leq  \mathcal{S} (n)$.  
The following lemma shows that the modified and original processes are uniformly $o(t_N)$ apart, in the first $\zeta_{N,m}$ visits to regular traps.
\begin{prop}\label{prop:approximation:clock:process} For all $\varepsilon > 0$ and $m \in \mathbb{N}$,    we have 
\[
\lim_{\delta \downarrow 0} \,
\limsup_{r \to \infty} \, \limsup_{N \to \infty} \bbP  \Bigl( \sup_{0\leq i \leq \zeta_{N,m}-1}  \{\mathcal{S}_{N,\delta,r}(i) - \mathcal{S}_{N,\delta,r}^{\star} (i)  \} > \varepsilon t_N \Bigr)  = 0.
\]  
\end{prop}

\begin{proof}[Proof of Proposition \ref{prop:approximation:clock:process}] By monotonicity, it suffices to verify the statement for the endpoint $\zeta_{N,m}$. Observe that whenever $\{L_{N,m} < H_{\mathsf{DT}(\delta)} \wedge H_{\mathsf{WT}(\delta,r)} \wedge H_{\mathsf{NIT}(\delta)} \} $ occurs,  
	the decomposition \eqref{eq:torus-decomp} implies 
\[ \mathcal{S}(\zeta_{N,m}-1) - \mathcal{S}^{\star} (\zeta_{N,m}-1)  \le \sum_{j=0}^{L_{N,m}} \mathbb{e}_{j} \tau_{Y_j} \ind{ Y_j \in \mathsf{ST}(\delta)}. \] 
Thus the probability $\bbP_{\mathbf{0}} (   \mathcal{S}(\zeta_{N,m}-1) - \mathcal{S}^{\star} (\zeta_{N,m}-1)   > \varepsilon t_N  ) $ is bounded above by
\begin{equation*}
 \sum_{i = \mathsf{DT}(\delta), \mathsf{WT}(\delta,r)
 , \mathsf{NIT}(\delta)}\bbP_\mathbf{0}  ( H_{i} \leq L_{N,m}  ) +  \bbP_{\mathbf{0}} \Bigl( \sum_{ j =0}^{ L_{N,m} } \mathbb{e}_{j} \tau_{Y_j} \ind{ Y_j \in \mathsf{ST}(\delta) } > \varepsilon t_N \Bigr) .
\end{equation*}
The claimed result then follows directly from Lemmas \ref{lem:sum_shallow_traps} and \ref{lem:nonreachable_traps} by
first sending $N \to \infty$, then $r \to \infty$, and finally $\delta \to 0$. This completes  the proof.
\end{proof}

Next, we show that the modified clock process converges to the desired subordinator after a proper space-time rescaling.

\begin{prop}
\label{lem:convergence:stable}
For $\delta \in (0,1)$ and $r \in \mathbb{N}$, and  $N \geq 1$, we define the
scaled process 
\begin{equation*}
	\mathscr{Y}^{\star}_{N,\delta,r}(u) \coloneq \frac{1}{t_N} \mathcal{S}_{N,\delta,r}^{\star} ( \lfloor u \delta^{-\frac{\alpha}{\beta}} \rfloor ) \ \text{ for } \ u \ge 0.
\end{equation*} 
Then, there exists a
constant $c_{\beta}>0$
such that, as $N \to \infty$,
followed by $r \to \infty$
and then $\delta \downarrow 0$,
the sequence
$(c_{\beta}\mathscr{Y}^{\star}_{N,\delta, r})$
converges weakly to a standard
$\alpha/\beta$-stable subordinator
$\mathscr{Y}$ in the
Skorokhod space $D[0,\infty)$. 
\end{prop}

\begin{proof} 
 For fixed $\delta$ and $r$, since $r\ll N/r_N$, Theorem~\ref{thm:distance_origin},
 implies that
$  \P (\Delta\mathcal{S}^{\star}(0)=0 ) \to 1$ as $N\to\infty$.  
Moreover, on every compact time interval and for fixed $\delta$, the
processes $ \mathscr{Y}^{\star}_{N,\delta,r}$ have only finitely many jumps, occurring at the same
deterministic times. Hence  the asymptotic independence and identical distribution of the increments $t_N^{-1} \Delta \mathcal{S}^{\star} (j) : j\ge 1 $, as stated in Lemma~\ref{lem:convergence:time:intervals}, together with the continuous mapping theorem imply that,
  as $N \to \infty$,  $ \mathscr{Y}^{\star}_{N,\delta,r}$ converges weakly in the Skorokhod space $D[0,\infty)$  to the process
\[
    \mathscr{Y}^{\star}_{\delta,r}(u) \coloneq
    \sum_{1 \le i \leq u \delta^{- {\alpha}/{\beta}}} \mathcal{E}^{i}_{\delta,r}  \, , \ u \ge 0\,,
\]
where $\mathcal{E}^{i}_{\delta,r}$ are i.i.d. copies of $\mathcal{E}_{\delta,r}= 
  \frac{2}{\pi} w^*_{\delta}\, \mathcal{W}_r \,  \mathbb{e} $, defined in Lemma~\ref{lem:convergence:time:intervals}.

  Since   $\mathcal{W}_r \coloneq \sum_{x \in \mathsf{B}_r} \bar{w}_x$   converges to $ \mathcal{W}_{\infty} \coloneq \sum_{x \in \mathbb{Z}^2} \bar{w}_x$ almost surely and in $L^1$ as $r \to \infty$ by Theorem \ref{thm:sumability_clusters},   another
application of the continuous mapping theorem   shows that, the process $ \mathscr{Y}^{\star}_{\delta,r}$ converges weakly in the Skorokhod space $D[0,\infty)$ to  
\[
  \mathscr{Y}^{\star}_{\delta}(u) \coloneq
    \sum_{1 \le i \leq u \delta^{- {\alpha}/{\beta}}} \mathcal{E}^{i}_{\delta}  \, , \ u \ge 0\,,
\]
where $\mathcal{E}^{i}_{\delta }$ are i.i.d. copies of $\mathcal{E}_{\delta }= 
  \frac{2}{\pi} w^*_{\delta} \, \mathcal{W}_{\infty} \,  \mathbb{e} $. 

  Therefore, the desired result follows immediately once we  show that
$(c_{\beta}\mathscr{Y}^{\star}_{\delta})$
converges weakly to $\mathscr{Y}$ as $\delta\downarrow0$, where 
\[
c_{\beta}
\coloneq
 \Bigl( 
 \Gamma\bigl(1- \frac{\alpha}{\beta} \bigr)
 \E \Bigl[ 
    \bigl( \frac{2}{\pi}\mathcal{W}_{\infty}  \bigr)^{\alpha/\beta}
  \Bigr]
 \E  \bigl[   \mathbb{e}^{\alpha/\beta} \bigr]
 \Bigr) ^{-\beta/\alpha}.
\]
The constant $c_{\beta}$ is finite and positive because
$\mathcal{W}_{\infty}\in L^1$ by
Theorem~\ref{thm:sumability_clusters} and $\alpha/\beta\in(0,1)$.

We first establish the finite dimensional convergence.
Recall the distribution of $w_{\delta}^{*}$ given by~\eqref{e:2.21}.
Conditioning on $\mathcal{W}_{\infty}$ and $\mathbb e$, and applying
the dominated convergence theorem, we obtain that,
\begin{equation}
\label{eq:E-delta-asymp}
\lim_{\delta \downarrow 0} \, \delta^{-\alpha/\beta}
 \P(c_{\beta}\mathcal{E}_{\delta}>x)
 = 
 \frac{x^{-\alpha/\beta}}
      {\Gamma(1-\alpha/\beta)} \  \text{ for every } x >0 .
\end{equation}  
Furthermore, we can find a constant $C>0$ depending only on $\beta$ such that 
\begin{equation}
	\label{eq:E-delta-tail}  
 \delta^{-\alpha/\beta}
 \P(c_{\beta}\mathcal{E}_{\delta}>x)
 \le Cx^{-\alpha/\beta} \quad \text{and} \quad   \E \bigl[ 
   c_{\beta}\mathcal{E}_{\delta}
   \ind{c_{\beta}\mathcal{E}_{\delta}\le x} \bigr] 
 \le
 C\delta^{\alpha/\beta}x^{1-\alpha/\beta}   
\end{equation}
 for all $ 0<\delta\le1/2$, $x > 0$. 
Indeed, conditionally on
$\frac{2}{\pi}\mathcal{W}_{\infty}\mathbb e=z$, the first expression
is bounded above, up to a multiplicative constant, by  $(z/x)^{\alpha/\beta}$, which is
integrable.    
For every $\lambda>0$, Fubini's theorem and
\eqref{eq:E-delta-asymp} yield 
 \[\lim_{\delta \downarrow 0} \delta^{-\alpha/\beta}
 \bigl( 
  1-\E \bigl[ 
     \rme^{-\lambda c_{\beta}\mathcal{E}_{\delta}}
  \bigr]
  \bigr)
 = \lim_{\delta \downarrow 0}  \int_{0}^{\infty} \lambda \,e^{-\lambda x} \, \delta^{-\frac{\alpha}{\beta}} \, \P( c_{\beta} \mathcal{E}_{\delta} > x ) \dif x =\lambda^{\alpha/\beta} . \]
 Thus we obtain 
\begin{equation} 
 \lim_{\delta \downarrow 0} \log\E \bigl[ 
   \rme^{-\lambda
   c_{\beta}\mathscr{Y}^{\star}_{\delta}(u)}
 \bigr]  
 = \lim_{\delta \downarrow 0}
 \lfloor u\delta^{-\alpha/\beta}\rfloor
 \log\E \bigl[  
   \rme^{-\lambda c_{\beta}\mathcal{E}_{\delta}}
\bigr]  
=
 -u\lambda^{\alpha/\beta}. \label{eq:stable-laplace-limit}
\end{equation} 
Since the sums defining
$\mathscr{Y}^{\star}_{\delta}$ have independent increments, the same
calculation for disjoint blocks proves convergence of  
finite-dimensional distributions to those of the  
$\alpha/\beta$-stable subordinator $\mathscr{Y}$.

It remains to prove tightness. We apply  Aldous' tightness
criterion. Fix $T,\varepsilon>0$, let
$\vartheta_{\delta}\le T$ be a stopping time for the natural
filtration of $\mathscr{Y}^{\star}_{\delta}$, and let
$0\le s\le\eta$. The interval
$(\vartheta_{\delta},\vartheta_{\delta}+s]$ contains at most
$\lceil\eta\delta^{-\alpha/\beta}\rceil+1$ new summands. Independence and 
\eqref{eq:E-delta-tail}
give 
\begin{align*}
 &\P \bigl( 
   c_{\beta}\mathscr{Y}^{\star}_{\delta}
      (\vartheta_{\delta}+s)
   -
   c_{\beta}\mathscr{Y}^{\star}_{\delta}
      (\vartheta_{\delta})
   >\varepsilon  \bigr)\\
 &\quad \le
 \bigl(\lceil\eta\delta^{-\alpha/\beta}\rceil+1\bigr)
 \Bigl\{
   \P(c_{\beta}\mathcal{E}_{\delta}>\varepsilon/2)
 +
   \frac{1}{\varepsilon}
   \E \bigl[  
     c_{\beta}\mathcal{E}_{\delta}
     \ind{c_{\beta}\mathcal{E}_{\delta}\le\varepsilon/2} \bigr]
 \Bigr\}  \lesssim _{\varepsilon} 
 \bigl(\eta+2\delta^{\alpha/\beta}\bigr).
\end{align*} 
Thus Aldous' stopping-time condition (\cite[Equation (A)]{Aldous1978}) holds.  
Moreover, since the paths
are non-decreasing,  the variables $\{\sup_{0\le u\le T}\mathscr{Y}^{\star}_{\delta}(u)=\mathscr{Y}^{\star}_{\delta}(T)\}_{\delta>0}$   are
tight by~\eqref{eq:stable-laplace-limit}. Aldous' 
criterion~\cite[Theorem 1]{Aldous1978}   gives  tightness in
$(D[0,T],J_1)$. Since $T>0$ is arbitrary, combining tightness with the
finite-dimensional convergence proves
$c_{\beta}\mathscr{Y}^{\star}_{\delta}
 \Rightarrow\mathscr{Y}$
 in $(D[0,\infty),J_1)$. 
This completes the proof.
\end{proof}

The proof of Theorem~\ref{p:2.4} is now straightforward.
\begin{proof}[Proof of Theorem~\ref{p:2.4}] Since the sample paths of
$\mathscr{Y}$
are continuous in probability,
it suffices to verify the
weak convergence of
$\mathscr{Y}_{N,\delta,r}|_{[0,T]}$
to
$\mathscr{Y}|_{[0,T]}$
in the Skorokhod space
$D[0,T]$ for
every $T>0$.

By Propositions~\ref{prop:approximation:clock:process} and~\ref{lem:cond4}, the $L^\infty$-distance between
$\mathscr{Y}_{N,\delta,r}|_{[0,T]}$ and $\mathscr{Y}^{\star}_{N,\delta,r}|_{[0,T]}$ converges to zero in probability. Indeed, we have 
\begin{multline}
\P \Bigl(\sup_{0\le u \le T} |\mathscr{Y}_{N,\delta,r}(u)-\mathscr{Y}^{\star}_{N,\delta,r}(u)|> \epsilon \Bigr) \le \P \bigl( \zeta_{N,m} \le 2T \delta^{-\alpha/\beta}  \bigr)
\\
 + \bbP  \Bigl( \sup_{0\leq i \leq \zeta_{N,m}-1}  \{\mathcal{S}_{N,\delta,r}(i) - \mathcal{S}_{N,\delta,r}^{\star} (i)  \} > \varepsilon t_N \Bigr) .
\end{multline}
Letting first $N \to \infty$, then $r \to \infty$, then $\delta \downarrow 0$, and finally $m \to \infty$, we obtain the desired approximation.

Since
$\mathscr{Y}^{\star}_{N,\delta,r}|_{[0,T]}$
converges weakly to
$\mathscr{Y}|_{[0,T]}$ in  
$D[0,T]$
by Proposition \ref{lem:convergence:stable},
the weak convergence of
$\mathscr{Y}_{N,\delta,r}|_{[0,T]}$
to
$\mathscr{Y}|_{[0,T]}$ 
immediately follows. This completes the proof. 
\end{proof}

\subsection{Proofs of key dynamical lemmas} 
We now turn to the proofs 
of the dynamical lemmas 
from Subsection~\ref{s:2.2}. 
The remaining Propositions~\ref{lem:cond4}--\ref{p:2.3}   
will be established in the next subsection.
\subsubsection{Simple random walk input}
\label{s:2.4.1}
We shall need the following input from the theory of the simple random walk on the torus. 
These statements are similar  to those in \cite[Section 4]{arous2006aging}; however, the formulation given here is slightly stronger---for instance, Lemma \ref{prop:hitting-time} covers an enlarged range of the parameter $\gamma$. While our proofs are different  and more streamlined, these statements 
are auxiliary to the main argument. 
We therefore relegate their proofs 
to Appendix~\ref{apx:2dsrw}.

The main objects of study are hitting times.   For  a subset $A$ of $\mathsf{T}_{N}$, we denote by $H(A)$ the first hitting time of $A$ by the jump chain $Y$:
\begin{equation*}
	H(A)\coloneq  \inf\{ n \ge 0: Y_n \in A \} ,
\end{equation*}
and abbreviate $H(x) \equiv H(\{x\})$. The first lemma states, among other things, that hitting times of far away vertices are asymptotically exponential after scaling.
\begin{lem} \label{lem:est_laplace}  
The following assertions hold:
\begin{enumerate}[(i)]
		\item   Let $\lambda_N$ be any sequence with $1\ll \lambda_N \ll \frac{\log N }{\log r_N}$. Then
\begin{equation} \label{equa:lem:est_laplace}
\lim_{N\to \infty} \sup_{0 \le \rho \le r_N^2} \ \sup_{x ,y \in \mathsf{T}_N, d(x,y) \ge N/r^{2}_N} \   \Big|\lambda_N \bbE_{x} \left[ \rme^{ - \lambda_N H( \mathsf{B}_{\rho} (y))/( |\mathsf{T}_N|\log N)}  \right] - \frac{\pi}{2} \Big| = 0 . 
\end{equation}
	\item We have 
\begin{equation}\label{eq:mean-hitting}
	 \limsup_{N \to \infty}  \sup_{x \neq y \in \mathsf{T}_N } \sup_{0 \le \rho  \le d(x,y)} \frac{ | \E_{x}[ H( \mathsf{B}_{\rho}(y) ) ] - \frac{2}{\pi} |\mathsf{T}_N| \log[ 1+d(x,y)]  | }{ N^2 [1 \vee \log(1+\rho)]}  < \infty . 
\end{equation}  
\end{enumerate} 
\end{lem} 

Next, we need similar statements but for sets that are unions of balls. Accordingly, we set
\begin{equation*}
	\overline{A}^{\rho} \coloneq \bigcup_{a \in A}  \mathsf{B}_{\rho}(a) \,,
\end{equation*}   
whenever $A$ is a subset of $\mathsf{T}_N$ and $\rho > 0$. Given such an $A$ and parameters $\epsilon, \theta > 0$,  let
\begin{equation*}
	\mathrm{Sub}_{\theta,\epsilon}(A) \coloneq \Big \{ A_1 \subset A : \: A_{1} \neq \emptyset,\, \Big|\tfrac{|A_1|}{|A|} - \theta\Big| < \epsilon \Bigr\} \,,
\end{equation*}
denote the collection of subsets of $A$ whose density in $A$ is $\epsilon$-approximately $\theta$. Finally, for the set of parameters on which our estimates hold uniformly, we take
\begin{equation}
	\label{def:K-gamma-N}
\begin{split}
  \mathscr{K}_{\gamma}^N  \coloneq \Bigl\{\!  (x,A,\rho) \,:\,\,  & x \notin A \subset \mathsf{T}_N \,,\,\,\,  1\le |A| \le (\log N)^{\gamma} \,,\,\,\,  0 \le \rho \le r_N  \\
  & \min \bigl\{ d(u,v) : u \neq v\in A\cup\{x\} \bigr\}\ge  {N}/{r^2_N} \Bigr\} \,,
\end{split}
\end{equation}
with $\gamma > 0$. 
We then have the following lemma.
\begin{lem}\label{prop:hitting-time} 
For any $\gamma \in [0,1)$ the following assertions hold:
\begin{enumerate}[(i)]
	\item The rescaled hitting time converges to the  exponential distribution: 
\begin{equation}\label{eq:hit-asy-exp} \lim_{N \to \infty} \sup_{  (x,A,\rho)  \in \mathscr{K}_{\gamma}^{N} }  \sup_{t \ge 0}
	 \bigg|\,
	 \P_x \Bigl({H\bigl(  \overline{A}^{\rho} \bigr)}  > {\frac{2}{\pi}\,\frac{|\mathsf{T}_N|\log N}{|A|}} t\Bigr)-e^{-t} \,
	 \bigg|
	  = 0 .
\end{equation} 
\item Fix $\theta \in [0,1]$. Assume that  $\epsilon_{N}>0$ satisfies $ \epsilon_{N} \to 0$ as $N \to \infty$. Then  
\begin{equation}\label{eq:hit-asy-exp-2}
	 \lim_{N\to\infty}\   \sup_{  (x,A,\rho)  \in \mathscr{K}_{\gamma}^{N} }  \ \sup_{A_{1} \in \mathrm{Sub}_{\theta, \epsilon_{N}}(A)}
	 \Big|\,
	 \P_x \Bigl({H(  \overline{A}_{1}^{\rho} )}  =  {H(  \overline{A}^{\rho} )}  \Bigr)-  \theta \,
	 \Big|
	 =0.
\end{equation} 
\end{enumerate}
\end{lem}
  
The last piece of input concerns the Green function of $Y$. To this end, if $\vartheta$ is a stopping time for
$Y$, we let $\mathbb{G}_{\vartheta}^{N} (\cdot, \cdot)$ be the Green function for $Y$ killed at $\vartheta$, namely
\[ 
 \mathbb{G}_{\vartheta}^{N} (x, y) \coloneq \sum_{j=0}^{\infty} \P_{x}( j < \vartheta , Y_j = y )= \E_x \Bigl[  \sum_{ j = 0}^{\vartheta-1} \ind{ Y_j =y} \Bigr].
\]  
 
\begin{lem}
\label{l:2.13}
For any sequence $\omega_N $  with $1 \ll \omega_N \leq (\log N)^{\gamma}$,
\begin{equation*}
	\lim_{N \to \infty}  \sup_{  (x,A,\rho)  \in \mathscr{K}_{\gamma}^{N} , |A| \ge \omega_N }  \, \sup_{y \in \mathsf{B}_{\rho}(x) }\frac{1}{\log N} \, \big| \mathbb{G}^{N}_{H(\overline{A}^{\rho})}(x,y) - \frac{2}{\pi} \log N  \big| = 0.
\end{equation*}
\end{lem}
 
\subsubsection{Proof of dynamical lemmas}
\begin{proof}[Proof of Lemma~\ref{lem:nonreachable_traps}] 
The proofs of the three statements are essentially identical; thus, we restrict our attention to the case of  \emph{deep traps}. The remaining cases follow analogously. Note that the $\delta$-dependence in \eqref{equa:DT:lem:nonreachable_traps} stems from \eqref{e:10} in Theorem~\ref{thm:non_reachable_traps}. Observe also that neither ${\mathsf{DT}}$ nor ${\mathsf{NIT}}$ depends on $r$.

Fix  $\varepsilon >0$.  Let $ E^{N}_{1} (\delta)  \coloneq \{ |\mathsf{DT}^{*}(\delta)| \leq  \varepsilon (\log N)^\gamma \} \cap \{ d(\mathbf{0}, \mathsf{T}(\delta)) \geq 2N/r_N \}$.
Applying   \eqref{e:10} in Theorem~\ref{thm:non_reachable_traps},  and Theorem~\ref{thm:distance_origin}, we have 
  \begin{equation}\label{eq:EN-1-negli}
	\limsup_{\delta \downarrow 0}\,\limsup_{N \to \infty} \P \bigl( (E^{N}_{1}(\delta))^{c}  \bigr)  = 0. 
\end{equation}  
  
Since $\{ H_{{\mathsf{DT}(\delta)}} \leq L_{N,m} \} \subset \cup_{x \in \mathsf{DT}^{*}(\delta)} \{  H( \mathsf{B}_{2 r_N}(x) ) \leq L_{N,m}\} $,  employing the union bound, we derive
\begin{align}
\bbP_\mathbf{0} \left( H_{{\mathsf{DT}(\delta)}} \leq L_{N,m} \mid \tau   \right) 
\le |\mathsf{DT}^{*}(\delta)| \, \sup_{x \in \mathsf{DT}^{*}(\delta)} \bbP_\mathbf{0} \left(  H( \mathsf{B}_{2 r_N}(x) ) \leq L_{N,m} \mid \tau   \right).\label{eq:DT-not-rea}
\end{align} 
On $E^N_1(\delta)$ we have   $d(\mathbf{0}, \mathsf{DT}(\delta)) > N/r_N$. Applying Markov's inequality and 
part (i) of Lemma \ref{lem:est_laplace}, we obtain, for sufficiently large $N$,  
\[
\sup_{x \in \mathsf{DT}^{*}(\delta)} 
 \bbP_\mathbf{0} \left(  H( \mathsf{B}_{2 r_N}(x) ) \le L_{N,m} \ \big| \ \tau   \right) \ind{E^N_1(\delta)} 
 \lesssim \sup_{d(x,\mathbf{0})>N/r_N}   \E_{\mathbf{0}}\bigl[ e^{-   H( \mathsf{B}_{2 r_N}(x) ) / L_{N,m}} \ \big| \ \tau \bigr] \le   \frac{ \pi L_{N,m} }{ N^2 \log N}.
\]
Substituting this upper bound into \eqref{eq:DT-not-rea} yields  
\begin{align*}
	\limsup_{N \to \infty} \bbP_\mathbf{0} \left( H_{{\mathsf{DT}(\delta)}} \leq L_{N,m}  , E^N_1(\delta)  \right)  
	\leq  \limsup_{N \to \infty}
	 \frac{\pi \, \varepsilon (\log N)^{\gamma} L_{N,m}}{N^2 \log N} = \pi m  \varepsilon .
\end{align*} 
Sending $\varepsilon \downarrow 0$ and using \eqref{eq:EN-1-negli}, we obtain the desired result.
\end{proof}

\begin{proof}[Proof of Lemma \ref{lem:sum_shallow_traps}] 
Let $(\mathbb{e}_{x,i})_{ x \in \bbZ^2, i \in \mathbb{N}}$ be a family of i.i.d. standard exponential r.v.'s, independent of the field $\tau$ and the jump chain $Y$.
Let $\mathbb{L}_{x}  \coloneq  \sum_{i=0}^{L_{N,m}} \ind{Y_i = x}$ be the local time at $x$ during the first $L_{N,m}$ jumps. We then have
\[
\mathcal{Z}_N(\delta;m) \coloneq \frac{1}{t_N}\sum_{j =0}^{ L_{N,m} }  {\mathbb{e}_{j} \tau_{Y_j}} \ind{ Y_j \in \mathsf{ST}(\delta)}
\stackrel{\mathrm{law}}{=}  
\sum_{x  \in \mathsf{ST}(\delta) }  \frac{ \wh{\tau}_{x}}{\log N} \sum_{i=1}^{\mathbb{L}_{x} } \mathbb{e}_{x,i} .
\] 

 Fix $\varepsilon >0$. We claim that it suffices to prove    
\begin{equation}
	  \limsup\limits_{\delta \downarrow 0} \,  \limsup\limits_{N \to \infty}  \P  \bigl(   \bbE  [ \mathcal{Z}_N(\delta;m)  \mid \tau  ]  \ge  \eta \varepsilon \bigr) = 0 \,,
	 \label{eq:cond-Z-claim}
\end{equation}
  for any $\eta>0$. Admitting this for now, and applying  Markov's inequality, we derive
\[ 
\P \bigl(  \mathcal{Z}_N(\delta;m) > \varepsilon ;    \bbE  [  \mathcal{Z}_N(\delta;m) \mid  \tau  ] \le \eta \varepsilon  \bigr) =\E \bigl[  \P(\mathcal{Z}_N(\delta;m) > \varepsilon \mid \tau) \, \ind{    \bbE  [  \mathcal{Z}_N(\delta;m) \mid  \tau  ] \le \eta \varepsilon }  \bigr]   \le \eta . 
\]
Letting $N \to \infty$, followed by $\delta \downarrow 0$ and then $\eta \downarrow 0$, the required result then follows. 

To establish \eqref{eq:cond-Z-claim}, we integrate the preceding expression for $\mathcal{Z}_N(\delta;m)$ first with respect to $\mathbb{e}_{x,i}$ and then with respect to $Y$. This yields
\[
 \bbE \bigl[ \mathcal{Z}_N(\delta;m)  \ \big|\  \tau \bigr] = \frac{1}{\log N} \sum_{x  \in \mathsf{ST}(\delta) } \wh{\tau}_{x} \mathbb{G}_{L_{N,m}}^N (\mathbf{0},x).
\] 
We proceed as in equation~(4.42) in \cite{arous2008} (see also Lemma \ref{eq:heat-kernal-torus}) to obtain
\begin{equation}\label{eq:upper:green}
 \mathbb{G}_{L_{N,m}}^N (\mathbf{0},x) \le C_{m} (\log N)^{1-\gamma} + C_{m} \log \frac{N^2}{1+|x|^2} \le C_{m}' \log N 
\qquad \forall x \in \mathsf{T}_N, 
\end{equation}
with constants $C_{m}, C_{m}' \geq 1$ depending on $m$. Consequently, 
   for all $N$ sufficiently large,    we have 
\begin{align}
 \E[ \mathcal{Z}_N(\delta;m) \mid \tau] &
   =  \frac{1 }{\log N} \sum_{x \in \mathsf{ST}(\delta)}  \wh{\tau}_{x}\mathbb{G}_{L_{N,m}}^N (\mathbf{0},x) \notag \\
  &\ \le    C'_{m} \sum_{x \in \mathsf{ST}(\delta), |x| \le N/r_N} \wh{\tau}_x   +  \frac{  C_m'   }{(\log N)^{\gamma}} \sum_{ x \in \mathsf{ST}(\delta), |x| \ge N/r_N}  \wh{\tau}_{x} . \label{eq:cexp-0}
\end{align}  
In the second line above we used    \eqref{eq:upper:green}: for   $|x| \le N/r_N$ we  directly bounded  $ \mathbb{G}_{L_{N,m}}^N (\mathbf{0},x) $ by $C'\log N $; and 
 for $|x| > N/r_N$ we   bounded $\mathbb{G}_{L_{N,m}}^N (\mathbf{0},x)$ by  $ C  (\log N)^{1-\gamma}  $. 

On the one hand, from Theorem~\ref{thm:sum_shallow_traps} we obtain  that for any $\eta >0$,
\begin{equation}
	\label{eq:cexp-1}
	 \limsup\limits_{\delta \downarrow 0} \ \limsup\limits_{N \to \infty} \P \Bigl( \frac{  C_m'   }{(\log N)^{\gamma}} \sum_{ x \in \mathsf{ST}(\delta), |x| \ge N/r_N}  \wh{\tau}_{x}  \ge    \eta  \Bigr) = 0\,.
\end{equation} 
On the other hand, we have 
\begin{align*}
\E  \Bigl[ \sum_{x \in \mathsf{ST}(\delta), |x| \le N/r_N} \wh{\tau}_x   \Bigr] &\le \sum_{ |x| \le N/r_N}  \int_{-\infty}^{\frac{\log \delta}{\beta}} e^{ \beta u} \,  \P  \bigl(  h_x - m_N+ \frac{\gamma}{\alpha} \log\log N \in \mathrm{d} u \bigr)  \\
& \le  \sum_{ |x| \le N/r_N}  \frac{(\log N)^{1+\gamma}}{N^2} \int_{-\infty}^{\frac{\log \delta}{\beta}} e^{ \beta u} e^{-\alpha u} \,     \mathrm{d} u
\le \frac{(\log N)^{1+\gamma}}{r_N^{2}} \,  \delta^{1-\frac{\alpha}{\beta}} .
\end{align*} 
Markov's inequality then yields that for any 
$\delta>0$,
\begin{equation}
	\label{eq:cexp-2}
  \limsup\limits_{N \to \infty}  \P  \biggl( \sum_{x \in \mathsf{ST}(\delta), |x| \le N/r_N} \wh{\tau}_x    \ge \frac{1}{\log N} \bigg ) = 0.
\end{equation}   
Combining \eqref{eq:cexp-0} with \eqref{eq:cexp-2} and \eqref{eq:cexp-1}, the required result \eqref{eq:cond-Z-claim} follows. This completes the proof.
\end{proof}

Next we turn to the limiting statements.
\begin{proof}[Proof of Lemma~\ref{lem:approximation:sigma}]
 We address the three assertions sequentially.
First we introduce the good event
\begin{equation}\label{eq:def-EN2-M}
  E^{N}_{2} (M, \delta ) = \Bigl\{ M^{-1}   \le  \frac{ | \mathsf{NT}^{*}(\delta )|}{\delta^{-\alpha/\beta} (\log N)^{\gamma}}  \le  M  \Bigr\} \cap \Bigl\{ d(\mathbf{0}, \mathsf{T}(\delta)) \geq \frac{2 N}{r_N} \Bigr\}.
\end{equation}
Set $M_N=\log\log N$. 
Then from Theorems \ref{thm:regular traps1} and \ref{thm:distance_origin}, we obtain, for any $\delta \in (0,1/2]$,  
\begin{equation}
	\label{eq:EN2-M}
	  \limsup_{N \to \infty} \P\bigl( 	E^{N}_{2} (M_N, \delta )^{c}  \bigr) = 0   \ \text{ and } \ \limsup_{M \to \infty} \limsup_{N \to \infty} \P\bigl( 	E^{N}_{2} (M, \delta )^{c}  \bigr) = 0 . 
\end{equation}

\noindent\underline{\emph{Proof of (i)}. }  
Recall that the distance 
between two local maximum points $x, y \in \mathsf{NT}(\delta)^{*}$ 
in different clusters is at least $N/r_N$. 
Therefore, when $E_2^{N}(M_N, \delta )$ occurs, the tuple 
$(Y(\sigma_{i}),  \mathsf{NT}^{*}(\delta)\setminus \{U_i\}  , r)$ 
belongs to $\mathscr{K}^{N}_{(\gamma+1)/2}$ for all $i \ge 0$ 
(setting $U_{0}= \mathbf{0} $ for convenience). 
Applying Lemma~\ref{prop:hitting-time} (i) and the strong Markov property 
yields the existence of a deterministic sequence 
$\varepsilon_N \to 0$ such that
\begin{equation*}
    \sup_{j \ge 1} \sup_{t>0} 
    \left| 
    \P_{\mathbf{0}}  \bigl(   
    \chi_N \Delta \sigma_{j} > t \ \big| \ \mathscr{F}^{Y}_{\sigma_{j-1}}, \tau  \bigr)
   - e^{-  t } 
    \right| \ind{E_2^{N}(M_N, \delta )} \le \varepsilon_N.
\end{equation*}
Here, for a stopping time $\vartheta$ of $Y$, $\mathscr{F}^{Y}_{\vartheta}$ represents the $\sigma$-field generated by $(Y_{k \wedge \vartheta})_{k \ge 0}$.
Iterating this conditioning argument implies that for all $k \ge 1$,
\begin{equation}
	\label{eq:delta-sigma-cond}
 \sup_{t_{j}>0, 1 \le j \le k} 
 \left| 
 \P_{\mathbf{0}} \bigl( 
 \chi_N \Delta \sigma_{j} > t_{j}, 1 \le j \le k \ \big| \ \tau 
 \bigr) - e^{-  \sum_{j=1}^{k} t_{j} } 
 \right| \ind{E_2^{N}(M_N, \delta )} \le k \varepsilon_N.
\end{equation}
Taking expectations on both sides, sending $N \to \infty$, 
and using \eqref{eq:EN2-M}, the first assertion follows.
\smallskip 

\noindent\underline{\emph{Proof of (ii)}. } Take any measurable subset
$I \subset [\delta, \delta^{-1}]$, 
and any configuration 
$(s_y)_{y \in \mathbb{Z}^2} \in 
(0,1]^{\mathbb{Z}^2}$ 
satisfying $s_{y}=1$ 
for $y \in \{0\} \cup 
(\mathbb{Z}^2 \setminus 
\mathsf{B}_{r})$.
To establish the weak convergence and the asymptotic independence, 
it suffices to show that for each $k \ge 0$, and any bounded 
random variable $G$ measurable with respect to 
$\sigma((\wh{\tau}_{U_{i}}, \bar{\tau}_{U_i}|_{\mathsf{B}_{r}}) : 0 \leq i \leq k)$, we have 
\begin{equation}\label{eq:asy-ind}
 \limsup_{N \to \infty}	\left| \, 
	\E_{\mathbf{0}}  \Bigl[  
	G \, \ind{\wh{\tau}_{U_{k+1}} \in I, \ \bar{\tau}_{U_{k+1}} (y) \le s_{y} ,\, y \in \mathsf{B}_{r} }\Bigr] 
 - \theta \, \E_{\mathbf{0}}[G] 
		\, \right| = 0,
\end{equation}
where 
\begin{equation*}
\theta \coloneq \P( w^*_{\delta} \in I) \P( \bar{w}_y \le s_y \ , \forall\  y \in \mathsf{B}_r ) 
= \frac{\int_{I} t^{-(\alpha/\beta + 1)} \rmd t}
    {\int_{\delta}^{\delta^{-1}} t^{-(\alpha/\beta + 1)} \rmd t} 
    \times  \sigma_\beta \Bigl(\prod_{y \in \bbZ^2} (0, s_y]\Bigr). 
\end{equation*} 
 
Recall that when $E_2^{N}(M_N, \delta )$ occurs, the tuples
$(Y(\sigma_{i}),  \mathsf{NT}^{*}(\delta)\setminus \{U_i\}    , r)$,
for $i \ge 0$, belong to $\mathscr{K}^{N}_{(\gamma+1)/2}$.
We apply point (ii) in Lemma~\ref{prop:hitting-time} with 
$(x, A, \rho) = (Y(\sigma_{i}), \mathsf{NT}^{*}(\delta) \setminus \{U_i\}, r)$ 
and define 
$A_1 = \{x \in A : \wh{\tau}_x \in I,  \tau_{x+y}/\tau_x  \le s_{y} \,, \forall \, y \in \mathsf{B}_r \}$. 
According to Theorem~\ref{thm:regular traps2}, the ratio $|A_1|/|A|$ converges 
in probability to $\theta$ as $N \to \infty$.
Thus, we can find a strictly increasing sequence $(a_k)_{k \ge 1}$ such that  
  \[   \P( |A_1|/|A| - \theta | \ge  2^{-k} ) \le 2^{-k} \ \text{ for all } \ N \ge a_{k}. \] 
  Let $\varepsilon'_{N}= 2^{-k}$ if $N \in [a_{k}  , a_{k+1}) \cap \mathbb{Z}$. Then 
  $  \P(| |A_1|/|A| - \theta | > \varepsilon'_{N}) \le \varepsilon'_{N}$, and  $\varepsilon'_{N} \to 0$  as $N \to \infty$.  Let 
\[  E^{3}_{N}(\delta,r) = E^{2}_N(M_N, \delta ) \cap \{ | |A_1|/|A| - \theta | \le \varepsilon'_{N} \}. \] 
Then Lemma~\ref{prop:hitting-time} (ii)  
yields the existence of a deterministic sequence 
$\varepsilon''_N \to 0$ satisfying 
\begin{equation*}
   \sup_{k \ge 0} 
    \Bigl|   \, 
    \P_{\mathbf{0}} \bigl( 
    \wh{\tau}_{U_{k+1}} \in I, \ \bar{\tau}_{U_{k+1}} (y) \le s_{y} ,\, y \in \mathsf{B}_{r}
    \ \big| \ \mathscr{F}^{Y}_{\sigma_{k}}, \tau 
    \bigr) - \theta \, 
    \Bigr| \ind{E^3_{N}(\delta,r)} \le \varepsilon''_N.
\end{equation*} 
By taking expectations on both sides 
and sending $N \to \infty$,   \eqref{eq:asy-ind} follows immediately. 
\smallskip 

\noindent\underline{\emph{Proof of (iii)}.} 
Let $(w^{*,i}_{\delta})_{i \ge 1}$ be i.i.d. copies of $w^{*}_{\delta}$. Note that the distribution of $w^*_{\delta}$ has no atoms.  It thus follows from  assertion (ii) that for any $l \in \mathbb{N}$,
\[ \limsup_{N \to \infty} \P( \, \exists \, 1 \le i< j \le l \text{ s.t. }  U_i = U_j )  \le  \P\bigl(\, \exists \, 1 \le i< j \le l \text{ s.t. }  w^{*,i}_{\delta} = w^{*,j}_{\delta} \bigr)  = 0 . \]
Then it is sufficient to show that for any $m \in \mathbb{N}$,
\begin{equation}\label{eq:zeta-N-k}
	  \lim_{l \to \infty}\limsup_{N \to \infty} \P( \zeta_{N,m} > l ) = \lim_{l \to \infty}\limsup_{N \to \infty} \P \Bigl( \sum_{i=1}^{l+1} \Delta \sigma_{i} < m N^2 (\log N)^{1-\gamma} \Bigr) = 0 .  
\end{equation} 
Fix $M \ge 1$. Since $E^{N}_{2}(M, \delta ) \subset E^{N}_{2}(M_N, \delta )$ for large $N$,  it follows again from 
\eqref{eq:delta-sigma-cond}   that  $\{ (\chi_N \Delta \sigma_{i})_{i \ge 1}, \P( \cdot \mid  E^{N}_2(M, \delta )) \}$ converges weakly as $N \to \infty$ to a sequence $(\mathbb{e}_i)_{i \ge 1}$ of i.i.d. standard exponential r.v.'s.  
On $  E^{N}_2(M, \delta )$, we have $ \chi_N   N^2 (\log N)^{1-\gamma} \le 2 M \delta^{-\alpha/\beta} $. Thus the double limit in \eqref{eq:zeta-N-k} is bounded above by
\begin{align*}
	& \lim_{l \to \infty}   \limsup_{N \to \infty} \P\Bigl( \sum_{i=1}^{l} \chi_N \Delta \sigma_{i} < 2 m \,M \delta^{-\frac{\alpha}{\beta}}  \mid  E^{N}_2(M, \delta )  \Bigr) + \limsup_{N \to \infty} \P( (E^{N}_2(M, \delta ) )^{c} )  \\
	&\quad =\lim_{l \to \infty}     \P\Bigl( \sum_{i=1}^{l} \mathbb{e}_{i} \le 2 m \, M \delta^{-\frac{\alpha}{\beta}}    \Bigr) + \limsup_{N \to \infty} \P( (E^{N}_2(M, \delta ) )^{c} )  .
\end{align*}
The first limit vanishes according to the law of large numbers.  The desired result \eqref{eq:zeta-N-k} then follows by sending $M \to  \infty$ and applying \eqref{eq:EN2-M}.
\end{proof}
 
Lastly, we have
\begin{proof}[Proof of Lemma~\ref{lem:convergence:time:intervals}]
Let $(\mathbb{e}_{x,i}, x \in \mathsf{T}_N, i \geq 1)$ be a family of i.i.d. standard exponential r.v.'s. Let 
$\mathbb{L}_{x}^n  \coloneq  \sum_{ \sigma_{n} \le i < \sigma_{n+1}}  \ind{ Y_i = x}$ be the local time of $Y$ at $x$ during $[\sigma_{n}, \sigma_{n+1})$. For notational simplicity, we denote by $\mathbf{y} \coloneq  Y_{\sigma_{n}}$. Then we have
\begin{equation}\label{eq:fml-d-sn}
	\frac{1}{t_N} \Delta \mathcal{S}^{\star} (n)  \stackrel{\mathrm{law}}{=}
 \sum_{x \in  \mathsf{B}_{r}(U_{n})} \frac{\tau_{x}}{t_N} \sum_{i=1}^{ \mathbb{L}_{x}^n } \mathbb{e}_{x,i} = \sum_{x \in  \mathsf{B}_{r}(U_{n})} \frac{\wh{\tau}_{x}}{\log N} \sum_{i=1}^{ \mathbb{L}_{x}^n } \mathbb{e}_{x,i}\,.
\end{equation} 

\noindent\underline{\emph{Step 1.}}  
We shall show that all local times $\mathbb{L}^n_{x}$ in the sum above can be replaced by $\mathbb{L}^n_{\mathbf{y}}$:
\begin{equation}
	\label{eq:Lx-Ly}
\frac{1}{\log N} \E \biggl[ \bigg| \sum_{x \in  \mathsf{B}_{r}(U_{n})}  \sum_{i=1}^{ \mathbb{L}_{x}^n } \mathbb{e}_{x,i} - \sum_{x \in  \mathsf{B}_{r}(U_{n})}  \sum_{i=1}^{ \mathbb{L}_{\mathbf{y}}^n } \mathbb{e}_{x,i} \bigg|  \ \bigg| \ \tau ,  \mathscr{F}^{Y}_{\sigma_n} \biggr] \le \sum_{x \in  \mathsf{B}_{r}(U_{n})} \E \biggl[ \frac{|\mathbb{L}_{x}^n -\mathbb{L}_{\mathbf{y}}^n  | }{\log N}   \ \bigg| \   \tau ,  \mathscr{F}^{Y}_{\sigma_n} \biggr] \xrightarrow[N \to \infty]{\text{in prob}} 0 .
\end{equation}
We begin by claiming that
 the strong Markov property implies
\begin{equation}
	\label{eq:local-time-covariance}
	 \E [ \mathbb{L}_{x}^n  \mathbb{L}_{x'}^n \mid \tau,   Y_{\sigma_{n}}  = \mathbf{y}]  = [\mathbb{G}^{N}_{\sigma_{1}} (\mathbf{y}, x) + \mathbb{G}^{N}_{\sigma_{1}} (\mathbf{y},x') ] \mathbb{G}^{N}_{\sigma_{1}} (x, x') - \ind{x= x'}\mathbb{G}^{N}_{\sigma_{1}} (\mathbf{y},x).
\end{equation} 
We then use the asymptotic result for the Green function obtained in Lemma~\ref{l:2.13}. Conditionally on $E^{N}_2(M_N)$ defined in \eqref{eq:def-EN2-M}, we have  $\mathbb{G}^{N}_{\sigma_{1}} (x, x') =  \frac{2}{\pi} \log N + o(\log N)  $
  uniformly in $x,x' \in \mathsf{B}_{2r}(\mathbf{y})$ and $ \mathbf{y} \in \mathsf{NT}^{*}(\delta)$.  We thus obtain 
\begin{equation*}
 \E [ \mathbb{L}_{x}^n  \mathbb{L}_{x'}^n \mid \tau,   Y_{\sigma_{n}}  = \mathbf{y}]  =2 \Bigl( \frac{2}{\pi} \log N  \Bigr)^2 + o(\log N)^2   \quad  \text{ on } E^{N}_2(M_N) , 
\end{equation*}
and hence 
\[ 
 \E [  |\mathbb{L}_{x}^n -\mathbb{L}_{\mathbf{y}}^n   |^2  \mid \tau,   Y_{\sigma_{n}} = \mathbf{y}] \ind{E^N_2(M_N)} =  o(\log N)^2 .\]
By using Jensen's inequality and  \eqref{eq:def-EN2-M}, the required result \eqref{eq:Lx-Ly} follows. 
 
\smallskip
\noindent\underline{\emph{Step 2.}} 
We next show that  the conditional Laplace transform, on $ E^{N}_2(M_N)$, satisfies
\begin{equation}
 \label{eq:laplace-ly}
\bbE \biggl[  \exp \Bigl( - \lambda \sum_{x \in \mathsf{B}_r(U_{n})} \frac{\wh{\tau}_{x}}{\log N} \sum_{i=1}^{ \mathbb{L}_{\mathbf{y}}^n } \mathbb{e}_{x,i} \Bigr) \ \bigg| \   \tau ,  \mathscr{F}^{Y}_{\sigma_n}\biggr] = [1+ o_N(1)] \frac{1}{1 +  \lambda  \frac{2}{\pi} \sum_{x \in \mathsf{B}_r(U_{n})} \wh{\tau}_x } \,.
\end{equation}
Indeed, observe that, conditioned on $\tau ,  \mathscr{F}^{Y}_{\sigma_n}$ , $\mathbb{L}_{\mathbf{y}}^n \ge 1$ has a geometric distribution with mean $\mathbb{G}_{ \sigma_{1}} (\mathbf{y}, \mathbf{y})$. Thus we can rewrite the Laplace transform as
\begin{align*} 
  \sum_{k=1}^{\infty}  \bbP_{\mathbf{y}}(\mathbb{L}_{\mathbf{y}}^n = k) \Bigl( \prod_{x \in \mathsf{B}_r(U_n) } \frac{1}{1+ \lambda \wh{\tau}_x / \log N } \Bigr)^k  
&
=  \frac{ \bbP_{\mathbf{y}}(\mathbb{L}_{\mathbf{y}}^n = 1)}{  
\prod_{x \in \mathrm{B}_r(U_n)} (1+ \lambda \wh{\tau}_x / \log N ) -1 + \P
_{\mathbf{y}}(\mathbb{L}_{\mathbf{y}}^n = 1) }\,.
\end{align*}
Notice that $\prod_{x \in \mathrm{B}_r(U_n)} (1+ \lambda \wh{\tau}_x / \log N ) - 1 =\lambda  \sum \wh{\tau}_x / \log N + O (\log N)^{-2}$ as $N\to \infty$, with error bounded uniformly in $\tau$ (since $\wh{\tau}_x \le \wh{\tau}_{U_n} \le \delta^{-1}   $). Furthermore, by Lemma~\ref{l:2.13}
we have $ \bbP_{\mathbf{y}} ( \mathbb{L}_{\mathbf{y}}^n = 1) = \mathbb{G}_{\sigma_{1}} (\mathbf{y}, \mathbf{y})^{-1} = [1+o_{N}(1)] \frac{\pi} {2}(\log N)^{-1}$ when $E^{N}_2(M_N)$ holds. Then the assertion \eqref{eq:laplace-ly} follows.   
\smallskip

  \noindent\underline{\emph{Step 3.}} 
Combining \eqref{eq:fml-d-sn} with \eqref{eq:Lx-Ly}, \eqref{eq:laplace-ly} and  \eqref{eq:EN2-M}, we obtain
\begin{equation*}
	 \bigg|\bbE \biggl[  \exp \Bigl( - \frac{\lambda}{t_N} \Delta \mathcal{S}^{\star} (n)  \Bigr)  \ \bigg| \   \tau ,  \mathscr{F}^{Y}_{\sigma_n} \biggr] -  \frac{1}{1 +   \lambda \frac{2}{\pi}  \wh{\tau}_{U_n} \sum_{y \in \mathsf{B}_r} \bar{\tau}_{U_n}(y) } \bigg|\xrightarrow[N \to \infty]{\text{in prob}} 0\,.
	\end{equation*}
The   dominated convergence theorem,
  together with  Lemma \ref{lem:approximation:sigma} (ii), implies the weak convergence  and asymptotic independence of the $(\frac{1}{t_N} \Delta \mathcal{S}^{\star} (n) )$.
It only remains to show \eqref{eq:local-time-covariance}. 
Indeed we have 
\begin{equation*}
	\mathbb{L}^{n}_x\mathbb{L}^{n}_{x'}
=
\sum_{\sigma_{n} \le i\le j<\sigma_{n+1}}
\mathbf 1_{\{Y_i=x\}}\mathbf 1_{\{Y_j=x'\}}
+
\sum_{\sigma_{n} \le j\le i<\sigma_{n+1}}
\mathbf 1_{\{Y_i=x\}}\mathbf 1_{\{Y_j=x'\}}
-
\mathbf 1_{\{x=x'\}}\sum_{\sigma_{n}\le i<\sigma_{n+1}}\mathbf 1_{\{Y_i=x\}}.
\end{equation*} 
Moreover from the strong Markov property  we get 
\begin{equation*}
	 \E \Bigl[  \sum_{\sigma_{n} \le i\le j<\sigma_{n+1}}
\mathbf 1_{\{Y_i=x\}}\mathbf 1_{\{Y_j=x'\}}  \ \Big| \  \tau,   Y_{\sigma_{n}}  = \mathbf{y}   \Bigr] 
 =
\sum_{i\ge 0}
\P_{\mathbf{y}}(i<\sigma_{1},Y_i=x)\,
 \E_x \Bigl[ 
\sum_{k=0}^{\sigma_{1}-1}\mathbf 1_{\{Y_k=x'\}}
  \Bigr] 
 = \mathbb{G}^{N}_{\sigma_{1}} (\mathbf{y}, x)\mathbb{G}^{N}_{\sigma_{1}} (x, x').
\end{equation*} 
Analogous calculations for the remaining two sums yield their respective conditional expectations, which establishes \eqref{eq:local-time-covariance}. This completes the proof.
\end{proof}

\subsection{Proof of additional high-level statements}
In this subsection we prove the remaining ingredients in the high-level proof of Theorem~\ref{thm:1}.
In this Subsection,
we write $\sigma_i^{(\delta,r)}$, 
$\mathcal{S}_{\delta,r}$, and 
$\mathcal{S}^{\star}_{\delta,r}$ 
to emphasize the dependence 
of $\sigma_i$, $\mathcal{S}$, 
and $\mathcal{S}^{\star}$ 
on $(\delta, r)$, respectively.

\begin{proof}[Proof of Proposition~\ref{lem:cond4}] Recall the good event $E^{N}_2(M;\delta)$ defined in \eqref{eq:def-EN2-M}. 
 By using \eqref{eq:EN2-M} and Lemma \ref{lem:nonreachable_traps},   there exists $M_0 \ge 1$, depending only on $\varepsilon$, such that
\begin{equation}\label{eq:nor-reach-wide} 
 \liminf_{N \to \infty} \P( E^{N}_{2} (M_0; \delta) )
\ge 1- \varepsilon \ \text{ and }  \
  \limsup_{N \to \infty}
\P(  H_{\mathsf{NIT}(\delta)} \le L_{N,m} )
\le \varepsilon \quad \text{ for all } \delta \in (0,1/2]. 
\end{equation}

 \smallskip
\noindent\underline{\emph{Step 1. }}
  We first prove there exists $m$ depending only on $\varepsilon,\theta$ such that for all $\delta$ and $r$,
\begin{equation}\label{eq:zeta-delta-0}
  \liminf_{N \to \infty} \, \P \bigl(  \zeta_{N,m}(\delta,r) \ge \theta \delta^{-\frac{\alpha}{\beta}}  \bigr)   \ge 1- 2 \varepsilon .  
\end{equation}
Indeed, the proof proceeds similarly to that of \eqref{eq:zeta-N-k}. By \eqref{eq:delta-sigma-cond}, assertion (i) of Lemma~\ref{lem:approximation:sigma}
remains valid under the conditional probability  
$\P(\cdot \mid E^{N}_{2}(M_0;\delta) )$.
Setting $k_{\delta} = \lceil 2\theta \delta^{-\alpha/\beta} \rceil $,  
the law of large numbers  implies the existence of a constant $D_0$, depending only on $\theta$ and $\varepsilon$, such that 
\[  \limsup_{N \to \infty} \P\Bigl( \sigma^{(\delta,r)}_{k_\delta}  > D_0 M_0 N^2 (\log N)^{1-\gamma} \mid   E^{N}_{2}(M_0;\delta) \Bigr) \le \P\Bigl( \sum_{i=1}^{k_{\delta}}  \mathbb{e}_{i}  > \frac{\pi}{8} D_0   \delta^{-\frac{\alpha}{\beta}}   \Bigr) \le  \varepsilon .\] 
Take $m = 1+ D_0 M_0$. The above inequality and \eqref{eq:nor-reach-wide}
together yield that  the unconditional
probability that $\sigma^{(\delta,r)}_{k_{\delta}} \leq L_{N,m}$
is at least $1-2 \varepsilon$. When 
$L_{N,m} \ge \sigma^{(\delta,r)}_{k_{\delta}}$, we have
$\zeta_{N,m}(\delta,r) \ge k_{\delta} \ge 2 \theta \delta^{-\alpha/\beta} $. Our assertion  \eqref{eq:zeta-delta-0} thus follows.

\smallskip
\noindent\underline{\emph{Step 2. }} Set  $\delta_0:=1/2$.  We claim that it is now sufficient to show there exists $m \in \mathbb{N}$ such that
  \begin{equation}\label{eq:calm-delta-0} \liminf_{N\to \infty}  \P(  \mathcal{S}_{\delta_0,r} (\zeta_{N,m}(\delta_0,r)-1) \ge  \theta t_N )  \ge 1-  \varepsilon   .
  \end{equation}
 Indeed,  for every $m \ge 1$ and   every $\delta<\delta_0 $,  on the event $\{L_{N,m}<H_{\mathsf{NIT}(\delta)}\}$, we have 
\begin{equation}\label{eq:S-monotone}
	\mathcal{S}_{\delta,r}\bigl(\zeta_{N,m}(\delta,r)-1\bigr)
\ge
\mathcal{S}_{\delta_0,r}\bigl(\zeta_{N,m}(\delta_0,r)-1\bigr) .
\end{equation} 
To see this, note that while
 $\mathsf{NT}^{*}(\delta)$ is not necessarily monotone in $\delta$, the following dichotomy holds: for every
$x \in \mathsf{NT}^{*}(\delta_0)$ and every $0<\delta<\delta_0$, either
$x \in \mathsf{NT}^{*}(\delta)$ or $x \in \mathsf{NIT}(\delta)$.
Under the event $\{L_{N,m}<H_{\mathsf{NIT}(\delta)}\}$, the walk does not visit any site in $\mathsf{NIT}(\delta)$ before time $L_{N,m}$. Hence every $\delta_0$-normal trap encountered by the walk within its first $L_{N,m}$ steps is also a $\delta$-normal trap, which proves \eqref{eq:S-monotone}.  
Consequently, 
combining \eqref{eq:calm-delta-0} with \eqref{eq:nor-reach-wide} and \eqref{eq:S-monotone}, we get 
\begin{align*} 
	\liminf_{N\to \infty}  \P(  \mathcal{S}_{\delta,r} (\zeta_{N,m}(\delta,r)-1) \ge  \theta t_N )  \ge 1- \varepsilon - \limsup_{N \to \infty}
\P(  H_{\mathsf{NIT}(\delta)} \le L_{N,m} ) \ge  1- 2 \varepsilon   .
\end{align*} 
This, together with \eqref{eq:zeta-delta-0}, completes the proof of Proposition \ref{lem:cond4}.

\smallskip
\noindent\underline{\emph{Step 3. }} It remains to prove \eqref{eq:calm-delta-0}. 
We  use   the trivial bound $\mathcal{S}_{\delta_0,r}  \ge \mathcal{S}^{\star}_{\delta_0,r} $. 
Lemma~\ref{lem:convergence:time:intervals}  implies that the sequence
$( t_N^{-1} \Delta \mathcal{S}^{\star}_{\delta_0,r}(j)) _{ j \ge 1}$
converges weakly to $(\mathcal{E}^{(j)}_{\delta_0,r})_{j \ge 1}$,  which consists of  i.i.d. copies of $\mathcal{E}_{\delta_0,r} $. Since  $\mathcal{E}_{\delta_0,r} \ge  \mathcal{E}_{\delta_0,1}$, the law of large numbers yields that
there exists $k_0 \ge 1$, depending only on $\varepsilon, \theta$, 
such that 
\begin{equation*} 
 \liminf_{N \to \infty} \P(  \mathcal{S}^{\star}_{\delta_0,r} (k_0)  \ge  2 \theta t_N) \ge    \P\Bigl(   \sum_{i=1}^{k_0}  \mathcal{E}_{\delta_0,1}^{(i)} \ge 2 \theta    \Bigr) \ge  1-\varepsilon . 
\end{equation*}   
Applying \eqref{eq:zeta-delta-0} with $\delta$ and $\theta$ replaced by $\delta_0$ and $k_0$  respectively, we get 
 $ \P  (  \zeta_{N,m}(\delta_0,r) \ge 2 k_0   )  \ge 1- 2 \varepsilon$ for all large $N$. 
This implies  $ \P  ( \mathcal{S}_{\delta_0,r}( \zeta_{N,m}(\delta_0,r)-1) \ge 2 \theta t_N   )  \ge 1- 3 \varepsilon$, which completes the proof.
\end{proof}

\begin{proof}[Proof of Proposition~\ref{lem:only:moderate:traps:visited}] 
The coarse-grained idea is that $X_t$ spends most of the time in the normal traps, thus if $t_N'/t_N$ were random, for example uniformly distributed in $[a, b]$, the probability of $X_{t_N'} \in \mathsf{NT}(\delta, r)$ would be proportional to the ratio
\[
\frac{ \mathrm{Leb}(\{ t; X_t \in \mathsf{NT}(\delta,r) \}\cap [a t_N, b t_N] )  } { (b-a) t_N} \approx 1 .
\]
We now make rigorous this intuitive idea.
  
\smallskip
\noindent\underline{\emph{Step 1. }}
To proceed, we introduce an auxiliary random time.  Fix $\eta\in(0,1)$, and
let $\mathbb{V}$ a uniform random variable on $(0,1)$, independent of both the DGFF 
$(h_x)_{x \in \mathsf{V}_{N}}$ 
and the walk $(X_t)$.    Define   the hitting time
\begin{equation*}
T(\eta)  \coloneq  \inf \{ t \ge 0 :  X_t \in  \mathsf{NT}^{*}(\eta) \}  \,.
\end{equation*} 
By the notation in \eqref{eq:def-U_i}, we have  $ X_{T(\eta)}=U^{(\eta,0)}_1$,  which we write as $U$ for short.  Define 
\begin{equation*}
	 T(\eta, v)  \coloneq  \inf \Bigl\{ t > 0 :  \int_{0}^{t} \ind{X_s= U }  \dif s \ge  v \, \eta \tau_{U}   \log N    \,  \Bigr\} \ \text{ for } \ v \ge 0 .
\end{equation*}
 
For each $\delta>0$ and $r\ge 1$, define the random subsets of $\mathbb{R}$ by
\begin{align*}
	\Xi &  \coloneq  \bigl\{ t \geq T(\eta,\mathbb{V})  ; X_{ t} \notin \mathsf{NT}(\delta,r)   \bigr\} \quad \text{ and }, \\
	 \Xi'  &\coloneq   {(t_N'-\Xi_0)} \cap [0,\infty)  \quad \text{ where  } \quad
	\Xi_{0} \coloneq  \Xi-T(\eta,\mathbb{V}) .
\end{align*}
 
On the event $\{T(\eta,\mathbb{V})\le a t_N\}$, the condition 
$X_{t_N'}\notin \mathsf{NT}(\delta,r)$ is equivalent to
$t_N'\in  T(\eta,\mathbb{V})+\Xi_0$, which can be 
rewritten as   $   T(\eta,\mathbb{V}) \in \Xi'   $. This gives 
\begin{equation}
 \label{eq-tprimeN-1}
    \P ( X_{t'_N} \notin \mathsf{NT}(\delta,r)  ) \le   \P   \bigl(     T(\eta,\mathbb{V}) \in \Xi'   , T(\eta,\mathbb{V}) \le a t_N    \bigr) + \P   \bigl(    T(\eta,\mathbb{V}) \ge a t_N \bigr) . 
\end{equation}  
Furthermore, whenever  $T(\eta,\mathbb{V}) \le a t_N \le b t_N$, it follows  that 
\begin{equation}
	|(t_N'-\Xi_0) \cap [0,\infty)| =| \Xi_0 \cap  [0, t'_N ]|
  \le  | \Xi \cap [0,   t'_N + T(\eta,\mathbb{V}) ] | \le  | \Xi \cap [0,  2b t_N] |  . \label{eq:Xi-0-is-thin}
\end{equation} 
Combining \eqref{eq-tprimeN-1} and \eqref{eq:Xi-0-is-thin} yields the following upper bound for $   \P (  X_{t'_N} \notin \mathsf{NT}(\delta,r))$:
 \begin{equation}
	\label{eq:error-terms}
\P   \bigl(  T(\eta,\mathbb{V}) \in \Xi' ,\ | \Xi' |  \le \varepsilon t_N  \bigr) + \P   \bigl(    T(\eta,\mathbb{V}) \ge a t_N \bigr) + \P \bigl( | \Xi \cap [0,  2b t_N] |  > \varepsilon t_N \bigr) . 
 \end{equation}

	 \noindent
\underline{\emph{Step 2.} } To control the first term in \eqref{eq:error-terms}, we show that for all large   $N$,
\begin{equation}
	\label{eq:error-1}
	 \P   \bigl(   T(\eta,\mathbb{V}) \in \Xi' ,\ | \Xi' |  \le \varepsilon t_N \bigr) \le  \eta^{-2}\, \varepsilon.
\end{equation}

First we claim 
\begin{equation}
	 \P   \bigl(   T(\eta,\mathbb{V}) \in \Xi' ,\ | \Xi' |  \le \varepsilon t_N \bigr) 
	= \E \bigl[   \mu_{\tau, U }  (  \Xi'  )  \ind{ |\Xi'| \le \varepsilon t_N}    \bigr] , \label{eq-constraint-on-E-1}
\end{equation} 
where $\mu_{\tau, U } $ is  the conditional distribution of $  T(\eta,\mathbb{V})$ given $(\tau, U)$;  that is,    
\[  \mu_{\tau, U }  (  A  ) \coloneq \P  \bigl(  T(\eta,\mathbb{V})  \in   A   \mid \tau , U \bigr)  \quad \text{ for all Borel set } A \subset [0,\infty).  \]  
Indeed, let  
$\mathcal{G}_{t}:=\sigma(\tau,  \mathbb{V}  ) \vee \mathcal{F}^{X}_{t} $, 
where $\mathcal{F}^{X}_{t}$ is the natural filtration of $X$.
Then $T(\eta,\mathbb{V})$ is a stopping time for this enlarged filtration
$(\mathcal{G}_{t})_{t\ge 0}$.  
Moreover, by  definition we have
$X_{T(\eta,\mathbb{V})}=U$ almost surely. 
Conditionally on   $\tau$, the process $X$ is a continuous-time
Markov chain, and adjoining the independent variable $ \mathbb{V} $ to the
filtration preserves the strong Markov property. So applying the strong
Markov property at $T(\eta,\mathbb{V})$, conditionally on
$(\tau, X_{T(\eta,\mathbb{V})}=U)$, the shifted process
$(X_{T(\eta,\mathbb{V})+s})_{s\ge0}$ is independent of
$\mathcal{G}_{T(\eta,\mathbb{V})}$ and has law $\P^\tau_{U}$.
The set $\Xi_0$, and hence $\Xi'$, is a
measurable function of $(\tau,U)$ and this shifted process.
Since $T(\eta,\mathbb{V})$ is
$\mathcal{G}_{T(\eta,\mathbb{V})}$-measurable, it follows that
$\Xi'$ and $T(\eta,\mathbb{V})$ are conditionally independent given
$(\tau,U)$.  This proves \eqref{eq-constraint-on-E-1}.

We next show that 
	 \begin{equation}
	\label{eq:anti-concentration-mu}
		\mu_{\tau, U }  (  A  )  \le  \eta^{-2}  t_N^{-1} |A| \  \text{ for all Borel set  } A \subset [0,\infty).
	 \end{equation}
Plugging \eqref{eq:anti-concentration-mu} back into \eqref{eq-constraint-on-E-1} yields the desired result \eqref{eq:error-1}.  
Define the inverse local time process  
 \[  \Psi(v) \coloneq \inf  \Big\{t>0:  \int_0^t \ind{X_s=U}\dif s \ge v  \Big\} \ \text{ for } \ v \ge 0 .\] 
  Then we have 
 	$ T(\eta,\mathbb{V})
	=
	\Psi( \eta\mathbb{V} \tau_U \log N )$.  
Since $\mathbb{V}$ is uniform on $(0,1)$, for every Borel set
$A\subset[0,\infty)$,
\begin{equation}\label{eq:anti-concentration-mu-path}
	\P\bigl( T(\eta,\mathbb{V})\in A \mid \tau,X\bigr)
	=  \frac{  1}{\eta \tau_U  \log N } \int_{0}^{ \eta \tau_U  \log N} \ind{  \Psi(v) \in A  } \dif v\,.
\end{equation}
To bounding the integral from above,  observe that for any interval
$ (a,b]\subset[0,\infty)$, $  \Psi(v) \in (a,b]$ implies  $v  \in (\int_0^a \ind{X_s=U}\dif s, \int_0^b \ind{X_s=U}\dif s] $. Since $\int_a^b \ind{X_s=U}\dif s \le b-a$, it follows that
$ \int_0^{\infty}   \ind{ \Psi(v) \in (a,b]} \dif v   \le  b-a  $.
Representing open sets as countable unions of disjoint intervals,   and applying outer regularity, this inequality extends to all Borel sets $A \subset [0,\infty)$:
\[
	\int_0^{\infty}  \ind{ \Psi(v) \in A} \dif v   \le |A|  .
\]
 Returning to \eqref{eq:anti-concentration-mu-path}, and using the fact that   $\wh{\tau}_U=\tau_U/\wh{t}_N= \tau_U \log N / t_N\ge\eta$ as $U\in\mathsf{NT}^*(\eta)$,  we  obtain 
\[
\P\bigl( T(\eta,\mathbb{V})\in A \mid \tau,X\bigr) \le  \frac{1}{\eta \tau_U  \log N } |A| \le  \frac{1}{\eta^2 t_N} |A| .
\]
Since $U$ is measurable with respect to $(\tau,X)$, taking the conditional expectation given $(\tau,U)$ yields \eqref{eq:anti-concentration-mu}, which completes the proof of \eqref{eq:error-1}.
   
\smallskip 
\noindent
\underline{\emph{Step 3.}} To control the second term in \eqref{eq:error-terms},
we  prove that $T(\eta,\mathbb{V})$ is indeed negligible on the scale $t_N$:
\begin{equation}
\label{eq:T-eta-uM-small}
 \lim_{\eta\downarrow0} \, 
\limsup_{N\to\infty}
\P\bigl(T(\eta,\mathbb{V})>a t_N\bigr)=0.
\end{equation}

Let $(\mathbb{e}_i)$ be i.i.d. standard exponentials. For $r \ge 1$, we define
\[
 \mathbb{L}_U^1 \coloneq \sum_{\sigma_1^{(\eta,r)}\le j < \sigma_2^{(\eta,r)}} \ind{Y_j=U}, \qquad \mathcal{L}_U^1 \coloneq \tau_U \sum_{i=1}^{\mathbb{L}_U^1} \mathbb{e}_i,
\]
where $\mathcal{L}_U^1$ denotes the local time accumulated at the $\eta$-normal trap bottom $U$ prior to entering the $r$-neighborhood of the next $\eta$-normal trap. Consequently, on the event $\{\mathcal{L}_U^1 \ge \eta \mathbb{V} \tau_U \log N\}$, the walk attains the target local-time level before leaving for a new $\eta$-normal trap, yielding $T(\eta,\mathbb{V}) \le \mathcal{S}_{\eta,r}(1)$.
Thus we get 
\begin{equation}
\label{eq:T-by-S-plus-F}
	\P\bigl(T(\eta,\mathbb{V})>a t_N\bigr)
	\le \P(  \mathcal{L}_U^1 <  \eta \mathbb{V} \tau_U\log N  ) + 
	\P\bigl(\mathcal{S}_{\eta,r} (1)>a t_N\bigr).
\end{equation}

For the first probability, after entering $\mathsf{B}_r(U^{(\eta,r)}_{1})$, the walk hits $U^{(\eta,r)}_{1}$ (which implies $U = U^{(\eta,r)}_{1}$) before exiting $\mathsf{B}_{N/r_N}(U^{(\eta,r)}_{1})$ and hence prior to $\sigma_2^{(\eta,r)}$ , with probability $1-o_N(1)$. Given $U$ is hit, the number of visits to $U$ before $\sigma_2^{(\eta,r)}$ is geometrically distributed with mean $\mathbb{G}_{\sigma_1^{(\eta,r)}}^N(U,U) \ge c\log N$ due to \eqref{eq:lbnd-Green}, yielding
\[
    \P(\mathbb{L}_U^1 \le \eta \log N ) \le C \eta + o_N(1).
\]
Furthermore, on $\{ \mathbb{L}_U^1 > 2 \eta \log N\}$, $\mathcal{L}_U^1 < \eta \mathbb{V} \tau_U \log N$ implies $\sum_{i=1}^{ 2 \eta \log N }\mathbb{e}_i < \eta \log N$. By large deviation principal, this latter event has probability $o_N(1)$. Combining these estimates gives, for any $r \ge 1$,
\begin{equation}
\label{eq:FNM-small}
    \limsup_{N\to\infty} \P( \mathcal{L}_U^1 < \eta \mathbb{V} \tau_U \log N ) \le C \eta.
\end{equation}
  
It remains to control $\mathcal{S}_{\eta,r} (1)$.  Fix $m \ge 1$ for the moment. On the event
$\{\zeta_{N,m}\ge2\}$, it holds
$
	\mathcal{S}_{\eta,r}(1)-\mathcal{S}_{\eta,r}^\star(1)
	\le
	\sup_{0\le i\le \zeta_{N,m}-1}
	\{\mathcal{S}_{\eta,r}(i)-\mathcal{S}_{\eta,r}^\star(i)\}$.
Hence Proposition~\ref{prop:approximation:clock:process}, together with
Proposition~\ref{lem:cond4}  yields
\begin{equation}
\label{eq:S-minus-Sstar-first}
	\lim_{\eta\downarrow0}\limsup_{r\to\infty}\limsup_{N\to\infty}
	\P\bigl( \mathcal{S}_{\eta,r} (1)- \mathcal{S}_{\eta,r} ^\star(1)>a t_N/2\bigr)=0.
\end{equation}  

For fixed $\eta$ and $r$, Theorem~\ref{thm:distance_origin}
implies that
$\P(\Delta\mathcal{S}_{\eta,r}^{\star}(0)=0)\to1$. 
Lemma~\ref{lem:convergence:time:intervals} then yields 
\[
\lim_{N \to \infty}	t_N^{-1} \mathcal{S}_{\eta,r} ^\star(1)
=  \mathcal{E}_{\eta,r}
	 \coloneq 
	\mathbb{e}\,\frac{2}{\pi}w^*_\eta\mathcal{W}_r \text{ in distribution.}
\]
Letting first $r\to\infty$ and then $\eta\downarrow0$, the right-hand side
converges to zero in probability: indeed $\mathcal{W}_r\uparrow\mathcal{W}_\infty
<\infty$ a.s., while
\[
 \P(w^*_\eta>a)
 = \frac{\int_a^{\eta^{-1}} t^{-(\alpha/\beta+1)}\,\dif t}
	     {\int_\eta^{\eta^{-1}} t^{-(\alpha/\beta+1)}\,\dif t}
	\xrightarrow{\eta\downarrow0}0
	\qquad \text{for every } a>0 .
\]
Consequently, we obtain 
\begin{equation}
\label{eq:Sstar-first-small}
	\lim_{\eta\downarrow0}\limsup_{r\to\infty}\limsup_{N\to\infty}
	\P\bigl(\mathcal{S}_{\eta,r}^\star(1)>a t_N/2\bigr)=0.
\end{equation}
Combining \eqref{eq:T-by-S-plus-F}, \eqref{eq:FNM-small},
\eqref{eq:S-minus-Sstar-first} and \eqref{eq:Sstar-first-small} proves
\eqref{eq:T-eta-uM-small}.

\smallskip
\noindent
\underline{\emph{Step 4.}}
To bound the last probability in \eqref{eq:error-terms}, we return to the clock process associated with the $\delta$-normal traps. Since
 $ |\Xi\cap[0,2bt_N]|
	\le
	\int_0^{2bt_N}\ind{X_s\notin\mathsf{NT}(\delta,r)}\dif s$,
it holds for any $m \ge 1$ that   
\begin{align*} 
\P\bigl(|\Xi\cap[0,2bt_N]|>\varepsilon t_N\bigr)	\le \P\bigl(\mathcal S(\zeta_{N,m}-1)<2bt_N\bigr) +
	\P\bigl(\mathcal S(\zeta_{N,m}-1)-\mathcal S^\star(\zeta_{N,m}-1)
		>\varepsilon t_N\bigr).
\end{align*}
Applying Proposition~\ref{lem:cond4} with $\theta_0 = 2b+1$ and choosing $m$ large enough, the first term is negligible,  while the second term vanishes by Proposition~\ref{prop:approximation:clock:process}. This implies
\[
	\lim_{\delta\downarrow0}\limsup_{r\to\infty}\limsup_{N\to\infty}
	\P\bigl(|\Xi\cap[0,2bt_N]|>\varepsilon t_N\bigr)=0.
\]
Combining the estimates from Steps 2–4 with \eqref{eq:error-terms}, and sending the auxiliary parameters to their limits in the order $N \to \infty$, $r \to \infty$, $\delta \downarrow 0$, $\varepsilon \downarrow 0$, and finally $\eta \downarrow 0$,   establishes \eqref{eq:lem:only:moderate:traps:visited}.
\end{proof}


\begin{proof}[Proof of Proposition~\ref{p:2.3}]
This is exactly the statement in part (iii) of Lemma~\ref{lem:approximation:sigma}.
\end{proof}

\section{The trapping landscape}
\label{sec:3} 
In this section we provide proofs for all results in Subsection~\ref{s:1.2}. As will be shown in the next subsection, $\delta$-trap vertices correspond to vertices where the DGFF is at height $\Theta(\log \log N)$ below the typical value of its maximum $m_N$. Therefore, results concerning their statistics translate to results concerning the (near) extreme values of the free field. We thus begin this section with an account of such results. These are mostly ``local'' estimates for the marginal density of the field at one or two points subject to additional constraints. Such estimates are technical and rather standard by now and thus their proof is deferred to a later section. Instead, we first provide, in the following subsection, the more interesting arguments which use these results to prove the desired trapping landscape statements. 

\subsection {Local estimates for DGFF near-extrema}
\label{sub:4.1}
Hereafter, let $f^{*}_{\mathsf{V}} \coloneq \max_{x \in \mathsf{V}} f(x)$   for any function $f$ and subset $V$ in its domain.   Expressions such as $\P(\xi\in a+\dif t)$ are understood in the measure-theoretic sense as $\P(\xi-a\in \dif t)$ with $t$ denoting the variable of the corresponding distribution measure. 
We write   $\P( \xi \in \dif t ) \le g(t) \dif t  $ as shorthand to mean that the density $f_{\xi}$ of a random variable $\xi$  is bounded by $g$: $f_{\xi}(t) \le g(t)$.  
Finally, for any $t \ge 0$ and $K>0$, we set 
\begin{equation*}
    \rho(t) \coloneq 2^{t^{50}} \quad \text{and} \quad \rho(t;K) \coloneq \min\bigl\{ \rho(t) , \sqrt{K} \bigr\}\,. 
\end{equation*} 
The first proposition establishes an estimate for the probability that a vertex in the bulk achieves a near-maximal value while satisfying specific local and global maximum constraints.

\begin{prop}
 \label{prop:one-point-estimate} 
 There exist constants $C_{\ref{prop:one-point-estimate}}, c_{\ref{prop:one-point-estimate}} > 0$ such that for all  large $N \in \mathbb{N}$, the following holds:
  Suppose $z \in \mathsf{V}_{N}$ satisfies $\wtN  =\wtN _{z}\coloneq \mathrm{dist}(z, \partial \mathsf{V}_{N}) \ge \sqrt{N}$. 
Then, for any $t \le s \le u$ with $t, s, u \in [- 10(\log N)^{1/2 }, 10(\log N)^{1/2 }]$, we have 
 \begin{align}
\P  \bigl( h_z \in m_{\wtN } +  \mathrm{d} t ,   h^*_{z+ \mathsf{B}_{\rho(u-t;  \wtN  )   } } & \leq m_{\wtN } + s
 ,  \,    h^*_{z+ \mathsf{B}_{\wtN /2} } \leq m_{\wtN } + u \bigr) \notag\\
	&  \qquad \leq   C_{\ref{prop:one-point-estimate}}  { (  1+u_+)   \exp( -c_{\ref{prop:one-point-estimate}  } \,  u_{-}^{2}  )    (1+s-t  ) } \, {\wtN ^{-2}}  e^{-\alpha t} 
 \dif t\,. \label{e:32}    
\end{align}
\end{prop}

Proposition \ref{prop:one-point-estimate} implies the following corollary, which provides upper bounds for the  probability of near-extremal values at any vertex,  
and the right tail distribution of the global maximum.

\begin{corollary}\label{cor-locmax}  
There exists   constants $C_{\ref{cor-locmax}}, c_{\ref{cor-locmax}}>0$ such that, for all large
$N\in\mathbb N$, any $z\in\mathsf{V}_N$ with
$\wtN=\mathrm{dist}(z,\partial\mathsf{V}_N)$, and any $t\le s\le u$
satisfying $|t|\le(\log N)^{1/2}$ and
$-(\log\log N)^{2}\le u\le(\log N)^{1/2}$,
\begin{align} \label{e:17}
  \P  \bigl( h_z \in m_N + \dif t
, h^*_{z+ \mathsf{B}_{\rho(u-t;  \wtN   )  } } & \leq m_{N} + s,  h^*_{\mathsf{V}_N} \leq m_N + u \bigr) \notag \\
&\qquad 
\le C_{\ref{cor-locmax}} (1+u_+)
\exp(- c_{\ref{cor-locmax}} u_-^2 )
(1+s-t)N^{-2}
\rme^{-\alpha t } \dif t .
\end{align}
Moreover for any $u \ge 0$, we have 
\begin{equation} \label{e:19}
 \P  \bigl( h^*_{\mathsf{V}_N} \geq m_N + u \bigr) \leq C_{\ref{cor-locmax}}   (u+1) e^{-\alpha u} \,. 
\end{equation} 
\end{corollary} 
  
The next proposition provides the corresponding lower bound to Proposition \ref{prop:one-point-estimate}.

\begin{prop} 
	\label{prop:one-point-lower-bnd}
 Fix parameters $\eta \in (1/2,1)$, $r \ge 1$, $M \ge 0$, and $\omega \in \mathbb{R}^{\mathbb{Z}^2}$ supported in $\mathsf{B}_{r}$ with $\omega_{\mathbf{0}} \ge 0$.
There exists $C_{\ref{prop:one-point-lower-bnd}} > 0$ depending on these parameters such that for all large $N$, any $z \in \mathsf{V}_{\eta N}$, and any $u,s,t \in [-10(\log N)^{1/2}, 10(\log N)^{1/2}]$ satisfying $u \ge \max\{s,t,0\}$ and $s \ge t-M$,
\begin{align}
\P  \bigl( h_z \in m_N +   \dif   t,  
 h_{z+y} \le h_z + \omega_y , \forall\, y \in \mathsf{B}_r ,\ & \
  h^*_{\mathsf{B}_{\sqrt{N}}(z) \setminus \{z\}} \leq m_N + s,
  h^*_{\mathsf{V}_N} \leq m_N + u \bigr) \notag \\
&\qquad \geq C_{\ref{prop:one-point-lower-bnd}}   (  1+u)     [1+(s-t)_+  ] \, N^{-2} \exp(- \alpha  t )
 \dif t\,.  \label{e:33}
\end{align} 
\end{prop}

The next proposition establishes the spatial isolation of near-maxima, showing that it is unlikely to find another high value in the annular regions surrounding a local extremum.

\begin{prop}
\label{prop:annular-high-point}
There exists a constant $C_{\ref{prop:annular-high-point}} > 0$ such that for all  large $N \in \mathbb{N}$, the following   hold. Let $z \in \mathsf{V}_N$ and set $\wtN  \coloneq \mathrm{dist}(z, \partial \mathsf{V}_{N})$.  
For all real numbers $t, s$ with absolute value at most $(\log N)^{1/2}$ and $\max\{s,t,0\} \le u    \le (\log N)^{1/2}$, we have 
\begin{align}
 \P  \bigl( h_z \in m_N +  \dif t  &,  h^*_{z+(\mathsf{B}_{\wtN /\rho(u-s;\wtN )} \setminus 
\mathsf{B}_{\rho(u-s;\wtN )}) } > m_N + s
 \,,\,\, h^*_{\mathsf{V}_N} \leq m_N + u \bigr) \notag \\
 &\qquad \leq C_{\ref{prop:annular-high-point}}  [1+\min\{(u-s)^{50}, \log N\}]^{-1/4}   \, (1+u ) (1+u-t)  \,  N^{-2}  \rme^{-\alpha t} \dif t\,. 	
 \label{e:26}
\end{align}  
Furthermore, if $s < t$, then for any integer $r \ge (1+t-s)^{4}$ and $2^{r} \le \rho(u-t ;  \wtN )$,  
 \begin{align}
	\P  \bigl( h_z = h^*_{z+\mathsf{B}_{\rho(u-t ; \wtN )}}  \in m_N +  \dif t  \,&,\,\,
h^*_{z+(\mathsf{B}_{\rho(u-s ; \wtN )} \setminus \mathsf{B}_{2^r})} > m_N + s
	\,,\,\, h^*_{\mathsf{V}_N} \leq m_N + u \bigr) \notag \\
 &\qquad \leq C_{\ref{prop:annular-high-point}} \, r^{-1/4}  \, (1+u) \, N^{-2} 
\rme^{-\alpha t } \dif t\,.
\label{e:26a}
 \end{align} 
\end{prop}

 
In the next proposition, we provide an upper bound on the probability of having two extreme local maxima simultaneously. This result is essential for the second moment argument used later to derive lower bounds of regular traps.
  
\begin{prop}
\label{prop:two-point-estimate} Fix any $\eta \in (1/2,1)$, and let $\eta'= {(3+\eta)}/{4}$.  
There exist constants  $C_{\ref{prop:two-point-estimate}} >0$ such that the following holds for all large $N \in \mathbb{N}$. Let $t,s,u \in \bbR$ be such that $u \in [0, (\log N)^{1/2}]$ and $t,s \in [-(\log N)^{1/2}, u]$. Then for all $x \neq y \in \mathsf{V}_{\eta  N}$     
\begin{align}
	  \P \bigl(  h^*_{x+\mathsf{B}_{\rho(u-t; N)}} & \, = h_x  \in  m_N + \dif t   ,\,
	  h^*_{y+\mathsf{B}_{\rho(u-s; N)}} = h_y  \in  m_N + \dif s   ,\,
	 h^*_{\mathsf{V}_{\eta' N}} \leq m_N + u \bigr) \notag\\
	& \quad \leq C_{\ref{prop:two-point-estimate}} \, \frac{ (1+u)^{10} e^{ \alpha u }} { N^{2} (1+|x-y|)^{2}} \,   
	\Bigl( \frac{\log N }{(\log_+|x-y| + 1)(\log_+ \frac{N}{|x-y| \vee 1} + 1)} \Bigr) ^{3/2} \rme^{-\alpha (t+s)} \dif t \dif s\,.\label{e:27} 
\end{align} 
\end{prop}

Finally, the following proposition establishes the asymptotic independence between the precise height of a local maximum and its local microscopic geometry. It demonstrates that
the local configuration converges to a universal distribution $\nu$, while the height fluctuations exhibit an exponential decay.  
 
\begin{prop}
\label{prop:cluster-law}
Let $\eta\in(1/2,1)$ and set $\eta'\coloneq(3+\eta)/4$. Let
 $K\ge4$, $M_0 \ge 1$ and $N \ge 1$. Set  $g_x \coloneq h_x + \psi_x  $ with  $(\psi_x)_{x \in \mathsf{V}_{N}}$ being a deterministic field, satisfying
\begin{equation}
\label{e:2.11}
\max_{x \in \mathsf{V}_{\eta' N}} |\psi_{x}| \leq 2 \sqrt{g} \log K -\frac{\log\log K}{4\alpha}  \ , \  \max_{x,y \in \mathsf{V}_{\eta' N}}{|\psi_x - \psi_y|} \le M_0\ , \
\max_{x \neq y \in \mathsf{V}_{\eta' N}} \frac{|\psi_x - \psi_y|}{|x-y|} \leq \frac{M_0  }{N} .
\end{equation}  
Fix $u, \bar{s}\ge0$, $r\ge1$, and
$\omega\in\mathbb R_-^{\mathbb Z^2}$ such that
$\omega_{\mathbf0}=0$ and
$\supp(\omega)\subset\mathsf{B}_r$.
Then
\begin{align}
&\P\Bigl(
 g_{z+y}\le g_z+\omega_y,\ \forall y\in\mathsf{B}_r,\,
 g_z=g^*_{z+\mathsf{B}_{2^\varrho}}
       \in m_{KN} +s+\dif t,\,
 g^*_{\mathsf{V}_N}\le m_{KN}+u
 \Bigr)
 \notag\\
&\quad=
 \rme^{-\alpha s}\nu(\omega)\,
 \P\Bigl(
 g_z=g^*_{z+\mathsf{B}_{2^\varrho}}
       \in m_{KN} +\dif t,\,
 g^*_{\mathsf{V}_N}\le m_{KN}+u
 \Bigr)
 +\mathcal R_{N,z}^{\omega}(t,s;\varrho)\,\dif t.
\label{e:28c}
\end{align}  
Moreover, uniformly over all deterministic fields $\psi$ satisfying
\eqref{e:2.11}, and uniformly in  $z \in \mathsf{V}_{\eta N}$, $s \in [0,\bar{s}]$ and $t \in [-2(\log\log N)^{2}, u-s]$, and  $\varrho_{N} \to \infty$ satisfying $r\le2^\varrho$ and $1\vee(u-t)^{50} \le \varrho_N \le \sqrt{\log N}$, we have 
\begin{equation}
 \lim_{N \to \infty}
\frac{
  |\mathcal R_{N,z}^{\omega}(t,s;\varrho_N)| \dif t
}{ \P\bigl(g_z\in m_{KN}+ \dif t\bigr)} \,
 \frac{\log N} {1+(u+m_{KN}-m_N-\psi_z)_+} = 0  .
\label{e:28c-error}
\end{equation}
\end{prop} 

\subsection{Proofs of trapping landscape results}

With the local estimates from the previous subsection at hand, we now turn to the proofs of the trapping landscape results. Since it will be slightly more convenient to work directly with the field $h$, we begin by introducing the following analogues    of the sets defined in Subsection~\ref{s:1.2.1}.
Given an interval $J \subset \bbR$, a real number $\rho > 0$ and a configuration $\omega \in \bbR_{-}^{\bbZ^2}$, we define
\begin{equation*}
\begin{aligned}
\mathsf{\Gamma}(J) &\coloneq \bigl\{ x \in \mathsf{V}_{N} :\: h_x - m_N \in J \bigr\} \,, \\
\mathsf{\Gamma}_{\rho} (J) &\coloneq \bigl\{ x \in \mathsf{\Gamma}(J) :\:
	\bigl(x + (\mathsf{B}_{N/\rho} \setminus \mathsf{B}_{\rho})\bigr) \cap \mathsf{\Gamma}(J) = \emptyset \bigr\}\,, \\
\mathsf{\Gamma}_{\rho}^*(J, \omega) &\coloneq \bigl\{ x \in \mathsf{\Gamma}_{\rho} (J) :
	h_{x+y} \le  h_x + \omega_y \,, \ \forall y \in \mathsf{B}_{\rho}
		\bigr\} \,,
\end{aligned}
\end{equation*}
and abbreviate $\mathsf{\Gamma}_{\rho}^*(J, \mathbf{0})$ as $\mathsf{\Gamma}_{\rho}^*(J)$. Furthermore, we introduce the shift parameters
\begin{equation*}
	s_N \coloneq -\frac{\gamma}{\alpha} \log_+ \log_+ N \quad \text{and} \quad  s^\pm_N(\delta) \coloneq s_N \pm \frac{1}{\beta} \log \frac{1}{\delta} \quad \text{for } \delta > 0 \,.
\end{equation*}
Observing that $\wh{t}_{N} = \exp  ( \beta(m_{N}+s_{N})  )$, we have the equivalence
\begin{equation}
\label{e:28b}
h_x = m_N + s_N + s \
\Longleftrightarrow \ \wh{\tau}_x = \rme^{\beta s} \,.
\end{equation}
Consequently, the sets defined in Subsection~\ref{s:1.2.1} can be expressed in terms of $\mathsf{\Gamma}$:
\begin{equation*}
\begin{aligned}
	\mathsf{T} (\delta)& = \mathsf{\Gamma}\bigl([s^{-}_{N}(\delta),  \infty) \bigr)   \quad  \text{and} \quad 
\mathsf{T} \bigl([\delta, \delta^{-1}] \bigr) = \mathsf{\Gamma}\bigl([s^{-}_{N}(\delta), s^{+}_{N}(\delta)] \bigr) ; \\
 \mathsf{IT}  (\delta) &= \mathsf{\Gamma} _{r_N} \bigl( [s^{-}_{N}(\delta),\infty) \bigr)   \quad  \text{and} \quad     \mathsf{IT}^{*} (\delta) = \mathsf{\Gamma}^{*}_{r_N}([s^{-}_{N}(\delta),  \infty)   .
\end{aligned} 
\end{equation*}
Recall that $\kappa_N \coloneq (\log N)^{\gamma}$.
 \smallskip

\begin{proof}[Proof of Theorem~\ref{thm:non_reachable_traps}] 
 \noindent\underline{\emph{Step 1.}}  
First we prove \eqref{e:10}. Note that 
\[
\mathsf{DT}^{*}(\delta)  \subset  \mathsf{\Gamma}^{*}_{r_{N}} ([s^{-}_{N}(\delta),\infty ) ) \cap    \mathsf{\Gamma}\bigl([  s^{+}_{N}(\delta), \infty) \bigr) \subset  \mathsf{\Gamma}^{*}_{r_{N}} ([s^{+}_{N}(\delta),\infty ) ) .
\] 

Applying \eqref{e:17} in Corollary~\ref{cor-locmax} with $s=t$,  we obtain that  for any $0 \leq u \leq (\log \log N)^2$ and any $-(\log  \log N)^2 < s \leq u$,  
\begin{align} 
 \E \bigl[ \big| \mathsf{\Gamma}^*_{r_N}\bigl([s, \infty) \bigr) \big| ;\; h^*_{\mathsf{V}_N} \leq m_N + u \bigr]  
 & 
\leq
\sum_{z \in \mathsf{V}_{N}} \int_{s}^{u} \P  \bigl( h_z = h^*_{z+\mathsf{B}_{\rho(u-t)}} \in m_N +  \dif t \,,\,\, h^*_{\mathsf{V}_N} \leq m_N + u \bigr) \notag  \\
& \lesssim   \int_{s}^u (u+1)   \rme^{-\alpha t} \dif t
\lesssim (u+1) \rme^{-\alpha s}  . \label{e:34}
\end{align}
 Here, we  used   $r_N \geq \rho(u-s)$ by the choice of $r_{N}$ in \eqref{def_rn}.  
 Consequently,  combining Markov's inequality with \eqref{e:19} in Corollary~\ref{cor-locmax} implies that  for any $\kappa > 0$, 
\begin{align} 
\P  \bigl( \big|\mathsf{\Gamma}^*_{r_N}\bigl([s, \infty) \bigr)\big| > \kappa \bigr) 
& \leq
\frac{\E \bigl[\big| \mathsf{\Gamma}^*_{r_N}\bigl([s, \infty) \bigr)\big| ;\; h^*_{\mathsf{V}_N} \leq m_N + u 
\bigr]}{\kappa}
+ \P  \bigl( h^*_{\mathsf{V}_N} > m_N + u \bigr) \notag \\
& \lesssim 
(u+1) \bigl( \rme^{-\alpha s}/\kappa + \rme^{-\alpha u} \bigr) \,. \label{e:28} 
 \end{align}
Substituting $s=s^{+}_{N}(\delta)$ and $\kappa = \epsilon \kappa_N$ into \eqref{e:28} yields 
\begin{equation*}
\P  \bigl(|\mathsf{DT}^*(\delta)| > \epsilon \kappa_N \bigr) \le \P  \bigl( 
\big|\mathsf{\Gamma}_{r_N}^*\bigl([s^{+}_{N}(\delta), 
\infty) \bigr)\big| > \epsilon \kappa_N \bigr)
\lesssim (u+1) \bigl( \epsilon^{-1} \delta^{\alpha/\beta} + \rme^{-\alpha u} \bigr) \,.
\end{equation*}
For any fixed $\epsilon > 0$, the right-hand side can be made arbitrarily small by setting $u=u_{\delta}= (1/\delta)^{\frac{\alpha}{2 \beta}}$ and choosing $\delta > 0$ sufficiently small. This proves~\eqref{e:10}.
 
 \smallskip
 \noindent\underline{\emph{Step 2.}} 
We next show  that   non-isolated traps $\mathsf{T}(\delta) \setminus \mathsf{IT}(\delta)=  \mathsf{\Gamma} ([s^{-}_{N}(\delta), \infty) ) \setminus
\mathsf{\Gamma}_{r_N}  ([s^{-}_{N}(\delta), \infty)$ are negligible.
Let $0 \leq u \leq (\log \log N)^2$ and $-(\log 
\log N)^2 
< s \leq u$ as above.  
Let $\mathsf{V}_{N}^{\mathrm{bd}} \coloneq \{z \in \mathsf{V}_N : \wt{N}_{z} \le N \rho(u-s) / r_N\}$ denote the near-boundary vertices, and $\mathsf{V}_{N}^{\mathrm{bk}} \coloneq  \mathsf{V}_N  \setminus \mathsf{V}_{N}^{\mathrm{bd}} $ denotes the bulk.  

Since $\rho(u-s)=o(r_N)$, for each $z \in \mathsf{V}^{\mathrm{bk}}_N$, we have
$\rho(u-s; \wt{N}_z ) = \rho(u-s)$, and 
$\mathsf{B}_{N/r_N}\setminus\mathsf{B}_{r_N}
\subset \mathsf{B}_{\wt{N}_{z}/
\rho(u-s)}
\setminus\mathsf{B}_{\rho(u-s)}$. Applying  \eqref{e:26} in  Proposition~\ref{prop:annular-high-point}, we obtain  
\begin{align}
\E \bigl[  &\, \big|\, \mathsf{V}_N^{\mathrm{bk}}\cap  \mathsf{\Gamma} \bigl([s, \infty)\bigr)  \cap
\mathsf{\Gamma}_{r_N} \bigl([s, \infty)\bigr)^{c}  \,\big| ;\; h^*_{\mathsf{V}_N} \leq m_N + u \bigr] \notag \\
& \leq \sum_{z \in \mathsf{V}^{\mathrm{bk}}_{N}} \int_{s}^{u} \P  \bigl( h_z = m_N + \mathrm{d} t \,,\,\, h^*_{z+(\mathsf{B}_{ \wt{N}_{z} /\rho(u-s)} \setminus 
\mathsf{B}_{\rho(u-s)})} > m_N + s \,,\,\, h^*_{\mathsf{V}_N} \leq m_N + u \bigr)   \notag  \\
& \lesssim  \int_{s}^{u}  (u + 1) (u-t+1) (u-s+1)^{-4}   \rme^{-\alpha t  }  \dif t \lesssim (u + 1) (u-s+1)^{-3}  \rme^{-\alpha s}  \,.\label{e:37-old} 
\end{align}     
For the near-boundary points, we simply drop the non-isolation requirement. Thus,  Corollary \ref{cor-locmax} applied with $s=u$, and the estimate $|\mathsf{V}^{\mathrm{bd}}_N|\lesssim N^2\frac{\rho(u-s)}{r_N}$, imply that  
\begin{equation}
  \E\Bigl[
	\big|\,\mathsf{\Gamma}([s,\infty)) \cap \mathsf{V}^{\mathrm{bd}}_N
	 \,\big|\,;
	h^*_{\mathsf{V}_N}\le m_N+u\Bigr] 
  \lesssim
	\frac{\rho(u-s)}{r_N}(u+1)
	\int_s^u
	(u-t+1)\rme^{-\alpha t}\dif t  = o_N(1)\,. \label{e:37-new}
\end{equation}  
Combining the preceding two bounds \eqref{e:37-old} and \eqref{e:37-new},  applying Markov’s inequality and using the  upper tail estimate of $h^*_{\mathsf{V}_N}$ as in~\eqref{e:28}, we obtain
\begin{equation}
\label{e:38}
\P  \bigl( \big| \mathsf{\Gamma}([s,\infty)) \setminus \mathsf{\Gamma}_{r_N} \bigl([s, \infty) \bigr) \big| > 
\kappa \bigr) 
\lesssim (u + 1) \bigl[(u-s+1)^{-3} \rme^{-\alpha s}/\kappa     + \rme^{-\alpha u} \bigr] + o_N(1)  \,.
\end{equation}
Plugging in $s=s^{-}_{N}(\delta)$, $\kappa = \epsilon \kappa_N$ and $u=\log \log N$, since $ {\rme^{-\alpha s_N^{-}(\delta)}}/{\kappa_N}
=\delta^{-\alpha/\beta}$,   the right-hand side of~\eqref{e:38} tends to zero as $N\to\infty$, for every fixed $\epsilon>0$ and $\delta>0$. This proves~\eqref{e:11}.
 
 \smallskip
 \noindent\underline{\emph{Step 3.}} 
 Finally, we handle the wide traps. Note the inclusion
\[ \mathsf{WT}^{*}(\delta,r) \subset  
\left\{ z \in      \mathsf{\Gamma}_{r_N}^{*} ([s^{-}_{N}(\delta), s^{+}_{N}(\delta)  ])  :    \bigl(z+ (\mathsf{B}_{r_N} \setminus \mathsf{B}_r)\bigr) \cap \mathsf{\Gamma}([s^{-}_{N}(\delta),\infty)) \neq \emptyset   \right\} .
\]

Let $0 \leq u \leq (\log \log N)^2$ and $-(\log \log N)^2 < s < s^{+}< u$. Define  $ \mathsf{V}^{\mathrm{bk}}_N$ as before. Then since $ \rho(u-s) \le r_N$   and $\wt{N}_{z}/\rho(u-s) \ge N/r_N$, we get  
\begin{align*}
& \E \bigl[ \big|\bigl\{ z \in \mathsf{V}^{\mathrm{bk}}_N \cap \mathsf{\Gamma}^*_{r_N} \bigl([s, s^+]   \bigr) :\: 
 z + (\mathsf{B}_{r_N} \setminus \mathsf{B}_{r} ) \cap \mathsf{\Gamma}([s,\infty)) \neq \emptyset \bigr\} \big| 
 ; h^*_{\mathsf{V}_N} \leq m_N + u   \bigr]  \\
& \leq \sum_{z \in \mathsf{V}^{\mathrm{bk}}_N} \int_{s}^{s^+} \P  \bigl( h_z = h^*_{z+\mathsf{B}_{\rho(u-t)}}  \in m_N +  \dif t  \,,\,\,
h^*_{z+(\mathsf{B}_{\rho(u-s)} \setminus \mathsf{B}_{r})} > m_N + s
 \,,\,\, h^*_{\mathsf{V}_N} \leq m_N + u  \bigr) \\
&\qquad  + \sum_{z \in \mathsf{V}^{\mathrm{bk}}_N}  \int_{s}^{s^+} \P  \bigl( h_z = m_N + \mathrm{d} t \,,\,\, h^*_{z+(\mathsf{B}_{ \wt{N}_{z} /\rho(u-s)} \setminus 
\mathsf{B}_{\rho(u-s)})} > m_N + s \,,\,\, h^*_{\mathsf{V}_N} \leq m_N + u \bigr)\,.
\end{align*} 
As in \eqref{e:37-old}, the second sum  is dominated above by $(u + 1) (u-s+1)^{-3}  \rme^{-\alpha s}$. 
For the first sum, \eqref{e:26a} in Proposition~\ref{prop:annular-high-point} shows that, whenever $(1+t-s)^4 \le r$ and $2^r \le \rho(u-s^+)$, it is dominated above by  
\[  \int_{s}^{s^+}(u + 1) (t-s+1)^3 (\log r)^{-1/4}  \rme^{-\alpha t }  \dif t   \lesssim (u + 1)   (\log r)^{-1/4} (s^{+}-s+1)^3  \rme^{-\alpha s}\,.  \]

Following the same argument as in \eqref{e:28}, and combining this with \eqref{e:37-new}, we deduce that  
\begin{align*} 
&\P  \bigl(  \big|\bigl\{ x \in \mathsf{\Gamma}^*_{r_N} \bigl([s, s^+)\bigr) :\: 
x + (\mathsf{B}_{r_N} \setminus \mathsf{B}_{r}) \cap \mathsf{\Gamma} \bigl([s, \infty)\bigr) \neq \emptyset \bigr\} \big| 
	  > \kappa)  \\
& \lesssim (u + 1) \bigl(  [(\log r)^{-1/4} (s^{+}-s+1)^3  + (u-s+1)^{-3}] \rme^{-\alpha s}/ \kappa 
	+ \rme^{-\alpha u} \bigr) + o_N(1). 
\end{align*}
Taking $s = s^{-}_{N}(\delta)$, $s^+ = s^{+}_{N}(\delta)$, and $\kappa = \epsilon \kappa_N$, we obtain
\begin{equation*}
\P  \bigl( |\mathsf{WT}^{*}(\delta, r)| > \epsilon \kappa_N \bigr)  
 \lesssim (u + 1) \bigl( [ (\log r)^{-1/4}   |\log^3 (1/\delta)|+o_N(1)]\epsilon^{-1} (1/\delta)^{\alpha/\beta}   
	+ \rme^{-\alpha u} \bigr) \,.
\end{equation*}
Sending $N \to\infty$ followed by $r \to \infty$ then $u \to \infty$, the desired result follows.
\end{proof} 
\smallskip

\begin{proof}[Proof of Theorem~\ref{thm:sum_shallow_traps}]
We first establish that, for every $x \in \mathsf{ST}(\delta)$,
$\mathsf{B}_{r_N}(x) \subset \mathsf{T}([0,
\delta^{-1}])$.
Suppose, to the contrary, that there exists
$y \in  \mathsf{B}_{r_N}(x) \cap \mathsf{T}((\delta^{-1},\infty))$. Let $z = \argmax_{w \in \mathsf{B}_{r_N}(y)} h_w$.
Since $x \notin \mathsf{NIT}(\delta)$, we necessarily have
$ y,z \in \mathsf{B}_{2 r_N}(x) \cap \mathsf{T}(\delta) \subset  \mathsf{IT}(\delta)$.
If $z \in \mathsf{IT}^{*}(\delta)$, then $z \in \mathsf{DT}^{*}(\delta)$, and hence $x  \in \mathsf{DT}(\delta)$, contradicting $x \in \mathsf{ST}(\delta)$.
Otherwise, there exists $w \in \mathsf{B}_{r_N}(z) \setminus \mathsf{B}_{r_N}(y)$ such that $h_w > h_z$. This yields $w \in y+(\mathsf{B}_{N/r_N} \setminus \mathsf{B}_{r_N})$, contradicting $y \in \mathsf{IT}(\delta)$. Thus the claimed inclusion holds.

Consequently, recalling \eqref{e:28b},
we have, for any $x \in \mathsf{ST}(\delta)$, $h^{*}_{x+ \mathsf{B}_{r_N}} \le m_{N}+s^{+}_{N}(\delta)$.
Moreover, the condition $\wh{\tau}_{x} \le \delta$ implies
$h_{x} \le m_{N}+s^{-}_{N}(\delta)$. We therefore decompose the expectation according to the value of $h_x$. For any
$u \in [0, (\log \log N)^2]$, this gives
\begin{align} 
\E & \Bigl[ \sum_{x \in \mathsf{ST} (\delta)} \wh{\tau}_x ;\; h^*_{\mathsf{V}_N} \leq m_N + u \Bigr] \notag\\
&\lesssim  \sum_{j=1}^{\infty} \delta \rme^{-\beta j} \sum_{x \in \mathsf{V}_{N}}
	\int_{s^{-}_{N}(\delta)-j}^{s^{-}_{N}(\delta)-j+1} \P  \bigl( h_x \in m_N + \dif t \,,\,\,
		h^*_{x+\mathsf{B}_{r_N}} \leq m_N + s^{+}_{N}(\delta) \,,\,\, 
h^*_{\mathsf{V}_N} \leq m_N + u \bigr)  . \label{eq-split-j}
\end{align}  

We split the sum over $j$ into the three disjoint ranges
$\mathcal{J}_1 = \{j \in \mathbb{N} : j \le (\log\log N)^2\}$,
$\mathcal{J}_2 = \{j \in \mathbb{N} : (\log\log N)^2 < j \le \sqrt{\log N}/2\}$,
and $\mathcal{J}_3 = \mathbb{N} \setminus (\mathcal{J}_1 \cup \mathcal{J}_2)$.
For $i=1,2,3$, let $\Sigma_i$ denote the contribution to \eqref{eq-split-j} from indices $j \in \mathcal{J}_i$.

For $j \in \mathcal{J}_{1}$, we have $r_N \geq \rho(u-t)$ for all $t \in [s^{-}_{N}(\delta)-j,s^{-}_{N}(\delta)-j+1]$. Applying Corollary~\ref{cor-locmax} to the integrand therefore gives
\[
\Sigma_{1}  \lesssim  (u+1)  \sum_{j \in \mathcal{J}_1}
	 \delta \rme^{-\beta j} \bigl( j  - \tfrac{2}{\beta} \log \delta \bigr) \rme^{-\alpha [s^{-}_{N}(\delta) - j ]}  \lesssim (u+1) \log(1/\delta) \delta^{1- {\alpha}/{\beta}} \, \kappa_{N} . 
\]
Here we used $\rme^{-\alpha s^{-}_{N}(\delta)} = \delta^{-\alpha/\beta} \kappa_N$ and $\beta > \alpha$.

For $j \in \mathcal{J}_{2}$, we drop the condition on
$h^*_{x+\mathsf{B}_{r_N}}$ from the probability in \eqref{eq-split-j}. Corollary~\ref{cor-locmax} then yields
\[
\Sigma_{2}  \lesssim (u+1) \delta   
	\sum_{j \in \mathcal{J}_2}
		\rme^{-\beta j} \bigl(u-s^{-}_{N}(\delta) + j   \bigr) \rme^{-\alpha [s^{-}_{N}(\delta) - j  ]}   \lesssim (u+1) \log(1/\delta) \delta^{1- {\alpha}/{\beta}} \, \kappa_{N} .
\] 
In the last step, we used $ \bigl(u-s^{-}_{N}(\delta) + j   \bigr) \le  3 j + \frac{1}{\beta} \log(1/\delta)$ for $j \in \mathcal{J}_{2}$.

For $j \in \mathcal{J}_{3}$, we estimate the probability in \eqref{eq-split-j} directly by the Gaussian tail bound
$ \P ( h_{x} -m_{N} - s^{-}_{N}(\delta)  +j \in   [0,1] ) \lesssim   \exp \bigl( - \frac{(m_{N}+ s^{-}_{N}(\delta) -j )^2}{2 \Var(h_{x})}  \bigr)  $.
Since $\max\limits_{x \in \mathsf{V}_{N}}\Var(h_{x}) = g \log N + O(1) $ and $\beta> \alpha$, we obtain
\begin{align*}
 \Sigma_{3} & \lesssim  \sum_{j \ge \sqrt{\log N}/2} \delta \rme^{-\beta j}  N^2    \exp \Bigl( - \frac{(m_{N}+ s^{-}_{N}(\delta)  -j )^2}{2 g \log N }  \Bigr)    
  \lesssim \delta  \sum_{j \ge \sqrt{\log N}/2}  (\log N)^{\gamma+3/2 } \rme^{-(\beta-\alpha) j} \xrightarrow{N \to \infty}  0 .
\end{align*}  

Combining the three estimates and applying Markov's inequality as in~\eqref{e:28}, we conclude that
\begin{equation*}
\P  \Bigl( \sum_{x \in \mathsf{ST} (\delta)} \wh{\tau}_x > \epsilon \kappa_N \Bigr) 
\lesssim 
	(u+1) \bigl[ \epsilon^{-1} \delta^{1 - \frac{\alpha}{\beta}} \log (1/\delta)  + 
	\rme^{-\alpha u} \bigr] . 
\end{equation*}
Sending $N \to\infty$, then $\delta \downarrow 0$, and finally $u \to \infty$ proves the theorem.
\end{proof}
\smallskip

\begin{proof}[Proof of Theorem~\ref{thm:regular traps1}: Upper Bound] 
The derivation of the upper bound~\eqref{e:14} proceeds analogously to the proof of the upper bounds in Theorem~\ref{thm:non_reachable_traps}. By substituting $s=s^{-}_{N}(\delta)$ and $\kappa = M \delta^{-\alpha/\beta} \kappa_N$ into~\eqref{e:28}, we obtain that for all $u \in [0, (\log \log N)^2]$,
\begin{equation*}
\P  \bigl(| \mathsf{RT}^*(\delta, r)| > M \delta^{-\alpha/\beta}  \kappa_N \bigr) \leq \P  \bigl( 
\big|\mathsf{\Gamma}_{r_N}^*\bigl([s^{-}_{N}(\delta), 
\infty) \bigr)\big| > M \delta^{-\alpha/\beta}  \kappa_N \bigr)
\lesssim (u+1) \bigl( M^{-1}   + \rme^{-\alpha u} \bigr) \,.
\end{equation*}
For any fixed $\delta > 0$, the right-hand side can be made arbitrarily small by first selecting $u$ sufficiently large, and subsequently choosing $M$ large enough.
\end{proof}

Proving the lower bound for $|\mathsf{RT}^*(\delta,r)|$ in Theorem~\ref{thm:regular traps1} is rather delicate. To streamline the exposition, we extract the core difficulties into the following lemma, and first provide the proof of the theorem assuming this lemma holds.

\begin{lem}
\label{lem:16a}
Fix any $\epsilon > 0$ and $\omega \in \bbR_{-}^{\mathbb{Z}^2}$ with  $\omega_{\mathbf{0}} = 0$ and $\inf_{y \in \mathbb{Z}^2} \omega_{y}> -\infty$. Then, for all non-degenerate   interval $I \subset \mathbb{R}_+$, there exists $\rho^{-} > 0$ (depending on $\epsilon$, $\omega$ and $I$) such that , 
\begin{equation}
\label{e:65}
\liminf_{N \to \infty} \inf_{s \in [-(\log\log N)^{3/2}, -(\log\log N)^{1/2}]}
\P  \Bigl( \big| \mathsf{\Gamma}^*_{r_N}( [s, \infty), \omega)\cap \mathsf{\Gamma}(s+I) \big| \geq \rho^{-} e^{-\alpha s} \Bigr) \geq 
1-\epsilon .
\end{equation}
\end{lem} 
\smallskip

 \begin{proof}[Proof of  Lower Bound in Theorem~\ref{thm:regular traps1} admitting Lemma \ref{lem:16a}]
 Take $\omega_{\mathbf{0}} = 0$ and $\omega_{x}=-2$ for $x \in \mathbb{Z}^2 \setminus \{\mathbf{0} \}$. Observe  that for all $\delta \in (0,1/2]$ and $r \ge 1$, 
\begin{align}
	 \mathsf{RT}^*(\delta,r) &= \bigl\{x \in \mathsf{\Gamma}^{*}_{r_N} ( [s_N^{-}(\delta), \infty) ) \cap \mathsf{\Gamma} ( [s_N^{-}(\delta), s_N^{+}(\delta) ] )  : h_{x+y} \le m_N+ s_N^{-}(\delta) , \forall y \in \mathsf{B}_{r_N} \setminus \mathsf{B}_{r} \bigr\}  \notag\\
	 &\supset  \bigl\{x \in \mathsf{\Gamma}^{*}_{r_N} ( [s_N^{-}(\delta), \infty) ) \cap \mathsf{\Gamma}  ( [s_N^{-}(\delta), s_N^{-}(\delta)+ \tfrac{1}{\beta} \log 2 ] )  : h_{x+y} \le h_{x}-2 , \forall y \in \mathsf{B}_{r_N} \setminus \{ \mathbf{0}\} \bigr\} \notag \\
	 &= \mathsf{\Gamma}^*_{r_N}( [s^{-}_{N}(\delta), \infty), \omega)\cap \mathsf{\Gamma}([s^{-}_{N}(\delta),s^{-}_{N}(\delta)+\tfrac{1}{\beta} \log 2]). \label{eq:RT-Ga-inclusion}
\end{align} 
Using Lemma \ref{lem:16a} with    $s=s_N^{-}(\delta)$, $I=[0,\tfrac{1}{\beta} \log 2]$  
 yields that for any $\epsilon>0$ there is $\rho^{-}=\rho^{-}_{\epsilon}$ such that 
 \[  \liminf_{N \to \infty} \P \Bigl(  |\mathsf{\Gamma}^*_{r_N}( [s^{-}_{N}(\delta), \infty), \omega)\cap \mathsf{\Gamma}([s^{-}_{N}(\delta),s^{-}_{N}(\delta)+\tfrac{1}{\beta} \log 2]) |\ge    \rho^{-}_{\epsilon} \delta^{-\alpha/\beta} \kappa_N \Bigr) \ge 1- \epsilon\,,   \]
 for all $\delta \in (0,1/2]$.
 The desired result then follows from \eqref{eq:RT-Ga-inclusion}.
\end{proof}

\begin{lem}[Gaussian oscillation bound]
\label{lem:gaussian-oscillation}
Let $T$ be a finite subset of a lattice box in $\mathbb Z^d$, $d\le4$,
of diameter at most $R$, and let $(X_t)_{t\in T}$ be a centered Gaussian
process  satisfying
 that for some
$\lambda,\sigma>0$,
\[
 \max_{t\in T}\E[X_t^2]\le\sigma^2,
 \qquad
 \bigl(\E[(X_s-X_t)^2]\bigr)^{1/2}\le \lambda  |s-t|,
 \quad s,t\in T.
\]
Let $|X|^{*}_T  \coloneq \max_{t\in T}|X_t|$. 
Then   the following assertions hold. 
\begin{enumerate}[(i)]
	\item  There exists a constant $C$ independent of $\sigma,\lambda$ such that $ \E[|X|^{*}_T]
 \le C\sigma [\log (2+\frac{\lambda R}{\sigma} )]^{1/2}$. 
 \item For every $u \ge 0$, $ \P (
  |X|^{*}_T
  >  \E[|X|^{*}_T]+u
 )
 \le 2\exp \{-\frac{u^2}{2\sigma^2} \}.$ 
\item If $\sigma\le A$ and $\lambda R/\sigma\le A$, then, for every
$b>0$, there exists  $C_{A,b}>0$ such that 
$ \E [ \exp\{ b|X|^{*}_T \}]
 \le C_{A,b}$. 
\end{enumerate}  
\end{lem}
\begin{proof}
Apply Dudley's inequality to the signed process
$ \widetilde X_{(\varepsilon,t)}\coloneq \varepsilon X_t$, $
 (\varepsilon,t)\in\{-1,1\}\times T$.
Its supremum is $\max_{t\in T}|X_t|$, its maximal variance is at most
$\sigma^2$, and its canonical-metric covering numbers satisfy
$
 \mathcal N(\widetilde T,d_{\widetilde X},r)
 \le C(1+\frac{\lambda R}{r})^d$. 
Since its canonical diameter is at most $2\sigma$, Dudley's entropy
integral is bounded by
$
 C\sigma [\log(2+\frac{\lambda R}{\sigma})]^{1/2}. $
This proves (i), while Borell--TIS gives (ii).  Integrating the tail in
(ii) proves (iii).
\end{proof}

We now proceed to establish Lemma \ref{lem:16a}. 
 Our approach relies on
a truncated second-moment argument, complemented by a concentration estimate.  
For what follows, 
for any  $N \ge 1$, we denote by $n$ the unique integer satisfying  $  2^n < N \le 2^{n+1} $.  
For   $\eta \in (0,1)$,
$I \subset   \mathbb{R}$  and $M \ge 0$,  define  
\begin{equation}
\label{e:50}
\mathsf{\Gamma}^{\dag}_{r_{N}, \eta}(I, M) \coloneq
\bigl\{ z \in \mathsf{V}_{\eta N}  : h_{z} \in m_{N} + I, \,
h_y \le h_{z} - M, \forall y \in \mathsf{B}_{r_{N}}(z) \setminus \{z\}  \bigr\} \,.
\end{equation}
Note that $\mathsf{\Gamma}^{\dagger}_{r_N,1}(I,\omega) \neq \mathsf{\Gamma}^{*}_{r_N}(I,\omega) $ because $x \in \mathsf{\Gamma}^{\dagger}_{r_N,1}(I,\omega) $ may not belong to $\mathsf{\Gamma}^{*}_{r_N}(I)$.
We first give the following  moment estimates for this set:

\begin{lem}
\label{lem:16c}
Fix any $\eta \in (0,1)$ and $M \ge 0$. Then there exists $c_{\ref{lem:16c}}>0$ such that for all sufficiently large $N \ge 1$, for any $u$ in the range $[0, (\log \log N)^2]$ and $s \in [-(\log \log N)^2, -1]$,
\begin{equation}
\label{e:51}
\E \bigl[ \big| \mathsf{\Gamma}^{\dag}_{r_N, \eta} ([s,s+1],  M) \big| ;\; h^*_{\mathsf{V}_N} \leq m_N + u \bigr] 
\ge c_{\ref{lem:16c}}(u+1) \rme^{-\alpha s} \,,  
\end{equation}
and
\begin{equation}\label{e:51.5}
 \E \bigl[ \big| \mathsf{\Gamma}^{\dag}_{r_N, \eta} ([s,\infty), 0) \big|^2 ;\; h^*_{\mathsf{V}_{  N}} \leq m_N + u \bigr] 
\lesssim  (1+u)^{10} e^{\alpha u } \,  \rme^{-2\alpha s} .
\end{equation} 
\end{lem}

\begin{figure}[tbp]
	\centering
	\includegraphics[scale=0.5]{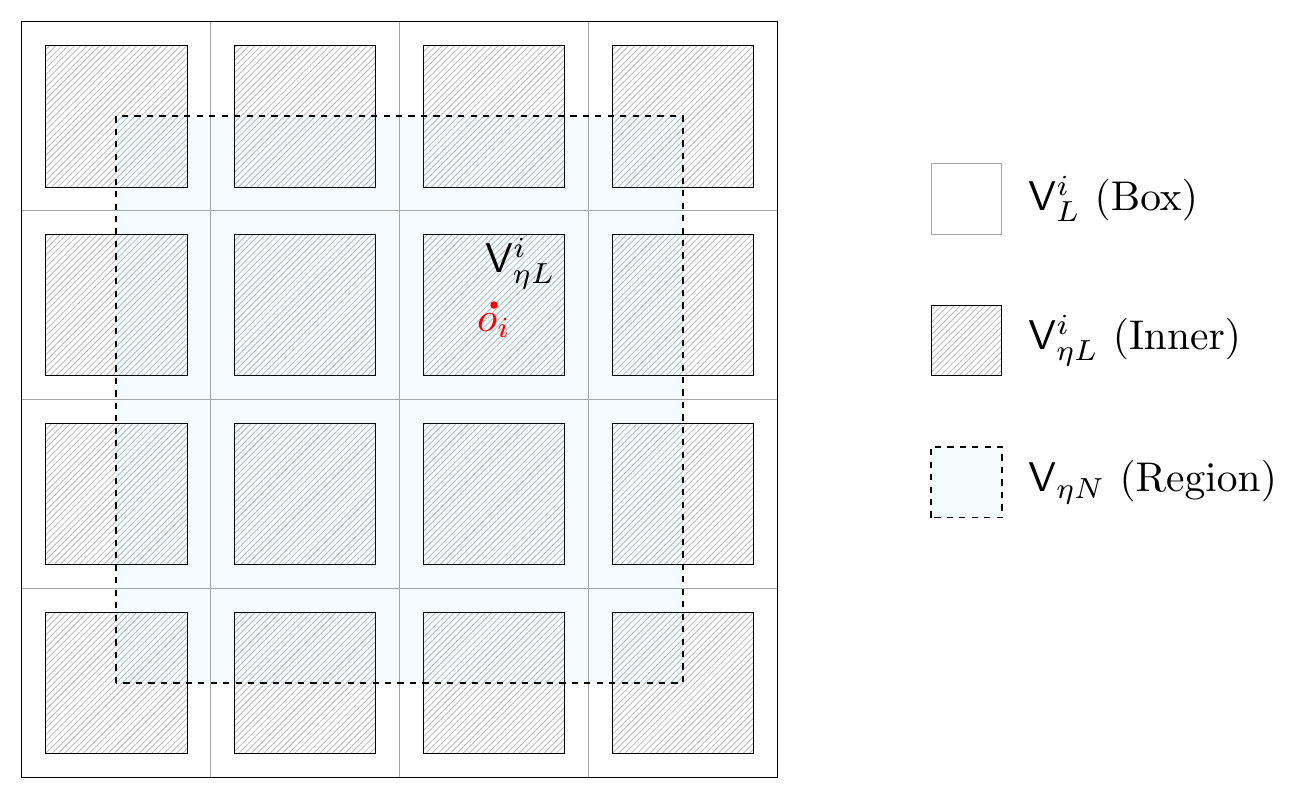}
\caption{Illustration of the decomposition of 
    $\ol{\mathsf{V}}_N$ into $K^2$ sub-boxes 
    $\ol{\mathsf{V}}^i_L$. Each hatched region
    represents an inner domain $\mathsf{V}^i_{\eta L}$.}
	\label{fig:branching}
\end{figure}

Next, fix a large integer $
K=2^{{k}}$ for some ${k} < n$. Set $L= N/K$. Since  $2N+1=K(2L+1)-(K-1)$, we can  tile $\ol{\mathsf{V}}_N = \mathsf{V}_{N} \cup \partial \mathsf{V}_{N}$ using $K^2$ translates of $\ol{\mathsf{V}}_L$, with two adjacent translates shared a column or row of vertices.
More precisely, let $o_1, \dots, o_{K^2} \in \mathsf{V}_{N}$ be such that if we set $\mathsf{V}^i_\bullet \coloneq o_i + \mathsf{V}_\bullet$, then
\begin{equation}
\label{e:57a}
\bigcup_{i=1}^{K^2} \ol{\mathsf{V}}^i_{L} = \ol{\mathsf{V}}_{N}
\ \  ; \ \ \ 
\forall i \neq j \,,\,\, \mathsf{V}^i_{L} \cap \mathsf{V}^j_{L} = \emptyset \,.
\end{equation}
See Figure \ref{fig:branching} for an illustration. 
We then decompose $h$ on $\mathsf{V}_N$ as $h = \varphi + \sum_{i=1}^{K^2} h^i$ where 
\begin{equation}
\label{e:58a}
\varphi_x \coloneq \Phi_{x}^{\mathsf{V}_{N} \mid  \cup_{j=1}^{K^2} \partial \mathsf{V}^j_L}= \E \bigl(h_x \,|\, h_y :\: y \in \cup_{j=1}^{K^2} \partial \mathsf{V}^j_L \bigr)
\  ; \    
h^i_x \coloneq \bigl( h_x - \varphi_x ) 1_{\{x \in \mathsf{V}^i_L\}}\,,
\end{equation}
for $ x \in \mathsf{V}_N $ and $i=1,\dots, K^2.$
The Gibbs--Markov property yields that,  for each $i$,
 $h^i$ is the DGFF on $  \mathsf{V}^i_L$ with zero boundary condition  
on $\partial \mathsf{V}^i_L$; 
and that these fields $(h^{i} : 1 \le i \le K^2)$ are independent of each other and also of the centered 
Gaussian field $\varphi$. 

Now, let us set
\begin{equation}
\label{e:100}
\cI_{N, \eta}  \coloneq \big \{ i = 1, \dots, K^2 :\: \mathsf{V}_L^i \subset \mathsf{V}_{\eta N} \bigr\}
\ ; \quad 
\mathsf{U}_{N,\eta} \coloneq \bigcup_{i \in \cI_{N, \eta}} \mathsf{V}_{\eta L}^i  \,.  
\end{equation}
The decomposition implies that 
\begin{equation}\label{eq-binding-1}
	\E [ 
\varphi_{x}^2] =  \E [ (h_x)^2 ]-\E [(h^{i}_{x})^2] = g \log K + O_{\eta}(1) \quad \forall x \in \mathsf{U}_{N,\eta}  .
\end{equation}
Furthermore it follows from \cite[(3.38) and (3.39)]{BL3} that  there exists a constant $C_{\eta}>0$ such that
\begin{equation}
	\label{eq-binding-2}
  \E  [ | \varphi_{x}  - \varphi_{y}  |^2 ]  = \E  [ | h_{x}  - h_{y}  |^2 ]  - \E  [ | h^{i}_{x}  - h^{i}_{y}  |^2 ]  \le C_{\eta}   \frac{|x-y|^2 }{L^2}  , \text{ for all } x ,y \in \mathsf{V}^i_{\eta L} ,\, i \in \cI_{N, \eta}.
\end{equation}

\begin{proof}[Proof of Lemma \ref{lem:16a} admitting Lemma \ref{lem:16c}] 
In this proof, we take $\eta=4/5$,
and for notation simplicity let $I=[0,1]$. Let $d:=1/8$. This argument readily extends to
general intervals. 
 
Set  $M_{\omega} \coloneq  d - \min_{y} \omega_{y} <\infty$. 
To control the binding field $\varphi$, we introduce  the quantities:
 \begin{equation*}
		|\varphi|^{*}_{i} \coloneq \max_{x \in \mathsf{V}^{i}_{\eta L}}| \varphi_x - 2d |  \  \text{ and } \  \mathrm{D}_{r_{N},\eta}(\varphi) \coloneq \max_{x \in \mathsf{U}_{N, \eta}}  \max_{y \in \mathsf{B}_{r_{N}}}|\varphi_x- \varphi_{x+y}| .  
 \end{equation*}
  
 \noindent\underline{\emph{Step 1.}} 
 We show that the event $ \{   \mathrm{D}_{r_{N},\eta}(\varphi) \leq d\}$ occurs with  high probability.   
 For pairs $(x,z)$ satisfying
$\{x,z\}\cap\mathsf{U}_{N,\eta}\ne\emptyset$ and $|x-z|\le r_N$, set
$\xi_{x,z}\coloneq\varphi_x-\varphi_z$.  Then we have 
$\mathrm{D}_{r_N,\eta}(\varphi)=\max_{x,z}|\xi_{x,z}|$. 

Apply Lemma~\ref{lem:gaussian-oscillation} to the centered Gaussian
process $(\xi_{x,z})$ indexed by these pairs, viewed as a subset of a
box in $\mathbb Z^4$. 
Using \eqref{eq-binding-2} with $\eta$ replaced by $(1+\eta)/2$,
together with the separation between distinct inner boxes, the
parameters in  Lemma~\ref{lem:gaussian-oscillation} can be taken as
$R\asymp N $, $
 \lambda\lesssim_\eta L^{-1} $, 
$ \sigma\lesssim_\eta r_N/L$. 
Consequently,
\begin{equation*}
 \E\bigl[\mathrm{D}_{r_N,\eta}(\varphi)\bigr]
 \lesssim_\eta \frac{r_N}{L}
 [\log (2+ {N}/{r_N} )]^{1/2}
 =o_N(1).
\end{equation*}
For all sufficiently large $N$,  $ \E [\mathrm{D}_{r_N,\eta}(\varphi) ]$ is at most $d/2$.
Part~(ii) of Lemma~\ref{lem:gaussian-oscillation} therefore gives
\begin{equation}
 \P\bigl(\mathrm{D}_{r_N,\eta}(\varphi)>d\bigr)
 \le 2\exp \bigl(  -c_\eta d^2 {L^2}/{r_N^2} \bigr) 
 =o_N(1).
 \label{e:56.5}
\end{equation}   
 
  \smallskip
  \noindent\underline{\emph{Step 2.}} Let $s \in [-(\log\log N)^{3/2},-(\log\log N)^{1/2}]$, and denote $s_K= s+ m_N-m_L = s + 2\sqrt{g}\log K +o_N(1)$.   
  Let $\mathsf{\Gamma}^{\dag, i}_{r_N,\eta}$ be defined as in~\eqref{e:50}, but with $m_L$, $h^i$ and $\mathsf{V}^i_{\eta L}$ in 
place of $m_N$, $h$ and $\mathsf{V}_{\eta N}$. 
 We claim that on the event $ \{   \mathrm{D}_{r_{N},\eta}(\varphi) \leq d \}$, 
 \[ 
 \text{ if }  |\varphi|^{*}_{i}\le d, \text{ then }   \mathsf{\Gamma}^{\dag,i}_{r_N, \eta}  ( [s_K,s_K+d], M_{\omega}  ) \subset  \bigl(\mathsf{\Gamma} ([s,s+1])\cap  \mathsf{\Gamma}^*_{r_N}([s, \infty), \omega)  \bigr) \cup \bigl(   \mathsf{\Gamma}([s,s+1]) \backslash \Gamma_{r_{N}}([s,\infty))\bigr) .
 \] 
To verify this, since $h_{x}-m_N-s=h^{i}_{x}-m_L-s_K+ \varphi_{x} \in [0,1] $, we have $x \in \mathsf{\Gamma}([s,s+1])$. If in addition $x \in \mathsf{\Gamma}_{r_{N}}([s,\infty)) $, then for any  $y \in \mathsf{B}_{r_{N}}\setminus \{0\}$, we have 
\begin{align*}
	h_{x+y}  =  h^{i}_{x+y} + \varphi_{x+y}  & \le   h^{i}_x - M_{\omega}   + (\varphi_{x+y}-\varphi_{x})+\varphi_{x}  \\
	& \le h^{i}_x - M_{\omega}   + d + \varphi_{x} \le   h_x+ \omega_{y}.
\end{align*}
This implies that $x $ lies in $ \mathsf{\Gamma}^*_{r_N}([s, \infty), \omega)$. Consequently,  on the event $\{  \mathrm{D}_{r_{N},\eta}(\varphi) \leq d \}$, we have 
\begin{align}
  & \big| \mathsf{\Gamma}([s,s+2]) \cap \mathsf{\Gamma}^*_{r_N}([s, \infty), \omega)\big| + \big| \mathsf{\Gamma}([s, s+2]) \setminus \mathsf{\Gamma}_{r_N}([s, \infty)) \big|  \notag \\
  &\qquad  \ge 
  \sum_{i \in \cI_{N, \eta}} \big| \mathsf{\Gamma}^{\dag, i}_{r_N, \eta} \bigl( [s_K,s_K+1],  \underline{\omega} \bigr) \big|   \ge    \sum_{i \in \cI_{N, \eta}} \varGamma_i \ind{|\varphi|^{*}_{i}  \le d }   =
  \sum_{i \in \cI^{\le 1}_{N, \eta}} \varGamma_i\,,
  \label{e:57} 
\end{align}
where for notation simplicity we denote 
\[ \varGamma_i \coloneq  | \mathsf{\Gamma}^{\dag,i}_{r_N, \eta}  ( [s_K,s_K+d],  M_{\omega}   )  | \, \mathbf{1}_{\{ h^{*i}_{V^{i}_L} \leq m_L  \}}  \quad \text{and} \quad  \mathcal{I}^{\le d}_{N, \eta} \coloneq \{ i \in | \mathcal{I}_{N, \eta} |: |\varphi|^{*}_{i}  \le d \}. 
\]
Note that $(\varGamma_i)_{i  \in \cI_{N, \eta}}$  are i.i.d.  and  independent of $\varphi$.  Let $i_0$ be any element in $\mathcal{I}_{N, \eta}$.  Conditionally on   $\varphi$,  the  Chebyshev  inequality (and the trivial bound when  $\cI_{N, \eta}^{\le d}$ is empty)   gives 
\begin{equation*}
\P  \Bigl(  \sum_{i \in \cI_{N, \eta}^{\le d} } \varGamma_i  
\leq  
 \frac{1}{2}   \E |\varGamma_{i_0} | \,  |\cI_{N, \eta}^{\le d} |
	\mid \varphi  \Bigr) 
\lesssim  \frac{1}{   |\cI_{N, \eta}^{\le d} | \vee 1} \ \frac{ \E |\varGamma_{i_0}^{2} | } {( \E [\varGamma_{i_0}  ])^2 } . 
\end{equation*}
The moment estimates \eqref{e:51} and \eqref{e:51.5} imply that the   ratio ${ \E |\varGamma_{i_0}^{2} | }/ {( \E [\varGamma_{i_0}  ])^2 }  $ is  bounded above  by a constant   
$c_{\eta,\underline{\omega}}>0$. 
Taking expectation on both sides, we obtain 
\begin{equation}
	\label{eq:sum-Gamma-i}
\P  \Bigl(  \sum_{i \in \cI_{N, \eta}^{\le d} } \varGamma_i  
\leq  
\frac{1}{2}   \E |\varGamma_{i_0} | \,  |\cI_{N, \eta}^{\le d} |
   \Bigr)   
	\le \, c_{\eta,\underline{\omega}} \E \Bigl[   \frac{1}{   |\cI_{N, \eta}^{\le d} |\vee 1 }  \Bigr] .
\end{equation}

  \noindent\underline{\emph{Step 3.}}   
It is sufficient to show that  $|\cI_{N, \eta}^{\le d} | =  \sum_{i \in \cI_{N, \eta}} \ind{|\varphi|^{*}_{i}  \le d }   $ diverges to infinity with high probability as $K \to \infty$. Note that $\ind{|\varphi|^{*}_{i}  \le d }   $  are correlated random variables.  
To extract the independence, we will decompose $\varphi$ as follows.  Fix any small constant $\theta \in (0,1/4)$. Let $ \wt{L}= K^{\theta} L = N/K^{1-\theta}$.
By a grid packing, we may choose
$\wt{\mathcal I}_{N,\eta}\subset\cI_{N,\eta}$ so that the boxes
$(\mathsf{V}_{\wt{L}}^i:i\in\wt{\mathcal I}_{N,\eta})$ are disjoint and
contained in $\mathsf{V}_{\eta N}$, with $|\wt{\mathcal I}_{N,\eta}| \asymp (N/\wt{L})^2  =K^{2-2\theta}$. Here $\mathsf{V}_{\wt{L}}^i$ is centred at
$o_i$ and contains $\mathsf{V}_L^i$.

Apply the Gibbs--Markov property first to the disjoint domains
$\mathsf{V}_{\wt{L}}^i \ ,  i\in\wt{\mathcal I}_{N,\eta}$, and then inside
each of them to $\mathsf{V}_L^i$. This gives the
exact decomposition
\begin{equation*}
 \varphi_x=\phi_x+\wt{\varphi}^{(i)}_{x} \ ,
 \  i\in\wt{\mathcal I}_{N,\eta} \ \text{ and }\  x \in \mathsf{V}^{i}_{L}\,.
\end{equation*}
Here $\phi$ is the conditional expectation of $h$ given the values of $h$ on   the
boundaries of these $\wt{L}$-boxes, while $\widetilde\varphi^{(i)}$ is
the binding field from $\mathsf{V}_{\wt{L}}^i$ to $\mathsf{V}_L^i$: that is, 
\[  (\phi_x) \overset{\mathrm{law}}=\Phi_{x}^{ \mathsf{V}_N \mid \cup_{i\in \wt{\mathcal{I}}_{N,\eta} } \partial \mathsf{V}^{i}_{\wt{L} }  } \quad   \text{ and } \quad  (\wt{\varphi}^{(i)}_x)  \overset{\mathrm{law}}= (\Phi_{x}^{\mathsf{V}^{i}_{ \wt{L} } \mid   \partial \mathsf{V}^i_L} ) .   \]
 The
fields $(\widetilde\varphi^{(i)}:
i\in\wt{\mathcal I}_{N,\eta})$ are independent and identically
distributed, and they are independent of $\phi$. Thus, 
conditioned on $\phi$,   the Bernstein inequality implies
\begin{equation}
	\label{eq:Bern-varphi-sum}
	\P\biggl[ \, \sum_{i\in\wt{\mathcal I}_{N,\eta} } \ind{|\varphi|^{*}_{i}  \le d } \le \frac{1}{2}   \sum_{i\in\wt{\mathcal I}_{N,\eta}} \P[|\varphi|^{*}_{i}  \le d \mid \phi ]  \, \Big | \, \phi \biggr] \le \exp \Bigl( - \frac{1}{100} \sum_{i} \P[|\varphi|^{*}_{i}  \le d \mid \phi ]   \Bigr)\,.
\end{equation} 
From \eqref{eq:Bern-varphi-sum} we deduce that 
\begin{equation}
	\P \Bigl( |\cI_{N, \eta}^{\le d} | \le K^{2-3\theta}  ,   \sum_{i \in \wt{\mathcal{I}}_{N,\eta} } \P[|\varphi|^{*}_{i}  \le 1 \mid \phi ] \ge 2K^{2-3\theta}  \Bigr) \lesssim e^{- K^{2-3\theta} /100} .
	\label{eq:cond-exp-large}
\end{equation}

  We further prove in Step 4 that  
\begin{equation} 	\P \Bigl( 
\sum_{i \in \wt{\mathcal{I}}_{N,\eta} } \P[|\varphi|^{*}_{i}  \le d \mid \phi ]  \le 2K^{2-3\theta} \Bigr) \lesssim \frac{1}{\log K } .
\label{eq:cond-exp-1} 
\end{equation}
 Using $\E[ \Gamma_{i_0}] \gtrsim _{\eta} e^{-\alpha s_K} \gtrsim   K^{-4} e^{-\alpha s} $ by \eqref{e:51}, and   substituting \eqref{eq:cond-exp-large}, \eqref{eq:cond-exp-1} into \eqref{eq:sum-Gamma-i} yields that  
\begin{equation}
	\label{eq:sum-Gamma-i-growing}
	\limsup_{N \to \infty} \sup_{s \in [-(\log\log N)^{3/2}, (\log\log N)^{1/2}]}  \P  \Bigl(  \sum_{i \in \cI_{N, \eta}^{\le 1} } \varGamma_i  
\leq  2 K^{-2-3\theta} e^{-\alpha s}
   \Bigr)  \lesssim\frac{1}{\log K }  \xrightarrow{K \to \infty} 0.	  
\end{equation} 
  Moreover, from   \eqref{e:38} in the proof of Theorem \ref{thm:non_reachable_traps}, we have,  for any $K>0$
\begin{equation}\label{eq:hei-s-nis} 
  \limsup_{N \to \infty}   \sup_{s \in [-(\log\log N)^{3/2}, (\log\log N)^{1/2}]} \,
\P  \Bigl( | \mathsf{\Gamma}([s, \infty)) \setminus \mathsf{\Gamma}_{r_N}([s, \infty)) |   > K^{-4} e^{-\alpha s} \Bigr)  = 0 .
\end{equation}   
Finally, combining \eqref{eq:hei-s-nis},   \eqref{eq:sum-Gamma-i-growing} with \eqref{e:57} and   \eqref{e:56.5},  
we obtain the desired  estimate \eqref{e:65} by taking $K$ sufficiently large.

\medskip   \noindent\underline{\emph{Step 4.}}  It remains to prove the claim  \eqref{eq:cond-exp-1}.  
We bound $\sum_{i \in \wt{\mathcal{I}}_{N,\eta} } \P[|\varphi|^{*}_{i}  \le d \mid \phi ] $ from below by
\[     \sum_{i \in \wt{\mathcal{I}}_{N,\eta} } \P \bigl[  |\wt{\varphi}^{(i)}_{o_{i}}  + \phi_{o_{i}} -2d| \le d/2 \mid \phi \bigr]
\, \P \bigl[ \max_{x \in \mathsf{V}^{i}_{\eta L}} |\wt{\varphi}^{(i)}_{x}-\wt{\varphi}^{(i)}_{o_{i}} | \le d/4  \mid \wt{\varphi}^{(i)}_{o_{i}}  \bigr]
\,  \mathbf{1}_{  \bigl\{ \max\limits_{x \in \mathsf{V}^{i}_{\eta L}} |\phi_x - \phi_{o_{i}}| \le d/4  , |\phi_{o_{i}}| \le \sqrt{\log K} \bigr\} }.
\]
The Green-function estimates used in \eqref{eq-binding-1}--
\eqref{eq-binding-2}, now at the scales $\wt{L},L$, give 
\begin{align}
 &\Var(\phi_{x}) = (1-\theta) g  \log K + O(1)
 \quad ,  \quad   \Var(\wt{\varphi}^{(i)}_{x})
 =\theta g\log K+O (1) \ , \ x \in  \mathsf{V}^{i}_{\eta L} \ ;
 \label{eq:new-var}\\
 &\E[(\phi_x-\phi_y)^2]
 \leq C \frac{|x-y|^2}{{\wt{L}}^2}  \quad ,  \quad 
 \E[(\wt{\varphi}^{(i)}_{x}-\wt{\varphi}^{(i)}_{y})^2]
  \leq C \frac{|x-y|^2}{L^2} \ ,  \ x,y\in\mathsf{V}_{\eta L}^i.
 \label{eq:new-increments}
\end{align} 
 
\begin{enumerate}[(1)]
	\item  By \eqref{eq:new-var} and the  independence, there is  a constant $c_{\theta}>0$,
\[  \P \bigl[  |\wt{\varphi}^{(i)}_{o_{i}}  + \phi_{o_{i}} -2d| \le d/2 \mid \phi \bigr] \ind{ |\phi_{o_{i}}| \le \sqrt{\log K} }\ge  \frac{c_{\theta} }{\sqrt{\log K}} . \] 
\item We assert that  there is a constant  $p_0>0$, independent of $N$, $K$ and $i$, such that 
\begin{equation}
	\label{eq:max-phi-i}
	\P \bigl[ \max_{x \in \mathsf{V}^{i}_{\eta L}} |\wt{\varphi}^{(i)}_{x}-\wt{\varphi}^{(i)}_{o_{i}} | \le d/4  \mid \wt{\varphi}^{(i)}_{o_{i}}  \bigr] \,  \ind{ |\wt{\varphi}^{(i)}_{o_{i}}| \le 2 \sqrt{\log K } }   \ge p_0\,.
\end{equation}
The gaussian orthogonal decomposition yields the field $Y^{(i)}$, define as  $Y_x^{(i)}\coloneq\wt{\varphi}^{(i)}_{x}- q_x \varphi_{o_i}^{(i)} $ with $  q_x\coloneq
 {\E[ \wt{\varphi}^{(i)}_{x} \wt{\varphi}^{(i)}_{o_{i}}]}/
 {\Var(\wt{\varphi}^{(i)}_{o_{i}})}$ for every  $x \in \mathsf{V}^{i}_{\eta L}$, 
is independent of $\wt{\varphi}^{(i)}_{o_{i}}$.
Moreover, from \eqref{eq:new-var},  it follows that $ \sup _{x\in\mathsf{V}_{\eta L}^i} |q_x-1|
 \leq\frac{C }{\log K}$. Since $|\wt{\varphi}^{(i)}_{x}-\wt{\varphi}^{(i)}_{o_{i}} | \le |Y^{(i)}_{x}| + |q_x-1| |\wt{\varphi}^{(i)}_{o_{i}}  |$, \eqref{eq:max-phi-i} follows immediately from the following assertion: 
There is a constant $p_0>0$, independent of $N$, $K$ and $i$, such
that
\begin{equation}
 \P\Bigl(\max_{x\in\mathsf{V}_{\eta L}^i}|Y_x^{(i)}|
             \leq d/2\Bigr)\geq p_0.
 \label{eq:new-residual-small-ball}
\end{equation}

This follows from chaining and the Gaussian correlation inequality. 
For each $j=0,\ldots, \log_{2} L$, let $\mathcal B_j$ be a partition of
$\mathsf{V}_{\eta L}^i$ into boxes whose side lengths are comparable
to $2^{-j}L$. We choose these partitions recursively: starting from
$\mathcal B_0=\{\mathsf{V}_{\eta L}^i\}$, obtain $\mathcal B_j$ by
bisecting each box in $\mathcal B_{j-1}$ in both coordinate
directions. Thus the partitions are nested, and every
$B\in\mathcal B_j$ is contained in a unique parent
$B^-\in\mathcal B_{j-1}$. 
Choose a lattice point $v_B\in B$ for every $B\in\mathcal B_j$, taking
$v_{\mathsf{V}_{\eta L}^i}=o_i$. At the terminal level, each box
contains exactly one lattice point, which is chosen as its
representative.  
Since $\sum_{j\ge1}\frac{d}{(j+4)^2}
\le \frac{d}{2}$, we have 
\[  
\bigcap_{j=1}^{\log_2 L}\Bigl\{ \forall \,  B \in \mathcal{B}_{j} \, , \,  \bigl|Y_{v_B}^{(i)}-Y_{v_{B^-}}^{(i)}\bigr|
\le \frac{d}{(j+4)^2}   \Bigr\} \subset \Bigl\{ \forall \, x \in \mathsf{V}^{i}_{\eta L} , \, |Y_x^{(i)}|
\le \frac{d}{2}  \Bigr\} . 
\]   
Since $v_B$ and $v_{B^-}$ both belong to $B^-$, we have $ 
|v_B-v_{B^-}|\le C2^{-j}L$. This implies  
$
\E [
 (Y_{v_B}^{(i)}-Y_{v_{B^-}}^{(i)} )^2
 ]\le C' 4^{-j}$ by \eqref{eq:new-increments}. 
Since $|\mathcal{B}_{j}| \asymp 4^{j}$, the union bound and Gaussian tail bound give 
\[
\P \Bigl(   \exists B \in \mathcal{B}_{j}, 
\bigl|Y_{v_B}^{(i)}-Y_{v_{B^-}}^{(i)}\bigr|
> \frac{d}{(j+4)^2}
\Bigr)
\leq C'  4^{j} \exp \Bigl\{  -c_\eta\frac{4^j}{(j+4)^4} \Bigr\}.
\]
By the Gaussian correlation inequality
\cite[Theorem 1]{Royen14},  we have 
\[
 \P \biggl(  \bigcap_{j=1}^{\log_2 L}\Bigl\{ \forall \,  B \in \mathcal{B}_{j} \, , \,  \bigl|Y_{v_B}^{(i)}-Y_{v_{B^-}}^{(i)}\bigr|
\le \frac{d}{(j+4)^2}   \Bigr\}  \biggr) \ge \prod_{j=1}^{\infty}  \P \Bigl(   \forall B \in \mathcal{B}_{j}, 
\bigl|Y_{v_B}^{(i)}-Y_{v_{B^-}}^{(i)}\bigr|
> \frac{d}{(j+4)^2}
\Bigr) . 
\]
Since $\sum_{j} 4^{j} \exp \{  -c_\eta\frac{4^j}{(j+4)^4} \}<\infty$,   
the product above  
is bounded below by  
a constant $p_0>0$, independent of 
$N$, $K$ and $i$. 
This proves \eqref{eq:new-residual-small-ball}.

\item Apply Lemma~\ref{lem:gaussian-oscillation} to the centered
Gaussian process $(\phi_x-\phi_{o_i}:x \in \mathsf{V}_{\eta L}^i)$.  
By \eqref{eq:new-increments}, its parameters may be taken as
$  R\asymp_\eta L $, $
 \lambda\lesssim_\eta\wt L^{-1} $ and $
 \sigma\lesssim_\eta L/\wt L=K^{-\theta}$. 
Parts~(i)  of Lemma~\ref{lem:gaussian-oscillation} imply
$ \E [\max_{x\in\mathsf{V}_{\eta L}^i}|\phi_x-\phi_{o_i}|]
 \lesssim_\eta K^{-\theta}$. 
Using the union bound and Lemma~\ref{lem:gaussian-oscillation}  (ii)  we get for all  large $K$,
\[
 \P \Bigl( \max_{i \in \wt{\mathcal I}_{N,\eta} } 
 \max_{x\in\mathsf{V}_{\eta L}^i}|\phi_x-\phi_{o_i}|>d/4
 \Bigr)
 \lesssim_\eta K^2 \, \exp\{-c_\eta K^{2\theta}\}.
\]

  \item   Set $
 Q_K\coloneq \sum_{i\in\wt{\mathcal I}_{N,\eta}}
 \ind{|\phi_{o_i}|\leq\sqrt{\log K}}$.   Moreover, the variance estimate \eqref{eq:new-var}   shows that
$\E Q_K\asymp | \wt{\mathcal I}_{N,\eta}|$. To bound the variance of $Q_K$, we use $\Var(Q_K) \le \E[Q_K]+\sum_{i \neq j} \Cov( \ind{|\phi_{o_i}|\leq\sqrt{\log K}}, \ind{|\phi_{o_j}|\leq\sqrt{\log K}})$. 
The  Gaussian maximal correlation inequality \cite{Geb41} yields 
\begin{equation*}
 \Big|\Cov \bigl(  
 \ind{|\phi_{o_i}|\leq\sqrt{\log K}},
 \ind{|\phi_{o_j}|\leq\sqrt{\log K}}\bigr) \Big| 
 \lesssim \frac{ \E[ \phi_{o_i} \phi_{o_j} ] }{\sqrt{\Var(\phi_{o_{i}}) }  \sqrt{\Var(\phi_{o_{j}}) } } \lesssim \frac{1+\log(N/|o_i-o_j|)}{\log K} . 
\end{equation*}
For each fixed $i$, the regular-grid geometry implies that the number of
centres whose distance from $o_i$ is between
$q \wt{L}$ and
$(q+1) \wt{L}$ is $O (q)$. Summing over
these annuli therefore gives
\begin{align*}
  \Var(Q_K) \lesssim   
  |\wt{\mathcal I}_{N,\eta}|+  \frac{ |\wt{\mathcal I}_{N,\eta}|}{\log K}
 \sum_{q \leq C N/\wt{L} }
 q \Bigl(   1+\log_+
 \frac{N}{q \wt{L} } \Bigr) 
 \lesssim   K^{2-2\theta} + 
 \frac{ K^{4-4\theta} }{\log K}.
\end{align*}
  Chebyshev's inequality gives $\P\left(Q_K\leq  \E[Q_K] /2 \right)
 \lesssim_\eta\frac1{\log K} $. 
\end{enumerate} 
 All together, the required claim \eqref{eq:cond-exp-1}  follows.
\end{proof}

\smallskip  
\begin{proof}[Proof of Lemma \ref{lem:16c}]
By the linearity of expectation, we lower bound the expectation
in~\eqref{e:51} by
\begin{align*} 
& \sum_{z \in \mathsf{V}_{\eta N}} 
\int_{s}^{s+1} \P  \bigl( h_z \in m_N +  \dif t  , h^{*}_{\mathsf{B}_{r_{N}}(z) \setminus \{z\}} \le  m_{N} + t -M , 
	  h^*_{\mathsf{V}_N} \leq m_N + u   \bigr)  \gtrsim (1+u)  e^{-\alpha s}\,,
\end{align*}
as required. Indeed, for $t \in [s,s+1]$ we have $t \le 0 \le u$, and $r_N \le \sqrt{N}$ for all sufficiently large $N$. Thus, the last inequality follows from \eqref{e:33} in Proposition \ref{prop:one-point-lower-bnd}, applied with $\eta_0=(1+\eta)/2$, $r=1$, $\omega\equiv0$, and with the parameter $s$ there set to $t-M$.

Turning to the second estimate \eqref{e:51.5},  note that $\E  [    | \mathsf{\Gamma}^{\dag}_{r_N,\eta} ( [s,\infty),0)  |^2 ; h^*_{\mathsf{V}_{N}} \leq m_N + u  ]$  can rewritten as 
\begin{equation}\label{eq-two-point-1}
  \sum_{x,y \in \mathsf{V}_{\eta N}}
	\P  \bigl( h_x = h^*_{x + \mathsf{B}_{r_N}} \geq m_N + s ,\,  h_y = h^*_{y + \mathsf{B}_{r_N}} \geq m_N + s ,\,  h^*_{\mathsf{V}_{N}} \leq m_N + u \bigr) .
\end{equation}
We decompose this sum into diagonal ($x=y$) and off-diagonal ($x \neq y$) terms.
 
For the diagonal terms, since $r_{N} \ge \rho(u-t)$, we may directly apply
the bound from Corollary~\ref{cor-locmax} to obtain 
\[ \sum_{x \in \mathsf{V}_{\eta N}}
	\P  \bigl( h_x = h^*_{x + \mathsf{B}_{r_N}} \geq m_N + s ,\,     h^*_{\mathsf{V}_N} \leq m_N + u \bigr) \lesssim (1+u) e^{-\alpha s}. \]  
 For the off-diagonal terms, note that the probability in \eqref{eq-two-point-1}
vanishes unless $|y-x| \ge r_{N}$.
Recall that $ 2^n < N \le 2^{n+1} $.  Since $|x-y|<2N\le 2^{n+2}$,  applying Proposition~\ref{prop:two-point-estimate}  with its parameter $\eta$ replaced by $\eta_0=(1+\eta)/2$  yields
\begin{align}
& \sum_{x \in \mathsf{V}_{\eta N}} \sum_{\ell=\log_2 r_N}^{n+2} \sum_{\substack{y \in \mathsf{V}_{\eta N}, \\ 2^{\ell-1}
\leq |x-y| < 2^\ell}}
	\int_{[s,u]^2}
		\P  \bigl( h_x = h^*_{x + \mathsf{B}_{\rho(u-t)}} \in m_N + \dif t ,\,  h_y = h^*_{y + \mathsf{B}_{\rho(u-t')}} \in m_N + \dif t' ,\,
	h^*_{\mathsf{V}_{N}} \leq m_N + u \bigr) \notag\\
& \lesssim (1+u)^{10} e^{\alpha u } e^{-2 \alpha s} \sum_{\ell=\log_2 r_N}^{n+2} \Bigl(
\frac{n+1}{ \ell[1\vee(n+1-\ell)]} \Bigr)^{3/2} \sum_{y \in \mathsf{B}_{2^\ell} \setminus \mathsf{B}_{2^{\ell-1}}} (1+|y|)^{-2} \lesssim (1+u)^{10} e^{\alpha u }   \rme^{-2\alpha s} \label{eq-two-point-2} .
\end{align}
Combining the estimates for the diagonal and off-diagonal contributions,  \eqref{e:51.5} follows.
\end{proof}
\smallskip

\begin{proof}[Proof of Theorem~\ref{thm:regular traps2}]  


For $\omega \in \mathbb{R}_{-}^{\mathbb{Z}^2}$ satisfying   
$\supp(\omega) \subset \mathsf{B}_r$  and $\omega_{\mathbf{0}} = 0$,  let us define 
\begin{equation*}
 \mathsf{NT}^*(\delta ; I, \omega) \coloneq  \{ x \in \mathsf{NT}^*(\delta ) :  \wh{\tau}_x \in I , h_{x+y} \le h_{x} + \omega_{y} ,\,\forall\, y \in  \mathsf{B}_r\} .
\end{equation*} 
Let $\nu$ be the cluster law in \eqref{e:1.10i} (see also \cite[Theorem 2.3]{BL3}).  We write $\nu(\omega) \coloneq \nu (\prod_{y \in \bbZ^2} (-\infty, \omega_y] )$. Define the distribution $\sigma_{\beta}$ on $(0,1]^{\mathbb{Z}^2}$ by  $\sigma_{\beta}(A)= \nu( \{(\omega_{y})_{y \in \mathbb{Z}^2}:   (e^{\beta \omega_{y}})_{y \in \mathbb{Z}^2} \in A \})$.
Equation \eqref{e:16.2} follows once we  establish   
\begin{equation} \label{eq-counting-1}  \limsup_{N \to \infty}\P\biggl( \Big|      \frac{  | \mathsf{NT}^*(\delta  ; I, \omega) | }
 { |\mathsf{NT}^*(\delta ) |}   \, 
- \, \frac{\int_{I} t^{-(\alpha/\beta + 1)} \rmd t}{\int_{\delta}^{\delta^{-1}} t^{-(\alpha/\beta + 1)} \rmd t} \nu(\omega)  \Big| > \epsilon \biggr)=0.
\end{equation}

Fix $\eta \in (0,1)$, and recall that $K=2^{{k}}$ with ${k} \in \mathbb{N}$.
We work with the decomposition introduced in \eqref{e:57a}, \eqref{e:58a}, and \eqref{e:100} (see also Figure \ref{fig:branching}). 
For any $u > 0$, define 
\begin{equation*}
	\cI_{N, \eta ,u} \coloneq  \bigl\{ i  \in \cI_{N, \eta} :  h^{*}_{ \mathsf{V}^{i}_{L} } \le m_{N}   +u \bigr\}  \ ; \ \mathsf{U}_{N,\eta,u} \coloneq \bigcup_{i \in \cI_{N, \eta,u}} \mathsf{V}_{\eta L}^i  \,.  
\end{equation*}
We will need several modifications of the set $\mathsf{\Gamma}$. 
  Accordingly,  for $I\subset \mathbb{R}$ and $\omega \in \bbR^{\bbZ^2}_{-}$ satisfying   
$\supp(\omega) \subset \mathsf{B}_r$  and $\omega_{\mathbf{0}} = 0$, we set 
\begin{equation*} 
\mathsf{\Gamma}^{\ddagger}_{\eta,{k},u}(I, \omega)  \coloneq \bigcup_{i \in \cI_{N, \eta ,u}} \bigl\{ 
	x \in \mathsf{V}_{\eta L}^i : h_x = h^*_{x+ \mathsf{B}_{r_{N}} } \in m_{N}+I ,\,  h_{x+y} \le h_x + \omega_{y} \, \forall \, y \in  \mathsf{B}_r \bigr\} \,,
 \end{equation*}
and write $\mathsf{\Gamma}^{\ddagger}_{\eta,{k},u}(I)$ as  shorthand for $\mathsf{\Gamma}^{\ddagger}_{\eta,{k},u}(I, 0)$.

Our first step is to demonstrate that $\mathsf{NT}^*$ in \eqref{eq-counting-1} can be
effectively replaced by $\mathsf{\Gamma}^{\ddagger}_{\eta,{k},u}$.
 
\begin{lem} \label{lem:15}
Fix $\delta \in (0,1/2]$ and  $\omega \in \bbR_{-}^{\bbZ^2}$ satisfying   
$\supp(\omega) \subset \mathsf{B}_r$  and $\omega_{\mathbf{0}} = 0$. Let $I \subset [\delta,\delta^{-1}]$ and set    $J \coloneq  \{ t: e^{\beta t} \in I \}$.    
Then for any  $\epsilon > 0$   
\begin{equation} \label{e:68}
 \lim_{u \to \infty} \,\limsup_{\eta \uparrow 1   }\,  \sup_{{k} \ge 1}\,  \limsup_{ N \to \infty} \, \P( \big| | \mathsf{NT}^*(\delta  ; I, \omega)
    \triangle \mathsf{\Gamma}^{\ddagger}_{\eta,{k},u}(s_N + J, \omega) \big|
   > \epsilon \kappa_N )  =0.
\end{equation}
\end{lem}

\begin{proof}[Proof of Lemma \ref{lem:15}]  
When $\{h^{*}_{\mathsf{V}_{N}} \le m_{N} + u\}$ occurs,    we have  $\cI_{N, \eta ,u}  = \cI_{N, \eta}$. Thus from \eqref{e:19} it follows 
 \begin{equation}
	\label{eq:I-N-eta+-u}
	\lim_{u \to \infty}  \limsup_{N \to \infty}  \P(\cI_{N, \eta ,u}  \neq \cI_{N, \eta}) \le \lim_{u \to \infty}  \limsup_{N \to \infty}\P( h^{*}_{\mathsf{V}_{N}} > m_{N} + u) = 0 .
 \end{equation}

Assume that $\cI_{N, \eta ,u}  = \cI_{N, \eta}$ holds.
If $x$ belongs to the symmetric difference of the two sets in \eqref{e:68}, then necessarily either $x \in (\mathsf{T}(\delta) \setminus \mathsf{IT}(\delta) )  $  or $x \in (\mathsf{V}_{N} \setminus \mathsf{U}_{N,\eta,u})$ with $h_x = h^*_{x+ \mathsf{B}_{r_{N}} } \ge m_{N}+s^{-}_{N}(\delta)$.
Since   $|\mathsf{V}_{N} \setminus \mathsf{U}_{N, \eta}| \lesssim  (1- \eta^2) K^2 L^2 \lesssim (1-\eta^2) N^{2}$,   
by repeating the argument in \eqref{e:34}, we obtain that for all $-(\log \log N)^2 < s \leq u$,
\begin{equation*}  
\E \Bigl[ \ind{ h^*_{\mathsf{V}_N} \leq m_N + u }\sum_{ x \in \mathsf{V}_{N} \setminus \mathsf{U}_{N,\eta}} \ind{ h_{x}= h^{*}_{x+\mathsf{B}_{r_{N}}}  \ge m_{N}+s }  
 \Bigr]  \lesssim \sum_{x \in \mathsf{V}_{N} \setminus \mathsf{U}_{N, \eta}} \int_{s}^u \frac{1+u}{N^2}\,  \rme^{-\alpha t} \dif t  \lesssim (1-\eta^2) (u+1) \rme^{-\alpha s} \,.
\end{equation*}   
Plugging in $s=s^{-}_{N} (\delta)$ and by using Markov's inequality, we get 
\begin{align}  
&\limsup_{\eta \uparrow 1}  \limsup_{N \to \infty} \P \Bigl(  h^*_{\mathsf{V}_N} \leq m_N + u ,
 \sum_{ x \in \mathsf{V}_{N} \setminus \mathsf{U}_{N,\eta}} \ind{ h_{x}= h^{*}_{x+\mathsf{B}_{r_{N}}}  \ge m_{N}+s^{-}_{N}(\delta) }  > \epsilon \kappa_N /2  \Bigr) \notag \\
 & \ 
\lesssim \limsup_{\eta \uparrow 1}  \  (1-\eta^2) \epsilon^{-1} (u+1)\delta^{-\alpha/\beta}  = 0 .\label{eq:V-minus-U-eta}
\end{align}  
By Theorem \ref{thm:non_reachable_traps}, we  have  
$   \P( | (\mathsf{T}(\delta) \setminus \mathsf{IT}(\delta) ) | > \epsilon \kappa_N /2 ) = o_N(1) $.
Combining this with \eqref{eq:I-N-eta+-u} and \eqref{eq:V-minus-U-eta}, the required result \eqref{e:68} follows.
\end{proof}

Fix $\delta \in (0,1/2]$. Theorem \ref{thm:regular traps1} shows that given $\epsilon_{1} $  there exists $\rho_{\epsilon_1} > 0$  
such that for   $N$ large enough, with probability at least 
$1-\epsilon_{1}/2$, $
|\mathsf{NT}^*(\delta) | > \rho_{\epsilon_1} \kappa_N$.
Then take $\epsilon_{2} < (\epsilon_{1} \wedge \epsilon_{1} \rho_{\epsilon_1}) /4 $. From Lemma \ref{lem:15},  it follows that for  large $u$, and $\eta$ close to $1$,    
 with probability $1-\epsilon_{2}$, we have  $\big| | \mathsf{NT}^*(\delta; I, \omega)  \triangle \mathsf{\Gamma}^{\ddagger}_{\eta,{k},u}(s_N + J, \omega)  \big| \leq 
\epsilon_{2} \kappa_N  $ and   $\big|      | \mathsf{NT}^*(\delta)  \triangle \mathsf{\Gamma}^{\ddagger}_{\eta,{k},u}(s_N + J)  \big| \leq 
\epsilon_{2} \kappa_N $. Then on the intersection of these events it holds
\[ \Big|    \frac{  |\mathsf{\Gamma}^{\ddagger}_{\eta,{k},u}(s_N + J, \omega) | }
 {  |\mathsf{\Gamma}^{\ddagger}_{\eta,k,u}
 ([s_N^-(\delta),s_N^+(\delta)] ) | } -      \frac{  | \mathsf{NT}^*(\delta ; I, \omega) | }
 { |\mathsf{NT}^*(\delta) |} \Big| \le \frac{2 \epsilon_2 \kappa_N  }{ | \mathsf{NT}^*(\delta  ) |} \le \epsilon_1 . \]
We therefore conclude that for any $\epsilon_1>0$, 
\begin{equation}\label{eq-replacing-RT-dagger}
 \limsup_{u \to \infty} \,\limsup_{\eta \uparrow 1  }\, \sup_{{k} \ge 1 }  \,\limsup_{ N \to \infty} \, \P \Bigl(  \Big|    \frac{  |\mathsf{\Gamma}^{\ddagger}_{\eta,{k},u}(s_N + J, \omega) | }
 { |\mathsf{\Gamma}^{\ddagger}_{\eta,k,u}
 ([s_N^-(\delta),s_N^+(\delta)] ) | } -   \frac{  | \mathsf{NT}^*(\delta ; I, \omega) | }
 { |\mathsf{NT}^*(\delta) |}  \Big| > \epsilon_1 \Bigr)  = 0 .
\end{equation}
Consequently, in order  
to show \eqref{eq-counting-1} it now suffices to prove the following result:

\begin{lem}
\label{lem:15.5} 
Fix $\delta \in (0,1/2]$ and  $\omega \in \bbR_{-}^{\bbZ^2}$ with finite support  and satisfying $\omega_{\mathbf{0}} = 0$. Let $I \subset [\delta,\delta^{-1}]$ and set    $J \coloneq  \{ t: e^{\beta t} \in I \}$.   Then the following assertions hold: 
\begin{align} 
&\lim_{u \to \infty} \,\limsup_{   \eta \uparrow 1,   {k}  \to \infty} \limsup_{ N \to \infty} \,  \P\biggl(\bigg|\, \frac{  |\mathsf{\Gamma}^{\ddagger}_{\eta,{k},u}(s_N + J, \omega) | }
 { |  \mathsf{\Gamma}^{\ddagger}_{\eta,{k},u}( [s^{-}_{N}(\delta), s^{+}_{N}(\delta)]) |}  
-  \frac{ \E [ | \mathsf{\Gamma}^{\ddagger}_{\eta,{k},u}( s_{N}+J, \omega) |   \mid \varphi ]}{\E [ | \mathsf{\Gamma}^{\ddagger}_{\eta,{k},u}( [s^{-}_{N}(\delta), s^{+}_{N}(\delta)]) |   \mid \varphi ]}  \, \bigg| > \epsilon \biggr) = 0 ,	\label{eq-counting-2} \\
& \lim_{u \to \infty} \,\limsup_{ \eta \uparrow 1,  \, {k} \to \infty  }    \limsup_{ N \to \infty} \,  \P\biggl( \bigg|\,
  \frac{ \E [ | \mathsf{\Gamma}^{\ddagger}_{\eta,{k},u}( s_{N}+J, \omega) |   \mid \varphi ]}{\E [ | \mathsf{\Gamma}^{\ddagger}_{\eta,{k},u}( [s^{-}_{N}(\delta), s^{+}_{N}(\delta)])  |   \mid \varphi ]} -  \frac{\int_{I} t^{-(\alpha/\beta + 1)} \rmd t}{\int_{\delta}^{\delta^{-1}} t^{-(\alpha/\beta + 1)} \rmd t} \nu(\omega) \, \bigg|  > \epsilon \biggr) = 0	 . \label{eq-counting-3}
\end{align} 
\end{lem}

\begin{proof}[Proof of \eqref{eq-counting-2} in Lemma \ref{lem:15.5}] 
It suffices to show that the following conditional probability vanishes:
\begin{equation}
	\label{eq:Gamma-dagg-devi}
 \limsup_{u \to \infty} \,\limsup_{  {k}  \to\infty}    \limsup_{ N \to \infty} \,  \P  \biggl(  \Big|
	 \frac{| \mathsf{\Gamma}^{\ddagger}_{\eta,{k},u}( s_{N}+J, \omega) |} {\E [ | \mathsf{\Gamma}^{\ddagger}_{\eta,{k},u}( s_{N}+J, \omega) | 
		\mid \varphi ]} - 1 \Big| > \epsilon  \, \Big| \,  \varphi \biggr)  = 0   \ \text{   in probability.}
\end{equation}
Then  \eqref{eq-counting-2} follows by working on the event that  ${  |\mathsf{\Gamma}^{\ddagger}_{\eta,{k},u}(s_N + J, \omega) | }$ and $
 { |  \mathsf{\Gamma}^{\ddagger}_{\eta,{k},u}( [s^{-}_{N}(\delta), s^{+}_{N}(\delta)]) |}  $ both concentrate around their conditional expectations respectively. 

 \smallskip   \noindent\underline{\emph{Step 1.}} 
For each $ i \in \cI_{N, \eta} $, set
\begin{equation*}
\mathsf{\Gamma}^{\ddagger, i}_{\eta,{k} }( s_{N}+J, \omega) \coloneq \{ 
x \in \mathsf{V}_{\eta L}^i :\: h_x = h^*_{x+ \mathsf{B}_{r_{N}} } \in  m_{N} + s_{N} + J, 
h_{x+y} \le h_x + \omega_{y} \, \forall \, y \in  \mathsf{B}_r \} .
\end{equation*} 
Observe that  $
|\mathsf{\Gamma}^{\ddagger}_{\eta,{k},u}( s_{N}+J, \omega)|  = \sum_{i \in \cI_{N, \eta}} 
|  \mathsf{\Gamma}^{\ddagger, i}_{\eta,{k}}( s_{N}+J, \omega)|\ind{i \in \cI_{N, \eta,u} }$. Conditioned on $\varphi$, the random variables   $(|  \mathsf{\Gamma}^{\ddagger, i}_{\eta,{k} }( s_{N}+J, \omega)|\ind{i \in \cI_{N, \eta,u} }: i \in \cI_{N, \eta} )$ are independent.  Chebyshev's inequality then yields 
\begin{equation}
\label{eq-Gd-cond-vari-1}
\P  \biggl[ \, \bigg|
	 \frac{| \mathsf{\Gamma}^{\ddagger}_{\eta,{k},u}( s_{N}+J, \omega) |} {\E [ | \mathsf{\Gamma}^{\ddagger}_{\eta,{k},u}( s_{N}+J, \omega) | 
		\mid \varphi ]} - 1 \bigg| > \epsilon  \,\Big| \, \varphi \biggr]
\leq
\frac{1}{\epsilon^{2}} \frac{  \sum_{i \in \cI_{N, \eta}} \E [ |  \mathsf{\Gamma}^{\ddagger, i}_{\eta,{k}}( s_{N}+J, \omega)|^2 \ind{ i \in \cI_{N, \eta,u} }  \mid \varphi ]}{( \E [ | \mathsf{\Gamma}^{\ddagger}_{\eta,{k},u}( s_{N}+J, \omega) | \mid \varphi ]  )^2} \,.
\end{equation}
 
 \smallskip   \noindent\underline{\emph{Step 2.}} 
We first show  a lower bound for the denominator in \eqref{eq-Gd-cond-vari-1}. By Markov's inequality, for any $\rho>0$ we have 
\begin{align*}
 & \P\bigl(  \E \bigl[ | \mathsf{\Gamma}^{\ddagger}_{\eta,{k},u}( s_{N}+J, \omega) | \mid \varphi \bigr] < \rho \kappa_N \bigr)
 	\le \P\bigl( \,
\P [| \mathsf{\Gamma}^{\ddagger}_{\eta,{k},u}( s_{N}+J, \omega) | < 2\rho \kappa_N \mid \varphi  ] \ge 1/2  \, \bigr) \\ 
& \quad \le 2 \P \bigl(  |\mathsf{\Gamma}^{\ddagger}_{\eta,{k},u}( s_{N}+J, \omega) |  < 2\rho \kappa_N \bigr) \\
& \quad\lesssim \P(  | \mathsf{NT}^*(\delta ; I, \omega) |
   < 4 \rho  \kappa_N ) +  \P( \big| | \mathsf{NT}^*(\delta ; I, \omega)
    \triangle \mathsf{\Gamma}^{\ddagger}_{\eta,{k},u}(s_N + J, \omega) \big|
	   > \rho \kappa_N )\,.
\end{align*} 
Note that $\mathsf{\Gamma}^*_{r_N}( [s^{-}_{N}(\delta), \infty), \omega)  \cap  \mathsf{\Gamma} ( s_N+J)   \subset   \mathsf{NT}^*(\delta ; I, \omega)  $. Then 
 Lemma~\ref{lem:16a}, and Lemma~\ref{lem:15}
 together imply that, for any given $\epsilon$,   we can find $\rho=\rho_{\epsilon,\delta,\omega}>0$  satisfying  
\begin{equation}
	\label{eq-Gd-cond-vari-2}
 \limsup_{u \to \infty} \,\limsup_{ \eta \uparrow 1 } \sup_{k \ge 1}    \limsup_{ N \to \infty}   \P\bigl(  \E \bigl[ | \mathsf{\Gamma}^{\ddagger}_{\eta,{k},u}( s_{N}+J, \omega) | \mid \varphi \bigr] < \rho \kappa_N \bigr) \le \epsilon/2 .
\end{equation}
 
 \smallskip   \noindent\underline{\emph{Step 3.}} 
Next, we derive an upper bound for the second moment appearing in \eqref{eq-Gd-cond-vari-1}.
By the Gibbs--Markov decomposition, conditional on $\varphi$, the fields $(h^{i}_{x} : x \in \mathsf{V}^{i}_{L})$ and $(h_{x} : x \in \mathsf{V}_N \setminus \mathsf{V}^{i}_{L})$ are independent. Consequently, we get 
\begin{align*}
 \sum_{i \in \cI_{N, \eta}} \E [    |  \mathsf{\Gamma}^{\ddagger, i}_{\eta,{k}}( s+J )|^2 \ind{i \in \cI_{N, \eta,u}}\mid \varphi]
 & = \sum_{i \in \cI_{N, \eta}} \frac{  \E  [   | \mathsf{\Gamma}^{\ddagger, i}_{\eta,{k}}( s+J ) |^2 \ind{i \in \cI_{N, \eta,u}}  \ind{ h^{*}_{\mathsf{V}_N \setminus \mathsf{V}^{i}_L} \le m_N +u } \mid \varphi ] }{ \P( h^{*}_{\mathsf{V}_N \setminus \mathsf{V}^{i}_L} \le m_N +u  \mid \varphi)}     \\
  &\leq \frac{1}{ \P(    h^{*}_{\mathsf{V}_{N}} \le m_{N} + u \mid \varphi)}  \sum_{i \in \cI_{N, \eta}} {  \E  [   |  \mathsf{\Gamma}^{\ddagger, i}_{\eta,{k}}( s+J ) |^2 \ind{  h^{*}_{\mathsf{V}_{N}} \le m_{N} + u } \mid \varphi ] }.
\end{align*}  
Since $\E[ \P(    h^{*}_{\mathsf{V}_{N}} > m_{N} + u \mid \varphi) ] = \P(    h^{*}_{\mathsf{V}_{N}} > m_{N} + u ) \lesssim u e^{-\alpha u}$ by Corollary \ref{cor-locmax},   Markov's inequality implies
\[  \limsup_{u \to \infty}\limsup_{N \to \infty} \, \P\bigl(\,  \P[    h^{*}_{\mathsf{V}_{N}} \le m_{N} + u \mid \varphi] < 1/2  \bigr) = 0 . \] 
Furthermore, 
following the argument
in \eqref{eq-two-point-1} and \eqref{eq-two-point-2},  
$\E \bigl[ \sum_{i \in \cI_{N, \eta}} | \mathsf{\Gamma}^{\ddagger, i}_{\eta,{k}}( s^{-}_{N}+J )|^2
\ind{h^{*}_{\mathsf{V}_{N}} \le m_{N}+u} \bigr]$ is bounded from above by
\begin{align*} 
 &   \sum_{i \in \cI_{N, \eta}}   \sum_{x,y \in \mathsf{V}^{i}_{\eta L}} \P  \bigl( h_x = h^*_{x + \mathsf{B}_{r_N}} \geq m_N + s^{-}_{N}(\delta) ,\,  h_y = h^*_{y + \mathsf{B}_{r_N}} \geq m_N + s^{-}_{N}(\delta) ,\,  h^*_{\mathsf{V}_N} \leq m_N + u \bigr) \\
 &\lesssim (1+u)^{10} e^{\alpha u } e^{-2 \alpha s^{-}_{N}(\delta)} \sum_{\ell=\log_2 r_N}^{n-{k}} \Bigl( 
\frac{n+1}{\ell(n+1-\ell)} \Bigr)^{3/2}   \lesssim (1+u)^{10} e^{\alpha u }   (\delta^{-\alpha/\beta} \kappa_{N} )^{2} \frac{1}{\sqrt{{k}}} .
\end{align*}
Therefore,  using Markov's inequality again we obtain  
\begin{equation}\label{eq-Gd-cond-vari-3}
	\limsup_{{k} \to \infty} \limsup_{N \to \infty} \P\biggl( 
\sum_{i \in \cI_{N, \eta}} {  \E  [   |  \mathsf{\Gamma}^{\ddagger, i}_{\eta,{k}}( s_{N}+J ) |^2 \ind{  h^{*}_{\mathsf{V}_{N}} \le m_{N} + u  } \mid \varphi ] } 
>  {\rho^2 \epsilon^{3}} \kappa^{2}_{N}  \biggr)  =  0.
\end{equation} 

Combining \eqref{eq-Gd-cond-vari-2} and \eqref{eq-Gd-cond-vari-3} with \eqref{eq-Gd-cond-vari-1}, we get that, on the complement of the events in  \eqref{eq-Gd-cond-vari-2} and \eqref{eq-Gd-cond-vari-3}, which has probability at least $1-\epsilon$ for large $u, {k},r,N$,
\[ \bigg|\,
	 \frac{| \mathsf{\Gamma}^{\ddagger}_{\eta,{k},u}( s_{N}+J, \omega) |} {\E [ | \mathsf{\Gamma}^{\ddagger}_{\eta,{k},u}( s_{N}+J, \omega) | 
		\mid \varphi ]} - 1 \,\bigg| \le \frac{1}{\epsilon^2} \frac{\rho^2 \epsilon^{3} \kappa^{2}_{N} }{\rho^2 \kappa_N^2}=\epsilon .\]
 Hence the desired result \eqref{eq:Gamma-dagg-devi} follows, and this completes the proof of \eqref{eq-counting-2}.  
\end{proof}

\begin{proof}[Proof of \eqref{eq-counting-3} in Lemma \ref{lem:15.5}]  
Let $G_K(\varphi)$ be the event that, for every
$i\in\cI_{N,\eta}$, the translated restriction of $\varphi$ satisfies
\eqref{e:2.11} on $\mathsf{V}_{\eta'L}^i$, with $M_0 $ replaced by $\theta_{\eta} \sqrt{\log K}$ and $N$ replaced by $L$.

 \smallskip   \noindent\underline{\emph{Step 1.}}
We claim that 
\begin{equation}
 \lim_{K\to\infty}\limsup_{N\to\infty}
 \P\bigl(G_K(\varphi)^{\rm c}\bigr)=0.
 \label{eq:G-K-M0}
\end{equation} 
To this end, we  first set $
 v_K\coloneq2\sqrt g\log K-\frac{\log\log K}{4\alpha}$ and show that 
\begin{equation}
 \limsup_{N\to\infty}
 \P \Bigl( 
 \max_{i\in\cI_{N,\eta}}
 \max_{y\in\mathsf{V}_{\eta'L}^i}|\varphi_y|>v_K
   \Bigr)
 \lesssim_\eta(\log K)^{-1/4}.
 \label{eq:binding-all-box-maximum}
\end{equation}
Fix any $i\in\cI_{N,\eta}$.
By the Gaussian orthogonal decomposition, 
the field $Z$ on $\mathsf{V}^{i}_{\eta' L}$, defined by  
$ Z_{y}= \varphi_y-q_{y} \varphi_{o_i} $ and 
$ q_{y}\coloneq
 {\E[\varphi_y\varphi_{o_i}]}/{\Var(\varphi_{o_i})}$,   is independent of $\varphi_{o_i}$.  The estimates
\eqref{eq-binding-1}--\eqref{eq-binding-2} yield
\[ 
 \sup_{y\in \mathsf{V}^{i}_{\eta' L}} |q_{y}-1|\lesssim_\eta(\log K)^{-1}, \quad \max_{y\in \mathsf{V}^{i}_{\eta' L}}  \Var(Z_{y})\lesssim_\eta1 , \quad
 \bigl(\E[( Z_y-Z_z)^2]\bigr)^{1/2}
 \lesssim_\eta   |y-z|/L.
\] 
Thus Lemma~\ref{lem:gaussian-oscillation}  with
$R\asymp_\eta L$, $\lambda\lesssim_\eta L^{-1}$ and
$\sigma\lesssim_\eta1$, shows that, for
$|Z|^{*} \coloneq\max_{y \in \mathsf{V}^{i}_{\eta' L}}|Z _{y}|$ and every fixed $b>0$, 
\begin{equation}
	\label{eq:Z-is-large}
	  \P(|Z|^{*}> v_K/2) \lesssim_\eta\rme^{-c_\eta v_K^2} \quad , \quad  \sup_{K\ge4}\limsup_{N\to\infty}
	 \max_{i\in\cI_{N,\eta}}\E[\rme^{bZ_i}]<\infty. 
\end{equation} 
Since
$
 \max_y|\varphi_y|
 \le |\varphi_{o_i}|+ O_{\eta}(1)|\varphi_{o_i}|/{\log K}  +|Z|^{*}$,
conditioning on $|Z|^{*}$ and applying the Gaussian tail probability to
$\varphi_{o_i}$ on $\{|Z|^*\le v_K/2\}$ yield
\begin{align}
 \P \Bigl(  \max_{y\in\mathsf{V}_{\eta''L}^i}|\varphi_y|>v_K, |Z|^*\le v_K/2\Bigr) 
 &\lesssim_\eta
 \frac{\sqrt{\log K}}{v_K}
 \exp \Bigl\{ -\frac{v_K^2}{2(g\log K+O_\eta(1))}  \Bigr\} \,
 \E [\rme^{C_\eta(1+Z_i)} ]
 +\rme^{-c_\eta v_K^2} \notag \\
 &\lesssim_\eta K^{-2}(\log K)^{-1/4}.
 \label{eq: phi-o--is-large}
\end{align} 
Combining \eqref{eq:Z-is-large} and \eqref{eq: phi-o--is-large} with the  union bound gives the required claim \eqref{eq:binding-all-box-maximum}. 

For the oscillation, apply  Lemma~\ref{lem:gaussian-oscillation} to
$ \varphi_y-\varphi_{o_i}$.  By
\eqref{eq-binding-2}, the parameters are again
$R\asymp_\eta L$, $\lambda\lesssim_\eta L^{-1}$ and
$\sigma\lesssim_\eta1$.  Parts~(i)--(ii), followed by a union bound,
give
\begin{equation}
 \limsup_{N\to\infty}
 \P \Bigl(  
 \max_{i\in\cI_{N,\eta}}
 \max_{x,y\in\mathsf{V}_{\eta''L}^i}
 |\varphi_x-\varphi_y|> \theta \sqrt{\log K} \Bigr) 
 \lesssim_\eta K^2\rme^{-c_\eta \theta^2  \log K}=o_K(1)\,,
 \label{eq:binding-all-box-oscillation}
\end{equation}
once $\theta=\theta_{\eta}$ is chosen sufficiently large.

Finally, since $\varphi-\varphi_{o_{i}}$ is discrete harmonic in each
$\mathsf{V}_L^i$, the standard interior difference estimate
\cite[Theorem 6.3.8]{LawL2010} gives, on the event that the maximum
inside \eqref{eq:binding-all-box-oscillation} is at most $\theta_{\eta}\sqrt{\log K}$,
\begin{equation}
 \max_{x\ne y\in\mathsf{V}_{\eta'L}^i}
 \frac{|\varphi_x-\varphi_y|}{|x-y|}
 \le\frac{C_\eta}{L}
 \max_{z\in\mathsf{V}_{\eta'L}^i}|\varphi_z-\varphi_{o_i}|
 \le\frac{C_\eta \sqrt{\log K}}{L}\,.
 \label{eq:diff-est-harmonic}
\end{equation}
The   three displays \eqref{eq:binding-all-box-maximum} \eqref{eq:binding-all-box-oscillation} and \eqref{eq:diff-est-harmonic}  verify all three conditions in \eqref{e:2.11}
and prove \eqref{eq:G-K-M0}.  As required for the use of
Proposition~\ref{prop:cluster-law}, $M_K$ is fixed when
$N\to\infty$ for each fixed $K$.

 \medskip   \noindent\underline{\emph{Step 2.}} Set 
$
 a_\delta\coloneq \frac1\beta\log\delta
 $ and $
 b_\delta\coloneq \frac1\beta\log\frac1\delta$.
For $i\in\cI_{N,\eta}$, $x\in\mathsf{V}_{\eta L}^i$ and
$v\in[a_\delta,b_\delta]$, let $p_{x}^{\omega}[v\mid\varphi]$ denote
the conditional density determined by
\begin{align*}
 \P\bigl[ h_x=h^*_{x+\mathsf{B}_{r_N}}
       \in m_N+s_N+\dif v,
       h_{x+y}\le h_x+\omega_y\ \forall y\in\mathsf{B}_r, 
     h^*_{\mathsf{V}_L^i}\le m_N+u\mid\varphi\bigr]=  p_{x}^{\omega}[v\mid\varphi]\dif v\,,
\end{align*}
and put $p_{x} [v\mid\varphi]=p_{x}^{\mathbf{0}}[v\mid\varphi]$. Moreover,
we define 
\[
 \mathcal{E}_{J,N}^{\omega}
 \coloneq \E\bigl[|\mathsf\Gamma^{\ddagger}_{\eta,k,u}
	       (s_N+J,\omega)|\mid\varphi\bigr]
 -\nu(\omega) 
	   \int_J\rme^{-\alpha(s-a_\delta)}\dif s  \, \sum_{i\in\cI_{N,\eta}}
 \sum_{x\in\mathsf{V}_{\eta L}^i}
 p_{x}[a_\delta\mid\varphi]  .\]
 We claim that for fixed parameters $K$, $\delta$, $u$,   $\omega$, 
 \begin{equation}
\lim_{N \to \infty} \, \frac{1}{\kappa_N} \, | \mathcal{E}_{J,N}^{\omega}| \, \ind{G_{K}(\varphi)} = 0 \quad  \text{ in probability. } \label{eq:full-box-mean-J} 
 \end{equation}
 
 Subtracting
the corresponding ratio of the main terms gives
\[
 \bigg| \, \frac{
  \E[|\mathsf\Gamma^{\ddagger}_{\eta,k,u}(s_N+J,\omega)|
       \mid\varphi]}
 {\E[|\mathsf\Gamma^{\ddagger}_{\eta,k,u}
       ([s_N^-(\delta),s_N^+(\delta)])|\mid\varphi]}
 -\nu(\omega)
  \frac{\int_J\rme^{-\alpha s}\dif s}
       {\int_{a_\delta}^{b_\delta}\rme^{-\alpha s}\dif s} \,\bigg|
\le 
 \frac{
 \mathcal{E}^{\omega}_{J,N}
 +
  \mathcal{E}^{\mathbf{0}}_{[a_{\delta},b_{\delta}],N}}
 {\E[|\mathsf\Gamma^{\ddagger}_{\eta,k,u}
       ([s_N^-(\delta),s_N^+(\delta)])|\mid\varphi]}\,.
\]
We then use the lower bound proved in
\eqref{eq-Gd-cond-vari-2}, with $\omega=\mathbf0$ and
$J=[a_\delta,b_\delta]$.  In the iterated limits of
Lemma~\ref{lem:15.5}, with probability tending to one ,
$
 \E\bigl[|\mathsf\Gamma^{\ddagger}_{\eta,k,u}
 ([s_N^-(\delta),s_N^+(\delta)])|\mid\varphi\bigr]$ is of order $ \kappa_N$.
Together with \eqref{eq:G-K-M0}, and \eqref{eq:full-box-mean-J}   we conclude  
\[
 \frac{
  \E[|\mathsf\Gamma^{\ddagger}_{\eta,k,u}(s_N+J,\omega)|\mid\varphi]}
 {
  \E[|\mathsf\Gamma^{\ddagger}_{\eta,k,u}
        ([s_N^-(\delta),s_N^+(\delta)])|\mid\varphi]}
  \xrightarrow{\text{in probability}} \nu(\omega)
  \frac{\int_J\rme^{-\alpha s}\dif s}
       {\int_{a_\delta}^{b_\delta}\rme^{-\alpha s}\dif s} 
\]
in the iterated limits of
Lemma~\ref{lem:15.5}.
Finally, the substitution $t=\rme^{\beta s}$ gives
$
 \int_J\rme^{-\alpha s}\dif s
 =\frac1\beta\int_I t^{-\alpha/\beta-1}\dif t$,
and the two endpoints $a_\delta,b_\delta$ are sent to
$\delta,\delta^{-1}$.  This is exactly \eqref{eq-counting-3}.

 \medskip   \noindent\underline{\emph{Step 3.}} It remains to prove \eqref{eq:full-box-mean-J}.   
Write $f_{x}^{\varphi}$ for the conditional density of $h_x$ given
$\varphi$ and $ \Delta_K\coloneq m_N-m_L$.
On $G_K(\varphi)$, apply Proposition~\ref{prop:cluster-law} in the
translated block $\mathsf{V}_L^i$, with the parameter $N$ there replaced
by $L$, and with
$\psi=\varphi$,  $t=s_N+a_\delta$, and $
 \bar{s}=b_\delta-a_\delta$. 
Since $N=KL$ and hence $m_{KL}=m_N$, the proposition gives, uniformly in
$x$ and $s\in[0,b_\delta-a_\delta]$,
\begin{equation}
 p_{x}^{\omega}[a_\delta+s\mid\varphi]
 =\rme^{-\alpha s}\nu(\omega) \, 
  p_{x}[a_\delta\mid\varphi]
  +R^{\omega}_{x,N}[s\mid\varphi],
 \label{eq:block-cluster-law}
\end{equation}
where, for each fixed $K$, there is a deterministic
$\epsilon_{N,K}\to0$ as $N\to\infty$ such that
\begin{equation}
 \ind{G_K(\varphi)}
 \sup_{0\le s\le b_\delta-a_\delta}
 |R^{\omega}_{x,N}[s\mid\varphi]|
 \le \epsilon_{N,K} \, 
 \frac{1+ u+\Delta_K +|\varphi_x|}{\log L}\,
 f_{x}^{\varphi}(m_N+s_N+a_\delta).
 \label{eq:block-cluster-error}
\end{equation}
Here we used that the convergence in \eqref{e:28c-error} is uniform over
all deterministic fields satisfying \eqref{e:2.11}.

Summing \eqref{eq:block-cluster-law} over
$i\in\cI_{N,\eta}$ and $x\in\mathsf{V}_{\eta L}^i$, and then
integrating over $q\in J-a_\delta$, gives
\[
 \mathcal E_{J,N}^{\omega}
 =
 \int_{J-a_\delta}
 \sum_{i\in\cI_{N,\eta}}
 \sum_{x\in\mathsf{V}_{\eta L}^i}
 R_{x,N}^{\omega}[q\mid\varphi]\,\dif q.
\]
Consequently, \eqref{eq:block-cluster-error} yields
\begin{align}
 \ind{G_K(\varphi)}
 \bigl|\mathcal E_{J,N}^{\omega}\bigr|
 &\le
 (b_\delta-a_\delta)\epsilon_{N,K}
 \sum_{i\in\cI_{N,\eta}}
 \sum_{x\in\mathsf{V}_{\eta L}^i}
 \frac{1+u+\Delta_K+|\varphi_x|}{\log L}\,
 f_x^\varphi(m_N+s_N+a_\delta) \notag\\
 &\lesssim_{K,\delta}  \epsilon_{N,K} \, L^{-2}\kappa_N\rme^{-\alpha\Delta_K}
 \sum_{i\in\cI_{N,\eta}}
 \sum_{x\in\mathsf{V}_{\eta L}^i}
 \bigl(1+u+\Delta_K+|\varphi_x|\bigr)
 \rme^{\alpha\varphi_x}.
 \label{eq:integrated-cluster-error}
\end{align}
Above, in the second inequality we used 
the Gaussian density estimate, the identity
$m_N=m_L+\Delta_K$, and $ 
 {f_x^\varphi(m_N+s_N+a_\delta)} 
 \lesssim_{K,\delta} (\log L )
 L^{-2}\kappa_N\rme^{-\alpha\Delta_K}
 \rme^{\alpha\varphi_x}$.
Hence the right-hand side of
\eqref{eq:integrated-cluster-error} is bounded above, up to a
constant depending only on the fixed parameters, by
\[
 \epsilon_{N,K}L^{-2}\kappa_N\rme^{-\alpha\Delta_K}
 \sum_{i\in\cI_{N,\eta}}
 \sum_{x\in\mathsf{V}_{\eta L}^i}
 \bigl(1+u+\Delta_K+|\varphi_x|\bigr)
 \rme^{\alpha\varphi_x}.
\]
Since, uniformly in $x\in\mathsf{V}_{\eta L}^i$, $\E[\rme^{\alpha\varphi_x}]=O_\eta(K^2)$, $
 \E\bigl[|\varphi_x|\rme^{\alpha\varphi_x}\bigr]
 =O_\eta(K^2\log K)$ and 
$\rme^{-\alpha\Delta_K}=O(K^{-4})$
and there are $O(K^2L^2)$ pairs $(i,x)$, the expectation of
the right-hand side of \eqref{eq:integrated-cluster-error} is
at most
\[
 C_{K,\delta,u} \, \epsilon_{N,K}\, \kappa_N
 =o_N(\kappa_N).
\]
Markov's inequality therefore gives
$
 \ind{G_K(\varphi)}
 \bigl|\mathcal E_{J,N}^{\omega}\bigr|
 =o_{\P}(\kappa_N)$
which proves \eqref{eq:full-box-mean-J}.
\end{proof}
 
\textit{Concluding the proof of Theorem~\ref{thm:regular traps2}. }
The desired result \eqref{eq-counting-1} follows immediately from the assertions in \eqref{eq-replacing-RT-dagger}, \eqref{eq-counting-2} and \eqref{eq-counting-3}. This establishes Theorem \ref{thm:regular traps2}.
\end{proof}

\begin{proof}[Proof of Theorem~\ref{thm:sumability_clusters}]
Let $X^* \coloneq \argmax _{x \in \mathsf{V}_{N}} h_x$ denote the 
location of the maximum. Define
\[ G_N \coloneq  \{   \mathrm{dist}( X^{*},  \mathsf{V}_{N}^{c} )  \ge 2 \sqrt{N},  | h^{*}_{\mathsf{V}_{N}}- m_{N} | \le \sqrt{\log\log N}  \} . \]
We shall show that $\P(G_N^c) \to 0$ as $N \to \infty$. Indeed, the tightness of $(h^{*}_{\mathsf{V}_{N}} - m_{N})_{N \ge 1}$ ensures that 
$\P(|h^{*}_{\mathsf{V}_{N}} - m_{N}| > \sqrt{\log\log N}) \to 0$. 
Furthermore, by Lemma 3.1 in \cite{LS24}, the variance 
satisfies $\max_{x \in \mathsf{V}_{N}}|\Var(h_x) - g \log \mathrm{dist}(x, \partial \mathsf{V}_N)| \le C$. 
A union bound combined with Gaussian tail estimates yields 
\begin{align*}
& \P \bigl(  \mathrm{dist}( X^{*}, \partial \mathsf{V}_{N} )  \le 2\sqrt{N},  | h^{*}_{\mathsf{V}_{N}}- m_{N} | \le \sqrt{\log\log N}  \bigr) 
\\
& \le  \sum_{x:\mathrm{dist}(x, \mathsf{V}_{N}^{c}) \le 2\sqrt{N}}  \P \bigl(h_{x} \ge  m_{N}- \sqrt{\log\log N}   \bigr) \lesssim N^{3/2} \exp \Bigl(- [1+o(1)]\frac{ (2 \sqrt{g} \log N)^2}{2 g \log \sqrt{N}} \Bigr) \xrightarrow{N \to \infty}  0 ,
\end{align*} 
which confirms $\P(G_N^c) \to 0$.

Next  we  claim that Theorem~\ref{thm:sumability_clusters} follows from the following assertion:
\begin{equation}
\label{e:42}
C_{\ref{e:42}}\coloneq \limsup_{N \to \infty} \E \Bigl[ \sum_{x : \mathrm{dist}(x,\mathsf{V}_{N}^{c}) \ge 2\sqrt{N}}   \exp \bigl( \beta(h_x -h^*_{\mathsf{V}_N}) \bigr)     \, \mathbf{1}_{{G_N}} \Bigr] < \infty \,.
\end{equation}
Indeed from Theorem 2.1 in~\cite{BL3}, it follows that, for any $r > 0$ and bounded continuous function $F$ on $\mathbb{R}_{-}^{\mathsf{B}_{r}}$, 
$
\E \bigl[ F \bigl( ( h_{X^* + y} - h_{X^*})_{y \in \mathsf{B}_{r}} \bigr)  \bigr] $ converges to $\int  F \bigl( ( \omega_{y}  )_{y \in \mathsf{B}_{r}} \bigr)   \nu(\dif \omega) 
$ as $N \to \infty$. In particular we have
\begin{equation*}
\E \bigl[ F \bigl( ( h_{X^* + y} - h_{X^*})_{y \in \mathsf{B}_{r}} \bigr) \mathbf{1}_{G_{N}} \bigr] \xrightarrow{N \to \infty}   \int  F \bigl( ( \omega_{y}  )_{y \in \mathsf{B}_{r}} \bigr)   \nu(\dif \omega) .
\end{equation*}  Since the map 
$(\omega_y)_{y \in \mathsf{B}_r} \mapsto \sum_{y \in \mathsf{B}_r} e^{\beta \omega_y}$ 
is continuous and bounded on $\mathbb{R}_{-}^{\mathsf{B}_r}$,  for any fixed $r > 0$, we have,
\begin{align*}
 \int \sum_{y \in \mathsf{B}_{r}} \rme^{\beta \omega_y}  \, \nu(\dif \omega) & =
\lim_{N \to \infty} \E \Bigl[  \sum_{y \in \mathsf{B}_{r}}\rme^{\beta(h_{X^*+y} -h_{X^*})} \mathbf{1}_{G_{N}}  \Bigr]
\\
& \leq \limsup_{N \to \infty}
\E \Bigl[ \sum_{x : \mathrm{dist}(x,\mathsf{V}_{N}^{c}) \ge 2\sqrt{N}}  \rme^{\beta(h_x -h^*_{\mathsf{V}_N})} \mathbf{1}_{G_{N}} \Bigr] 
\le C_{\ref{e:42}}.
\end{align*}
Then letting $r \to \infty$,   the monotone convergence theorem yields $
	\int \sum_{y \in \mathbb{Z}^2} \rme^{\beta \omega_y}    \nu(\dif \omega)   \le   C_{\ref{e:42}} $ as desired.

It  remains to show \eqref{e:42}. We decompose the sum 
based on the level of the field. For low values, 
 \begin{align} 
& \E \Bigl[  \sum_{x: \mathrm{dist}(x,\mathsf{V}_{N}^{c}) \ge 2\sqrt{N}} \rme^{\beta(h_x -h^*_{\mathsf{V}_N})} \mathbf{1}_{\{   h_{x } \le  m_{N} - \sqrt{\log N}\}}   \ind{ | h^{*}_{\mathsf{V}_{N}}- m_{N} | \le \sqrt{\log\log N}   }  \Bigr] \notag\\
& \leq 
e^{\beta \sqrt{\log\log N}}
\sum_{x: \mathrm{dist}(x,\mathsf{V}_{N}^{c}) \ge 2\sqrt{N}}  
 \int_{-\infty}^{-\sqrt{\log N}}   \rme^{\beta t  } \P  \bigl( h_x \in m_N + \dif t\bigr)   \lesssim  e^{\beta \sqrt{\log\log N}} 
N^2  \int_{-\infty}^{-\sqrt{\log N}}  \rme^{\beta t  } e^{- \frac{ (m_{N}+t)^2}{2 g \log N}} \dif t \notag\\
& \lesssim (\log N)^{\frac{3}{2} + o(1)}  \int_{-\infty}^{-\sqrt{\log N}}  \rme^{ (\beta - \alpha ) t  } e^{- \frac{ t^2}{2 g \log N}} \dif t  \lesssim (\log N)^{\frac{3}{2} + o(1)} e^{-(\beta-\alpha) \sqrt{\log N}} \xrightarrow{N \to \infty} 0, \label{eq:low-posi}
\end{align} 
where we used $\beta > \alpha$. For the remaining sites, applying  Corollary~\ref{cor-locmax}  yields that  for any $u$ in the range $[-\sqrt{\log\log N}, 
\sqrt{\log\log N}]$,
\begin{align*} 
& \E \Bigl[  \sum_{x: \mathrm{dist}(x,\mathsf{V}_{N}^{c}) \ge 2\sqrt{N}} \rme^{\beta(h_x -h^*_{\mathsf{V}_N})}  \mathbf1_{\{h_x\ge m_N-\sqrt{\log N}\}} \mathbf{1}_{\{h^*_{\mathsf{V}_N} \in (m_{N}+u-1, m_{N}+u] \}}  \Bigr] \\
& \leq  \sum_{x: \mathrm{dist}(x,\mathsf{V}_{N}^{c}) \ge 2\sqrt{N}}   \int_{-\sqrt{\log N}}^{u}  \rme^{\beta(t - u + 1)} \P  \bigl( h_x \in m_N + \dif t
	\,,\,\, h^*_{\mathsf{V}_N} \leq m_N + u \bigr) \\ 
& \lesssim (1+u_+) \rme^{-c_{\ref{cor-locmax}}(u^-)^2}  \int_{-\infty}^u
	 \rme^{- \beta(u-t)} \rme^{- \alpha t} (1+u-t) \dif t  \, \lesssim (1+u_+) \rme^{-c_{\ref{cor-locmax}}(u_-)^2} \rme^{-\alpha u}  .
\end{align*} 
Summing over $u \in \mathbb{Z} \cap [-\sqrt{\log\log N}, \sqrt{\log\log N}]$ 
and incorporating \eqref{eq:low-posi}, we obtain
\begin{equation*} 
\E \Bigl[   \sum_{x: \mathrm{dist}(x,\mathsf{V}_{N}^{c}) \ge 2\sqrt{N}}   \rme^{\beta(h_x -h^*_{\mathsf{V}_N})}  \ind{ | h^{*}_{\mathsf{V}_{N}}- m_{N} | \le \sqrt{\log\log N}   } \Bigr] \lesssim 1+  \sum_{u \in \mathbb{Z}} (1+u_+) \rme^{-c_{\ref{cor-locmax}}(u_-)^2} \rme^{-\alpha u}  \lesssim 1 .
\end{equation*} 
This establishes~\eqref{e:42} and hence completes the proof. 
\end{proof} 
\smallskip

\begin{proof}[Proof of Theorem~\ref{thm:distance_origin}]The desired result follows by a simple union bound and the Gaussian tail estimate, similar to \eqref{eq:low-posi}: The mean number of vertex in the smaller box   $V_{N/r_N}$ with height  exceeding $2\sqrt{g} \log N - \Theta(\log\log N)$ is   dominated by  that of $|V_{N/r_N}|$-dimensional i.i.d. center Gaussian vector  with variance $g \log N + O(1)$. The typical value of its maximum is $2\sqrt{g} \log N -  \Theta(\log r_N) $; thus typically there is no such vertex. We leave the details to the reader.
\end{proof}
 
\section{Local estimates for DGFF near-extrema}
\label{s:4} 
In this section we prove the estimates stated in Subsection~\ref{sub:4.1}. 
The proofs rely crucially on a \emph{concentric decomposition} of the DGFF 
along dyadic scales, introduced in~\cite{BL3}. We begin by briefly recalling 
this decomposition and collecting several accompanying results from~\cite{BL3} 
that will be used below.  We invite the reader to look in~\cite{BL3}, Section 3.2 in particular, for more details.
\subsection{Preliminaries: Concentric decomposition}
\label{sec:5.1} 
We begin with introducing some notation used in this section. 
We will decompose the DGFF on  $\mathsf{V}_{N}$ along a sequence of concentric boxes. Let us choose a reference point $o \in \mathsf{V}_{N}$, 
 and set 
\begin{equation}\label{eq-def-n'}
	\mathfrak{n}=\mathfrak{n}_{o}\coloneq\max \bigl\{k \ge 0:  \mathsf{B}_{2 ^{k}}(o) \subset \mathsf{V}_{N} \bigr\}\,.
\end{equation}
 Recall that $\wtN _o \coloneq \mathrm{dist}(o,\partial \mathsf{V}_{N})$ and hence 
$   2^{\mathfrak{n}_o} \le  \wtN _o \le  2^{\mathfrak{n}_o+1} $. Moreover, define 
\begin{equation*}
\mathsf{D}_0 \coloneq \mathsf{V}_N \ , \ 
\mathsf{D}_{k} \coloneq   \mathsf{B}_{2^{\mathfrak{n}-k}}(o) \ \text{ for }  1 \le k \le \mathfrak{n}-1   \ , \
\mathsf{D}_{\mathfrak{n}} \coloneq \{o\} \ \text{ and } \   \mathsf{D}_{\mathfrak{n}+1} \coloneq \emptyset .
\end{equation*}
  Observe    that we have 
$ 	\ol{\mathsf{D}}_{k} \coloneq \mathsf{D}_{k} \cup \partial \mathsf{D}_{k} \subseteq \mathsf{D}_{k-1} $ for all $ 1 \le k \le \mathfrak{n}$.
These boxes $\{\mathsf{D}_{k}\}$ depend on the reference point $o$ but we will omit this in notation for simplicity.  Let 
\begin{equation*}
	\mathsf{A}_k \coloneq \mathsf{D}_{k-1} \setminus \mathsf{D}_k \, , \ \forall \, 1 \le k \le \mathfrak{n}+1 .
\end{equation*}  

Recall that for any finite subset $\mathsf{D}$ of $\mathbb{Z}^2$, $h^{\mathsf{D}}$ denotes the DGFF on $\mathsf{D}$ with zero boundary condition on $\mathbb{Z}^2 \setminus \mathsf{D}$. Moreover set 
\begin{equation}\label{def-binding-field}
	 	\Phi^{\mathsf{D} \mid   \mathsf{D}'}_{x} \coloneq \E  \bigl[    h^{\mathsf{D}}_{x} \mid h^{\mathsf{D}}_{y}: y\in \mathsf{D'} \bigr] ,  \quad \forall \, x \in \mathbb{Z}^2.   
\end{equation}
The following decomposition of $h$ is obtained by conditioning successively on the values of $h$ at $\partial \mathsf{D}_k$.  We refer to Proposition 3.3 in~\cite{BL3} for a proof.   

\begin{prop}[Concentric decomposition]
\label{prop:38}
Let $(\varphi_{k})_{k=1}^{\mathfrak{n}+1}$  and $(h_{k})_{k=1}^{\mathfrak{n}}$ be independent of each other, both are sequences of independent random fields  such that 
\begin{enumerate}[(i)]
\item  for   $1 \le k \le \mathfrak{n}$,  $\varphi_{k} \eqd \Phi^{\mathsf{D}_{k-1} \mid \partial \mathsf{D}_{k} }$; and  $ \varphi_{\mathfrak{n}+1}  \eqd h^{\mathsf{D}_{\mathfrak{n}}} = h^{\{o\}}$; 
\item for   $1 \le k \le \mathfrak{n} $,  $h_{k} \eqd h^{\mathsf{D}_{k-1} \setminus \ol{\mathsf{D}}_{k}}$.
\end{enumerate}
Moreover define for each $1 \le k \le \mathfrak{n}+1$, and $x \in \mathbb{Z}^2$,
\begin{equation*}
	\chi_{k}(x) \coloneq \varphi_{k}(x) - \E  \bigl[\varphi_{k}(x) \mid
	\varphi_{k}(o) \bigr]  \, \text{ and } \, b_{k}(x)=\E  \bigl[ \varphi_{k}(x) - \varphi_{k}(o) \,|\, \varphi_{k}(o) = 1 \bigr] .
\end{equation*}
In particular, $\chi_{\mathfrak{n}+1} \equiv 0$ and $b_{\mathfrak{n}+1}= \delta_{o}-1$. Then the following decomposition holds:  
\begin{equation}
	\label{eq-concentric-decomposition}
\bigl( h_x=	h^{\mathsf{D}_0}_x : x \in \mathsf{D}_0 \bigr) \eqd  \Bigl( \sum_{k=1}^{\mathfrak{n}+1}[1+b_{k}(x)]\varphi_{k}(o) +   \sum_{k=1}^{\mathfrak{n}} \chi_{k}(x) + \sum_{k=1}^{\mathfrak{n}} h_{k}(x) : x \in \mathsf{D}_0 \Bigr) . 
\end{equation} 
\end{prop}


An immediate consequence of the decomposition~\eqref{eq-concentric-decomposition}
is the following representation of the field on the \(\ell\)-th annulus
\(\mathsf{A}_{\ell}\). Indeed, on \(\mathsf{A}_{\ell}\) we have
\(\varphi_k\equiv 0\) for \(k\ge \ell+1\) and \(h_k\equiv 0\) for
\(k\neq \ell\). Hence
\begin{equation}
\label{eq-h-annulus-k}
h_x
=
S_{\ell}
+
\sum_{k=1}^{\ell} b_k(x)\varphi_k(o)
+
\sum_{k=1}^{\ell} \chi_k(x)
+
h_{\ell}(x),
\qquad x\in \mathsf{A}_{\ell},
\end{equation}
where
\begin{equation}
\label{def-S-l}
S_{\ell}
\coloneq
\sum_{k=1}^{\ell} \varphi_k(o),
\qquad 1\le \ell\le \mathfrak{n}+1,
\qquad
S_0\coloneq 0.
\end{equation}
Moreover,  Proposition \ref{prop:38}  implies that
\[
S_{\mathfrak{n}+1}=h_o,
\ \text{ and } \
S_{\ell}
=
\Phi^{\mathsf{D}_0\mid \partial \mathsf{D}_{\ell}}_o,
\qquad 1\le \ell\le \mathfrak{n}.\footnote{If one adopts the convention \(\partial \mathsf{D}_{\mathfrak{n}+1}=\{o\}\),
then the case \(\ell=\mathfrak{n}+1\) need not be treated separately.}
\]

The following discussion explains the heuristic usefulness of the above 
decomposition. 
For simplicity, assume that \(o\in \mathsf{V}_{2^{n-1}}\), so
that \(|n-\mathfrak{n}_o|=\Theta(1)\). From~\eqref{eq-h-annulus-k}, together
with the fact that the maximum of
\(h_\ell\eqd h^{\mathsf{D}_{\ell-1}\setminus \overline{\mathsf{D}}_{\ell}}\)
is of order \(m_{2^{\mathfrak{n}-\ell+1}}+O_{\P}(1)\), we see that the
maximum of \(h\) on \(\mathsf{A}_\ell\), after centering by
\(m_{2^{\mathfrak{n}-\ell+1}}\), is well approximated---up to controllable
noise terms---by the inhomogeneous partial sum \(S_\ell\). Moreover, since
\( \Var[\varphi_k(o)]\) is bounded uniformly in \(n\) and \(k\), the process
\((S_\ell)_{\ell=0}^{\mathfrak{n}+1}\) is close to a random walk.

Conditioning on $\{h_{o} = S_{\mathfrak{n}+1} = m_N + t\}$ amounts, by the Gibbs--Markov property,  to adding a deterministic harmonic shift which, on $\mathsf{A}_\ell $ is of order $m_{2^\ell}$.  
This shift   pushes the high values of the GFF $h$ on $\mathsf{A}_{\ell}$ to level $m_{2^{n-\ell+1}}+m_{2^\ell} \simeq m_N + O(\log(\ell \wedge (n-\ell)))$, making them essentially extreme.
Thus, controlling the extreme values of $h$, conditioned on $\{h_{o} = m_N + t\}$ reduces to controlling the ``backbone'' random walk $(S_{\ell})_{\ell=0}^{\mathfrak{n}+1}$, conditioned on $S_{\mathfrak{n}+1} = m_N + t$. 
For example, requiring that $z$ is in addition a global maximum, translates into the requirement that the random walk stays (roughly) below the curve $\ell \mapsto \tfrac{\ell}{n}(m_N + t)$  given the endpoint constraint  $S_{\mathfrak{n}+1} = m_N + t$.  In this way, asymptotics for DGFF extreme events can be read off from corresponding random-walk problems, for which a rich theory is available.

In the remaining of this section, we make this heuristics precise. This involves showing that $(S_{\ell})$ is indeed a random-walk-type process, controlling the ``noise'' terms involving $b_k(x), \chi_{k}(x)$ in~\eqref{eq-h-annulus-k}, estimating the shift in mean resulting from conditioning the field to reach the order $m_N$ at $o$, i.e., $S_{\mathfrak{n}+1}=m_N+t$. Indeed we will  deal with the general case of conditioning on $\{S_{\ell} = t\}$ for any $\ell=1, \dots, \mathfrak{n}+1$.

\subsubsection{The ``shift'' field}
 Now, for $1 \le \ell \le \mathfrak{n}+1$, define the (deterministic) shift field $f_{\ell}$ by
\begin{equation}
\label{e:129}
f_{\ell}(x) : =   \E  \bigl[ h_x \mid S_{\ell} = 1 \bigr] \quad   \text{ for all } x \in \mathsf{D}_0 \setminus \mathsf{D}_{\ell} . 
\end{equation} 
Then the property of Gaussian vectors yields that  for all $t \in \bbR$, 
\begin{equation*}
\P  \bigl( (h_x)_{x \in \mathsf{D}_0 \setminus \mathsf{D}_{\ell}} \in \cdot \, \big| \, S_{\ell} = t \bigr) 
=  \P  \bigl( (h_x + t f_{\ell}(x))_{x \in \mathsf{D}_0 \setminus \mathsf{D}_{\ell}} \in \cdot \, \big| \, S_{\ell} = 0 \bigr)  \,.
\end{equation*}
Moreover, recall that  $S_{\ell}=  \Phi^{\mathsf{D}_0 \mid \partial \mathsf{D}_{\ell}}_o  $ for $1 \le \ell \le \mathfrak{n}+1$ (regarding $\partial \mathsf{D}_{\mathfrak{n}+1}=\{o\}$). We have 
\begin{align}\label{eq-f-by-Green-function}
f_{\ell}(x) &= \frac{\E  [h_x \,    \Phi^{\mathsf{D}_0 \mid \partial \mathsf{D}_{\ell}}_o   ]}{\E  [ (\Phi^{\mathsf{D}_0 \mid \partial \mathsf{D}_{\ell}}_o )^{2}] } = \frac{\E  \bigl[h _x h_o \bigr]}
{\E  [ (h^{\mathsf{D}_{0}}_o )^{2} ]  - \E  [ (h^{\mathsf{D}_{\ell}}_o )^{2} ]}  \notag \\
&=   \frac{G_{\mathsf{D}_0}(x,o)}{G_{\mathsf{D}_0}(o,o) - G_{\mathsf{D}_{\ell}}(o,o)} \quad \text{ for all } x \in \mathsf{D}_0 \setminus \mathsf{D}_{\ell} . 
\end{align} 
Here, recall that $G_\mathsf{V}$ for  $\mathsf{V} \subset \bbZ^2$ is   the discrete Green function on $\mathsf{V}$ with zero boundary conditions outside with   convention $G_{\emptyset} \equiv 0$. We define 
\[ \wedge_{m}^{n}(k) \coloneq \min\{(k-m)_+,(n-k)_+\}\,,  \]
and write  $ \wedge_{0}^{n}(k) $ as $ \wedge^{n}(k)$ for short. 
The following lemma shows that $f_{\ell}$  is rather flat on the annuli in the concentric decomposition.    

\begin{lem}
\label{lemma-3.8} 
There exists $C_{\ref{lemma-3.8}}>0$ such that  the following properties hold for the shift fields $f_{\ell}$ with $1  \le k \le \ell \le \mathfrak{n}+1$:
\begin{align}
	\max_{x \in \mathsf{A}_k} \big| f_{\ell}(x)   &- \frac{k}{\ell} \big| \le  \, \frac{C_{\ref{lemma-3.8}}}{\ell} \quad \text{ and }, \label{eq-f-asymptotic} \\
	\max_{x \in \mathsf{A}_k} \Big|  m_{2^{\mathfrak{n}}}  
  - \bigl(m_{2^{\mathfrak{n}}} &- m_{2^{ \mathfrak{n}-\ell+1}}  \bigr) f_{\ell}(x) - m_{2^{ \mathfrak{n}-k+1}}  \Big| \le  C_{\ref{lemma-3.8}} [1 + \log_+ \wedge^{\ell}(k)].  \label{eq-m-and-f}
\end{align} 
\end{lem}

\begin{proof}
Firstly, by (3.8) in \cite[  Lemma 3.1]{LS24}, there exists $C>0$ such that  for all $x \in \mathsf{D}_0$
\begin{equation}\label{eq-Green-asymp}
	\Big| G_{\mathsf{D}_0}(x,o) - g \log \frac{\mathrm{dist}( \{ x,o\},  \partial \mathsf{D}_0) }{ |x-o|\vee 1}  \Big| \le C \Bigl[  1+\frac{|x-o|}{\mathrm{dist}(\{x,o\}, \partial \mathsf{D}_0)} \Bigr] . 
\end{equation}
Note that uniformly for $x \in \mathsf{D}_{1}$, $\mathrm{dist}(x,\partial \mathsf{D}_0) \asymp \mathrm{dist}(o,\partial \mathsf{D}_0) \asymp   2^{\mathfrak{n}} $; and  for $x \in \mathsf{A}_k \coloneq\mathsf{D}_{k-1} \setminus \mathsf{D}_k $ with $k \ge 2$, we have  $|x-o|  \asymp 2^{\mathfrak{n}-k}$.  Employing \eqref{eq-f-by-Green-function} and \eqref{eq-Green-asymp} we get  
\begin{equation*}
f_{\ell}(x) =  
\frac{ g \log 2^{\mathfrak{n}-(\mathfrak{n}-k)} + O (1)}
{g \log 2^{\mathfrak{n}} - g \log 2^{\mathfrak{n}-\ell}  + O (1)} = \frac{k + O(1)}{\ell+O(1)} \quad   \text{ for all } x \in \mathsf{A}_k, 
\end{equation*} 
which shows the first assertion \eqref{eq-f-asymptotic} for the case  $k \ge 2$.  

Furthermore,  we claim that there exists a constant $C>0$ such that  
\begin{equation*}
	G_{\mathsf{D}_0}(x,o) \le C \quad \text{ for all } \quad  x \in \mathsf{A}_{1}= \mathsf{D}_0 \setminus \mathsf{D}_{1} .
\end{equation*} 
Then repeating the previous argument gives the estimate \eqref{eq-f-asymptotic} in the case $k=1$. 
To this end, we can embed $\mathsf{D}_0=\mathsf{V}_{N}$ in a discrete upper half plane $\mathsf{H}$ with $\mathsf{V}_{N} \subset \mathsf{H}$ and $\mathrm{dist}(o, \partial \mathsf{H})=\mathrm{dist}(o,\partial \mathsf{V}_{N})$.  The monotonicity of the Green function yields that for $x \in \mathsf{V}_N$,
\begin{equation*}
	G_{\mathsf{D}_0}(x,o)  \le  G_{\mathsf{H}}(x,o)  = a(x-\overline{o})- a(x-o) = g \log \frac{|x-\overline{o}|}{|x-o|} + O(1)\,,
\end{equation*}
where $a$ is the potential kernel with well-known asymptotic $a(y)= g\log |y| + O(1)$ and $\overline{o}$ is the conjugate of $o$ (i.e., $\overline{o}$ is the reflection of $o$ across $\partial \mathsf{H}$). The claim   follows from the facts that $ |x-\overline{o}| \le 9 |x-o| $ for all $x \in \mathsf{H}$ with $|x-o| \ge \mathrm{dist}(o,\partial \mathsf{H})/4$.

Finally we  prove  \eqref{eq-m-and-f}. Using the estimate in \eqref{eq-f-asymptotic}, we get that  for $x \in \mathsf{A}_k$,
\begin{equation*} 
\begin{split}
(m_{2^{\mathfrak{n}}} - m_{2^{\mathfrak{n}+1-\ell} }) f_{\ell}(x)
& = \sqrt{g} \Bigl[ (2\log 2)  \ell - \frac{3}{4} (\log (\mathfrak{n}+1) - \log (\mathfrak{n}+2-\ell)) + O(1) \Bigr] \frac{k+O(1)}{\ell}    \\
& = 2\sqrt{g} (\log 2) k- \frac{3}{4} \sqrt{g}  \frac{k}{\ell} \log \frac{\mathfrak{n}+1 }{\mathfrak{n}+2 -\ell} + O(1) \,.
\end{split}
\end{equation*}
Thus the left-hand side in \eqref{eq-m-and-f} is   dominated above by 
\begin{equation*}
	1+ \big|\frac{k}{\ell}   \log \frac{\mathfrak{n}+1 }{\mathfrak{n}+2 -\ell} -   \log_+ \frac{\mathfrak{n}+1 }{\mathfrak{n}+2 -k} \big| \le \frac{\ell-k}{\ell}   \log_+ \frac{\mathfrak{n}+1 }{\mathfrak{n}+2 -k}      +  1+ \frac{k}{\ell}      \log_+ \frac{\mathfrak{n}+2-k }{\mathfrak{n}+2 -\ell} .
\end{equation*}
On the right-hand side, the first term is bounded by $ \frac{ \ell-k}{\ell} \log_+ \frac{\ell }{\ell+1 -k} \le \sup_{t \in (0,1]} t \, \log (1/t)<\infty$.  For the second term,  using the inequality $\lambda \log(1+ t) < \log(1+ \lambda t )$ for $\lambda\in(0,1), t>0 $, we have 
$\frac{k}{\ell}      \log_+ \frac{\mathfrak{n}+2-k }{\mathfrak{n}+2 -\ell}   \le \log (1+k)$.  
On the other hand, the same term is bounded by  $\log ( 1+ \ell-k ) \le 1 + \log_+(\ell-k)$. This completes the proof. 
\end{proof}

\subsubsection{The ``Back-bone'' random walk} 
For $1 \le k \le \mathfrak{n}$, define 
\begin{equation}\label{eq-centered-h}
	 \wh{h}_{k}(x)\coloneq h_{k}(x)- (m_{N} \mathbf{1}_{\{k=1\}} +m_{2^{\mathfrak{n}-k}} \mathbf{1}_{\{k \ge 2\}}) .
\end{equation} 
The next lemma controls the process 
\((S_\ell)_{\ell=0}^{\mathfrak {n}+1}\), together with the accompanying 
``noise'' terms in~\eqref{eq-h-annulus-k}.

\begin{lem}[{\cite[Lemmas 3.6--3.9]{BL3}}]
\label{lem:39} 
There exist constants $c_{\ref{lem:39} }, C_{\ref{lem:39} } > 0$ such that the following hold:
\begin{enumerate}[(i)]
	\item There are constants $0 < \sigma_{\min} < \sigma_{\max} < \infty$ with $\Var \bigl[ \varphi_{k}(o) \bigr] \in  (\sigma^2_{\min}, \sigma^2_{\max} )$  for all $ 1\le k \le \mathfrak{n}+1 $, and 
	\[
	 \Var \bigl[ \varphi_{k}(o) \bigr] \to  g \log 2  \ \text{ as  } \wedge^{\mathfrak{n}+1}(k) \to \infty. 
	\]
	\item For every $1 \le k \le \mathfrak{n}$, the function $b_k$ is discrete harmonic on $\mathbb{Z}^2 \setminus (\partial \mathsf{D}_k \cup \partial \mathsf{D}_{k-1})$ and uniformly bounded in $k$. It satisfies $b_k(x) = -1$ for all $x \in \mathbb{Z}^2 \setminus \mathsf{D}_{k-1}$ and $b_k(x) \ge -1$ for $x \in \mathsf{D}_{k-1}$. Moreover, for every $x \in \mathsf{D}_k$,
	\[
		|b_{k}(x)| \le C_{\ref{lem:39} } \, \frac{\mathrm{dist}(x,o)}{\mathrm{dist}(o, \partial \mathsf{D}_k)} .
	\]
	\item For $1 \le k \le \mathfrak{n}$ we have $\chi_{k}(x) = 0$ whenever $x \notin \mathsf{D}_{k-1}$. For   $k+1 \le r \le \mathfrak{n}$, $
		\E  \bigl[ \max_{x \in \mathsf{D}_{r}} |\chi_{k}(x)| \bigr] \le C_{\ref{lem:39} } \, 2^{-(r-k)} $.
	In addition, for every $\lambda > 0$,
	\[
	\P  \Bigl( \big| \max_{x \in \mathsf{D}_{r}} |\chi_{k}(x)| -  \E \bigl[ \max_{x \in \mathsf{D}_{r}} |\chi_{k}(x)| \bigr] \big| > \lambda\, 2^{-(r-k)} \Bigr) \le 2\, \mathrm{e}^{-c_{\ref{lem:39} } \lambda^{2}} .
	\]
	\item  Then for $1 \le k \le \mathfrak{n}$ with $\chi_0 \equiv 0$, we have 
\begin{align*}
\P  \Bigl(  \max_{x \in \mathsf{A}_{k}} \{ \chi_{k-1}(x) + \chi_{k}(x) + \wh{h}_{k}(x) \}   \in (-\infty,-u)\cup(u^2,\infty)  \Bigr)  \le C_{\ref{lem:39} }\, \mathrm{e}^{-c_{\ref{lem:39} } u^2} \ , \   \forall \, u >0. 
\end{align*} 
\end{enumerate} 
\end{lem}

\begin{proof} 
By adapting the arguments from \cite[Lemmas 3.6--3.9]{BL3}, we obtain assertions (ii)--(iv) and assertion (i) for $k \ge 2$. The remaining case is a direct consequence of \eqref{eq-Green-asymp} and the relation $\Var  [ \varphi_{1}(o) ] = G_{\mathsf{D}_0}(o,o) - G_{\mathsf{D}_1}(o,o) \asymp 1$.
\end{proof}
 
Next, we introduce the control variables $ \mathscr{M}_{\ell} $ as follows. 
Define for $m \ge 2$, 
\[ 
\vartheta_{m}( \wedge^{\ell}(k) ) \coloneq  \log \bigl(\max\{m,   \wedge^{\ell}(k)\} \bigr).  \]
For each $1 \le k \le \mathfrak{n}+1$,  we introduce 
\begin{align*}
	R^{1}_{k} \coloneq |\varphi_{k}(o)|, \quad R^{2}_{k}\coloneq \max_{ k+1 \le r \le \mathfrak{n}}  \, \max_{x \in \mathsf{D}_{r}} \,  2^{r-k} | \chi_{k}(x)| ,\quad  R^{3}_{k} \coloneq  \max_{x \in \mathsf{A}_k}   \bigl\{ \chi_{k-1}(x) + \chi_{k}(x) + \wh{h}_{k}(x) \bigr\}  ,
\end{align*}
with the understanding that $R^{2}_{k}=0 $ for $k \ge \mathfrak{n}$ and $\wh{h}_{\mathfrak{n}+1}\equiv 0$.
 Let  $R_{k} \coloneq \max\{ R^{1}_{k} ,R^{2}_{k} , (R^{3}_{k})_+ ^{1/2},   (R^{3}_{k})_{-}\}$.  Finally,  we define , for $1 \le \ell \le \mathfrak{n}+1$,
\begin{equation}\label{eq-control-variable}
	 \mathscr{M}_{\ell} \coloneq \min \bigl\{ m \ge 3 , m \in \mathbb{N}: \  R_k  \le   \vartheta_{m}( \wedge^{\ell}(k) ) \text{ for all }  1 \le k \le \ell \bigr\}.    
\end{equation}
By definition, either  $\mathscr{M}_{\ell} \le \frac{\ell}{2}$ or  $\mathscr{M}_{\ell} \ge  {\ell}/{2}$ and   $\max_{1 \le k \le \ell}R_k \le \log \mathscr{M}_{\ell} $.
The following lemma shows that this control variable provides a rather tight comparison between the maximum of the field in the annuli $\mathsf{A}_k$  and the corresponding increments of the above random walk.
 

\begin{lem}
\label{lem:40}
There exists a constant $C_{\ref{lem:40}}>0$ such that for all  $1 \le \ell \le  \mathfrak{n}+1$ and $ 1 \le k \le \ell $, 
\begin{align*}
   \max_{x \in \mathsf{A}_k} h_{x}  - (m_{N} \mathbf{1}_{\{k=1\}} +m_{2^{\mathfrak{n}-k}} \mathbf{1}_{\{k \ge 2\}}) &   
 \in [ S_{k}- C_{\ref{lem:40}} \, \vartheta_{\mathscr{M}_{\ell}}(\wedge^{\ell}(k))  \, , \,S_{k}+  C_{\ref{lem:40}} \, \vartheta_{\mathscr{M}_{\ell}}(\wedge^{\ell}(k))^{2} ] \, ; \\
  \max_{x \in \mathsf{D}_{k+1}} \big| \Phi^{\mathsf{D}_0 \mid \partial \mathsf{D}_{k}}_{x} - S_k \big|  &  \leq C_{\ref{lem:40}}   \,   \vartheta_{\mathscr{M}_{\ell}}(\wedge^{\ell}(k)) \,.
\end{align*} 
Moreover,   $
|S_{j}| \vee |S_{\ell} - S_{\ell-j}|  \le  j\log  (\mathscr{M}_{\ell} \vee j)$ for every $j \le  \ell/2$.  
\end{lem}

\begin{proof}
	It follows from 
 the decomposition \eqref{eq-h-annulus-k}   that  $\max_{x \in \mathsf{A}_k} h_{x}  - (m_{N} \mathbf{1}_{\{k=1\}} +m_{2^{\mathfrak{n}-k}} \mathbf{1}_{\{k \ge 2\}}) $ is equal to
 \begin{equation}
	\label{eq:centered-annuli-max}
  S_k + \max\limits_{x \in \mathsf{A}_k}  \Bigl\{ \sum_{j=1}^{k}   b_{j}(x)  \varphi_{j}(o)    +   \sum_{j=1}^{k-2}  \chi_{j}(x)  +  \{ \wh{h}_{k}(x) + \chi_{k}(x) + \chi_{k-1}(x) \} \Bigr\} .  
 \end{equation}
By using Lemma~\ref{lem:39} (ii) and the definitions of  $R^{i}_{k}, i=1,2,3$, \eqref{eq:centered-annuli-max} bounded from above by $C  \sum_{j=1}^{k} (R^{1}_{j}+R^{2}_{j})2^{-(k-j)}  +  (R^{3}_{k})_+ $; and \eqref{eq:centered-annuli-max} bounded from below by  $- C\,[\sum_{j=1}^{k} (R^{1}_{j}+R^{2}_{j})2^{-(k-j)}   +  (R^{3}_{k})_- ] .$
Similarly, since 
\[  \Phi^{\mathsf{D}_0 \mid \partial \mathsf{D}_{k}}_{x} - S_k= \sum_{j=1}^{k} b_{j}(x)   \varphi_{j }(o)   + \sum_{j=1}^{k}  \chi_{j}(x) \quad \text{ for }  x \in \mathsf{D}_{k}, \]
 we derive  that  $ \max\limits_{x \in \mathsf{D}_{k+1}} \big| \Phi^{\mathsf{D}_0 \mid \partial \mathsf{D}_{k}}_{x} - S_k \big| \leq C \sum_{j=1}^{k} (R^{1}_{j}+R^{2}_{j})2^{-(k-j)}  $. The desired bounds then follow  immediately from the definition of the control variable $\mathscr{M}_{\ell}$.
\end{proof}

The next lemma shows that the event that $\mathscr{M}_{\ell}$ is large has negligible probability.

\begin{lem}
\label{lem:41}
There exists constants $C_{\ref{lem:41}}, c_{\ref{lem:41}}$ and $c_0 =  {\sigma^2_{\max}}/{\sigma^{2}_{\min}}$   such that the following assertions hold.
Let $1 \leq \ell \leq \mathfrak{n}+1$, and  $t \in \bbR$. For all $1 \leq k \leq \ell$, we have 
\begin{equation}
\label{e:151}
\P  \bigl( R_k > u \mid  S_{\ell} = t \bigr)
	\leq C_{\ref{lem:41}}\exp \bigl( - c_{\ref{lem:41}} [(u -  c_0 |t|/\ell)_+]^2 \bigr) \ \text{ for all } u > 0.
\end{equation} 
Furthermore,   for each $m \ge 3$, we have  
\begin{align} 
 &\{ \mathscr{M}_{\ell} = m \} \subset \bigl\{ \exists  1 \le k \le \ell  \text{ with }  \wedge^{\ell}(k) \le m-1 \text{ s.t. }  R_k  >\log  (m-1) \bigr\}  \label{e:151.5} \\ 
 \P  \bigl(  \exists  1 \le k \le \ell  &\text{ with }  \wedge^{\ell}(k) \le m-1 \text{ s.t. }  R_k  > \log  (m-1)    \mid 
	S_{\ell} = t \bigr) \leq  C_{\ref{lem:41}} \, \rme^{- c_{\ref{lem:41}} [(\log m - c_0 |t|/\ell)_+]^2 } .\label{e:152}
\end{align}
\end{lem}

\begin{proof}
The inclusion relation \eqref{e:151.5} follows directly from the definition: on $\{ \mathscr{M}_{\ell} = m \}$ there must exists $k$ such that $R_{k}>\log ( \max\{ m-1, \wedge^{\ell}(k) \} )$ and  $R_{k}\le \log ( \max\{ m, \wedge^{\ell}(k) \} )$. This forces $\wedge^{\ell}(k) \le m-1$.   
The inequality \eqref{e:152} is a direct consequence of \eqref{e:151} and the union bound. So it suffices  to show  \eqref{e:151}. 

Note that $R^2_k$, $R^{3}_k$ are   independent of $S_{\ell}$.   Lemma \ref {lem:39} (ii)-(iv) then implies $\P  ( R_k^{2} \vee (R^{3}_{k})_+ ^{1/2} \vee   (R^{3}_{k})_{-} > u \mid  S_{\ell} = t ) \lesssim e^{- c_{\ref{lem:39} } u^{2}}$.
As for $R^1_k=|\varphi_{k}(o)|$,  by using Lemma~\ref{lem:39} (i) we obtain 
\begin{equation*}
\Var [\varphi_{k}(o) \,\big|\, S_{\ell} = t ] \leq \Var [\varphi_{k}(o) ] \le \sigma^2_{\max}, \ \text{ and } \ \E   [\varphi_{k}(o) \,\big|\, S_{\ell} = t ]= \frac { t \, \E  [\varphi_{k}(o)^2]}{\E  [S^2_{\ell}]} \leq  \frac{\sigma^2_{\max}}{\sigma^{2}_{\min}} \frac{|t|}{\ell} . 
\end{equation*}   
Then the   standard Gaussian tail bounds yields 
\begin{equation*}
\P  \bigl( R^1_k > u \, \big|\, S_{\ell} = t \bigr) \lesssim \exp \bigl( -   [(u - c_0|t|/\ell)_+]^2 / (2 \sigma_{\max}^2) \bigr)\,,
\end{equation*} 
as required.  
\end{proof}
 
\subsection{Preliminaries: Ballot estimates with two barriers}
\label{sub:3.1} 
The concentric decomposition reduces the analysis of the extreme values of $h$ to the study of a backbone random walk $(S_{k})$, subject to specific constraints. Specifically, our derivation leads to events in which the walk is required to stay above or below curves of small polynomial growth. Estimates of this type are typically referred to as ballot-type theorems. In this section, we adapt the approach developed in \cite[Section 4.3]{BL3} and \cite{CHL19}.

\smallskip
\noindent \textit{Notation for barriers and events.}
Assume  $0 < \sigma_{\rm min} < \sigma_{\rm max} < \infty$.  
Let $( \xi_k )_{k = 1}^{n}$ be a sequence of independent, centered Gaussian r.v.'s with variances $
\E  [ \xi_{k}^2  ] \in (\sigma^2_{\rm min}, \sigma^2_{\rm max})  $.
 We define the (inhomogeneous) random walk $(Z_k)_{k \geq 0}$ by  $
Z_k \coloneq Z_0 +  \sum_{j=1}^k \xi_j $ for all $1 \le k \le n$. 
For $x,y\in \bbR$, and $n \ge 1$, we denote the law of the random walk bridge from $x$ to $y$ in $n$ steps by $\P _{0,n}^{x, y}$, defined as  $\P (\,\cdot\, \mid Z_0 = x, Z_n = y)$.
 
Fix $a \in (0,1/2)$ and $b>0$.  For  $0\le m <n$ and $x, y \in \mathbb{R}$,  we define 
\begin{equation*}
	\mathfrak{l}_{m,n}^{x,y}(k) \coloneq    \frac{n-k}{n-m} x + \frac{k-m}{n-m} y \ \text{ and } \ \zeta_{m,n}^{a,b} (k) \coloneq b [\wedge_{m}^{n}(k)  ]^{a}  \quad  \text{ for all  } \,  m \le k \le n. 
\end{equation*}
Regard $\mathfrak{l}_{0,m}^{x_0,x_1}+\mathfrak{l}_{m+1,n}^{y_0,y_1}$ a function defined on $[0,n] \cap \mathbb{N}$.
 Given a subset $I \subset \mathbb{R}$ and two functions $g, \tilde{g}$ on $I$, we write $g \preceq_{I} \tilde{g}$ if $g(k) \le \tilde{g}(k)$ for all $k \in I \cap \mathbb{N}$. 

 \begin{figure}[tbp]
    \centering
\begin{tikzpicture}[scale=0.62, every node/.style={font=\small}]

    \pgfmathsetmacro{\nn}{12}      
    \pgfmathsetmacro{\rr}{8}       
    \pgfmathsetmacro{\xzero}{3.5}    
    \pgfmathsetmacro{\xone}{5}     
    \pgfmathsetmacro{\yzero}{4}    
    \pgfmathsetmacro{\yone}{3}     
    \pgfmathsetmacro{\aexp}{0.4}   
    \pgfmathsetmacro{\walkx}{2}    
    \pgfmathsetmacro{\walky}{1.8}  

    \draw[->] (-0.5,0) -- (\nn+1,0) node[right] {$k$};
    \foreach \k/\lbl in {0/0, \rr/r, \rr+1/r+1, \nn/n}
        \draw (\k,0.08) -- (\k,-0.08) node[below=1pt] {$\lbl$};

    \draw[RoyalBlue, dashed, thick] (0,\xzero) -- (\rr,\xone);
    \draw[BurntOrange, dashed, thick] (\rr+1,\yzero) -- (\nn,\yone);

    \draw[RoyalBlue, thick, domain=0:\rr, samples=80, smooth]
        plot (\x, {\xzero + (\xone-\xzero)/\rr*\x + (min(\x,\nn-\x))^\aexp});
    \draw[BurntOrange, thick, domain=\rr+1:\nn, samples=40, smooth]
        plot (\x, {\yzero + (\yone-\yzero)/(\nn-\rr-1)*(\x-\rr-1) + (min(\x,\nn-\x))^\aexp});

    \pgfmathsetmacro{\barLeft}{\xone + (min(\rr,\nn-\rr))^\aexp}
    \pgfmathsetmacro{\barRight}{\yzero + (min(\rr+1,\nn-\rr-1))^\aexp}
    \draw[gray, dotted] (\rr,\barLeft) -- (\rr+1,\barRight);

    \pgfmathsetmacro{\half}{\nn/2}
    \pgfmathsetmacro{\skelHalf}{\xzero + (\xone-\xzero)/\rr*\half}
    \pgfmathsetmacro{\barHalf}{\skelHalf + (min(\half,\nn-\half))^\aexp}
    \draw[<->, gray] (\half,\skelHalf) -- (\half,\barHalf)
        node[midway, right] {\scriptsize $\zeta_{0,n}^{a,b}(k)$};

    \fill (0,\xzero) circle (1.6pt) node[below left] {$x_0$};
    \fill (\rr,\xone) circle (1.6pt) node[below] {$x_1$};
    \fill (\rr+1,\yzero) circle (1.6pt) node[above] {$y_1$};
    \fill (\nn,\yone) circle (1.6pt) node[below right] {$y_0$};

    \draw[very thick, black!70]
        (0,\walkx) -- (1,2.6) -- (2,2.1) -- (3,3.0) -- (4,2.4) -- (5,3.3) --
        (6,2.7) -- (7,3.6) -- (\rr,2.9) -- (\rr+1,3.4) -- (10,2.2) -- (11,1.3) -- (\nn,\walky);
    \foreach \k/\v in {0/\walkx,1/2.6,2/2.1,3/3.0,4/2.4,5/3.3,6/2.7,7/3.6,8/2.9,9/3.4,10/2.2,11/1.3,12/\walky}
        \fill[black!70] (\k,\v) circle (1pt);
    \node[left] at (0,\walkx) {$x$};
    \node[right] at (\nn,\walky) {$y$};

    \node[RoyalBlue, anchor=west] at (0.3,7.2) {$\mathfrak{l}_{0,r}^{x_0,x_1}+\zeta$};
    \node[BurntOrange, anchor=west] at (\rr+0.5,7.2) {$\mathfrak{l}_{r+1,n}^{y_0,y_1}+\zeta$};
    \draw[very thick, black!70] (4.6,7.2) -- (5.3,7.2);
    \node[anchor=west] at (5.3,7.2) {$Z_k$};

\end{tikzpicture}
    \caption{Illustration of the barrier $\mathfrak{l}_{0,r}^{x_0,x_1}+\mathfrak{l}_{r+1,n}^{y_0,y_1}+\zeta_{0,n}^{a,b}$ appearing in Lemma~\ref{lem-two-barrier-ballot-upper-bound-rw}. The dashed segments show the piecewise-linear skeleton $\mathfrak{l}_{0,r}^{x_0,x_1}$ and $\mathfrak{l}_{r+1,n}^{y_0,y_1}$; the solid curves add the perturbation $\zeta_{0,n}^{a,b}(k)=b[\wedge_0^n(k)]^a$, which vanishes at $k=0,n$ and is largest near $k=n/2$. A trajectory of the random walk bridge $(Z_k)_{0\le k\le n}$ from $x$ to $y$ staying below the barrier is shown in gray.}
    \label{fig:barriers}
\end{figure}
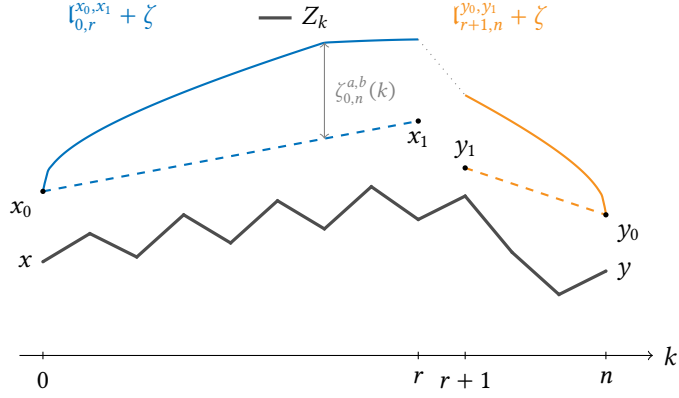


\begin{lem}[Upper bound]
\label{lem-two-barrier-ballot-upper-bound-rw}
For  every  $a \in (0,1/2)$, $b>0$, and $K>0$, there exist constants $C_{\ref{lem-two-barrier-ballot-upper-bound-rw}}, c_{\ref{lem-two-barrier-ballot-upper-bound-rw}} > 0$ such that, for every tuple $(n,r; x,x_0,x_{1};y,y_0,y_1) \in \mathbb{N}^{2} \times \mathbb{R}^6$ satisfying $n \ge 3$, $n/2 \le r \le n-2$, $x \le x_0$, $y \le y_0$, and $|x-y| \le K\sqrt{n}$, we have
\begin{align}
   &  \P_{0,n}^{x,y} \Bigl( Z \preceq_{[0,n]} \mathfrak{l}_{0,r}^{x_0,x_{1}}  + \mathfrak{l}_{r+1,n}^{y_1,y_{0}} + \zeta_{0,n}^{a,b} \Bigr) \notag \\
    &\ \leq C_{\ref{lem-two-barrier-ballot-upper-bound-rw}}\frac{(1+x_0-x)(1+y_0-y)}{n} \Bigl[ 1+\frac{(x_1-y)_+}{(n-r)^{1/2}} \Bigr] \Bigl[ 1+\frac{(y_1-y)_+}{(n-r)^{1/2}} \Bigr]  \exp \Bigl( - c_{\ref{lem-two-barrier-ballot-upper-bound-rw}}\frac{(x_{1}-y)_{-}^2 + (y_{1}-y)_{-}^2 }{n-r} \Bigr). \notag
\end{align} 
\end{lem}

The restriction $r \ge n/2$ is without loss of generality, by time reversal. The assumption $|x-y| \le K\sqrt{n}$ is likewise mild: by subtracting the linear interpolation between the endpoints, one may reduce to a bridge from $0$ to $0$. The resulting upper bound, however, need not be asymptotically sharp. To obtain a matching lower bound, we therefore impose the stronger assumptions below:

\begin{lem}[Sharp estimate]
\label{lem-two-barrier-ballot-asy-rw} Fix $K>0$, and let $\mathscr{H}_K$ denote the set of all tuples $(n,r; x,x_0,x_{1};y,y_0,y_{1}) \in \mathbb{N}^2 \times \mathbb{R}^{6}$ satisfying $n \ge 3$, $n/2 \le r \le n-2$, and
\begin{align*}
    0 & < x_0 - x \le K \sqrt{n}, \quad 0 < y_0 - y \le K \sqrt{n-r}, \\
    | x_{1}- x_0 | & \le K \sqrt{n}, \quad | y_{1}- y_0 | \le K \sqrt{n-r}, \quad -K \sqrt{n-r} \le x_{1} - y_{1} \le K \sqrt{n}.
\end{align*}
Then, for every  $a\in (0,1/2)$ and $b>0$, uniformly over
$(n,r;x,x_0,x_1;y,y_0,y_1)\in\mathscr{H}_K$,
\begin{equation*}
    \P_{0,n}^{x,y} \Bigl( Z \preceq_{[0,n]} \mathfrak{l}_{0,r}^{x_0,x_{1}} +\mathfrak{l}_{r+1,n}^{y_1,y_{0}} \pm \zeta_{0,n}^{a,b} \Bigr)
    \asymp_{K,a,b}  \left[ 1+\frac{(x_1-y_1)_+}{\sqrt{n-r}} \right] \frac{(1+x_0-x)(1+y_0-y)}{n}.
\end{equation*}
The implicit constants in $ \asymp_{K,a,b} $ depend on $K, a,$ and $b$.
\end{lem}
 
The assumptions of Lemma~\ref{lem-two-barrier-ballot-asy-rw}
are not optimal, but they suffice for our purposes. In the
piecewise-constant barrier setting, where $x_0=x_1$ and
$y_0=y_1$, the natural condition for a matching lower bound is
typically
$(1+x_0-x)(1+y_0-y)=O(n)$
see, for example, \cite{CHL19}.
 
Our final lemma quantifies the entropic repulsion of a random walk
bridge from the upper barriers when the bridge is constrained to remain below them.

\begin{lem}[Entropic repulsion]
\label{lem-ent-rw}
For any $a \in (0,1/2)$, $b>0$, and $K>0$, there exists a constant $C_{\ref{lem-ent-rw}}$ such that for all $n \ge 3$, $n/2 \le r \le r' \le n-2$, $x \le x_0$, $y \le y_0$ with $x_0 \ge y_0$, $|x-y_0| \le K \sqrt{n}$, and $0 \le M \le K \sqrt{n-r}$, the following holds:
\begin{align}
    \P_{0,n}^{x, y} \Bigl( & \{ Z \preceq_{[0,r]} x_0 + \zeta_{0,n}^{a,b}, \ Z \preceq_{[r+1,n]} y_0+\zeta_{0,n}^{a,b} \} \cap \{ Z \preceq_{[r+1,r']} y_0 - \zeta_{0,n}^{a,b} - M \}^{c} \Bigr) \notag \\
    &\leq C_{\ref{lem-ent-rw}} \frac{(1+x_0-x)(1+y_0-y)}{n} \frac{M+(n-r')^{a}}{(n-r')^{1/2}} \Bigl[ 1+ \frac{(x_0-y_0)}{(n-r)^{1/2}} \Bigr].
    \label{ent:right}
\end{align}
Additionally, if $|y_0-y| \le K \sqrt{n-r}$ and $1 \le r'' \le n-r$, then
\begin{align}
    \P_{0,n}^{x, y} \Bigl( & \{ Z \preceq_{[0,r]} x_0 + \zeta_{0,n}^{a,b}, \ Z \preceq_{[r+1,n]} y_0+\zeta_{0,n}^{a,b} \} \cap \{ Z \preceq_{[r'',r]} x_0 - \zeta_{0,n}^{a,b} - M \}^{c} \Bigr) \notag \\
    &\leq C_{\ref{lem-ent-rw}} \frac{(1+x_0-x)(1+y_0-y)}{n} \frac{M+(r'')^{a}}{(r'')^{1/2}} \Bigl[ 1+ \frac{(x_0-y_0)}{(n-r)^{1/2}} \Bigr].
    \label{ent:left}
\end{align}
\end{lem}

\subsection{One point estimate: Proof of Proposition~\ref{prop:one-point-estimate} and \ref{prop:annular-high-point}}

\begin{proof}[Proof of Proposition~\ref{prop:one-point-estimate}] 
Throughout the proof, we fix the reference point $o=z$.  
Recall the definition of $\mathfrak{n}=\mathfrak{n}_o$ from \eqref{eq-def-n'} and let $\wtN =\wtN _o\coloneq\mathrm{dist}(o,\partial \mathsf{V}_N)$. 
These definitions imply $2^{\mathfrak{n}} \le \wtN  \le 2^{\mathfrak{n}+1}$. Furthermore, the assumption $\wtN  \ge \sqrt{N}$, combined with $2^{n} < N \le 2^{n+1}$, ensures that $\mathfrak{n} \in [n/2-1, n+1]$.
  
 \smallskip  
Since $|m_{\wtN }-m_{2^{\mathfrak{n}}}|=O(1)$, we may without loss of generality,  replace each occurrence of $m_{\wtN }$ in \eqref{e:32} by $m_{2^{\mathfrak{n}}}$. 
By \eqref{eq-Green-asymp}, we have $\Var(h_o) = g \log \wtN  + O(1)$. Consequently, there exists a constant $c>0$ such that, uniformly for all $|t| \le 10 \sqrt{\log N}$, 
 \begin{equation}\label{eq-ho-density}
	 \P (h_o \in m_{2^{\mathfrak{n}}}  + \dif t ) \lesssim \frac{1}{\sqrt{\mathfrak{n}} }  \exp \Bigl(   - \frac{[2 \sqrt{g} \log \wtN  - \frac{3}{2\alpha} \log \mathfrak{n} +t]^2}{2g \log \wtN } \Bigr) \lesssim  \,  \wtN  ^{-2} \times    \mathfrak{n}  \times e^{- \alpha t - c \frac{ t^2}{\log N}} \dif t.
 \end{equation}  
 Let $\tilde{\rho} \coloneq \min\{ (u-t)^{50}, \mathfrak{n}/2 \}$. Note that $\mathsf{D}_{1} \subset \mathsf{B}_{\wtN /2}(o)$ and $\mathsf{D}_{\mathfrak{n}-\tilde{\rho}}\subset\mathsf{B}_{\rho(u-t;\wtN )}(o)$. 
 Thus in view of \eqref{eq-ho-density}, it suffices for the proof of \eqref{e:32} to establish that, uniformly for  $-10 \sqrt{\log N} \le t \le s \le u \le 10\sqrt{\log N}$,
\begin{equation}
  \P \bigl( h^*_{  \mathsf{D}_{\mathfrak{n}-\tilde{\rho}}} \leq m_{2^{\mathfrak{n}}} + s , \, h^*_{ \mathsf{D}_1 } \leq m_{2^{\mathfrak{n}}} + u \mid h_o = m_{2^{\mathfrak{n}}} +t \bigr) 
  \lesssim \frac{(1+u_+)e^{-c (u_{-})^2} (1+s-t)}{\mathfrak{n}}. \label{eq-o-is-local-extreme}
\end{equation} 

We abbreviate the shift field $f_{\mathfrak{n}+1}$ defined in \eqref{e:129} by $f$, set $\hat{f} \coloneq 1- f$, and denote the control variable $\mathscr{M}_{\mathfrak{n}+1}$ from \eqref{eq-control-variable} by $\mathscr{M}$. Note that $0 \le f_x \le 1$.
Using the Gaussian property, the conditional probability in \eqref{eq-o-is-local-extreme} can be decomposed as $\sum_{m \ge 3} \mathrm{Pr}_{\eqref{eq-switch-shift}}(m)$, where
\begin{equation}\label{eq-switch-shift}
\mathrm{Pr}_{\eqref{eq-switch-shift}}(m)
\coloneq  \P  \bigl(  \mathscr{M} = m ,\
h_x \le  \hat{f}_{x} m_{2^{\mathfrak{n}}} - t f_x + u\ind{x \in \mathsf{D}_{1} \setminus \mathsf{D}_{\mathfrak{n} - \tilde{\rho} }}  +s \ind{x \in \mathsf{D}_{\mathfrak{n} - \tilde{\rho} }}     \ , \  \forall x \in \mathsf{D}_{1} 
\mid  h_o = 0 \bigr) .
\end{equation}
The  lemma below provides a useful upper bound for $\mathrm{Pr}_{\eqref{eq-switch-shift}}(m)$.

\begin{lem}\label{lem-put-M-in}
Under the same condition in Proposition~\ref{prop:one-point-estimate}, there exists a constant  $c>0$ such that 
\begin{equation}\label{eq-put-M-in}
\mathrm{Pr}_{\eqref{eq-switch-shift}}(m) \lesssim  \mathbf{1}_{\{ \log m \ge c\, u_{-} \}}  e^{-  c \log^{2} m} \,  \frac{ (1+u_+ )  (1 +s- t)  }{\mathfrak{n} }    .
\end{equation}
\end{lem}
 
Indeed, summing \eqref{eq-put-M-in} over $m$ yields \eqref{eq-o-is-local-extreme}. Thus, it remains to prove Lemma \ref{lem-put-M-in}. 

\begin{proof}[Proof of Lemma \ref{lem-put-M-in}]
Applying Lemma~\ref{lemma-3.8} with $\ell = \mathfrak{n}+1$, we have  
 $\max_{x \in \mathsf{A}_{k}} | m_{2^{\mathfrak{n}}}  \hat{f}_{x} - m_{2^{ \mathfrak{n}-k}} | \leq C_{\ref{lemma-3.8}} [1 + \log_+ \wedge^{\mathfrak{n}+1}(k)] 
$ and $\max_{x \in \mathsf{A}_k} \big| f_{x}    - \frac{k}{\mathfrak{n}+1} \big| \le\frac{C_{\ref{lemma-3.8}}}{\mathfrak{n}+1} $. Then  Lemma \ref{lem:40} yields that for each $2 \le k \le \mathfrak{n}+1$,
\begin{align}
h_x - m_{2^{\mathfrak{n}}} \hat{f}_x &\leq u  - t f_{x} ,\, \forall \, x \in \mathsf{A}_{k}  \Longrightarrow   \max_{x \in \mathsf{A}_k } h_{x}  - m_{2^{\mathfrak{n}-k}} \le  C_{\ref{lemma-3.8}} [1 + \log_+ \wedge^{\mathfrak{n}+1}(k)] + u - t \tfrac{k}{\mathfrak{n}} \notag\\
&\Longrightarrow S_{k} \le C_0 \,  \vartheta_{\mathscr{M}}(\wedge^{\mathfrak{n}+1}(k))  + u  - t \tfrac{k}{\mathfrak{n}+1}  \ , \label{eq-cons-S-k}  
\end{align} 
where we set $C_0 = 1+ C_{\ref{lemma-3.8}} \vee C_{\ref{lem:40}}$.  Let $g(k) \coloneq u\mathbf{1}_{\{k \le \mathfrak{n}-\tilde{\rho}\}} +s \mathbf{1}_{\{k > \mathfrak{n}-\tilde{\rho}\}} - t \frac{k}{\mathfrak{n}+1}$, and define $g_{m}(k) \coloneq g(k)+ C_0\vartheta_{m}(\wedge^{\mathfrak{n}+1}(k))$. 
We may then bound $\mathrm{Pr}_{\eqref{eq-switch-shift}}(m)$ from  above by
\begin{equation}
	\mathrm{Pr}_{\eqref{eq-ballot-S-M-m}}(m) \coloneq \P \Bigl( \{\mathscr{M} = m\} \cap \{ S \preceq_{[2,\mathfrak{n}+1]} g_m   \}   \mid S_{\mathfrak{n}+1}=0 \Bigr)  . 	\label{eq-ballot-S-M-m}
\end{equation} 

First, on the event $\{ \mathscr{M} =m\}$ with $m \ge 3$, we have $|S_{2}|  \le 2\log m$. Since $|t|/ \mathfrak{n+1}=o(1)$, if $g_m(2)=g(2)+C_0 \log m < - 2\log m$ (equivalently $(C_0+2)\log m \le -u$), then necessarily  $ \mathrm{Pr}_{\eqref{eq-ballot-S-M-m}}(s,t;m) =0 $.  
Second,  using \eqref{e:151.5} and \eqref{e:152} in Lemma \ref{lem:41} yields there is a constant $c>0$ such that  for all $ m \ge \mathfrak{n}^{1/4} $,
\begin{equation*}
	\P ( \mathscr{M}  \ge m    \mid  S_{\mathfrak{n}+1}=0) \lesssim e^{- 2c  (\log m)^2} \lesssim \frac{1}{\mathfrak{n}} e^{- c  (\log m)^2}\,,
\end{equation*}
which implies Lemma~\ref{lem-put-M-in}. 

It  remains to bound $\mathrm{Pr}_{\eqref{eq-ballot-S-M-m}}(m) $ in the regime $3\le m \le \mathfrak{n}^{1/4}$.
Define the event 
\begin{equation}\label{def:Fm}
	F_{m}\coloneq \{ \exists \, k \text{ with }  \wedge^{\mathfrak{n}+1}(k) \le m-1 \text{ s.t. }  R_k  >  \log  (m-1) \} \cap \{ |S_{m-1}|, | S_{\mathfrak{n}+1}-S_{\mathfrak{n}+2-m}| \le m \log m \}\,,
\end{equation}
which is measurable with respect to $\sigma( \{ \varphi_{k}(o), \chi_{k}, h_{k} : k \le m-1 \text{ or } k \ge \mathfrak{n}+2-m  \} )$.  
By Lemmas \ref{lem:40} and \ref{lem:41} we have $\{\mathscr{M} =m\} \subset F_{m}$. Applying    the Markov property gives 
\begin{equation}\label{eq-extrac-ballot-1}
 \mathrm{Pr}_{\eqref{eq-ballot-S-M-m}}(m) \le 
\E  \Bigl[  \P \Bigl(  S \preceq_{[m-1,\mathfrak{n}+2-m]} g_{m}       \mid S_{m-1}, S_{\mathfrak{n}+2-m}  \Bigr) \mathbf{1}_{F_{m}}  \  \big| \  S_{\mathfrak{n}+1}= 0 \Bigr] . 
\end{equation} 
We then split the analysis into two sub-cases:

\smallskip  
 \noindent \underline{\textit{Case 1: }}  $3 \le m \le (\tilde{\rho}+1)/2$.  
 We apply Lemma \ref{lem-two-barrier-ballot-upper-bound-rw} with $a=1/4$ and $b>0$ chosen so that $\vartheta_m(\wedge^{\mathfrak{n}+1}(k)) \le  \log m+\zeta^{a,b}_{m-1,\mathfrak{n}+2-m}(k)$.
On $F_m \cap \{ S_{\mathfrak{n}+1}=0\}$, we have $|S_{\mathfrak{n}+2-m} -S_{m-1}| \lesssim m \log m \lesssim \mathfrak{n}^{1/2} \lesssim (\mathfrak{n}-2m )^{1/2}$, so the hypotheses of Lemma \ref{lem-two-barrier-ballot-upper-bound-rw} are satisfied.
Combining this with the bounds $(u-t)\lesssim \tilde{\rho}^{1/2}$ and $\frac{\tilde{\rho}}{\mathfrak{n}}|t| \lesssim \tilde{\rho}^{1/2}$, we obtain, uniformly over $S_{m-1} < g_{m}(m-1)$ and $S_{\mathfrak{n}+2-m} < g_{m}(\mathfrak{n}+2-m)$,
\begin{align}
 & \P \Bigl(  S \preceq_{[m-1,\mathfrak{n}+2-m]} g_{m}       \mid S_{m-1}, S_{\mathfrak{n}+2-m}  \Bigr) \lesssim \P \Bigl(  S \preceq_{[m-1,\mathfrak{n}+2-m]} g + \zeta^{a,b}_{m-1,\mathfrak{n}+2-m}     \mid S_{m-1}, S_{\mathfrak{n}+2-m}  \Bigr)  \notag\\
&\ \lesssim  \frac{1}{\mathfrak{n}} \,  [ 1+g_{m}(m-1)  -S_{m-1}] [1+ g_{m}(\mathfrak{n}+2-m)  -S_{\mathfrak{n}+2-m}] \Bigl( 1+ \frac{ (u- \frac{\mathfrak{n}-\tilde{\rho}}{\mathfrak{n}} t)_{+}}  {(\tilde{\rho}+2-m)^{1/2}} \Bigr)^2\notag\\
&\ \lesssim  \frac{ (m \log m)^2}{\mathfrak{n} }   (1+u_+ )  (1 +s- t)  .\label{eq-extrac-ballot-2}
\end{align}
Substituting \eqref{eq-extrac-ballot-2} into \eqref{eq-extrac-ballot-1} and using $\P (F_{m} \mid S_{\mathfrak{n}+1}=0) \lesssim e^{- c \log^{2} m}$, which follows from Lemma \ref{lem:41}, we obtain \eqref{eq-put-M-in}. 
\smallskip 

\noindent \underline{\textit{Case 2: }} $ (\tilde{\rho}+1)/2 \le m \le \mathfrak{n}^{1/4}$ (if applicable). 
Here we simply bound $g(k)$ by $u - \frac{k}{\mathfrak{n}} t$. The standard ballot estimate (Lemma \ref{lem-ballot-rw}) then gives
\begin{equation}\label{eq-M-m-large-bal}
     \P \Bigl( S \preceq_{[m-1,\mathfrak{n}+2-m]} g_{m} \mid S_{m-1}, S_{\mathfrak{n}+2-m} \Bigr) \lesssim \frac{(m \log m)^2}{\mathfrak{n} } {(1+ u_+)(1 +u - t)} .
\end{equation}
Combining \eqref{eq-M-m-large-bal} with \eqref{eq-extrac-ballot-1} and the bound on $\P (F_{m} \mid S_{\mathfrak{n}+1}=0)$ yields
\begin{equation*}
	\mathrm{Pr}_{\eqref{eq-ballot-S-M-m}}(m)
\lesssim   (m \log m)^2e^{-  c \log^{2} m} \,  \frac{(1+u_+ )  (1 +u- t)}{\mathfrak{n} } \lesssim  e^{-  \frac{c}{2} \log^{2} m} \,  \frac{(1+u_+ )  }{\mathfrak{n} },
\end{equation*}
where we used the assumption $[1+(u-t)] \lesssim 1+\tilde{\rho} \lesssim m$. This establishes \eqref{eq-put-M-in}.
\end{proof}

This completes the proof of Lemma \ref{lem-put-M-in} as well as  Proposition~\ref{prop:one-point-estimate}.
\end{proof}

  \begin{proof}[Proof of Corollary \ref{cor-locmax}]
Recall from \eqref{eq-Green-asymp} that  $\Var(h_z) \le g \log   \wtN _z   + C$. Furthermore,    the trivial lower bound $\Var(h_z)=G_{\mathsf{V}_{N}}(z,z)\ge 1/4$ holds for all $z \in \mathsf{V}_{N}$.  

\smallskip
 \noindent
\underline{\emph{Step 1.} } 
 We first consider the near boundary regime where   $\wtN _z  \coloneq\mathrm{dist}(z,\partial\mathsf{V}_{N}) $ is less than $   N e^{ - d_{N} } $ with $\log\log N \ll d_{N} \ll \log N$.    Then as $N \to \infty$, the Gaussian tail estimate yields the following bound, uniform in $t \in [- (\log N)^{1/2},\infty)$: 
	\begin{align}
	& \P ( h_z \in m_N + \dif t ) \lesssim \exp\Bigl( - \frac{(m_{N}+t)^2}{2g   \log N  } \frac{\log N}{\log N- d_{N} +C } \Bigr)  \dif t \notag \\
	 &=\exp\Bigl( - \bigl[2\log N + \alpha t - \frac{3}{2} \log\log N +\frac{(t- \frac{3}{2\alpha} \log\log N)^2 }{2 g \log N} \bigr] \bigl[1+\frac{d_{N}}{\log N -d_{N}+C} \bigr] \Bigr)  \dif t \notag \\
	 & \le   N^{-2} e^{-\alpha t}  \exp\bigl(- [1+o(1)]( d_{N} + \alpha t d_{N}/\log N) -  \Theta( t^2/ \log N)\bigr)  \dif t \label{eq-bnd-h-z} . 
\end{align}
Consequently, since $(u_-)^2\le(\log\log N)^{4}$, whenever $\wtN_z\le N e^{-(\log N)^{1/4}}$ the last display is bounded by the right-hand side of \eqref{e:17}.

\smallskip
 \noindent
\underline{\emph{Step 2.}} 
Next, we consider the bulk regime where $\wtN _z \in [ N e^{ - (\log N)^{1/4} },N]$. For brevity, define the centering shift
\begin{equation*}
\Delta_{N,z}\coloneq m_N-m_{\wtN_z}
=2\sqrt g\log(N/\wtN_z)+o_N(1).
\end{equation*}
Here we apply the estimate from Proposition~\ref{prop:one-point-estimate}.
Note that $\Delta_{N,z}=o(\sqrt{\log N})$. If $u\ge0$, Proposition~\ref{prop:one-point-estimate}, applied with the shifted parameters $\Delta_{N,z}+t,\Delta_{N,z}+s,\Delta_{N,z}+u$, implies 
\begin{align}
 & \P  \bigl( h_z \in m_N + \dif t
, h^*_{z+ \mathsf{B}_{ \rho(u-t;\wt{N})}   }   \leq m_{N} + s,  h^*_{\mathsf{V}_N} \leq m_N + u \bigr)  \notag\\
&  \le  \P  \bigl( h_z \in m_{\wtN } +\Delta_{N,z}  + \dif t
 , h^*_{z+ \mathsf{B}_{\rho(u-t;\wt{N})}   }  \leq m_{\wtN } + \Delta_{N,z}+ s,  h^*_{ z+ \mathsf{B}_{\wtN /2}}  \leq m_{\wtN } +\Delta_{N,z} + u \bigr) \notag\\
 & \lesssim (  1+ \Delta_{N,z}+u )  (1+s-t  ) N^{-2} (N/\wtN )^{ 2}  e^{ - \alpha (\Delta_{N,z}+ t) } \dif t\lesssim  (   1+u )  (1+s-t  ) N^{-2}  e^{ - \alpha t } \dif t \label{eq-bnd-h-z-2} 
\end{align}
In the last inequality,  we have  used  $(1+\Delta_{N,z}+u) \le (1+\Delta_{N,z})(1+u)$, $(N/\wtN _{z})^2 = [1+o_{N}(1)] e^{\alpha \Delta_{N,z}/2}$, and   $(1+\Delta_{N,z})e^{-\alpha \Delta_{N,z}/2} \lesssim 1$. This completes the proof of \eqref{e:17} in the case $u \ge 0$.

 It remains to consider $u<0$. We claim that repeating  the proof of Lemma~\ref{lem-put-M-in}   gives 
\begin{equation}
	\label{eq:global-cap-1}
	 \P\Bigl(
	 h^*_{z+\mathsf{B}_{\rho(u-t;\wtN)}}\le m_N+s,\,
	 h^*_{\mathsf{V}_N}\le m_N+u
	 \,\Big|\,h_z=m_N+t
	 \Bigr)
	 \lesssim
	 \frac{(1+\Delta_{N,z})(1+s-t)}{\mathfrak n}
	 e^{-c u_-^2}.
\end{equation} 
Here we have to retain the constraint $h^*_{\mathsf{V}_N}\le m_N+u$ instead of replacing it by the constraint on $z+\mathsf{B}_{\wtN/2}$.    The outermost annulus $\mathsf{A}_1=\mathsf D_0\setminus\mathsf D_1$ is thus constrained at height $m_N+u$. Recall that $h_1$ is centered by $m_N$ in \eqref{eq-centered-h}. Lemma~\ref{lem:40} then implies that, on $\{\mathscr M=m\}$, the event under consideration is empty unless $\log m\ge c u_-$. Consequently the indicator in \eqref{eq-put-M-in} remains $\mathbf1_{\{\log m\ge c u_-\}}$, rather than $\mathbf1_{\{\log m\ge c(\Delta_{N,z}+u)_-\}}$. 
The only other modification to the proof of
Lemma~\ref{lem-put-M-in} is that, for $k\ge2$, replacing the centering
$m_{\wtN}$ by $m_N=m_{\wtN}+\Delta_{N,z}$ replaces the barrier
$g(k)$ there by
$ g^\Delta(k)
 =
 g(k)+\Delta_{N,z}(1-\frac{k}{\mathfrak n} )$,
up to an $o(1)$ error, uniformly in $k$.  
in the ballot estimate this costs at most a factor $C(1+\Delta_{N,z})$. Then \eqref{eq:global-cap-1} follows.
  
Moreover, \eqref{eq-ho-density}, applied with
$\Delta_{N,z}+t$ in place of $t$, yields $
 \P(h_z\in m_N+\dif t)
 \lesssim
 \wtN^{-2}\mathfrak n
 e^{-\alpha(\Delta_{N,z}+t)}\dif t$. 
Combining this with \eqref{eq:global-cap-1} and using
$\wtN^{-2} 
 =[1+o(1)]N^{-2}e^{\alpha\Delta_{N,z}/2}$,
$
 (1+\Delta_{N,z})e^{-\alpha\Delta_{N,z}/2}\lesssim1$
proves \eqref{e:17} for the case $u \in (-(\log\log N)^2,0)$. 

 \smallskip
 \noindent
\underline{\emph{Step 3.}} 
Finally, we turn to the proof of \eqref{e:19}. Setting $s=t=u$ in \eqref{e:32}, the argument used to derive \eqref{eq-bnd-h-z-2} yields that for all $z \in \mathsf{V}_{N}$ and $u \in [0, \log N]$\footnote{Although  Proposition \ref{prop:one-point-estimate},is stated for $u \lesssim \sqrt{\log N}$, the result also holds for $t \le s \le u$,  $|t|, |u| \lesssim \log N$ satisfying $(1+s-t) \asymp (1+u-t)$. In this case, the two-barrier ballot estimate in \eqref{eq-extrac-ballot-2} is not required anymore, as the standard ballot estimate in \eqref{eq-M-m-large-bal} suffices for the proof.}, 
\begin{equation*}
    \P \bigl( h_z \in m_N + \dif u, h^*_{\mathsf{V}_N} \leq m_N + u \bigr) \lesssim ( 1+u ) N^{-2} e^{ - \alpha u } \dif u.
\end{equation*}
Summing over $z$, we obtain
$\P \bigl( h^*_{\mathsf{V}_N} \in m_N + \dif u \bigr)  \lesssim ( 1+u ) e^{ - \alpha u } \dif u.$ for $u \in [0, \log N]$. 
For the tail  $u' \ge \log N$, repeating the argument in \eqref{eq-bnd-h-z} shows that
\begin{equation*}  
    \P \bigl( h^*_{\mathsf{V}_N} \ge m_N + u' \bigr) \lesssim \sum_{z \in \mathsf{V}_{N}}\P \bigl( h_z \geq m_N +u') \lesssim (\log N)^{3/2} e^{- \alpha u' - \frac{(u')^2}{2 g \log N}} \lesssim e^{-\alpha u'} .
\end{equation*}
This establishes \eqref{e:19} and concludes the proof. 
\end{proof}

\begin{rem} 
Regarding the estimate \eqref{e:32}, it is worth noting that for negative $u$, the decay is significantly faster than a Gaussian tail; indeed, a double exponential decay was established by Ding in \cite{ding2013exponential}. Furthermore, the sharp decay rate for the right tail in \eqref{e:19} should be  $ 
\P ( h^*_{\mathsf{V}_N} \geq m_N + u ) \leq C_{\ref{cor-locmax}} \min\{(1+u),\log N\} e^{-\alpha u - \frac{u^2}{2 g \log N} }$. However, as the coarser bounds provided in \eqref{e:32} and \eqref{e:19} already suffice for our applications, we do not pursue these optimal estimates in the present work.
\end{rem}

\smallskip
\begin{proof}[Proof of Proposition \ref{prop:one-point-lower-bnd}]
Analogous to  the  upper bound in Proposition~\ref{prop:one-point-estimate}, it suffices to prove that there exists a constant $c > 0$ such that for $-(\log N)^{1/2}\le t ,s \le u \le (\log N)^{1/2}$, $s \ge t-M$, and $u \ge 0$: 
\begin{align}
\mathrm{Pr}_{\eqref{eq-lower-bnd}} &\coloneq \P  \bigl(  h_{x+o} \le h_{o} +\omega_{x} \, \forall x \in \mathsf{B}_r, \,
h^*_{\mathsf{B}_{\sqrt{N}}(o)\setminus \{o\}} \leq m_N + s
	 , \, h^*_{\mathsf{V}_{N}} \leq m_N + u \,\big|\, h_o = m_N + t \bigr) \notag \\
& = \P  \bigl(  h_{x} \le (m_N+t)\hat{f}_{x} + \omega_{x-o}, \forall x \in \mathsf{B}_r (o) 
	  ; \, h_x \le  m_N  \hat{f}_{x} + s - t f_x ,\, \forall x \in \mathsf{B}_{\sqrt{N}}(o)\setminus \{o\} 
	 \, ; \notag \\
&   \quad \qquad \qquad   h_x \le m_N \hat{f}_{x} +  u - t f_x  ,\, \forall x \in \mathsf{D}_{0}
 \mid  h_o = 0 \bigr)  \ge c \,  \frac{(1+u)[1+(s-t)_+]}{\mathfrak{n}}.   \label{eq-lower-bnd}
\end{align} 
Fix a constant $m \ge 2\log_{2} r$ (to be chosen later) and decompose the above event into three parts:  
\begin{align*} 
B^{1}_{m} &\coloneq \{ h_{x}  \le \omega_{x-o}, \forall x \in \mathsf{B}_r (o)   ; \, h_x \le  m_N  \hat{f}_{x} + s - t f_x ,\, \forall x \in \mathsf{D}_{\mathfrak{n}-m} \setminus \{o\}  \} ; \\
B^{2}_{m} &\coloneq  \{ h_x - m_N \hat{f}_{x} \leq u - t f_x \:: x \in \mathsf{D}_{m} \setminus \mathsf{D}_{\mathfrak{n} /3} \, ; \, h_x - m_N \hat{f}_{x} \leq  s - t {f}_{x} \:: x \in D_{ \mathfrak{n} /3} \setminus D_{\mathfrak{n}-m}  \} ;\\ 
B^{3}_{m} &\coloneq  \{ h_x \leq  m_{N} \hat{f}_{x} -t f_{x} , \,   \forall  x \in \mathsf{D}_{0} \setminus \mathsf{D}_{m}  \} . 
\end{align*}
Using $\mathsf{B}_{\sqrt{N}}(o) \subset \mathsf{D}_{\mathfrak{n}/3}$ and $0\le \hat{f}_{x} \le 1$, we have  $
\mathrm{Pr}_{\eqref{eq-lower-bnd}}  \ge \P ( B^{1}_{m} \cap B^{2}_{m} \cap B^{3}_{m} \mid h_{o}= 0) $. 
Since the conditioned field   $(h_{x} , x \in \mathsf{V}_{N} \mid h_o=0)$ is positive associated, by the FKG inequality, we get 
\begin{equation*}
  \mathrm{Pr}_{\eqref{eq-lower-bnd}} \ge \P(B^{1}_{m} \mid h_{o} = 0) \, \P(B^{2}_{m} \mid h_{o} = 0) \, \P(B^{3}_{m} \mid h_{o} = 0) .
\end{equation*}

 \smallskip
\noindent \underline{\textit{Bound for $B^1_m$}: } Conditional on $h_o = 0$, the variables $(h_x)_{x \in \mathsf{D}_{\mathfrak{n}-m} \setminus \{o\}}$ converge weakly to a non-degenerate Gaussian vector as $N \to \infty$. Thus, for fixed $r$ and $m\ge 2 \log_{2} r$, there is a constant $c^{(1)}_{M,r,m}>0$ such that
\begin{equation*}
 \P \bigl( B^{1}_{m} \mid h_{o} = 0 \bigr) \geq \P \bigl( h_{x} \le \omega_{x-o}, \forall x \in \mathsf{B}_r (o) ; \, h_x \leq - M , \forall x \in \mathsf{B}_{2^{m}} (o) \backslash \{o\} \mid h_{o} = 0\bigr) \geq c^{(1)}_{M,r,m}.
\end{equation*}
 
\smallskip\noindent \underline{\textit{Bound for $B^3_m$}: }
Conditioned on $h_o =0$, $(h_{x})_{x \in \mathsf{V}_N} $ follows the distribution of the unconditioned field $(h_x - f_x h_0)_{x \in \mathsf{V}_{N}}$.
By \eqref{eq-f-asymptotic}, there exists $C_{1}>0$ such that $m_{N} \hat{f}_{x} -t f_{x} \ge m_{N}-   C_{1}(1+m)$ for $x \in \mathsf{D}_0 \setminus \mathsf{D}_{m}$. Also, Lemma \ref{lemma-3.8} ensures $\max_{x \in \mathsf{D}_0 \setminus \mathsf{D}_{m}} |f_{x}| \leq C_1 \frac{m}{n}$ (possibly enlarging $C_1$).
Therefore, for large $N$, there exists $c^{(2)}_{m}>0$ such that
\begin{align*}
\P  \bigl( B^{3}_{m} \mid h_o =0 \bigr)  &\geq  \P \bigl(  \max_{x \in   \mathsf{D}_{0} \setminus \mathsf{D}_{m} } [h_x - f_x h_0]\le  m_{N} -   C_{1}[1+m] \bigr)  \\
& \ge \P \bigl(  \max_{x \in  \mathsf{V}_{N}  } h^{\mathsf{V}_{N}  }_{x} \le  m_{N} -   2C_{1} [1+m] \bigr) -\P( |h_o|> \sqrt{n})  \geq c^{(2)}_{m}  .
\end{align*}
 
\smallskip\noindent \underline{\textit{Bound for $B^2_m$}: }  Employing Lemmas \ref{lemma-3.8} and \ref{lem:40}, as in \eqref{eq-cons-S-k}, on the event $\{\mathscr{M} \leq m\} \cap \{S_{\mathfrak{n}+1} = 0\}$, $B^{2}_{m}$ is implied by
\begin{align*}
B'_m \coloneq \bigl\{& \forall k \in [m, \mathfrak{n}/3], S_k \leq u - t \tfrac{k}{\mathfrak{n}} - C_0 \, \vartheta_{m}(\wedge^{\mathfrak{n}+1}(k))^2 \text{ and } \\
& \forall k \in [\mathfrak{n}/3,\mathfrak{n}-m] , S_k \leq s - t \tfrac{k}{\mathfrak{n}} - C_0 \, \vartheta_{m}(\wedge^{\mathfrak{n}+1}(k))^2 \bigr\} .
\end{align*} 
Hence, $\P (B^{2}_{m} \mid h_o =0)$ is bounded below by
\begin{equation}
\label{eq-M-con}
 \E  \Bigl[ \P  \bigl( B'_m  \mid  S_m, S_{n-m}   \bigr)   \mid S_{\mathfrak{n}+1}=0 \Bigr]
-  \P  \bigl( B'_m \cap\{\mathscr{M}  > m\}  \mid S_{\mathfrak{n}+1}=0 \bigr)  \,.
\end{equation} 
A direct adaptation of the proof of Lemma~\ref{lem-put-M-in}, using the same decomposition according to the value of $\mathscr{M}$ and the same ballot estimate, yields
\begin{equation}\label{eq-M-con-1}
	\P  \bigl( B'_m \cap\{\mathscr{M}  > m\}  \mid S_{\mathfrak{n}+1}=0 \bigr)  \le C \frac{[1+(s-t)_+] (1+u)}{\mathfrak{n}} \rme^{-c (\log m)^{2}  }  \,. 
\end{equation}
We apply Lemma \ref{lem-two-barrier-ballot-asy-rw} to provide a lower bound for the first term in \eqref{eq-M-con}. To this end, 
choose $m$ large enough such that both $s-t\frac{\mathfrak{n}-m}{\mathfrak{n}}-C_0 \log^2 m$ and 
$u- t\frac{m}{\mathfrak{n}}- C_0 \log^2 m$ are greater than $-\sqrt{m}$. Since $|t| \lesssim \mathfrak{n}^{1/2}$ and  $ 0\le u-s \le u-t+M \lesssim  \mathfrak{n}^{1/2}$,   applying
Lemma \ref{lem-two-barrier-ballot-asy-rw} with $a=1/4$ and  $b>0$ chosen such that $\vartheta_m(\wedge^{\mathfrak{n}+1}(k))^2 \le (\log m)^2 +\zeta^{a,b}_{m-1,\mathfrak{n}+2-m}(k)$; it follows    that there is constant $c_{*}>0$ independent of $m$ satisfying    
\begin{equation}\label{eq-M-con-2}
 \P \bigl( B'_m \mid S_m , S_{\mathfrak{n}-m} \bigr) \ind{ S_m , S_{\mathfrak{n}-m} \in [-2 \sqrt{m}, - \sqrt{m}] } \ge c_{*} \, \frac{[1+(s-t)_+](1+u)}{\mathfrak{n} } .
\end{equation} 
  
Furthermore, a direct density computation yields
\begin{equation}\label{eq-M-con-3}
 \P( S_m ,S_{\mathfrak{n}-m} \in [-2 \sqrt{m}, - \sqrt{m}] \mid S_{\mathfrak{n}+1} = 0) \gtrsim \min_{\lambda_{1},\lambda_{2} \in [1,2]} e^{- \frac{\lambda_{1}^2 + \lambda_{2}^2}{2 \sigma^2_{\max}} } e^{ - (\lambda_{1}-\lambda_{2})^2 \frac{m }{2 \sigma^2_{\max}(\mathfrak{n}-m)} }. 
\end{equation} 
Combining \eqref{eq-M-con}, \eqref{eq-M-con-1}, \eqref{eq-M-con-2}, and \eqref{eq-M-con-3}, we find a constant $c_{**}>0$ such that for large $m$ and   $\mathfrak{n} \gg m$
\[
 \P (B^{2}_{m} \mid h_o =0) \ge c_{**} \frac{[1+(s-t)_+](1+u)}{\mathfrak{n} } .
\] 
This establishes \eqref{eq-lower-bnd} and completes the proof of Proposition \textcolor{magenta}{\ref{prop:one-point-estimate}}.  
\end{proof}

\smallskip
\begin{proof}[Proof of Proposition~\ref{prop:annular-high-point}]
	By the estimate \eqref{eq-bnd-h-z}, it suffices to consider the case where $\wtN  \ge N e^{-(\log N)^{1/4}}$, which implies $\mathfrak{n} \asymp \log N$.
The proof relies on the following two estimates, which are analogous to those established in Proposition~\ref{prop:one-point-estimate}. Let $|t|, |s| \le 2(\log N)^{1/2}$ and $u \in [0, 2(\log N)^{1/2}]$.
\smallskip

\noindent \textbf{Claim 1.} Assume $s < u $ and let $ \hat{\rho} \coloneq \min\{ (u-s)^{50}, \mathfrak{n}/2 \}$. Then,
\begin{equation} 
 \P \bigl( h^*_{z+( \mathsf{B}_{ \wtN  /\rho(u-s;\wtN )} \setminus 
\mathsf{B}_{\rho(u-s;\wtN )})} > m_{\wtN } + s
 , h^*_{z+ \mathsf{B}_{\wtN  }/2} \leq m_{\wtN  } + u  \mid  h_z = m_{\wtN } + t   \bigr)    
  \lesssim  \frac{(1+u)(1+u-t)}{ \mathfrak{n} \, (1+\hat{\rho})^{1/4} } . \label{eq-entrop-gff}
\end{equation}  
 
\noindent \textbf{Claim 2.} Assume   $s < t$ and let $r \in [(1+t-s)^{4}, \tilde{\rho}]$, where $ \tilde{\rho} \coloneq \min\{ (u-t)^{50}, \mathfrak{n}/2 \}$. Then,
\begin{equation}
\P  \bigl(   h^*_{z+\mathsf{B}_{\rho(u-t ; \wtN )}}  \le m_{\wt{N}} +t   , \,
h^*_{z+(\mathsf{B}_{\rho(u-s ; \wtN )} \setminus \mathsf{B}_{2^r})} > m_{\wt{N}} + s  ,\, h^*_{z+ \mathsf{B}_{\wtN  }/2} \leq m_{\wtN  } + u \mid  h_z  =  m_{\wt{N}} +  t  \bigr) \lesssim \frac{(1+u)}{ \mathfrak{n}\, r^{1/4} }\,.
\label{eq-entrop-gff-2}
\end{equation}

\noindent\textit{Conclusion of the proof of Proposition~\ref{prop:annular-high-point}.}
Assuming the two claims for the moment, we combine \eqref{eq-entrop-gff} with the density estimate \eqref{eq-ho-density} to obtain
\begin{align*}
& \P  \bigl( h_z \in m_{\wtN } +  \dif t  ,  h^*_{z+(\mathsf{B}_{\wtN /\rho(u-s;\wtN )} \setminus 
\mathsf{B}_{\rho(u-s;\wtN )})} > m_{\wtN }+ s
 \,,\,\, h^*_{\mathsf{V}_N} \leq m_{\wtN }+ u \bigr) \\
 & \qquad  \lesssim  (1+\hat{\rho})^{-1/4} (1+u) (1+u-t)    \wtN ^{-2}  \rme^{-\alpha t} \dif t\,.
\end{align*}  
Recall that $\Delta_{N,z} \coloneq m_{N} - m_{\wtN_z} = 2\sqrt{g} \log(N/\wtN_{z}) + o_{N}(1)$. Substituting $\Delta_{N,z}+t$, $\Delta_{N,z}+s$, $\Delta_{N,z}+u$ for $t$, $s$, $u$ in the inequality above, and repeating the computation in \eqref{eq-bnd-h-z-2}, we arrive at the claimed inequality \eqref{e:26}. The second assertion \eqref{e:26a} follows analogously, by applying Claim~2 with the same substitution.
\medskip 

\noindent
\underline{Proof of claim \eqref{eq-entrop-gff-2}.} We adopt the notation of Proposition~\ref{prop:one-point-estimate}. Without loss of generality, we may replace $z+\mathsf{B}_{\wtN/2}$ by $\mathsf{D}_{1}$, and $z+(\mathsf{B}_{\rho(u-s;\wtN)} \setminus \mathsf{B}_{2^{r}})$ by $\mathsf{D}_{\mathfrak{n}-\hat{\rho}} \setminus \mathsf{D}_{\mathfrak{n}-r}$. By standard properties of Gaussian vectors, it then suffices to bound, under $\P(\cdot \mid h_o = t)$, the probability of the event
\begin{equation*}
    E^{h} \coloneq \Bigl\{ h_x \le   m_{2^{\mathfrak{n}}}  \hat{f}_{x} + t +(u-t)\ind{x \in \mathsf{D}_{1} \setminus  \mathsf{D}_{\mathfrak{n} - \tilde{\rho} }} , \, \forall x \in \mathsf{D}_{1} \Bigr\} \cap \Bigl\{ \max_{ \mathsf{D}_{ \mathfrak{n}-\hat{\rho}} \setminus \mathsf{D}_{\mathfrak{n}-r}} (h_x - m_{2^{\mathfrak{n}}} \hat{f}_{x}  ) \geq s \Bigr\}\,.
\end{equation*}
More precisely, \eqref{eq-entrop-gff-2} follows at once from the bound
\begin{equation}
 \P(E^{h} \mid h_o = t) =
\sum_{m \ge 2} \P ( \{ \mathscr{M} =m\}\cap E^{h} 
 \mid h_o = t ) \lesssim r^{-1/4} (1+u) \, {\mathfrak{n}}^{-1} . \label{eq-switch-shift-2}
\end{equation}

To establish \eqref{eq-switch-shift-2}, we split the sum according to the size of $m$. For $m \ge r^{1/4}$, dropping the second requirement in the definition of $E^{h}$ and adapting the argument of Lemma~\ref{lem-put-M-in} yields
\[ \P ( \{ \mathscr{M} =m\}\cap E^{h} 
 \mid   h_o = t ) \lesssim e^{-2c \log^2 m}      (1+ u)    \mathfrak{n}^{-1} \lesssim   e^{-c \log^2 m}   r^{-1/4}  (1+u) \, \mathfrak{n}^{-1} .
  \]

Suppose now that $2 \le m \le r^{1/4} \le \tilde{\rho}^{1/4}$. Following the derivation of \eqref{eq-cons-S-k}, the event $\{\mathscr{M} = m\} \cap E^{h}$ is contained in the random walk event $E^{S}_{m}$ defined as follows: with $\tilde{g}(k) \coloneq u\mathbf{1}_{\{k \le \mathfrak{n}-\tilde{\rho}\}} + t \mathbf{1}_{\{k > \mathfrak{n}-\tilde{\rho}\}}$, set
\begin{equation*} 
 E^S_{m} \coloneq \Bigl\{ \max_{1 \le k \le \mathfrak{n}+1 } [S_{k} - \tilde{g}(k) - C_0 \vartheta_{m}(\wedge^{\mathfrak{n}+1}(k))] \le 0 \ ; \ \max_{\mathfrak{n}-\hat{\rho} \le k \le \mathfrak{n}-r } [S_{k} + C_0 \vartheta_{m}(\wedge^{\mathfrak{n}+1}(k))^2] \ge s -1\Bigr\}  .
\end{equation*} 
Split the interval   in the second maximum in $E_m^S$ at $\mathfrak{n}-\tilde{\rho}$: let $E_{m,\mathrm{l}}^S$ and $E_{m,\mathrm{r}}^S$ denote the events obtained by restricting this maximum to $[\mathfrak{n}-\hat{\rho}, \mathfrak{n}-\tilde{\rho})$ and to $[\mathfrak{n}-\tilde{\rho}, \mathfrak{n}-r]$, respectively. Since $\hat{\rho} \ge \tilde{\rho}$, we have 
 \[ 
 E_m^S\subset E_{m,\mathrm{r}}^S\cup E_{m,\mathrm{l}}^S \ ,
\]
 where $E_{m,\mathrm{l}}^S$ is omitted when $\hat\rho=\tilde{\rho}$,   in which case the interval $[\mathfrak{n}-\hat{\rho}, \mathfrak{n}-\tilde{\rho})$ is empty.

Recall that $\{\mathscr{M} = m\} \subset F_{m}$, where $F_{m}$ is the event defined in \eqref{def:Fm}, and that on $F_{m} \cap \{S_{\mathfrak{n}+1} = t\}$ we have $|S_{m-1}|, |S_{\mathfrak{n}+2-m} - t| \le m \log m$. We apply Lemma~\ref{lem-ent-rw} with $a = 1/4$, choosing $b > 0$ so that $\vartheta_m(\wedge^{\mathfrak{n}+1}(k))^2 \le (\log m)^2 + \zeta^{a,b}_{m-1,\mathfrak{n}+2-m}(k)$. Applying \eqref{ent:right} to $E_{m,\mathrm{r}}^S$, with repulsion depth $t-s+O((\log m)^2)$, then gives, uniformly over all admissible values of $S_{m-1}$ and $S_{\mathfrak{n}+2-m}$,  
\[ \P\bigl(E_{m,\mathrm{r}}^S\mid S_{m-1},S_{\mathfrak n+2-m}\bigr) 
  \lesssim \frac{(1+u)(m\log m)^2}{\mathfrak n}
 \frac{(t-s)+r^{1/4}}{r^{1/2}}
 \left[1+\frac{u-t}{(\tilde{\rho}+2-m)^{1/2}}\right] \ \text{ on }\  F_{m} .\]
Since $t-s\le r^{1/4}$, $u-t \lesssim \tilde{\rho}^{1/2}$, and  $\tilde{\rho}+2-m \asymp \tilde{\rho}$, the right-hand side is dominated by $\frac{ (m \log m)^2}{\mathfrak{n} } \frac{(1+u)}{r^{1/4}}.$  
 
It remains to treat $E_{m,\mathrm{l}}^S$, which is present only when $\hat{\rho} > \tilde{\rho}$. In this case $\tilde{\rho}=(u-t)^{50}<\mathfrak n/2$.  Put $r ''= (\tilde{\rho}+2-m)/2 \asymp\tilde{\rho}$. The endpoint assumptions in \eqref{ent:left} hold as $|S_{\mathfrak n+2-m}-t|\le m\log m\lesssim\tilde{\rho}^{1/2}$ on $F_m$.  Hence \eqref{ent:left}, applied with repulsion depth $u-s+O((\log m)^2)$, yields
\begin{align*}
 \P\bigl(E_{m,\mathrm{l}}^S\mid S_{m-1},S_{\mathfrak n+2-m}\bigr)
 & \ \lesssim \frac{(1+u)(m\log m)^2}{\mathfrak n}
 \frac{(u-s)+\tilde{\rho}^{1/4}}{\tilde{\rho}^{1/2}}
 \left[1+\frac{u-t}{(\tilde{\rho}+2-m)^{1/2}}\right] \ \text{ on }\  F_{m} .
\end{align*}
Since  $ 
 u-s=(u-t)+(t-s) \lesssim \tilde{\rho}^{1/50}+ r^{1/4}\lesssim\tilde{\rho}^{1/4}$,  combining the two estimates above and using $\tilde{\rho} \ge r$, we conclude that
\[ \P \bigl( E_{m}^{S} \mid S_{m-1} , S_{\mathfrak{n}+2-m} \bigr) 
  \lesssim \frac{ (m \log m)^2}{\mathfrak{n} } \frac{(1+u)}{r^{1/4}}. \] 
Finally, using the estimate $\P (F_{m} \mid S_{\mathfrak{n}+1}=t) \lesssim e^{-2 c \log^{2} m}$ from Lemma~\ref{lem:41} and $|t| \lesssim \mathfrak{n}^{1/2}$, we obtain:
\[ 
\P ( \{ \mathscr{M} =m\}\cap E^{h} \mid h_o = t ) \le \P ( F_{m} \cap E^{S}_{m} \mid S_{\mathfrak{n}+1} = t ) \lesssim e^{-c \log^2 m} \frac{1}{r^{1/4}}   \frac{ (1+u)}{ \mathfrak{n}}\,.
\]
Summing these bounds over $m \ge 2$ yields \eqref{eq-switch-shift-2}, and hence completes the proof of Claim~2.

\smallskip 
\noindent
\underline{Proof of claim \eqref{eq-entrop-gff}.} 
The proof of \eqref{eq-entrop-gff} follows the same strategy and is in fact simpler, since it requires an entropic repulsion estimate with only a single barrier (namely, Lemma~\ref{lem-ent-rw} applied with $x_0 = y_0$); we therefore omit the details.
\end{proof}

\subsection{Two point estimate: Proof of Proposition \ref{prop:two-point-estimate}}

 Since the local maximum of the DGFF is almost surely unique, the probability in \eqref{e:26a} is nonzero only if
\begin{equation}\label{e:176}
 |x-y|  \ge \rho(u- t ;   N) \vee \rho(u-s ;   N).
\end{equation}
 Indeed, suppose $s\ge t$  so that $h_{y}\ge h_{x}$ (the other case is symmetric), the condition $h_{x}= h^{*}_{x + \mathsf{B}_{\rho(u-t;N)}}$  implies that $y \not\in x + \mathsf{B}_{\rho(u-t;N)}$.

\medskip
 \noindent
\underline{\emph{Step 1.} } We show that in the near diagonal case $|x-y| \le 10^{10}$,   \eqref{e:27} follows from  Proposition~\ref{prop:one-point-estimate}.  
 Suppose  $t\le s$, then we have  $y\notin x+\mathsf{B}_{\rho(u-t;N)}$. We claim that 
  the left-hand side of \eqref{e:27}, is bounded above by
\begin{equation}
	\label{eq:two-to-one-point-1}
	 \P( h_x=h^*_{x+\mathsf{B}_{\rho(u-t;N)}}\in m_N+\dif t,\ 
	 h^*_{ \mathsf{V}_{\eta' N\setminus\{y\}}} \le m_N+u ) \dif s . 
\end{equation}
This is because,  
conditionally on    $\cF^{ \{y\}^{c} }\coloneq  \sigma( h_{z}:  z \in \mathbb{Z}^2 \setminus \{y\} )$,   $h_y$ is Gaussian with uniformly bounded density. Moreover,      on the event $ \{ h^*_{ \mathsf{V}_{\eta' N}\setminus\{y\}} \le m_N+u \} $ all the neighbors of $y$ are at most $m_N+u$, and hence the Gibbs--Markov property gives
\begin{equation*}
 \P\bigl[h_y\le m_N+u\mid \cF^{ \{y\}^{c} } \bigr] \ge 1/2 \ \text{ on } \ \{ h^*_{ \mathsf{V}_{\eta' N} \setminus\{y\}}\le m_N+u \}.
\end{equation*}
Thus \eqref{eq:two-to-one-point-1} is bounded above by  
\begin{equation}
	\label{eq:two-to-one-point-2}
	 2 \P\bigl(h_x=h^*_{x+\mathsf{B}_{\rho(u-t;N)}}\in m_N+\dif t,
	h^*_{ \mathsf{V}_{\eta' N}}\le m_N+u\bigr) \dif s  \, \lesssim \, (1+u)N^{-2}e^{-\alpha t}\dif t\dif s.  
\end{equation}
The last inequality follows from Proposition~\ref{prop:one-point-estimate}, together with the fact $\wtN_x \asymp_{\eta}N$, and  $x+\mathsf{B}_{\wtN_x/2}\subset\mathsf{V}_{\eta' N}$.Since $|x-y| \le 100$,  $u-t,u-s = O(1)$, the right-hand side in  \eqref{eq:two-to-one-point-2} is bounded by the right-hand side of \eqref{e:27}, after adjusting the implicit constant.  This proves \eqref{e:27} in the case $|x-y| \le 100$.   

\medskip
 \noindent
\underline{\emph{Step 2.} }
We fix a small constant   $\delta \in (0,  2^{-4}(1-\eta)) $, and henceforth focus on the case $|x-y| <\delta N$. The complementary case  $|x-y| \ge \delta N$ is technically simpler and is addressed in Remark \ref{case-x-far-from-y}.

 Throughout this proof, we take the reference point  $o = (x+y)/2 \in \mathsf{V}_{\eta N}$,  and employ the concentric decomposition from Section~\ref{sec:5.1}.  
Let  $\ell = \ell_{x,y} \geq 2$ be the largest integer such that $x,y \in \mathsf{D}_{\ell+2}$. Note that such an $\ell \ge 2$ must exist, otherwise $|x-y| \ge 2^{\mathfrak{n}_o -3} \ge \mathrm{dist}(o,\partial \mathsf{V}_{N}) 2^{-4} \ge 2^{-4}(1-\eta) N  $ which contradicts the choice of $\delta$.  
By the choice of $\ell$, we have $2^{\mathfrak{n}_o - \ell -2}  \le 	|x-y| \le 2^{\mathfrak{n}_o - \ell-1} $. Combining this with \eqref{e:176}, and recalling that $ \tilde{\rho} \coloneq \min\{ (u-t)^{50} , \mathfrak{n}_o /2 \} $, $ \hat{\rho} \coloneq \min\{ (u-s)^{50} , \mathfrak{n}_o /2 \}$, we obtain
\begin{equation}\label{e:176b}
\mathfrak{n}_o - \ell \ge \tilde{\rho} \vee \hat{\rho} . 
\end{equation} 
Moreover, the bound $|x-y|\le 2^{\mathfrak n_o-\ell-1}$ together with $|x-y|\ge100$ implies that $\mathfrak n_o-\ell-4\ge4$.

Set $\ell' \coloneq \ell+4$ and define
\begin{equation*}
\mathsf{D}^x_{j}  \coloneq  x + \mathsf{B}_{2^{\mathfrak{n}_o-\ell'-j}} \ , \ \ 
\mathsf{D}^y_{j}  \coloneq  y + \mathsf{B}_{2^{\mathfrak{n}_o-\ell'-j}}  \ , \ \ 
\mathsf{D}^{x,y}_{j}  \coloneq \mathsf{D}^x_{j}  \cup \mathsf{D}^y_{j}  \text{ for } 0 \le j \le \mathfrak{n}_o- \ell'.
\end{equation*} 
See Figure \ref{fig-two-point} for an illustration. 
Observe that  $\mathsf{D}^{x,y}_{0} \subset  \mathsf{D}_{\ell+1} $. 
\begin{figure}[tp]
	\centering
	\includegraphics[scale=0.4]{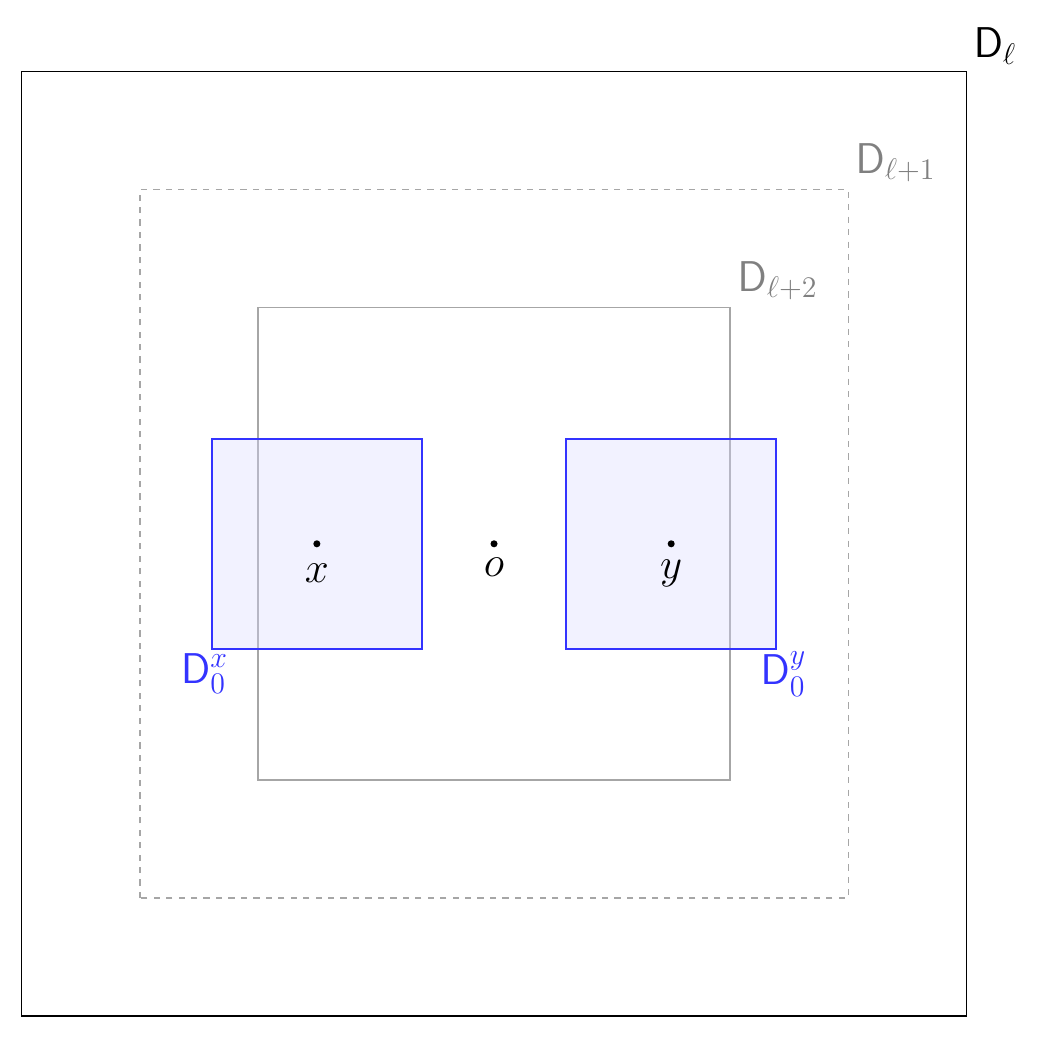}
	\caption{$\mathsf{D}^x_0$ and $\mathsf{
	D}^y_0$}\label{fig-two-point}
\end{figure}  

Next we bound the probability in~\eqref{e:27}.  
 Note that   $\{h_x = h^*_{x+\mathsf{B}_{\rho(u-t; N)}} \in m_N + \dif t, h^*_{\mathsf{V}_{\eta' N}} \leq m_N + u \}$ is included in the following event
\begin{equation*}
 A_{x,t}   \coloneq \bigl\{ h_x \in m_{N}+\dif t,   h^*_{\mathsf{D}^x_{   (\mathfrak{n}_o -\ell'- \tilde{\rho} )\vee 1} } \le m_N + t   \,,\,  h^*_{\mathsf{D}^x_{1}} \leq m_N + u \bigr\} .
\end{equation*}
 Define $A_{y,s}$  analogously, replace $(x,t,\tilde{\rho})$ by $(y,s,\hat\rho)$. Furthermore, let 
\[B_{\ell}  \coloneq \bigl\{ h^*_{\mathsf{D}_1 \setminus \mathsf{D}_{\ell}} \leq m_N +u \bigr\}\,. \]
 The choice $\eta'=(3+\eta)/4 $ implies   
$\mathsf{D}_1\subset\mathsf{V}_{(1+\eta)N/2}\subset\mathsf{V}_{\eta' N}.$
In particular, the constraint $h^*_{\mathsf{V}_{\eta' N}}\le m_N+u$ implies $B_\ell$. 
Conditioning on the $\sigma$-field $\cF^{  (\mathsf{D}^{x,y}_{0})^{c} }\coloneq  \sigma( h_{z}:  z \in \mathbb{Z}^2 \cap  (\mathsf{D}^{x,y}_{0})^{c} )$, and applying the Gibbs--Markov property, the probability in~\eqref{e:27} is bounded from above by
\begin{equation}
\label{e:182}
\mathrm{Pr}_{\eqref{e:182}} \coloneq 
\E  \Bigl[ \P  \bigl( A_{x,t} \mid \cF^{(\mathsf{D}^{x,y}_{0})^\rmc} \bigr) \P  \bigl( A_{y,s} \mid \cF^{(\mathsf{D}^{x,y}_{0})^\rmc} \bigr) \ind{B_{\ell}} \Bigr] .  
\end{equation}

\smallskip
 \noindent
\underline{\emph{Step 3.} }
To bound the conditional probabilities of $A_{x,t}$ and $A_{y,s}$ as well as the probability of $B_{\ell}$, we rely on Lemmas~\ref{l:10} and \ref{l:11} stated below (with proofs postponed to the end of this subsection). 

The Gibbs--Markov property of the DGFF yields the decomposition
\begin{align}
\label{e:180}
h_z  \,=\, \Phi^{ \mathsf{D}_{0} \mid  \partial \mathsf{D}^{x,y}_{0} }_z +  h^{\mathsf{D}^{x}_{0}}_z + h^{\mathsf{D}^{y}_{0}}_z    \, , \, z \in \mathsf{D}^{x,y}_{0}\,,
\end{align}
where  these three fields are independent and are defined in \eqref{def-binding-field}.  

The estimate in Lemma \ref{l:10}  involves both the backbone random walk $S_{\ell} = \Phi^{\mathsf{D}_0 \mid \partial \mathsf{D}_{\ell}}_o$  (see \eqref{def-S-l}), 
and the oscillations of the binding field $\Phi^{\mathsf{D}_0 \mid \partial \mathsf{D}^{x,y}_{0}}$ around $S_{\ell}$. To make this precise, we define
\begin{equation*}
\Delta^{x,y}_{z}  \coloneq    \Phi^{\mathsf{D}_0 \mid \partial \mathsf{D}^{x,y}_{0}}_z - S_{\ell}  \ , \  |\Delta|^{x,y}  \coloneq \max_{z \in \mathsf{D}^{x,y}_{1}} \big|   \Delta^{x,y}_{z}   \big|  \ \text{ and } \
\wh{S}_{\ell}  \coloneq S_{\ell} - ( m_N - m_{2^{\mathfrak{n}_o-\ell'}} ) \,.
\end{equation*}
Note that the quantities   $\Delta^{x,y}_{z}$, $|\Delta|^{x,y} $ and $S_{\ell}$ are measurable with respect to $\cF^{(\mathsf{D}^{x,y}_{0})^\rmc}$.
   
\begin{lem}\label{l:10}
	There exists a constant $C_{\ref{l:10}}>0$ depending only on $\eta$ such that  almost surely 
\begin{equation}
	\label{e:187}
\P  \bigl(A_{x,t} \mid \cF^{(\mathsf{D}^{x,y}_{0})^\rmc} \bigr)
\leq C_{\ref{l:10}} \, 4^{- (\mathfrak{n}_o-\ell')} \, 
(1+|\wh{S}_l|+|\Delta|^{x,y})^2 \,\rme^{ \alpha |\Delta|^{x,y}+\alpha \wh{S}_l} \, (1+u)\,  \rme^{-\alpha t}  \dif t  .
\end{equation}
An identical statement holds for $A_{y,s}$  with $y,s$ replacing $x,t$. 
\end{lem}

\begin{lem}
\label{l:11} Set $w_{n}(\ell) \coloneq \frac{\ell(n-\ell)}{n}$. There exists a constant $c>0$ such that 
for all $v \in \bbR$ and $m \ge 2$,  
\begin{equation}\label{e:201}
 \mathrm{Pr}_{\eqref{e:201}}(v;m) \coloneq \P  \bigl( B_{\ell} \, , \, \mathscr{M}_{\ell} = m  \, , \, \wh{S}_{\ell} \in \dif v \bigr)  \lesssim    \rme^{-c \log^2 m}  \,\frac{ (1+u)^2(1+|v|)^{2}  }{  4^{\ell} \, [ w_n(\ell)]^{3/2}}
 \rme^{-\alpha v} 
  \,\ind{v \le u+C_{\ref{lem:40}}  \log  m} \dif v .
\end{equation} 
\end{lem}

\begin{proof}[\underline{Step 4. Proof of Proposition~\ref{prop:two-point-estimate} admitting Lemmas \ref{l:10} and \ref{l:11}}]
	It follows from   Lemma \ref{l:10}   that 
\begin{equation}
\label{e:187.1}
\mathrm{Pr}_{\eqref{e:182}}  \lesssim   (1+u)^2 
 4^{-2(\mathfrak{n}_o-\ell)}
\rme^{-\alpha (t+s)} \dif t  \rmd s \times 
\E  \bigl[  \rme^{2\alpha |\Delta|^{x,y}}  \rme^{2\alpha \wh{S}_l} (1+|\wh{S}_l|)^{4}(1+|\Delta|^{x,y})^{4} \, \ind{B_\ell} \bigr] \,.
\end{equation}  
To get rid of the term $|\Delta|^{x,y}$, 
by the Gibbs--Markov property again we may write
\begin{equation*}
\Delta^{x,y}_{z} =  \Phi^{ \mathsf{D}_0  \mid \partial \mathsf{D}^{x,y}_{0}}_{z} - S_{\ell} = \Phi^{ \mathsf{D}_0 \mid   \partial \mathsf{D}_{\ell}  }_z - S_{\ell}  + \Phi^{    \mathsf{D}_{\ell} \mid   \partial \mathsf{D}^{x,y}_{0} }_z 
\text{ for }
z \in \mathsf{D}^{x,y}_{0} \subset \mathsf{D}_{\ell+1}\,,
\end{equation*}
where  $ \Phi^{ \mathsf{D}_0  \mid \partial \mathsf{D}_{\ell}}  $ and $ \Phi^{  \mathsf{D}_{\ell} \mid   \partial \mathsf{D}^{x,y}_{0} }  $ are independent.  Then the  triangle inequality gives 
\begin{equation*}
|\Delta|^{x,y} \leq \max_{z \in \mathsf{D}_{\ell+1}} \big| \Phi^{ \mathsf{D}_0 \mid   \partial \mathsf{D}_{\ell}  }_z - S_{\ell} \big| + \max_{z \in \mathsf{D}_{\ell+1}} | \Phi^{    \mathsf{D}_{\ell} \mid   \partial \mathsf{D}^{x,y}_{0} }_z  |.
\end{equation*}   
Since $\max_{z \in \mathsf{D}_{\ell+1}} | \Phi^{    \mathsf{D}_{\ell} \mid   \partial \mathsf{D}^{x,y}_{0} }_z  |$ has a Gaussian right tail with uniform constants\footnote{This follows from the Borell--TIS inequality together with Dudley’s entropy bound, similarly to the derivation of~\eqref{e:56.5}.} and is independent of the other terms, 
we can integrate it out from the expectation in~\eqref{e:187.1} at the cost of multiplicative constant.   
Moreover,  Lemma~\ref{lem:40} yields
$ \max_{z \in \mathsf{D}_{\ell+1}} \big| \Phi^{ \mathsf{D}_0 \mid   \partial \mathsf{D}_{\ell}  }_z - S_{\ell} \big|  \leq  C_{\ref{lem:40}}\log \mathscr{M}_{\ell}$. 
Plugging these  into \eqref{e:187.1} yields  
\begin{equation}
\label{e:192}
\mathrm{Pr}_{\eqref{e:182}}  \lesssim 
 (1+u)^{2} 
 4^{-2(\mathfrak{n}_o-\ell)}
\rme^{-\alpha (t+s)} \dif t  \rmd s \times 
\E  \bigl[  \rme^{ C \log \mathscr{M}_{\ell}} \rme^{2\alpha \wh{S}_l} (1+|\wh{S}_l|)^{4} (1+\log \mathscr{M}_{\ell})^{4} \,\ind{B_{\ell}} \bigr] \,.
\end{equation}
Lemma~\ref{l:11} bounds the expectation $ \E  [ \rme^{ C \log \mathscr{M}_{\ell} + 2\alpha \wh{S}_l} (1+|\wh{S}_l|)^{4} (1+\log \mathscr{M}_{\ell})^{4} \,\ind{B_{\ell}}  ]$ by
\begin{align}  
 &\sum_{m=2}^\infty    \rme^{-c \log^2 m+(C +4)\log m}  \int_{ -\infty}^{u+C_{\ref{lem:40}}  \log  m}
 \,\frac{ (1+u)^2(1+|v|)^{6} }{  4^{\ell} \, [ w_n(\ell)]^{3/2}}
 \rme^{ \alpha v} 
  \, \dif v  \,. \notag \\
 &\lesssim \frac{ (1+u)^{8} \rme^{\alpha u} }{  4^{\ell} \, [ w_n(\ell)]^{3/2}} \sum_{m=2}^{\infty}   \rme^{-c \log^2 m + C' \log m} \lesssim  \frac{ (1+u)^{8} \rme^{\alpha u}  }{  4^{\ell} \, [ w_n(\ell)]^{3/2}}. \label{e:194}
\end{align}
 Finally recall that $\mathfrak{n}_o= n +O_{\eta}(1)$ so that  $4^{\mathfrak{n}_{o}} \asymp N^{2}$; by the choice of $\ell$ we have  $4^{\mathfrak{n}_o-\ell} \asymp [1+|x-y|]^{2}$. Moreover $w_{n}(\ell)=  {\ell(n-\ell)}/{n} \asymp  {\ell(\mathfrak{n}_o-\ell)}/{n} \asymp (\log_+ |x-y|)(\log_+ \frac{N}{|x-y|} )/ \log N $.  The   desired result then follows by combining  \eqref{e:192} and \eqref{e:194}. 
\end{proof}

\begin{rem}
\label{rmk:two-point} 
It is worth emphasizing that while this proposition is sufficient for our current needs, it is likely not optimal. We expect that the pre-factor $(1+u)^{10}e^{\alpha u}$ in \eqref{e:27}  can be sharpened to a term of the form $(1+u)^{c} e^{\alpha u}$ for some $c \in (0,1)$. For a comparable result in the context of branching Brownian motion, see \cite[Lemma A.2]{HLW25}, where the exponent is $c=3/4$.
\end{rem}

\begin{rem}\label{case-x-far-from-y}
In the regime $|x-y| \ge \delta N$, the argument is simpler because Lemma~\ref{l:11} is not required.   The proof follows by a straightforward adaptation of the argument above and is therefore omitted.
\end{rem}
  
It remains to state and prove Lemma~\ref{l:10} and~\ref{l:11} on which the above proof relies.
 
\begin{proof}[\underline{Step 5. Proof of Lemma \ref{l:10}}] 
 For simplicity, we write $\check{h}$ for $h^{\mathsf{D}^{x}_0}$; and $\check{N}= 2^{\mathfrak{n}_o-\ell'}$. Then \eqref{e:180} becomes 
	\begin{equation}\label{eq-decom-h-on-D-x} 
 h_z  = \check{h}_{z} + \Delta^{x,y}_{z} +  \wh{S}_{\ell}+  (m_{N}-m_{\check{N}})  \, ,  \ z \in \mathsf{D}^{x}_{0}\,.
\end{equation}
  Conditioned on $\cF^{(\mathsf{D}^{x,y}_{0})^\rmc}$, the field $\check{h}$ follows the law of the DGFF on a box of side length $2 \check{N}+1$, while $\wh{S}_{\ell}$ and $\Delta^{x,y}_{z}$ act as constants. Hence we may apply Proposition \ref{prop:one-point-estimate}.    
  In light of~\eqref{eq-decom-h-on-D-x},  we have  $A_{x,t} \subseteq \check{A}_{x,t}$, where   $\check{A}_{x,t}$ is defined as: 
 \begin{equation*}
 \bigl\{  \check{h}_x \in  m_{\check{N}} - \wh{S}_{\ell}-\Delta^{x,y}_{x}+\dif t  \,,\, 
	 \check{h}^*_{\mathsf{D}^x_{   (\mathfrak{n}_o -\ell'- \tilde{\rho} )\vee 1} } \le  m_{\check{N}} + |\wh{S}_{\ell}| + |\Delta|^{x,y} + t+1 \, , \, 
 \check{h}^*_{\mathsf{D}^x_{1}} \leq m_{\check{N}} +| \wh{S}_{\ell}| + |\Delta|^{x,y} + u  \bigr\} .
 \end{equation*} 
 We now verify the requirements for Proposition \ref{prop:one-point-estimate}. First, note that $\mathrm{dist}(x,\partial \mathsf{D}^{x}_{0})=\check{N}$.
Furthermore, by the  assumption $-(\log N)^{1/2} \le t \le u$ and $ 0 \le u < (\log N)^{1/2}$, \eqref{e:176b}, and the scaling $\mathfrak{n}_o \asymp \log N $,    we have:
\begin{equation*}
 u- t \lesssim 1+\tilde{\rho}^{1/2} \lesssim (\mathfrak{n}_o- \ell' )^{1/2} \lesssim (\log \check{N})^{1/2} .
\end{equation*}   
However, the bound required for the shift term $\wh{S}_{\ell}-\Delta^{x,y}_{z}+  t$ is not guaranteed to hold. We therefore distinguish between two cases: 
\begin{enumerate}[(1)]
	\item If $|\wh{S}_{\ell}|+|\Delta^{x,y}| \leq (\mathfrak{n}_o - \ell')^{1/2}$, then since $|t| \le u-t \lesssim (\mathfrak{n}_o - \ell)^{1/2}$,  the conditions of Proposition \ref{prop:one-point-estimate} are satisfied. Applying the proposition yields: 
	\begin{equation*}
	\P  \bigl( \check{A}_{x,t} \mid \cF^{(\mathsf{D}^{x,y}_{0})^\rmc} \bigr)
\lesssim  \check{N}^{-2}
  \rme^{-\alpha t}  \rme^{ \alpha |\Delta|^{x,y}} \rme^{\alpha \wh{S}_l} (1+|\wh{S}_l|+|\Delta|^{x,y})(1+u)  \dif t  \,. 
	\end{equation*}  
	\item 
	If $|\wh{S}_{\ell}|+|\Delta^{x,y}| \geq (\mathfrak{n}_o - \ell')^{1/2}$, we employ a crude upper bound by relaxing the maximum constraints and considering only the   density $\{\check{h}_x \in m_{\check{N}} - \wh{S}_{\ell}-\Delta^{x,y}_{z}+\dif t \}$. Analogously to \eqref{eq-ho-density}, we obtain:  
	\begin{equation*}
	\P  \bigl( \check{A}_{x,t} \mid \cF^{(\mathsf{D}^{x,y}_{0})^\rmc} \bigr)
\lesssim (\mathfrak{n}_o - \ell') \check{N}^{-2}
  \rme^{-\alpha t} \rme^{ \alpha |\Delta|^{x,y}}\rme^{\alpha \wh{S}_l}   \dif t  \le
 (|\wh{S}_l|+|\Delta|^{x,y})^2 \check{N}^{-2}
  \rme^{-\alpha t} \rme^{ \alpha |\Delta|^{x,y}}\rme^{\alpha \wh{S}_l}   \dif t\,.
	\end{equation*} 
 \end{enumerate}   
Combining these two cases establishes the claimed result \eqref{e:187}. The same argument holds for $A_{y,s}$, replacing $x,t$ with $y,s$.
\end{proof}

\begin{proof}[\underline{Step 6. Proof of Lemma \ref{l:11}}] 
According to Lemma~\ref{lem:40}, on the event $\{\mathscr{M}_{\ell}=m\}$ we have  $   h^{*}_{\mathsf{A}_{k}}   - m_{2^{\mathfrak{n}_o- k} } \ge  S_{k} - C_{\ref{lem:40}} \log (\max\{m ,\wedge^{\ell}(k)\})$.
Combining this with the constraint $h^{*}_{\mathsf{D}_{1}\setminus \mathsf{D}_{\ell}} \le m_{N}+u$, we derive that
\begin{equation*}
    S_{k}  - (m_{N} -m_{2^{\mathfrak{n}_o - k}} )  \le u + C_{\ref{lem:40}} \log ( \max\{ m  , \wedge^{\ell}(k) \} ) \ \text{ for any } 2 \le k \le \ell.
\end{equation*}
In particular, taking $k=\ell$, this implies that  $\mathrm{Pr}_{\eqref{e:201}}(v;m)$ is non-zero only if $v \le u+ C_{\ref{lem:40}} \log m$.

Next, we write $
 \mathrm{Pr}_{\eqref{e:201}}(v;m) $   as $ \P  ( B_{\ell} \,,\, \mathscr{M}_{\ell} =m \mid \wh{S}_{\ell} = v ) \,
 \P  (\wh{S}_{\ell} \in \dif v ) $.
Since  
$\mathfrak{n}_o =n+O_{\eta}(1)$ and $\ell'=\ell+4$, we have $(m_N - m_{2^{\mathfrak{n}_o-\ell'}}) = m_{2^{\ell}} + \frac{3}{2 \alpha} \log w_{n}(\ell) + O_{\eta}(1)$. Moreover, $S_{\ell}$ is a centered Gaussian with variance  $\E  [S_{\ell}^2] = g \log 2^{\ell} + O(1)$. Using the standard Gaussian density, we obtain:
\begin{equation}
\P  \bigl(\wh{S}_{\ell}  \in  \dif v \bigr) 
 = \P  \Bigl(S_{\ell} \in  m_{2^{\ell}}  + \frac{3}{2 \alpha} \log w_{n}(\ell) + O(1) + \dif v \Bigr)  \lesssim \ell \, [w_n(\ell)]^{-3/2} \, 4^{-\ell}  \rme^{-\alpha v - c v^2/{\ell}} \,.   \label{eq-hatS-v}
\end{equation} 
  
We then obtain a uniform upper bound by temporarily discarding the event $B_{\ell}$ in the definition of $\mathrm{Pr}_{\eqref{e:201}}(v;m)$ and applying Lemma~\ref{lem:41} directly: 
\begin{align}
 \mathrm{Pr}_{\eqref{e:201}}(v;m) &\le  \P \bigl( 
  \mathscr{M}_{\ell} = m \mid \wh{S}_{\ell} =v \bigr) \P  \bigl(\wh{S}_{\ell}  \in  \dif v \bigr)  \notag\\
  & \lesssim \frac{ \ell\, \rme^{-\alpha v }   }{[w_n(\ell)]^{ 3/2}4^{ \ell}}   \, \rme^{- c v^2/{\ell}} \rme^{-c [ \log  m - \frac{|v|}{\ell}-C ]_+^{2}}  \lesssim \frac{\ell\,   \rme^{-\alpha v }   }{[w_n(\ell)]^{ 3/2}4^{ \ell}} \,  e^{-c' \log^2 (m)} . \label{eq:tem-bnd}
\end{align}
Here we  used the elementary inequality $  v^2/{\ell} +[\log  m - \frac{|v|}{\ell}-C ]_+^{2} \ge \frac{1}{4} \log^{2} (m)-\frac{1}{2} C^2$, which can be verified by considering separately the cases $\tfrac{|v|}{\ell} < \tfrac{1}{2} \log m$ and $\tfrac{|v|}{\ell} \ge  \tfrac{1}{2} \log m$. This estimate allows us to conclude \eqref{e:201} in two specific regimes:  
\begin{enumerate}[(1)]
    \item If $|v| \ge \sqrt{\ell}$, using $\ell \lesssim (1+|v|)^2$ in \eqref{eq:tem-bnd} yields \eqref{e:201}.
    \item If $m \ge \sqrt{\ell}$, \eqref{e:201} follows since $\ell \, \rme^{-c' \log^{2} m} \lesssim \rme^{-c'' \log^{2} m}$.
\end{enumerate}

It remains to handle the case $|v| \le \sqrt{\ell}$ and $ 2 \le m \le \sqrt{\ell}$. We proceed as in the proof of Proposition \ref{prop:one-point-estimate} and   claim that
\begin{equation}\label{e:165} 
\mathrm{Pr}_{\eqref{e:165} }(m,v) \coloneq
\P \bigl( h^*_{\mathsf{D}_1 \setminus \mathsf{D}_{\ell}} \leq m_N + u    \ , \ \ 
  \mathscr{M}_{\ell} = m \mid \wh{S}_{\ell} =v \bigr) 
\lesssim  \frac{ (1+u)  [1+ (u-v)_+]   }{\ell} \rme^{-c \log^2 m}.
\end{equation}  
To show this, for $K \ge 0$ we define the event 
 $F_{m}(K)\coloneq \{ \, \exists \, k  \text{ s.t. } \wedge^{\ell}(k) \le m-1 ,  R_k +K  >  \log  (m-1) \} \cap \{ |S_{m}|, |S_{\ell}-S_{\ell-m}| \le (1+K)m \log m \}$.  
By Lemmas~\ref{lem:40} and \ref{lem:41} we have $\{\mathscr{M}_{\ell}=m\} \subset F_{m}(0)$. Hence,
 \[
\mathrm{Pr}_{\eqref{e:165} } \le \P \bigl( F_{m}(0) , S_{k}  - (m_{N} -m_{2^{\mathfrak{n}_o - k}} )  \le u + C_{\ref{lem:40}} \log ( \wedge^{\ell}(k)  ) \ \text{ for all }  m \le k\le \ell-m  \mid \wh{S}_{\ell} = v \bigr) .
 \]
We may replace the condition $\{\wh{S}_{\ell} = v \}$ by $\{ S_{\ell}= v\}$, since   $ \mathrm{Law}( \{S_{k}\}_{k=1}^{\ell} | \wh{S}_{\ell}= v )= \mathrm{Law} ( \{ S_{k} \}_{k=1}^{\ell} + \frac{\E[S_{k}^2]}{\E[S_{\ell}^2]  }(m_{N}-m_{2^{\mathfrak{n}-\ell'}}) | S_{\ell}= v ) $.
By Lemma~\ref{lem:39} (i) we have 
$\frac{\E[ S_{k}^2 ]}{\E[S_{\ell}^2]} = \frac{k+O(1)}{\ell} $.    Applying Lemma \ref{lemma-3.8}, we find   $ |\frac{k+O(1)}{\ell}[m_{N}-m_{2^{\mathfrak{n}-\ell'}}]- (m_{N}-m_{2^{\mathfrak{n}-k}})| \lesssim \log (\wedge^{\ell}(k))$. 
Moreover, $\mathrm{Law}(\{|\varphi_{k}(o)|\}_{k=1}^{\ell} | \wh{S}_{\ell}= v )$ is stochastically dominated by $ \mathrm{Law}(  \{  |\varphi_{k}(o) |+ C \}_{k=1}^{\ell} \mid S_{\ell}= v   )$ for a sufficiently large constant $C>0$. Combining these observations, we obtain  
\begin{align*}
 \mathrm{Pr}_{\eqref{e:165} }  
 &\leq \P \bigl( F_{m}(C) , S_{k}   \le u + C  \log ( \wedge^{\ell}(k)  ) \ \text{ for all }  m \le k\le \ell-m  \mid S_{\ell} = v \bigr) \\
 & \leq \E\bigl[  \P(   S_{k}   \le u + C  \log ( \wedge^{\ell}(k)  )  \mid S_{m}, S_{\ell -m})  \ind{F_{m}(C)} \mid S_{\ell}= v\bigr] \lesssim \frac{ (1+u)  [1+ (u-v)_+]   }{\ell} \rme^{-c \log^2 m}\,,
\end{align*} 
as claimed. 
Here,  we  applied the classical Ballot theorem (Lemma~\ref{lem-ballot-rw}) to bound the inner probability and used Lemma~\ref{lem:41},   with the assumptions $|v|\le\sqrt{\ell}, m\le\sqrt{\ell}$, to get
$\P(F_m(C)\mid S_\ell=v)\lesssim e^{-c\log^2 m}$.

Finally, combining the bound \eqref{e:165} with the density estimate \eqref{eq-hatS-v} establishes \eqref{e:201} for the remaining regime. This completes the proof.
\end{proof}

\subsection{Asymptotic statements: Proof of Proposition \ref{prop:cluster-law}}  
Throughout the proof we take the reference point $o=z$.
 Let $\Delta_{K} \coloneq m_{KN}-m_{N}=2\sqrt{g}\log K +o_{N}(1)$.  Define the shifted levels $t_{o} \coloneq t + \Delta_{K}-\psi_{o}$ and   $u_o \coloneq u+\Delta_K-\psi_o$. Note that the condition $g_{o}= m_{KN}+t$ is equivalent to  $h_{o}=m_{N}+t_{o}$. We define 
 \begin{equation*} 
\mathrm{Pr}_{\eqref{eq:g-locmax-w}} (t,u;\omega)   \coloneq \P  \bigl(  g_{o+y} \le g_{o} + \omega_{y} , \forall y \in \mathsf{B}_{r},   g^*_{ \mathsf{D}_{\mathfrak n-\varrho}} \le m_{KN} + t ,  g^*_{\mathsf{V}_{N}} \leq m_{KN} + u \mid  g_o =   m_{KN} +   t \bigr)\,.   \label{eq:g-locmax-w}
 \end{equation*} 
Since $o\in\mathsf{V}_{\eta N}$ and $\eta' >(1+\eta)/2$, for all sufficiently large $N$, we have $1\le\mathfrak n-\varrho\le\mathfrak n-1$ and $  \mathsf{B}_{2^\varrho}(o)=\mathsf{D}_{\mathfrak n-\varrho}\subset\mathsf{D}_1
 \subset\mathsf{V}_{(1+\eta)N/2}\subset\mathsf{V}_{\eta'N}$.
 Recalling from \eqref{e:2.11} that
 $\max_{x\in\mathsf{V}_{\eta'N}}|\psi_o-\psi_x|\le M_0$, we have, for every $x\in\mathsf{V}_{\eta'N}$, $
 \{h_x\le m_N+t_o-M_0\}\subset\{g_x\le m_{KN}+t\}
 \subset\{h_x\le m_N+t_o+M_0\}$,
 and the analogous inclusions hold with $(t,t_o)$ replaced by $(u,u_o)$.
 Consequently,  Propositions \ref{prop:one-point-estimate} implies the bound  
 \begin{equation} 
\mathrm{Pr}_{\eqref{eq:g-locmax-w}} (t,u;\omega)  
 \lesssim_{\eta, M_0, \omega} \frac{(1+u_o)}{\mathfrak{n}}\,. \label{eq:g-locmax-w-bnd}
 \end{equation}

Let   $\mathfrak{a}$ denote the potential kernel associated to the simple random walk on $\mathbb{Z}^2$ (see e.g. \cite[Section 4.4]{LawL2010} for precise definition).
For $\ell \ge 1$, define 
\begin{equation*}
	\tilde{h}_{\ell}(x) \coloneq  \sum_{j=\mathfrak{n}-\ell+1}^{\mathfrak{n}+1} b_{j}(x) \varphi_{j}(o) + \sum_{j=\mathfrak{n}-\ell+1}^{\mathfrak{n}} \chi_{j}(x) + \sum_{j=\mathfrak{n}-\ell+1}^{\mathfrak{n}} h_{j}(x) , \, x \in \mathsf{D}_{\mathfrak{n}-\ell} .
\end{equation*}
In particular $\tilde{h}_{\ell}\overset{\mathrm{d}}{=}  h^{\mathsf{B}_{2^{\ell}}(o)} - h^{\mathsf{B}_{2^{\ell}}(o)}_{o} $. 
 We define the  quantities
\begin{align*}
&  \Xi^{\mathrm{in}}_{\mathfrak{n}-\ell}(\omega) 
\coloneq \E\Bigl[|S_{\mathfrak{n}+1}-  S_{\mathfrak{n}-\ell}|
 \mathbf{1}_{\{ \tilde{h}_{\ell}(x) \le \tfrac{2}{\sqrt{g}} \mathfrak{a}(x-o) + \omega_{x-o} \text{ in } \mathsf{B}_{2^{r}}(o) \}} 
    \mathbf{1}_{\{ \tilde{h}_{\ell}(x) \le \tfrac{2}{\sqrt{g}} \mathfrak{a}(x-o)  \text{ in } \mathsf{B}_{2^{\ell}}(o) \}}
    \mathbf{1}_{\{ S_{\mathfrak{n}+1}- S_{\mathfrak{n}-\ell}\in[\ell^{1/3},\, \ell^2]\}}     \Bigr]  ; \\
	& \Xi^{\mathrm{out}}_{N,\ell}(t,u) \coloneq \E\Bigl[
|S_{\ell}|\ind{-S_{\ell}\in[\ell^{1/3},\ell^2]}
\ind{h_x\le[1-f_{\mathfrak {n}+1}(x)](m_N+t_o)+\psi_o-\psi_x+u-t,
	\ \forall x\in \mathsf{V}_{N}\setminus\mathsf{D}_\ell}
\,\Big|\,S_{\mathfrak {n}+1}=0\Bigr]. 
\end{align*} 
The following lemma provides the key asymptotic expansion for  $\mathrm{Pr}_{\eqref{eq:g-locmax-w}} (t,u;\omega) $.

\begin{lem}\label{lem:extrmal-asymp}
Under the hypotheses of  Proposition~\ref{prop:cluster-law}, for every $\epsilon>0$, there exists 
$k_0\ge 1$ such that for all  $k_0 \le \ell \le n^{1/10}$,  $z \in \mathsf{V}_{\eta N}$, $t \in [-2(\log\log N)^{2}, u]$, and $ k_0\vee(u-t)^{50} \le \varrho \le \sqrt{\log N}$,  we have: 
\begin{equation}\label{eq-neigh-g}
 \Big| \, \mathrm{Pr}_{\eqref{eq:g-locmax-w}} (t,u;\omega)  - 
 \frac{2  \Xi^{\mathrm{in}}_{\mathfrak n-\ell}(\omega)\,\Xi^{\mathrm{out}}_{N,\ell}(t,u)}{(g \log 2)\mathfrak{n}} \, \Big| \le \frac{\epsilon (1+u_o)}{n} . 
\end{equation} 
Moreover,  for every fixed $0<\bar{s}<\infty$ and $\ell$, uniformly over the
admissible $t$ and $s \in[0,\bar{s}]$,
\begin{equation}
\big|\Xi^{\mathrm{out}}_{N,\ell}(t+s,u)
-\Xi^{\mathrm{out}}_{N,\ell}(t,u)\big|
\lesssim_{\ell, \bar{s}} \frac{s}{\mathfrak n} . 
\label{eq:Xi-out-stability-additive}
\end{equation}  
\end{lem}

\begin{proof}[Proof of Proposition~\ref{prop:cluster-law} admitting
Lemma~\ref{lem:extrmal-asymp}]
For fixed $K\ge4$, \eqref{e:2.11} gives
$u_o\ge u+(\log\log K)/(4\alpha)+o_N(1)>0$ for all large $N$.
Applying Lemma~\ref{lem:extrmal-asymp} at $(t+s,\omega)$ and
$(t,\mathbf0)$, and then using
\eqref{eq:Xi-out-stability-additive}, gives, for all sufficiently large
fixed $\ell$,
\begin{align*}
&\Big|
 \mathrm{Pr}_{\eqref{eq:g-locmax-w}}(t+s,u;\omega)
 -\frac{\Xi^{\mathrm{in}}_{\mathfrak n-\ell}(\omega)}
       {\Xi^{\mathrm{in}}_{\mathfrak n-\ell}(\mathbf0)}
  \mathrm{Pr}_{\eqref{eq:g-locmax-w}}(t,u;\mathbf0)
 \Big| \\
&\qquad\lesssim_\omega
 \epsilon\frac{1+u_o}{n}
 +\frac{2\Xi^{\mathrm{in}}_{\mathfrak n-\ell}(\omega)}
       {(g\log2)\mathfrak n}
  \left|\Xi^{\mathrm{out}}_{N,\ell}(t+s,u)
        -\Xi^{\mathrm{out}}_{N,\ell}(t,u)\right|  \lesssim_{\ell,\bar s }
 \epsilon\frac{1+u_o}{n}+\frac1{\mathfrak n^2},
\end{align*}
uniformly in the parameters of the proposition. Note that the ratio 
$\Xi^{\mathrm{in}}_{\mathfrak n-\ell}(\omega)/
\Xi^{\mathrm{in}}_{\mathfrak n-\ell}(\mathbf0)$ depends only on $\ell$ and  by
\cite[Proposition 5.8]{BL3}, it tends to
$\nu(\omega)$ as $\ell\to\infty$.  Combining this with
\eqref{eq:g-locmax-w-bnd}, and choosing first $\epsilon$ small and then
$\ell$ large, yields 
\begin{equation}\label{eq:Pr-asymp}
	 \mathrm{Pr}_{\eqref{eq:g-locmax-w}}(t+s,u;\omega)
	 =\nu(\omega)\mathrm{Pr}_{\eqref{eq:g-locmax-w}}(t,u;\mathbf0)
	 +o_N(1)\frac{1+(u_o)_+}{\log N}.
\end{equation} 
Here and below, the error $o_N(1)$ is uniform in the parameters of the
proposition.

Since $\Var(g_z)=g\log N+O_\eta(1)$ uniformly in
$z\in\mathsf{V}_{\eta N}$, a direct Gaussian-density calculation gives
\[
 \frac{\P(g_z\in m_{KN}+t+s+\dif t)}
      {\P(g_z\in m_{KN}+t+\dif t)}
 =\rme^{-\alpha s}[1+o_N(1)].
\]
This holds uniformly for $s\in[0,\bar s]$ and the admissible $t$.
When $\omega=\mathbf0$, the local condition in
$\mathrm{Pr}_{\eqref{eq:g-locmax-w}}(t,u;\mathbf0)$ is
$g_{z+y}\le g_z$ for all $y\in\mathsf{B}_r$.  This is implied by
$g^*_{z+\mathsf{B}_{2^\varrho}}\le g_z$, because
$r\le2^\varrho$.  
Consequently,  
multiplying the preceding conditional asymptotic identity in \eqref{eq:Pr-asymp} by
$\P(g_z\in m_{KN}+s+\dif t)$ and using the density ratio above
therefore proves \eqref{e:28c} and \eqref{e:28c-error}.
\end{proof}

As before we abbreviate   $f_{\mathfrak{n}+1}$ defined in \eqref{e:129} by $f$, set $\hat{f} \coloneq 1- f$, and denote  $\mathscr{M}_{\mathfrak{n}+1}$  by $\mathscr{M}$. 

\begin{proof}[Proof of Lemma \ref{lem:extrmal-asymp}] 
The properties of Gaussian vectors imply that the conditional field $(g_{x}) \mid g_o=m_{KN}+t$ has the same distribution as $(h_{x}+ f_{x}(m_{N}+t_{o})+\psi_{x}) \mid h_o=0$. Note that  $\mathsf{D}_{\mathfrak n-\varrho}=o+\mathsf{B}_{2^\varrho}$. For each $0 \le j\le \mathfrak{n} $,  define 
 \begin{equation*}
 \mathscr{X}_{j} \coloneq \ind { h_{x} \le \hat{f}_{x} (m_{N}+t_{o})+\psi_{o}-\psi_{x}+ (u-t)\ind{x\in \mathsf{D}_{0}\setminus\mathsf{D}_{\mathfrak n-\varrho}} \,,\ 
	\forall x \in  \mathsf{D}_{j} \backslash \mathsf{D}_{j+1}    } \ \text{ and } \  \mathscr{X}_{j_{1},j_{2}} = \prod_{j=j_{1}}^{j_{2}-1}  \mathscr{X}_{j}  .
 \end{equation*}
 Moreover, define for each $k \ge r$, 
$\mathscr{X}^{\omega}_{\mathfrak{n}-k,\mathfrak{n}+1}  \coloneq \mathscr{X}_{\mathfrak{n}-k,\mathfrak{n}+1}   \ind { h_{x} \le \hat{f}_{x} (m_{N}+t_{o})+\psi_{o}-\psi_{x} + \omega_{x-o}  , \forall \, x   \in  \mathsf{B}_r(o) }$. Then for $k \ge r$, we can express the probability of interest as:
  \begin{equation*}
	\mathrm{Pr}_{\eqref{eq:g-locmax-w}} (t,u;\omega) = \E[ \mathscr{X}_{0,\mathfrak{n}-k} \mathscr{X}^{\omega}_{\mathfrak{n}-k,\mathfrak{n}+1}  \mid S_{\mathfrak{n}+1}=0 ]  .
  \end{equation*}
Thanks to  Lemma \ref{lem-put-M-in}, there exists a large integer $m$, depending only on $M_0$ and $\epsilon$,  satisfying
 \begin{equation}
  \big| \E [  \mathscr{X}_{0,\mathfrak{n}+1} \ind{\mathscr{M}  \ge m} \mid  S_{\mathfrak{n}+1} =0 ]  \big| 
 \lesssim_{M_0} \frac{ (1+u_o)}{\mathfrak{n}} e^{-c \log^{2} m } \le \frac{\epsilon (1+u_o)}{n} \label{eq-F-exclude-large-M}  .
 \end{equation} 
 We fix this $m$ and we will always work on the event  $\{\mathscr{M} \le m\}$ in the remaining proof. 
  
\smallskip
\noindent
\underline{\textit{Step 1.}}  Fix $k,\ell \in \mathbb{N}$ such that $\ell \ge 2k \ge 4r$. We define the indicator functions relative to $\tilde{h}_{\ell}$:
 \begin{equation*}
	\tilde{\mathscr{X}}_{j} \coloneq \ind { \tilde{h}_{\ell}(x) \le  \frac{2}{\sqrt{g}} \mathfrak{a}(x-o)  \forall x \in \mathsf{D}_{j} \backslash \mathsf{D}_{j+1}  }\ \text{ and } \  \tilde{\mathscr{X}}_{j_{1},j_{2}} = \prod_{j=j_{1}}^{j_{2}-1}  \tilde{\mathscr{X}}_{j}  .
 \end{equation*} 
Set $\tilde{\mathscr{X}}^{\omega}_{ \mathfrak{n}-k,\mathfrak{n}+1}  \coloneq   \tilde{\mathscr{X}}_{ \mathfrak{n}-k , \mathfrak{n}+1} \ind{ \tilde{h}_{\ell}(x) \le \frac{2}{\sqrt{g}} \mathfrak{a}(x-o) + \omega_{x-o} \,, \forall x \in \mathsf{D}_{\mathfrak{n}-r}}$.
We claim that for each fixed $k$, provided $\ell$ is sufficiently large, then the following bound holds for all large  $n$
\begin{equation}\label{eq-replace-in}
 E_{\eqref{eq-replace-in}} \coloneq \E  \bigl[    \mathscr{X}_{0,\mathfrak{n}-k}  \big|\mathscr{X}^{\omega}_{\mathfrak{n}-k,\mathfrak{n}+1} - \tilde{\mathscr{X}}^{\omega}_{\mathfrak{n}-k,\mathfrak{n}+1}   \big|   \ind{\mathscr{M}  \le m}  \mid  S_{\mathfrak{n}+1} =0 \bigr] \le \frac{\epsilon(1+u_o)}{  n} .
\end{equation}
To verify this, note that   as $n \to \infty$, we have uniformly for $x \in \mathsf{D}_{\mathfrak{n}-k}$, 
\[  \hat{f}_{x} (m_{N}+t_o) + \psi_{o}-\psi_x = [1+o(1)] \frac{2}{\sqrt{g}}    [G_{\mathsf{D}_0}(o,o)-G _{\mathsf{D}_0}(x,o)  ] = [1+o(1)] \frac{2}{\sqrt{g}}\mathfrak{a}(x-o). \] 
Additionally, on  $\{S_{\mathfrak{n}+1}=0\} \cap \{ \mathscr{M} \le m \}$,  Lemma \ref{lem:39} and the definition \eqref{eq-control-variable} imply 
\begin{equation}\label{eq:h-vs-htilde}
	| h_{x} - \tilde{h}_{\ell} (x)|  \le \sum_{i=1}^{\mathfrak{n}-\ell} |b_{i}(x)| |\varphi_{i}(o)| + \sum_{i=1}^{\mathfrak{n}-\ell} |\chi_{i}(x) | \lesssim_{m} (\log \ell)  2^{-(\ell-j)} , \  \text{ for } x \in \mathsf{D}_{\mathfrak{n}-j}  \text{ with } j \le \ell .
\end{equation}
Consequently,  given any $\delta>0$ and $k$,  provide $\ell$, $N$ sufficiently large, the difference $\mathscr{X}^{\omega}_{\mathfrak{n}-k,\mathfrak{n}+1} - \tilde{\mathscr{X}}^{\omega}_{\mathfrak{n}-k,\mathfrak{n}+1}   $ is nonzero   only if one of the following conditions is satisfied: 
(a) ${\mathscr{X}}^{\omega}_{ \mathfrak{n}-k , \mathfrak{n}+1} =1$ and there exists $x \in \mathsf{D}_{\mathfrak{n}-k}$ with $h_{x} \ge \hat{f}_{x} (m_{N}+t_{o})+\psi_{o}-\psi_{x}+\omega_{x-o} - \delta$. (b) $\tilde{\mathscr{X}}^{\omega}_{ \mathfrak{n}-k , \mathfrak{n}+1} =1$ and there is $x \in \mathsf{D}_{\mathfrak{n}-k}$ with $\tilde{h}_{\ell}(x) \ge \frac{2}{\sqrt{g}} \mathfrak{a}(x-o) +\omega_{x-o}- \delta$. This leads to the bound:
 \begin{align}
 E_{\eqref{eq-replace-in}} 
 & \leq  \E  \bigl[    \mathscr{X}_{0,\mathfrak{n}-k}     \tilde{\mathscr{X}}^{\omega}_{\mathfrak{n}-k,\mathfrak{n}+1} \ind{ \exists\, x \in \mathsf{D}_{\mathfrak{n}-k}, \tilde{h}_{\ell}(x) \ge \frac{2}{\sqrt{g}} \mathfrak{a}(x-o) +\omega_{x-o}-\delta   }    \mid  h_o=0 \bigr] \notag \\
  & \ +    \E  \bigl[    \mathscr{X}_{0,\mathfrak{n}-k}     \mathscr{X}^{\omega}_{\mathfrak{n}-k,\mathfrak{n}+1}\ind{ \exists\, x \in \mathsf{D}_{\mathfrak{n}-k}, h_{x} > \hat{f}_{x} (m_{N}+t_{o})+\psi_{o}-\psi_{x}+\omega_{x-o} -\delta }   \mid  h_o=0 \bigr]\,. \label{eq:FKGerror}
 \end{align}  

Since the DGFF satisfies the strong FKG property, by applying \cite[Lemma 4.22]{BL3}, we obtain:
\begin{align*}
	 & \P  \bigl(   \exists\, x \in \mathsf{D}_{\mathfrak{n}-k},  \tilde{h}_{\ell}(x) \ge \tfrac{2}{\sqrt{g}} \mathfrak{a}(x-o) + \omega_{x-o} -\delta    \mid  h_o=0 , \mathscr{X}_{0,\mathfrak{n}-k} \tilde{\mathscr{X}}^{\omega}_{\mathfrak{n}-k,\mathfrak{n}+1}     =1 \bigr)  \\
	 & \le \sum_{x \in \mathsf{D}_{\mathfrak{n}-k} \setminus \{o\}} \P\bigl( \tilde{h}_{\ell}(x) \ge \tfrac{2}{\sqrt{g}} \mathfrak{a}(x-o) + \omega_{x-o} -\delta  \mid \tilde{h}_{\ell}(x) \le \tfrac{2}{\sqrt{g}} \mathfrak{a}(x-o) +\omega_{x-o}\bigr) \lesssim 4^{k} \delta .
\end{align*}
The last inequality   holds because  $\inf_{k,\ell \ge 2k}\inf_{x \in \mathsf{D}_{\mathfrak{n}-k} \setminus \{o\} } \Var(\tilde{h}_{\ell}(x)  )>0$. 
Thus, provided $\delta$ is chosen sufficiently small (depending on $k$ and $\epsilon$), we have:
\begin{align*}
	& \E  \bigl[    \mathscr{X}_{0,\mathfrak{n}-k}     \tilde{\mathscr{X}}^{\omega}_{\mathfrak{n}-k,\mathfrak{n}+1} \ind{ \exists\, x \in \mathsf{D}_{\mathfrak{n}-k}, \tilde{h}_{\ell}(x) \ge \frac{2}{\sqrt{g}} \mathfrak{a}(x-o) +\omega_{x-o} -\delta   }    \mid  h_o=0 \bigr]  \lesssim 4^{k} \delta 	\, \E  \bigl[    \mathscr{X}_{0,\mathfrak{n}-k}     \tilde{\mathscr{X}}_{\mathfrak{n}-k,\mathfrak{n}+1}  \mid  h_o=0 \bigr] \\
&\qquad \lesssim 4^{k} \delta 	 \E  \bigl[    \mathscr{X}_{0,\mathfrak{n}-k}       \ind { h_{x} \le  \hat{f}_{x} (m_{N}+t_{o})+\psi_{o}-\psi_{x} + 1, \forall \, x   \in \mathsf{D}_{\mathfrak{n}-k}  }   \mid  h_o=0 \bigr] \lesssim \frac{\epsilon (1+u_o)}{ 2n} .
\end{align*}  
Here the last inequality follows from the conditional  
estimate \eqref{eq-o-is-local-extreme}. 
Applying a similar argument to the second expectation in  \eqref{eq:FKGerror} yields that it is also bounded by $ \frac{\epsilon (1+u_o)}{ 2 n} $ thereby establishing \eqref{eq-replace-in}.

\smallskip 
\noindent
\underline{\textit{Step 2.}}
Recall that  $k,\ell \in \mathbb{N}$ are chosen such that $\ell \ge 2k \ge 4r$.
We define:
\begin{equation*}
\tilde{\mathscr{S}}_{\ell,\mathfrak{n}-k}
\coloneq
\ind{-S_{\ell}\in[\ell^{1/3},\ell^2]}
\ind{S_j\le0,\ \forall\,\ell\le j\le\mathfrak n-\ell}
\ind{-S_{\mathfrak n-\ell}\in[\ell^{1/3},\ell^2]}
\tilde{\mathscr X}_{\mathfrak n-\ell,\mathfrak n-k}.
\end{equation*} 
   We claim that provided that $k$ (depending on $ m,\epsilon$) subsequently $\ell$ large enough,   
\begin{equation}\label{eq-replace-mid}
	E_{\eqref{eq-replace-mid}} \coloneq \E  \bigl[      {\mathscr{X}}_{ 0,\ell } \big| {\mathscr{X}}_{ \ell,\mathfrak{n}-k } - \tilde{\mathscr{S}}_{ \ell,\mathfrak{n}-k }  \big|\tilde{\mathscr{X}}_{ \mathfrak{n}-k , \mathfrak{n}+1}    \ind{\mathscr{M}  \le m}  \mid  S_{\mathfrak{n}+1} =0 \bigr] \le \frac{\epsilon (1+u_o)}{  n} .
\end{equation}   
Indeed, this estimate follows from the entropic repulsion of a random walk conditioned to stay negative.
 Proceeding as in \eqref{eq-cons-S-k},  Lemmas \ref{lemma-3.8} and \ref{lem:40} imply the existence of a constant $C>0$ (depending on $M_0$) satisfying the following inclusions:
  \begin{align*} 
  & \{ S_{j} \le - C \vartheta_{m}(\wedge^{\mathfrak{n}+1} (j) )^{2}+  u_{o} \ind{j < \mathfrak{n}-\varrho} +   t_{o}  \ind{j\ge \mathfrak{n}-\varrho} - t_{o} \tfrac{j}{\mathfrak{n}}  \} \\
	& \subset   \{  \mathscr{X}_{j}=1 \} \subset \{ S_{j} \le C \vartheta_{m}(\wedge^{\mathfrak{n}+1} (j) )  +  u_{o} \ind{j < \mathfrak{n}-\varrho} +   t_{o}  \ind{j\ge \mathfrak{n}-\varrho} - t_{o} \tfrac{j}{\mathfrak{n}}  \}   \quad \text{ for all } 1 \le j \le \mathfrak{n}. 
 \end{align*} 
 Furthermore, combining Lemma~\ref{lem:40} with
$\mathfrak a(x-o)=g\log|x-o|+O(1)$ shows that, after increasing
$C$ if necessary, the condition
$S_j\le-C\vartheta_m(\wedge^{\mathfrak {n}+1}(j))^2$ is sufficient
to ensure that
$h_x\le \frac{2}{\sqrt g}\mathfrak a(x-o)$ for every
$x\in\mathsf{D}_j\setminus\mathsf{D}_{j+1}$, and the latter
event implies
$S_j\le C\vartheta_m(\wedge^{\mathfrak {n}+1}(j))$.
  Using \eqref{eq:h-vs-htilde} and the fact that
$
\vartheta_m (\wedge^{\mathfrak {n}+1}(j) )^2
2^{-(\mathfrak n-j)}
\ge
(\log\ell)2^{-\ell} 
$ for $\mathfrak n-\ell\le j\le\mathfrak n$,  
, we deduce that when $S_{\mathfrak{n}+1}=0$, 
 \begin{align*}
	 \{    S_{ j} \le   - C   \vartheta_{m}(\wedge^{\mathfrak{n}+1} (j) )^{2} \}  
	  \subset \{ \tilde{\mathscr{X}}_{j}=1  \} \subset \{   S_{ j} \le     C   \vartheta_{m}(\wedge^{\mathfrak{n}+1} (j) ) ^{2}  \}  \ \text{ for all } \   j \ge \mathfrak{n}- \ell.
 \end{align*}   
  
By taking a sufficiently large constant $C'$ (depending on $m$ and $M_0$), we have   
$ C   \vartheta_{m}(\wedge^{\mathfrak{n}+1} (j) )^{2} \le C'   [\wedge^{\mathfrak{n}+1} (j)]^{1/4}   +   t_{o} \tfrac{\mathfrak{n}-j}{\mathfrak{n}}   $ for $j \in [\mathfrak{n}-\varrho,\mathfrak{n}]$.  Thus whenever  ${\mathscr{X}}_{ 0,\ell }  {\mathscr{X}}_{ \ell,\mathfrak{n}-k } \tilde{\mathscr{X}}_{ \mathfrak{n}-k,\mathfrak{n}+1} =1$ or ${\mathscr{X}}_{ 0,\ell }   \tilde{\mathscr{S}}_{ \ell,\mathfrak{n}-k } \tilde{\mathscr{X}}_{ \mathfrak{n}-k , \mathfrak{n}+1} =1$, 
the following "weak" event must occur:
 \[ G^{\mathrm{w}} \coloneq \{ S_{j} \le C'   [\wedge^{\mathfrak{n}+1} (j)]^{1/4}   +  u_{o}\ind{j < \mathfrak{n}-\varrho} +   t_{o} \ind{j\ge \mathfrak{n}-\varrho} - t_{o} \tfrac{j}{\mathfrak{n}} , \, \forall 1 \le j \le \mathfrak{n} \}.  \]  
Conversely,   assume that $ {\mathscr{X}}_{ \ell,\mathfrak{n}-k } = 0$ or $ \tilde{\mathscr{S}}_{ \ell,\mathfrak{n}-k }  = 0$. Since when $\mathscr{M} \le m$, from Lemma \ref{lem:40} it follows $-S_{\ell}    \le \ell \log \ell \le \ell^2$ for large $\ell$. Provided $k,\ell$ is large we then have $-C'   [\wedge^{\mathfrak{n}+1} (j)]^{1/4}/2   +  u_{o} \le 0$ for $j \in [\ell,\mathfrak{n}-\varrho)$ and $ -C   \vartheta_{m}(\wedge^{\mathfrak{n}+1} (j) )^{2} \ge -C'   [\wedge^{\mathfrak{n}+1} (j)]^{1/4}   +   t_{o} \tfrac{\mathfrak{n}-j}{\mathfrak{n}}   $ for $j \in [\mathfrak{n}-\varrho,\mathfrak{n}-k]$. 
Thus the random walk must fail to  make the  following "strong" event happen:
 \[
 G^{\mathrm{s}}_{\ell,\mathfrak{n}-k} \coloneq \{ S_{j} \le -C'   [\wedge^{\mathfrak{n}+1} (j)]^{1/4}   +  u_{o}\ind{j < \mathfrak{n}-\varrho} +   t_{o} \ind{j\ge\mathfrak{n}-\varrho} - t_{o} \tfrac{j}{\mathfrak{n}} ,  \forall \ell \le j \le \mathfrak{n}-k \}.  
  \]
Recall that $(S_j)_{j=0}^{\mathfrak{n}+1}$ is a Gaussian vector and $ \frac{\E[S_{j}^2]}{\E[S_{\mathfrak{n}+1}^2]}= \frac{j+O(1)}{\mathfrak{n}}$.  Applying  Lemma \ref{lem-ent-rw} with $a=1/4$ and $b=C'$ we obtain
\begin{align*} 
E_{\eqref{eq-replace-mid}}  & \le  \P \bigl(   G^{\mathrm{w}} \setminus G^{\mathrm{s}}_{\ell,\mathfrak{n}-k}  \mid  S_{\mathfrak{n}+1} =0\bigr)  \\
& \lesssim \P_{0,\mathfrak{n}+1}^{0,t_o} \bigl( \{  S \preceq_{[0,\mathfrak{n}+1]} u_o + C' [\wedge^{\mathfrak{n}+1} (\cdot)]^{1/4} \} \cap \{ S \preceq_{[k,\mathfrak{n}-k]} t_o - C' [\wedge^{\mathfrak{n}+1} (\cdot)]^{1/4} \}^{c} ) \\
& \lesssim \frac{(1+u_{o})}{\mathfrak{n}} \frac{1}{k^{1/4}}   \Bigl[ 1+ \frac{u-t}{ \varrho^{1/2}} \Bigr]  \le \frac{\epsilon (1+u_o)}{\mathfrak{n}} ,
\end{align*}    
provided that $k$ is sufficiently large.

\smallskip
 \noindent
\underline{\textit{Step 3.}} As in \eqref{eq-F-exclude-large-M}, adapting the argument in  Lemma \ref{lem-put-M-in} yields $ \E   [   \mathscr{X}_{ 0,\ell }  \tilde{\mathscr{S}}_{ \ell,\mathfrak{n}-k }   \tilde{\mathscr{X}} _{ \mathfrak{n}-k , \mathfrak{n}+1}  \ind{\mathscr{M} > m}  \mid  S_{\mathfrak{n}+1} =0  ]  \lesssim_{M_0} \frac{ (1+u_o)}{\mathfrak{n}} e^{-c \log^{2} m } \le \frac{\epsilon (1+u_o)}{\mathfrak{n}} $.   Combining this with the estimates \eqref{eq-replace-in} and \eqref{eq-replace-mid}, we obtain:  
\begin{equation}\label{eq:after-replace}
	\big| \mathrm{Pr}_{\eqref{eq:g-locmax-w}} (t,u;\omega) -   \E  \bigl[   \mathscr{X}_{ 0,\ell }    \ind{- S_{\ell}, - S_{\mathfrak{n}-\ell} \in [\ell^{1/3},\ell^2 ] }  \,  \ind{ S_{j} \le 0  ,\,\forall\, \ell \le j \le \mathfrak{n}-\ell}  \, \tilde{\mathscr{X}}^{\omega}_{ \mathfrak{n}-\ell , \mathfrak{n}+1}   \mid  S_{\mathfrak{n}+1} =0 \bigr] \big| \le \frac{3\epsilon (1+u_o)}{\mathfrak{n}} . 
\end{equation}
By the concentric decomposition (Proposition \ref{prop:38}), conditioned on $S_{\ell}$ and $S_{\mathfrak{n}-\ell}$ and the event $\{S_{\mathfrak{n}+1}=0\}$,  the “outer” field $\{h_{x}:x \in \mathsf{D}_0 \setminus \mathsf{D}_{\ell} \}$, the ``middle'' random walk bridge $\{S_{\ell}, S_{\ell+1}, \cdots, S_{\mathfrak{n}-\ell}\}$, and the ``inner" field $\{\tilde{h}_{\ell}(x): x\in \mathsf{D}_{\mathfrak{n}-\ell}\}$ are independent. The Ballot estimate with error terms (see \cite[Lemma 5.6.]{BL3})   yields that
there is a constant $c>0$ such that for  all $- S_{\ell}, - S_{\mathfrak{n}-\ell} \in [\ell^{1/3},\ell^2 ]$, 
\begin{equation*}
	\Big| \P \Bigl( \bigcap_{j=\ell+1}^{\mathfrak{n}-\ell} \{S_j \leq 0\} \mid  S_{\ell}, S_{\mathfrak{n}-\ell}  \Bigr) - \frac{2}{g \log 2} \frac{|S_{\ell}|| S_{\mathfrak{n}-\ell}|}{\mathfrak{n}} \Big| \leq c \frac{\ell^4}{\mathfrak{n}} \frac{S_{\ell} S_{\mathfrak{n}-\ell}}{\mathfrak{n}}  . 
\end{equation*}  
Plugging this in the  expectation in \eqref{eq:after-replace}  and using independence  we find:
\begin{equation}
	\Big| \mathrm{Pr}_{\eqref{eq:g-locmax-w}} (t,u;\omega) -  \frac{2   }{(g \log 2) \mathfrak{n}}  \, \E \bigl[   \mathscr{X}_{ 0,\ell }    \ind{- S_{\ell}, - S_{\mathfrak{n}-\ell} \in [\ell^{1/3},\ell^2 ] } |S_{\ell}| \, | S_{\mathfrak{n}-\ell}|    \tilde{\mathscr{X}}^{\omega}_{ \mathfrak{n}-\ell , \mathfrak{n}+1}   \mid  S_{\mathfrak{n}+1} =0 \bigr] \Big| \le \frac{3\epsilon (1+u_o)}{\mathfrak{n}}\,. \label{eq:to-be-decouple}
\end{equation}

Next we decouple the scales $j \le \ell$ from the scales $j \ge \mathfrak{n}-\ell$.
Conditionally on $S_{\ell} = -v $ with $v \in [\ell^{1/3},\ell^2]$ in the expectation above, we have  
\begin{align*}
	& \E \bigl[   | S_{\mathfrak{n}-\ell}|    \ind{  - S_{\mathfrak{n}-\ell} \in [\ell^{1/3},\ell^2 ] } \, \tilde{\mathscr{X}}^{\omega}_{ \mathfrak{n}-\ell , \mathfrak{n}+1}   \mid S_{\ell} = -v, S_{\mathfrak{n}+1} =0 \bigr]  \\
	& = 	\E \bigl[   | S_{\mathfrak{n}+1} -S_{\mathfrak{n}-\ell}|     \ind{  S_{\mathfrak{n}+1}- S_{\mathfrak{n}-\ell} \in [\ell^{1/3},\ell^2 ] }  \, \tilde{\mathscr{X}}^{\omega}_{ \mathfrak{n}-\ell , \mathfrak{n}+1}   \mid S_{\ell}=-v, S_{\mathfrak{n}+1} =0 \bigr] \\
	&=  	\E \bigl[   | S_{\mathfrak{n}+1}^{v} -S^{v}_{\mathfrak{n}-\ell}|   \ind{   S_{\mathfrak{n}+1}^{v}-S_{\mathfrak{n}-\ell}^{v} \in [\ell^{1/3},\ell^2 ] } \, \tilde{\mathscr{X}}^{v,\omega}_{ \mathfrak{n}-\ell , \mathfrak{n}+1}     \bigr] ,
\end{align*} 
where in the last equality, 
note that the law of $\{\varphi_{j}(o)\}_{j=\ell+1}^{  \mathfrak{n}+1}$   conditionally on  $S_{\ell}=- v, S_{\mathfrak{n}+1}=0$ is the same  as the unconditional distribution of  $\{\varphi^{v}_{j}(o) \}$ with ${\varphi}^{v}_{j}(o) \coloneq  \varphi_{j}(o) -  \frac{ \E[ \varphi_{j}(o)^2 ]}{\Var(S_{\mathfrak{n}+1}-S_{\ell}) } (S_{\mathfrak{n}+1}-S_{\ell})  +   \frac{ \E[ \varphi_{j}(o)^2 ]}{\Var(S_{\mathfrak{n}+1}-S_{\ell}) } v $; and $S^{v}$, $\tilde{h}_{\ell}^{v}$  $\tilde{\mathscr{X}}^{v,\omega}_{ \mathfrak{n}-\ell , \mathfrak{n}+1}$ are the corresponding process defined via $\{{\varphi}^{v}_{j}\}_{j=\ell+1}^{\mathfrak{n}+1}, \{ \chi_{j}\}$, $\{h_{j}\}$. 

Observe that  on $\{  |S_{\mathfrak{n}+1}-S_{\ell}| \le \mathfrak{n}^{2/3} \}$,   $\max_{j \ge \mathfrak{n}-\ell} |\varphi^{v}_{j}(o)- \phi_{j}(o)|$ is  bounded by $ O(1)\frac{1}{\mathfrak{n}^{1/3}}$. Hence  both  $\max_{x \in \mathsf{D}_{\mathfrak{n}-\ell}}|\tilde{h}_{\ell}^{v}(x)-    \tilde{h}_{\ell}(x)|$,  and     $||S^v_{\mathrm{n}+1}-S^v_{\mathrm{n}-\ell}|-|S_{\mathrm{n}+1}-S_{\mathfrak{n}-\ell}||$ are bounded by $ O_{\ell}(1)\frac{1}{\mathfrak{n}^{1/3}}$. Then repeating the argument in  Step 1, (or see also the proof of Proposition 5.2. in \cite{BL3}) we obtain, for any $\epsilon$, $v \in [\ell^{1/3},\ell^2]$, provided $N$ is sufficiently large
\begin{equation*}
	\Big |\E \bigl[   | S_{\mathfrak{n}+1}^{v} -S^{v}_{\mathfrak{n}-\ell}|   \ind{   S_{\mathfrak{n}+1}^{v}-S_{\mathfrak{n}-\ell}^{v} \in [\ell^{1/3},\ell^2 ] } \, \tilde{\mathscr{X}}^{v,\omega}_{ \mathfrak{n}-\ell , \mathfrak{n}+1}     \bigr]  - \E \bigl[   | S_{\mathfrak{n}+1} -S_{\mathfrak{n}-\ell}|   \ind{   S_{\mathfrak{n}+1}-S_{\mathfrak{n}-\ell} \in [\ell^{1/3},\ell^2 ] } \, \tilde{\mathscr{X}}^{\omega}_{ \mathfrak{n}-\ell , \mathfrak{n}+1}     \bigr] \Big| \le \epsilon .
\end{equation*}
The second expectation is exactly $\Xi^{\mathrm{in}}_{\mathfrak{n}-\ell}(\omega)$. Plugging this back into \eqref{eq:to-be-decouple}  and recalling that $ \Xi^{\mathrm{out}}_{N,\ell}(t,u) =\E [   \mathscr{X}_{ 0,\ell }    \ind{- S_{\ell} \in [\ell^{1/3},\ell^2]} |S_{\ell}|   \mid  S_{\mathfrak{n}+1} =0 ]$, the desired result \eqref{eq-neigh-g} follows.

\medskip
 \noindent
\underline{\textit{Step 4.}}
It remains to show  \eqref{eq:Xi-out-stability-additive}.  
Let $V_{\ell}:= \Var(S_{\ell})$.
Note that  the conditioned law of $\{
 (h_x)_{x\in\mathsf{V}_{\eta'N}\setminus\mathsf{D}_\ell},S_\ell \} $ given 
 $S_{\mathfrak {n}+1}=s $, is the same as   the conditioned law of 
$\{ (h_x+s f_{\mathfrak {n}+1}(x))_{x\in\mathsf{V}_{\eta'N}\setminus\mathsf{D}_\ell},
S_\ell+\frac{ s V_{\ell}}{V_{\mathfrak{n}+1}}\}$ given $S_{\mathfrak {n}+1}=0$. Hence  
the left-hand side in \eqref{eq:Xi-out-stability-additive} is bounded above by 
\begin{align}
	& \E\Bigl[
\Big|\,
\bigl|S_\ell- \frac{sV_{\ell}}{V_{\mathfrak{n}+1}} \bigr|
\ind{-S_\ell+ \frac{sV_{\ell}}{V_{\mathfrak{n}+1}}
\in[\ell^{1/3},\ell^2]}
-|S_\ell|\ind{-S_\ell\in[\ell^{1/3},\ell^2]}
\, \Big| \, 
\Big| \,  
S_{\mathfrak {n}+1}=s  \Bigr] \label{eq:outer-weight-shift} \\
&  \quad  + \Big| \E\Bigl[
|S_{\ell}|\ind{-S_{\ell}\in[\ell^{1/3},\ell^2]}
\ind{h_x\le[1-f_{\mathfrak {n}+1}(x)](m_N+t_o)+\psi_o-\psi_x+u-t,
	\ \forall x\in \mathsf{V}_{N}\setminus\mathsf{D}_\ell}
\,\Big|\,S_{\mathfrak {n}+1}=s\Bigr] -\Xi^{\mathrm{out}}_{N,\ell}(t,u) \Big|
\label{eq:outer-barrier-shift} .
\end{align} 

For $1\le i,j\le\ell$,   we have
\[  \mathrm{Cor} [\varphi_i(o)\varphi_j(o)\mid S_{\mathfrak {n}+1}=s ] =\delta_{ij}\Var(\varphi_i(o))
-\frac{\Var(\varphi_i(o))\Var(\varphi_j(o))}
{\Var(S_{\mathfrak {n}+1})}. \]  
For fixed $\ell$, the conditional variance of $S_\ell$ is bounded away from zero and infinity, and hence its conditional density is uniformly bounded. Therefore we obtain \eqref{eq:outer-weight-shift} is dominated above by $\frac{s V_{\ell}}{V_{\mathfrak{n}+1}} =\ell O ( \frac{s}{\mathfrak{n}})$.  

For fixed $\ell$, the inverse of the conditional covariance matrix  of $(\varphi_i(o))_{i=1}^{\ell}$ has uniformly bounded norm. The formula for the relative entropy of Gaussian vectors with the same covariance   gives
\[ D_{\mathrm{KL}}\Bigl(
\mathrm{Law}\bigl((\varphi_j(o))_{j\le\ell}\mid S_{\mathfrak {n}+1}=s\bigr)
\,\Big\|\,
\mathrm{Law}\bigl((\varphi_j(o))_{j\le\ell}\mid S_{\mathfrak {n}+1}=0\bigr)
\Bigr)
\lesssim_\ell\frac{s^2}{\mathfrak n^2}\,.
 \]
 Pinsker's inequality implies the total variation distance  is  of order $O({s}/{\mathfrak n})$. 
 All other variables in the first $\ell$ levels are independent of $S_{\mathfrak {n}+1}$, and the outer field is measurable with respect to these levels. Hence    the total-variation distance between the conditioned law of $ \{
 (h_x)_{x\in\mathsf{V}_{\eta'N}\setminus\mathsf{D}_\ell},S_\ell \}  $  given $S_{\mathfrak{n}+1} =s$ or $S_{\mathfrak{n}+1} =0$ is also of order $O(s/\mathfrak{n})$. This implies  \eqref{eq:outer-barrier-shift} is bounded by $\ell^2 O(s/\mathfrak{n})$, which proves \eqref{eq:Xi-out-stability-additive}.
\end{proof}


\appendix
\section{Random walk estimates}
\label{sec:RWEstimates_Proofs}

\subsection{Random walk on the real line}

In this section we provide some random walk estimates which are needed in Section~\ref{sub:3.1}.  We use the same notation as there. Furthermore, we  set   \[V_{k} \coloneq  \sum_{j=1}^{k} \E[\xi_{j}^2] \quad \text{for} \quad 1 \le k \le n  \]
and 
\[  U(x) \coloneq 1+ x_+ \quad  \text{for} \quad  x \in \mathbb{R}. \]
In the following, all the constants depend on $(\sigma^2_{\min},\sigma^{2}_{\max})$.

\begin{lem}[Standard ballot estimate]
	\label{lem-ballot-rw}
 Fix an $a\in (0,1/2)$, $b>0$.   
Then there exists constant $C_{\ref{lem-ballot-rw}}$ depending  on $a,b$ such that for all  $n \ge 1$ and for all real numbers $x_0, x, y_0 ,y $,  
\begin{equation*} 
\P_{0,n}^{x, y} 
	\Bigl(   Z \preceq_{[1,n-1]} \mathfrak{l}_{0,n}^{x_0,y_0} + \zeta_{0,n}^{a,b} \Bigr) 
\leq C_{\ref{lem-ballot-rw}} \min\Bigl\{  \frac{U(x_0-x)U(y_0-y)}{n} ,1 \Bigr\} \,.
\end{equation*}
Moreover,  for any $K>0$ there exists constant $c_{\ref{lem-ballot-rw}}$ depending only on $a,b $, $K$   such that for all  $x_0 \ge x-
K, y_0 \ge y-K$,  
\begin{equation*}
\P_{0,n}^{x, y} 
	\Bigl(    Z \preceq_{[1,n-1]} \mathfrak{l}_{0,n}^{x_0,y_0} - \zeta_{0,n}^{a,b}   \Bigr) 
\geq c_{\ref{lem-ballot-rw}} \min\Bigl\{\frac{U(x_0-x)U(y_0-y)}{n} ,1 \Bigr\}  \,.
\end{equation*}
\end{lem}

\begin{proof}[Proof of Lemma \ref{lem-ballot-rw}]  
 For the  Brownian motion case, the proof  can be found in \cite[Proposition 2.1.]{CHL19}, using Girsanov--Cameron--Martin change of measure technics.  
 To transfer the result of  Brownian motion to $(Z_k)$, one can then apply \cite[Lemma 4.15]{BL3}. 
\end{proof}

\begin{lem}\label{lem-two-barrier-ballot-upper-bound} Assume that $(n,r; x,x_0,x_{1};y,y_0,y_1) \in \mathbb{N}_{+}^2 \times \mathbb{R}^{6}$ satisfies  $
 n \ge 4 $, $ 2  \le r \le n-2 $, $ x< x_0 $  and $ y<y_0  $. Define $
 \mu_{k} \coloneq \frac{V_n-V_k}{V_n} x + \frac{V_k}{V_n} y $ for $1 \le k \le n$.
 Then for any $a \in (0,1/2),b>0$ there exist   constants $C_{\ref{lem-two-barrier-ballot-upper-bound}} >0 $ such that uniformly in $(n,r; x,x_0,x_{1};y,y_0,y_1) $, 
 \begin{align}
  & \P _{0,n}^{x,y}\Bigl( Z \preceq_{[0,r]} \mathfrak{l}_{0,r}^{x_0,x_1} + \zeta_{0,r}^{a,b}  ,  Z \preceq_{[r+1,n]} \mathfrak{l}_{r+1,n}^{y_1,y_0} + \zeta_{r+1,n}^{a,b}       \Bigr)  \notag \\
  &  \leq C_{\ref{lem-two-barrier-ballot-upper-bound}}   \frac{U(x_0-x)U(y_0-y)}{n} \Bigl[ 1+\frac{(x_1-\mu_r)_+}{\wedge^{n}(r)^{1/2}} \Bigr] \Bigl[ 1+\frac{(y_1-\mu_{r+1})_+}{\wedge^{n}(r)^{1/2}} \Bigr] \, 
 \exp\Bigl\{    -\frac{(x_{1}-\mu_r)_{-}^2 + (y_{1}-\mu_{r+1})_{-}^2 }{4\wedge^{n}(r)} \Bigr\}\,.  \label{eq-two-barrier-upp-1}
\end{align}  
 In particular, if $|x-y| \le K\sqrt{n}$, then one may replace both $\mu_r$ and $\mu_{r+1}$ by $y$ in \eqref{eq-two-barrier-upp-1}, after adjusting the constants by factors depending only on $K$. 
\end{lem}

\begin{proof} 
By time reversal, we may assume without loss of generality that \(n/2 \le r \le n-2\). Denote the probability in \eqref{eq-two-barrier-upp-1}   by  $\mathrm{Pr}_{\eqref{eq-two-barrier-upp-1}}$. 
Conditioning first on $Z_r$ and then on $Z_{r+1}$, and applying
Lemma~\ref{lem-ballot-rw} to the bridges on $[0,r]$ and $[r+1,n]$,
respectively, gives
\begin{align}
\mathrm{Pr}_{\eqref{eq-two-barrier-upp-1}}
&\lesssim
\frac{  {U(x_0-x)U(y_0-y)}  }{r (n-r-1)}
\E_{0,n}^{x,y}\bigl[  U(x_1-Z_{r})\ind{Z_r \le x_1} U(y_1-Z_{r+1})\ind{Z_{r+1} \le y_1}  \bigr].
\label{eq-two-barrier-upper-new-1}
\end{align}  
We use the elementary Gaussian estimate that, for a standard normal distributed r.v. $G$ and fixed constant $s_0>0$, there exists constant $ C$ depending only on $s_0$ such that 
\begin{equation}
\E\bigl[ U(x- s G)^2 \ind{s G \le x} \bigr]
\le C s^2 \Bigl(  1+\frac{x_+}{s}  \Bigr)^2
\exp \Bigl( -\frac{x_-^2}{2s^2}  \Bigr)  \quad \forall \, s \ge s_0, x \in \mathbb{R}. 
\label{eq-truncated-gaussian-second-moment}
\end{equation}
Indeed this claim follows directly by   splitting into
$x\ge0$ and $x<0$ and using Gaussian tail estimate.

Under $\P_{0,n}^{x,y}$, the variable  $Z_k$, $k \ge n/2$  has mean 
$\mu_k$   and variance 
$
s_k^2=\frac{V_k(V_n-V_k)}{V_n}\asymp n-k$. 
  By Cauchy--Schwarz and
\eqref{eq-truncated-gaussian-second-moment}, we obtain
\begin{align*}
&\E_{0,n}^{x,y}\bigl[  U(x_1-Z_{r})\ind{Z_r \le x_1} U(y_1-Z_{r+1})\ind{Z_{r+1} \le y_1}  \bigr] \\ 
&\lesssim (n-r)
\Bigl(1+\frac{(x_1-\mu_r)_+}{\sqrt {n-r}}\Bigr)
\Bigl(1+\frac{(y_1-\mu_{r+1})_+}{\sqrt {n-r}}\Bigr)  
\ 
\exp\Bigl(-\frac{(x_1-\mu_r)_-^2+(y_1-\mu_{r+1})_-^2}{4(n-r)}\Bigr).
\end{align*}
Since $r(n-r)\asymp n(n-r)$, substituting this into
\eqref{eq-two-barrier-upper-new-1} proves \eqref{eq-two-barrier-upp-1}.

Finally, if $|x-y|\le K\sqrt n$, then for all  $n/2 \le k \le n$ we have 
\[
|\mu_k-y|= \frac{V_n-V_k}{V_n}|x-y|\lesssim_K\frac{n-k}{\sqrt n}\lesssim  \sqrt {n-k}\,.
\]
The last assertion follows by absorbing the  shifts $\frac{\mu_{r}-y}{\sqrt{n-r}}$, $\frac{\mu_{r+1}-y}{\sqrt{n-r}}$ into the constants in
the polynomial and Gaussian factors above. This completes the proof.
\end{proof}

\begin{lem} 
\label{lem-two-barrier-ballot-asy} Given $K>0$, let $\mathscr{H}_{K}$ denote the set of all tuples $(n,r; x,x_0,x_{1};y,y_0,y_{1}) \in \mathbb{N}^2 \times \mathbb{R}^{6}$ satisfying $n \ge 4$, $n/2 \le r \le n-2$, and
\begin{align*}
    0 & < x_0 - x \le K \sqrt{n}, \quad 0 < y_0 - y \le K \sqrt{n-r}, \\
    | x_{1}- x_0 | & \le K \sqrt{n}, \quad | y_{1}- y_0 | \le K \sqrt{n-r}, \quad -K \sqrt{n-r} \le x_{1} - y_{1} \le K \sqrt{n}.
\end{align*}
Then, for each $K>0$, $a\in (0,1/2)$, and $b>0$, and  for any tuple in $\mathscr{H}_{K}$,
\begin{equation}
 \P _{0,n}^{x,y}\Bigl( Z \preceq_{[0,r]} \mathfrak{l}_{0,r}^{x_0,x_1} \pm \zeta_{0,r}^{a,b}  ,  Z \preceq_{[r+1,n]} \mathfrak{l}_{r+1,n}^{y_1,y_0} \pm \zeta_{r+1,n}^{a,b}       \Bigr)  
    \asymp_{K,a,b}  \left[ 1+\frac{(x_1-y_1)_+}{\sqrt{n-r}} \right]  \frac{U(x_0-x)U(y_0-y)}{n}\,.
    \label{eq-two-barrier-ballot-asy}
\end{equation}
The implicit constants in $ \asymp_{K,a,b} $ depend on $K, a,$ and $b$.
\end{lem}

\begin{proof}
The
assumptions defining $\mathscr H_K$ imply
\begin{equation*}
|x-y|\lesssim_K\sqrt n,
\qquad
|y_1-y|\lesssim_K\sqrt{n-r},
\qquad
(x_1-y_1)_-\lesssim_K \sqrt{n-r}. 
\end{equation*}
Consequently,  we have 
$
1+\frac{(x_1-y)_+}{\sqrt {n-r}}
\asymp_K 1+\frac{(x_1-y_1)_+}{\sqrt{n-r}}$ and $
1+\frac{(y_1-y)_+}{\sqrt{n-r}} \asymp_K 1 $.
The upper bound in \eqref{eq-two-barrier-ballot-asy}, for the plus sign,
now follows from Lemma~\ref{lem-two-barrier-ballot-upper-bound}. By
monotonicity, this is also an upper bound for the choice of the minus sign.

It remains to prove the lower bound for the minus sign. Choose constants
$A>B>K+2$, depending only on $K$, and set
$
I\coloneq [y_1-A\sqrt{n-r},\,y_1-B\sqrt{n-r}]$. 
Since $\Var_{0,n}^{x,y}(Z_r)\asymp n-r$ and
$|y_1-\mu_r|\lesssim_K\sqrt{n-r}$, the Gaussian density formula gives
\[
\P_{0,n}^{x,y}(Z_r\in I)\gtrsim_K1.
\]
Moreover, conditionally on $Z_r=z \in I$ and $Z_n=y$, the increment
$Z_{r+1}-Z_r$ is Gaussian with mean and variance
\[
\frac{\E[\xi_{r+1}^2]}{V_n-V_r}(y-z)
\quad\text{and}\quad
\E[\xi_{r+1}^2]
\Big(1-\frac{\E[\xi_{r+1}^2]}{V_n-V_r}\Big),
\]
respectively. Uniformly for $z\in I$, the mean is bounded in absolute value
by a constant depending only on $K$, while the variance is bounded above
and below by positive constants. Thus
\begin{equation}
\P_{0,n}^{x,y}( Z_r\in I,\ |Z_{r+1}-Z_r|\le1 )\gtrsim_K1.
\label{eq-sharp-two-barrier-interface-event}
\end{equation}

For $ z\in I$ and $|w -z |\le 1$, our choice of $A$ and $B$, together with
$x_1-y_1\ge-K\sqrt{n-r}$, yields
\begin{equation}
1+x_1-z\asymp_K \sqrt{n-r}+ (x_1-y_1)_+,
\qquad
1+y_1-w\asymp_K\sqrt{n-r}.
\label{eq-sharp-two-barrier-interface-gaps}
\end{equation}
In particular, $z\le x_1$ and $w\le y_1$. Conditional on
$(Z_r,Z_{r+1})=(z,w)$, Lemma~\ref{lem-ballot-rw}, applied to the two
sub-bridges, gives
\begin{align*}
&\P_{0,n}^{x,y}\Bigl(
Z\preceq_{[0,r]}\mathfrak l_{0,r}^{x_0,x_1}-\zeta_{0,r}^{a,b},\ 
Z\preceq_{[r+1,n]}\mathfrak l_{r+1,n}^{y_1,y_0}-\zeta_{r+1,n}^{a,b}
\,\Big|\,Z_r=z,Z_{r+1}=w\Bigr)
\\
&\quad\gtrsim_{K,a,b}
\frac{U(x_0-x)U(x_1-z)}{r}
\frac{U(y_0-y)U(y_1-w)}{n-r-1}.
\end{align*}
Here the minima in Lemma~\ref{lem-ballot-rw} can be removed because both
displayed products are bounded above by constants depending only on $K$;
this follows from the defining assumptions of $\mathscr H_K$ and
\eqref{eq-sharp-two-barrier-interface-gaps}. Integrating over $z \in I,w \in [z-1,z+1]$ and using
\eqref{eq-sharp-two-barrier-interface-event}--\eqref{eq-sharp-two-barrier-interface-gaps},
we obtain
\begin{align*}
&\P_{0,n}^{x,y}\Bigl(
Z\preceq_{[0,r]}\mathfrak l_{0,r}^{x_0,x_1}-\zeta_{0,r}^{a,b},\ 
Z\preceq_{[r+1,n]}\mathfrak l_{r+1,n}^{y_1,y_0}-\zeta_{r+1,n}^{a,b}
\Bigr)
 \gtrsim_{K,a,b}
\frac{ U(x_0-x)U(y_0-y) }{n} \Bigl[ 
 1+\frac{(x_1-y_1)_+}{\sqrt{n-r}}  \Bigr],
\end{align*}
where we used $r\asymp n$. This proves the lower bound
for the minus sign. The lower bound for the plus sign follows by
monotonicity, completing the proof.
\end{proof}

\noindent\textbf{Proof of Lemmas in  Section~\ref{sub:3.1}}
   
\begin{proof}[Proof of Lemmas \ref{lem-two-barrier-ballot-upper-bound-rw} and \ref{lem-two-barrier-ballot-asy-rw}]
 Notice that we can choose $b'$ large, depending only on $a,b$ such that 
	\begin{equation*}
		 \zeta_{0,n}^{a,b}(k) \le    b'\zeta_{0,n}^{a,b}(r)  +\zeta_{0,r}^{a,b'}(k) \ \forall k \in [0,r]\,;
	\end{equation*}
	similarly  $\zeta_{0,n}^{a,b}(k) \le    b'\zeta_{0,n}^{a,b}(r)  +\zeta_{r+1,n}^{a,b'}(k)$ for all $   k \in [r+1,n]$. Note that  $\zeta_{0,n}^{a,b}(r) \lesssim \sqrt{n-r}$,  thus    Lemma \ref{lem-two-barrier-ballot-upper-bound-rw} follows from Lemma \ref{lem-two-barrier-ballot-upper-bound} by replacing $x_1$, $y_1$  with $x_1 +b'\zeta_{0,n}^{a,b}(r)$ and $y_1 +b'\zeta_{0,n}^{a,b}(r)$ respectively. Similarly Lemma \ref{lem-two-barrier-ballot-asy-rw} immediately follows from \ref{lem-two-barrier-ballot-asy}.
\end{proof}

\begin{proof}[Proof of Lemma \ref{lem-ent-rw}] For simplicity write $\zeta =\zeta_{0,n}^{a,b}$ throughout this proof. 

	\smallskip
 \noindent   
\underline{Proof of \eqref{ent:right}.} 
Define   $\tau \coloneq \min\{k \ge r+1: Z_k +  \zeta (k) \ge y_0-M\}$. Observe that 
\[ \{Z \preceq_{[r+1,r']} \mathfrak{l}_{r+1,n}^{y_0,y_{0}} -\zeta-M \}^{c} \subset \{ \tau \le r'\}. \]
  On the event $\{\tau=l\}$, we necessarily have $Z_{l} +   \zeta (l) \ge y_0-M$ and  $Z_{k} +  \zeta (k) < y_0-M$ for $r+1 \le k < l$. 
Define for each $r+1 \le l \le r'$
\begin{equation}
\mathrm{Pr}_{\eqref{eq-entro-prob}} (l)  \coloneq  \P_{0,n}^{x, y} 
\Bigl( Z \preceq_{[0,r]} x_0 +  \zeta  ; Z \preceq_{[r+1,n]} \le y_0+  \zeta  ;  \tau  = l  \Bigr)  \label{eq-entro-prob}.
	\end{equation} 
It suffices to bound $\sum_{l=r+1}^{r'}\mathrm{Pr}_{\eqref{eq-entro-prob}} (l)$. 
 
Consider the case $l>r+1$ first. When $\tau=l$, for every $r+1 \le k < l$, we have $Z_{k} \le Z_{l} + \zeta(l)-\zeta(k) \le Z_{l}$ since   $r\ge n/2$. 
Applying the Markov property  at times $l-1$ and $l$ yields 
\begin{align}
\mathrm{Pr}_{\eqref{eq-entro-prob}} (l)  
& \le \int_{- \zeta(l)-M }^{\zeta(l)} \P_{0, n}^{x,y} (  Z_{l} - y_0 \in \dif z )
 \int_{-\infty}^{- \zeta(l)-M} \P_{0,l}^{x,y_0+z} ( Z_{l-1} - y_0 \in \dif w) 
\notag \\
 & \quad \times \P_{l, n}^{y_0+z,y} \Bigl(  Z \preceq_{[l,n]} y_0+ \zeta  \Bigr) 
	 \, \P_{0,l-1}^{x, y_0+w} 
 \Bigl( Z \preceq_{[0,r]} x_0+\zeta, 
	   Z \preceq_{[r+1,l-1]}   y_0+z  \Bigr)  . \label{eq-ballot-bd-1}
 \end{align}  

We now estimate the two bridge densities appearing in the integrand. Since $\Var_{0,n}^{x,y} (Z_{l}) \asymp {n-l}$, we have $ \P_{0,n}^{x,y} ( Z_{l} - y_0 \in \dif z) \lesssim \frac{1}{\sqrt{n-l}} \dif z$. Moreover, we claim that there exists  constants $C,c>0$ such that  for all large $n$ and any $w < - \zeta(l)-M < z < \zeta(l) $, 
  \begin{equation}\label{eq:overshoot}
	\P_{0,l}^{x,y_0+z} (  Z_{l-1} - y_0 \in \dif w )  \leq C \exp( - c |z-w|^2) \dif w .
 \end{equation}
Note that  $\Var_{0,l}^{x,y_0+z}  (Z_{l-1})\asymp 1$,  and  $\E_{0, l}^{x,y_0+z} [Z_{l-1}]= y_0+z+ \frac{x-(y_0+z)}{l}=y_0+z+O(1)$. 
Hence, we have $|y_0+w-\E_{0,l}^{x,y_0+z} [Z_{l-1}] | \ge c' (z-w)$ for all sufficiently large $z-w$. Estimate \eqref{eq:overshoot} therefore follows directly from the Gaussian density formula. 


Next, we bound the two probabilities in \eqref{eq-ballot-bd-1}. 
Note that there exists a constant  $b '$ depending only on $b $ such that 
\begin{equation}
	 x_0 +   \zeta_{0,n}^{a,b}(k) \le \mathfrak{l}_{0,r}^{x_0,x_0 +b'+\zeta(r) } (k) + \zeta_{0,r}^{a,b'}(k) \,, \, \forall \, k \in [0,r] \cap \mathbb{N} .   \label{eq:linear-inter}
\end{equation}
Thus, we  apply  Lemma~\ref {lem-two-barrier-ballot-upper-bound} and derive  that, 
 for all $ -  \zeta(l)-M \le z \le  \zeta(l)$,   
\begin{align} 
	&  \P_{0,l-1}^{x, y_0+w} 
  \Bigl(   Z \preceq_{[0,r]} x_0+\zeta, 
	   Z \preceq_{[r+1,l-1]}   y_0+z \Bigr)  \notag\\
&\ \le  \P_{0,l-1}^{x, y_0+w} 
 \Bigl( Z \preceq_{[0,r]} \mathfrak{l}_{0,r}^{x_0,x_0 +b'+\zeta(r) }  + \zeta_{0,r}^{a,b'} ,
	   Z \preceq_{[r+1,l-1]}   y_0+z  \Bigr)   \notag\\
&\ \lesssim  \frac{(1+x_0-x)(1+z-w)}{n} \Bigl[ 1+ \frac{(x_0+  \zeta(r)- \E_{0,l-1}^{x, y_0+w}[Z_r ])_+ }{\sqrt{l-r}} \Bigr] \Bigl[ 1+ \frac{(y_0+z- \E_{0,l-1}^{x, y_0+w}[Z_{r+1} ])_+}{\sqrt{l-r}} \Bigr] .  \label{eq:bal-2-r}
\end{align}  
Note that $|\E_{0,l-1}^{x, y_0+w}[Z_r ] -  (y_0+z) | \le  O(\frac{l-r}{l}) |x-y_0-z| +O( z-w )$.  By our assumption, we have 
$|x-y_0| \lesssim \sqrt{n} \lesssim \sqrt{l}$, and $|z| \le \zeta(l)+M \lesssim \sqrt{n} \lesssim \sqrt{l}$. Thus $  O(\frac{\sqrt{l-r}}{l}) |x-y_0-z|= O(1)$.  
Combining with  the fact  $0\le \zeta(r) - \zeta(l)\le \zeta(l-r) \lesssim \sqrt{l-r}$, we obtain that 
\eqref{eq:bal-2-r} is dominated from above by 
 \begin{equation}
	\frac{(1+x_0-x)}{n} \Bigl[ 1+ \frac{x_0-y_0  +\zeta(l) + M  }{\sqrt{l-r}} \Bigr] (1+z-w)^3. \label{eq:bal-2-r-0new}
 \end{equation} 
Similarly, the standard Ballot estimate (Lemma \ref{lem-ballot-rw}) implies 
\begin{equation}
	  \label{eq:bal-2-r-3}
\P_{l, n}^{y_0+z,y} \Bigl(  Z \preceq_{[l,n]} y_0+ \zeta  \Bigr)  \lesssim \frac{[1+ \zeta(l)-z] (1+y_0-y) }{n-l} \lesssim  [\zeta(l)+M]\frac{ (1+y_0-y) }{n-l}. 
\end{equation}

Plugging  inequalities \eqref{eq:bal-2-r-0new} and \eqref{eq:bal-2-r-3} back into  \eqref{eq-ballot-bd-1}, 
we  obtain that for every  $r+1<l \le r'$, 
\begin{equation}
	  \mathrm{Pr}_{\eqref{eq-entro-prob}} (l) \lesssim \frac{U(x_0-x)U(y_0-y)}{n} \, \frac{ \zeta(l)+M }{(n-l)^{3/2}}  \, \Bigl[ 1+ \frac{x_0-y_0  +\zeta(l) + M  }{(l-r)^{1/2}} \Bigr]  \ I_{l} , \label{eq:bnds-Pr-l}
\end{equation}  
where 
\begin{align*}
I_{l}\coloneq \int_{-   \zeta(l)-M }^{  \zeta(l)} \dif z \,   \int_{-\infty}^{ - \zeta(l)-M} \dif w \,     e^{-c|z-w|^2}     ( 1+ z-w  )^3   
=\int_0^{2\zeta(l)+M}
\int_\lambda^\infty
e^{-c\Delta^2}(1+\Delta)^3
\dif \Delta\dif \lambda =O(1). 
\end{align*}
Here in the second equality we made the change of variables $\lambda=z+\zeta(l)+M$ and $\Delta=z-w$. 

Summing \eqref{eq:bnds-Pr-l}  from $l=r+2$ to $r'$ we have 
\begin{align*}
& \sum_{l=r+2}^{r'} \, \frac{ \zeta(l)  }{(n-l)^{3/2}}    \lesssim  \frac{1}{(n-r')^{\frac{1}{2}-a}}  \, ;
\sum_{l=r+2}^{r'} \frac{ (x_0-y_0+ M) \zeta(l) }{(n-l)^{3/2}(l-r)^{1/2}}  \lesssim \frac{x_0-y_0+ M}{(n-r')^{1/2-a}(n-r)^{1/2}} ; \\
& \sum_{l=r+2}^{r'} \frac{ \zeta(l)^2 }{(n-l)^{3/2}(l-r)^{1/2}}  \lesssim
\frac{
1+\ind{a=1/4}
\log\left(\frac{n-r}{n-r'}\right)}
{(n-r)^{\frac12\wedge(1-2a)}
 (n-r')^{(\frac12-2a)_+}}
 \ll \frac{1}{(n-r')^{\frac{1}{2}-a}}; \\
& \sum_{l=r+2}^{r'} \frac{M}{(n-l)^{3/2}} \lesssim \frac{M}{(n-r')^{1/2}} \	\text{ and } \
\sum_{l=r+2}^{r'} \frac{M(x_0-y_0+M)}{(n-l)^{3/2} (l-r)^{1/2} } \lesssim \frac{M(x_0-y_0+M)}{(n-r')^{1/2}(n-r)^{1/2}} .
\end{align*}  

When $r'\geq r+1$, conditioning on $Z_{r+1}$ in the boundary case
$l=r+1$ and repeating the same barrier comparison and ballot estimates
as above give,  
\[ 
\mathrm{Pr}_{\eqref{eq-entro-prob}}(r+1)
\lesssim_K
\frac{U(x_0-x)U(y_0-y)}{n}
\frac{M+(n-r)^a}{(n-r)^{1/2}}
 \Bigl( 1+\frac{x_0-y_0}{(n-r)^{1/2}}   \Bigr) . 
\]
Combining this with the preceding sum proves \eqref{ent:right}.

\smallskip
 \noindent
\underline{Proof of \eqref{ent:left}.}   
Since $x_0 \ge y_0 $ we can assume without loss of generality that $r \ge 4n/5 $.    
Indeed, if $r<4n/5$, replace $r$ and $r''$ by
$
\widetilde r\coloneq  \lceil\frac{4n}{5} \rceil,$ and $
\widetilde r''\coloneq r''\wedge(n-\widetilde r)$. 
Since $x_0\ge y_0$, these replacements can only enlarge the event on
the left-hand side of \eqref{ent:left}. Moreover, since $r\ge n/2$
and $r''\le n-r$, we have $
n-\widetilde r\asymp n-r
$ and $
\widetilde r''\asymp r''$.
Thus the assumptions and the right-hand side of \eqref{ent:left}
change only by constant factors.

Define  $\wh{\tau}= \max\{ r'' \le k \le r : Z_k + \zeta(k) > x_0-M \}$. On the event $\{ Z \preceq_{[r'',r]} x_0 - \zeta - M \}^{c}$,   there exists an index $r'' \le l  \le r$ such that $\wh{\tau}= l$. In this case, it holds 
\[  Z_{l} + \zeta(l) \ge x_0 - M \text{ and } Z_k + \zeta(k) \le Z_{l} + \zeta(l) \text{ for all } l  < k \le r .  \]
As in the previous proof, set 
\begin{equation}
	\mathrm{Pr}_{\eqref{eq-ballot-bd-2}} (l)   \coloneq  \P_{0,n}^{x, y} 
\bigl( Z \preceq_{[0,r]} x_0 +  \zeta  ; Z \preceq_{[r+1,n]} \le y_0+  \zeta  ;  \wh{\tau}  = l  \bigr)\,, \label{eq-ballot-bd-2}
\end{equation}
and it suffices to control $\sum_{l=r''}^{r}\mathrm{Pr}_{\eqref{eq-ballot-bd-2}} (l)$.

Assume $r'' \le l < r$ first. Applying the Markov property at times $l$ and $l+1$ yields
\begin{align*}
\mathrm{Pr}_{\eqref{eq-ballot-bd-2}} (l)    
& \le \int_{- \zeta(l)-M }^{\zeta(l)} \P_{0, n}^{x,y} (  Z_{l} - x_0 \in \dif z ) \int_{-\infty}^{- \zeta(l+1)-M} \P_{l,n}^{x_0+z,y} ( Z_{l+1} - x_0 \in \dif w) 
\notag \\
 & \qquad \times \P_{0,l}^{x,x_0+z} \Bigl(  Z \preceq_{[0,l]} x_0+  \zeta  \Bigr) 
	 \, \P_{l+1,n}^{x_0+w, y} 
 \Bigl(  Z \preceq_{[l+1,r]} x_0+\zeta \wedge (z+\zeta(l)-\zeta), 
	   Z \preceq_{[r+1,n]}   y_0+\zeta  \Bigr)  .
 \end{align*}  
  We estimate the terms in the integrand as follows:
  \begin{enumerate}[(i)]
	\item   $\P_{0, n}^{x,y} (  Z_{l} - x_0 \in \dif z )  \P_{l,n}^{x_0+z,y} ( Z_{l+1} - x_0 \in \dif w)   \lesssim \frac{1}{\sqrt{\wedge^{n}(l)}} e^{-c|z-w|^2} \dif z \dif w$.
  \item The standard Ballot estimate (Lemma \ref{lem-ballot-rw}) implies
     \[ \P_{0,l}^{x,x_0+z} (  Z \preceq_{[0,l]} x_0+ \zeta  ) \lesssim \frac{(1+x_0-x)(1+\zeta(l)-z)}{l} \lesssim [\zeta(l)+M]\frac{U(x_0-x) }{l}\,. \]
	\item  
 Using a similar middle step as in \eqref{eq:linear-inter}, Lemma~\ref {lem-two-barrier-ballot-upper-bound} implies that we can dominate the probability $\P_{l+1,n}^{x_0+w, y} 
 (  Z \preceq_{[l+1,r]} x_0+\zeta \wedge (z+\zeta(l)-\zeta), 
	   Z \preceq_{[r+1,n]}   y_0+\zeta   ) $  from above by  
\begin{equation}
	\label{eq:a.2-bnd-2}
	 \Bigl[ 1+\frac{
	\bigl(x_0+ \mathfrak{b}_l(z) -
	\E_{l+1,n}^{x_0+w,y}[Z_r]\bigr)_+
	}{\sqrt{\wedge_l^n(r)}} \Bigr] 
	\Bigl[ 1+\frac{
	\bigl(y_0+\zeta(r)-
	\E_{l+1,n}^{x_0+w,y}[Z_{r+1}]\bigr)_+
	}{\sqrt{\wedge_l^n(r)}} \Bigr] \frac{U(z-w)U(y_0-y)}{n-l} . 
\end{equation} 
Here, $\mathfrak{b}_l(z) \coloneq \zeta(r) \wedge( z+ \zeta(l)-\zeta(r))$.    
\item We claim that the first factor in \eqref{eq:a.2-bnd-2} is dominated above by  
$(1+|z-w|)
 (
1+\frac{x_0-y_0}{\sqrt{n-r}}
 )$. 
 
Suppose first that $r-l\le n-r$. Since
$|\zeta(l)-\zeta(r)| 
\lesssim\sqrt{r-l}$, we have 
$\mathfrak{b}_l(z)  \le z+C\sqrt{r-l}$.
Moreover,
$x_0-\E_{l+1,n}^{x_0+w,y}[Z_r]
=-w+\frac{V_r-V_{l+1}}{V_n-V_{l+1}}(x_0+w-y)$. Since $w\le0$,   we obtain
$(x_0+\mathfrak{b}_l(z) -
\E_{l+1,n}^{x_0+w,y}[Z_r])_+
\lesssim
|z-w|+\sqrt{r-l}
+\frac{r-l}{n-l}(x_0-y_0+y_0-y)  $. Combining with 
$\wedge_l^n(r)=r-l$,
$\frac{r-l}{n-l}\frac1{\sqrt{r-l}}
\lesssim\frac1{\sqrt{n-r}}$, and
$y_0-y\le K\sqrt{n-r}$, the desired estimate follows.

Suppose next that $r-l>n-r$. In this case
$\mathfrak{b}_l(z) \le\zeta(r)$ and
$\wedge_l^n(r)=n-r$. Writing
$\E_{l+1,n}^{x_0+w,y}[Z_r]
=y+\frac{V_n-V_r}{V_n-V_{l+1}}(x_0+w-y)$,
we get
$(x_0+\mathfrak{b}_l(z) -
\E_{l+1,n}^{x_0+w,y}[Z_r])_+
\lesssim
(x_0-y_0+y_0-y)+\zeta(r)
+\frac{n-r}{n-l}(-w)$.
Since
$-w\le |z-w|+\zeta(l)+M$,
$\zeta(r)\lesssim\sqrt{n-r}$,
$M\le K\sqrt{n-r}$, and
$\frac{\sqrt{n-r}}{n-l}\zeta(l)
\lesssim1$,
the required estimate follows again.

\item  We show that the second factor  in \eqref{eq:a.2-bnd-2} is dominated by $(1+|z-w|)(1+
\frac{\zeta(r)}{\sqrt{\wedge_l^n(r)}}
+
\frac{n-r}{n-l}
\frac{\zeta(l)+M}{\sqrt{\wedge_l^n(r)}})$.  
Indeed,  
$y_0 -\E_{l+1,n}^{x_0+w,y}[Z_{r+1}]
= \frac{V_{r+1}-V_{l+1}}{V_n-V_{l+1}}(y_0-y)
-\frac{V_n-V_{r+1}}{V_n-V_{l+1}}(x_0-y_0+w)
 $.
Since $x_0-y_0\ge0$,   and 
$-w \le |z-w|+\zeta(l)+M$.
It follows that
$ (y_0+\zeta(r)-\E_{l+1,n}^{x_0+w,y}[Z_{r+1}] )_+
\lesssim \zeta(r)
+\frac{n-r}{n-l} (|z-w|+\zeta(l)+M )
+\frac{r-l}{n-l}(y_0-y)$. Combining with 
$\frac{n-r}{n-l}\frac1{\sqrt{\wedge_l^n(r)}}\lesssim1$ and
$\frac{r-l}{n-l}\frac1{\sqrt{\wedge_l^n(r)}}
\lesssim\frac1{\sqrt{n-r}}$,  $y_0-y\le K\sqrt{n-r}$,   the 
claim follows.  
 \end{enumerate}

Substituting the preceding estimates and integrating first in $w$ and
then in $z$, as in the calculation of $I_l$ in \eqref{eq:bnds-Pr-l}, shows that  
\begin{align*}
\mathrm{Pr}_{\eqref{eq-ballot-bd-2}} (l) \lesssim &\frac{U(x_0-x)U(y_0-y)}
{l(n-l)\sqrt{\wedge^n(l)}}
(\zeta(l)+M)
\left(1+\frac{x_0-y_0}{\sqrt{n-r}}\right)
\left[
1+\frac{\zeta(r)}{\sqrt{\wedge_l^n(r)}}
+\frac{n-r}{n-l}
 \frac{\zeta(l)+M}{\sqrt{\wedge_l^n(r)}}
\right].
\end{align*}
It remains to use
\begin{align*}
 \sum_{l=r''}^{r-1}
\frac{\zeta(l)+M}{l(n-l)\sqrt{\wedge^n(l)}}
\left[
1+\frac{\zeta(r)}{\sqrt{\wedge_l^n(r)}}
+\frac{n-r}{n-l}
 \frac{\zeta(l)+M}{\sqrt{\wedge_l^n(r)}}
\right] \lesssim 
\frac1n\frac{M+(r'')^a}{\sqrt{r''}}.
\end{align*}
\begin{enumerate}[(1)]
	\item  For the term without the two fractions, we have 
\begin{align*}
 \sum_{l=r''}^{r-1}
\frac{\zeta(l)+M}{l(n-l)\sqrt{\wedge^n(l)}} \lesssim
\frac1n\sum_{l=r''}^{n/2}\frac{l^a+M}{l^{3/2}}
+\frac1n\sum_{l=n/2}^{r-1}
\frac{(n-l)^a+M}{(n-l)^{3/2}}  
\lesssim
\frac1n\frac{M+(r'')^a}{\sqrt{r''}}.
\end{align*}
	\item 
For the term containing $\zeta(r)$, if $r-l\geq n-r$, then
$\frac{\zeta(r)}{\sqrt{\wedge_l^n(r)}} \lesssim 1.$
If $1\leq r-l<n-r$, then $l\asymp n$, $n-l\asymp n-r$, and
$\zeta(l)\asymp(n-r)^a$, so
\begin{align*}
 \sum_{r-(n-r)<l<r}
\frac{\zeta(l)+M}{l(n-l)\sqrt{\wedge^n(l)}}
\frac{\zeta(r)}{\sqrt{\wedge_l^n(r)}} 
 \lesssim
\frac{M+(n-r)^a}{n\sqrt{n-r}}
(n-r)^{a-1/2}
\lesssim
\frac1n\frac{M+(r'')^a}{\sqrt{r''}}.
\end{align*}
\item 
The remaining term satisfies the desired inequality. 
Indeed, for $r-l\leq n-r$, we use
$\zeta(l)+M\lesssim\sqrt{n-r}$ to obtain
\begin{align*}
 \sum_{r-(n-r)\leq l<r}
\frac{(n-r)(\zeta(l)+M)^2}
{l(n-l)^2\sqrt{\wedge^n(l)}\sqrt{\wedge_l^n(r)}} 
 \lesssim
\frac{(M+(n-r)^a)^2}{n(n-r)^{3/2}}
\sum_{l=1}^{n-r}l^{-1/2} 
\lesssim
\frac{M+(n-r)^a}{n\sqrt{n-r}} .
\end{align*}
For $r-l>n-r$, the interpolation factor yields
\begin{align*}
\sum_{\substack{r''\leq l<r\\r-l>n-r}}
\frac{(n-r)(\zeta(l)+M)^2}
{l(n-l)^2\sqrt{\wedge^n(l)}\sqrt{\wedge_l^n(r)}} 
& \lesssim
\frac{\sqrt{n-r}}{n^2}
\sum_{l=r''}^{n/2}\frac{(l^a+M)^2}{l^{3/2}}
+\frac{\sqrt{n-r}}{n}
\sum_{l=n/2}^{r-(n-r)}
\frac{((n-l)^a+M)^2}{(n-l)^{5/2}}\\
&\quad\lesssim 
\frac1n\frac{M+(r'')^a}{\sqrt{r''}},
\end{align*}
where the last inequality follows by integral comparison, using
$r''\leq n-r$, $M\leq K\sqrt{n-r}$, and $a<1/2$.  This proves the
summation estimate.
\end{enumerate}

For $l=r$, conditioning directly on $Z_r$ and applying
Lemma~\ref{lem-ballot-rw} to the remaining bridges gives
\[
\mathrm{Pr}_{\eqref{eq-entro-prob}}(r)
\lesssim 
\frac{U(x_0-x)U(y_0-y)}{n}
\frac{M+(n-r)^a}{\sqrt{n-r}}
\left(1+\frac{x_0-y_0}{\sqrt{n-r}}\right).
\]
Since $r''\leq n-r$ and $a<1/2$, this is bounded by the right-hand side
of \eqref{ent:left}.  Combining it with the preceding sum proves
\eqref{ent:left}.   
\end{proof}

\subsection{Simple random walk on the 2D torus}
\label{apx:2dsrw}

In this section we prove the random walk statements from subsection~\ref{s:2.4.1}.
Let us denote $P^{\mathsf{T}_N}_{k}(x,y)$ the transition matrix of the simple random walk on the 2D torus $\mathsf{T}_N$.
Let $\mathbf{0}=(0,0)$. 
To get rid of the periodicity of the random walk, let us denote by $q_k^{N}(x,y)\coloneq [P^{\mathsf{T}_N}_{k}(x,y) + P^{\mathsf{T}_N}_{k+1}(x,y)]/2$.
The following heat kernel estimate can be found in \cite[Lemma 4.3]{arous2008} (see also \cite[Lemma 3.1]{Cox89}).

\begin{lem}
	\label{eq:heat-kernal-torus}
	Fix an arbitrary sequence $\omega_{N} \to \infty$. Then the following assertions hold.
	\begin{enumerate}[(i)]
		\item There exists constants $C,c>0$ so that $ q^{N}_{k}(\mathbf{0},x) \le  C [\frac{1}{k} e^{-c |x|^2/k} +  \frac{1}{N^2}e^{-c N^2/k} ]$. 
	\item Uniformly in $k \le   {N^2}/{\omega_{N}}$, $x \in  \mathsf{T}_N$ and $\omega_{N}(1+|x|^{2}) \le k$, we have
\[ \pi k   \, q^N_{k}(\mathbf{0},x) = 1  + o_{N}(1). \]
	\item Uniformly in $k \ge \omega_{N} N^2$ and $x \in  \mathsf{T}_{N}$, we have
\[ {|\mathsf{T}_N|} \, q^{N}_{k}(  \mathbf{0} ,x) = 1 + o_{N}(1)\,. \]
	\end{enumerate} 
\end{lem}


We can now give:
\begin{proof}[Proof of part (i) in Lemma~\ref{lem:est_laplace}] 
Denote by $\varsigma_N \coloneq   \lambda_N / (|\mathsf{T}_N| \log N) $. By our assumption  then, the translation invariance  
reduces the problem to estimating $\bbE_{\mathbf{0}}[ e^{ -\varsigma_N H( \mathsf{B}_{\rho}(x) ) }]$ with $|x| \geq N/r_N^2$.

\smallskip
\noindent\underline{\textit{Step 1.}}
Let us start with the case $\rho=0$. The strong Markov property yields 
\begin{equation}
	\label{eq:hit-Green}
\bbE_{z}[\exp (-\varsigma_N  H (\mathbf{0}))  ] = G(z \mid \varsigma_N) /G(\mathbf{0} \mid \varsigma_N) ,
\end{equation}
where $G( \cdot \mid \varsigma_N)$ is the Green function of the random walk with geometric killing rate {$1-\rme^{-\varsigma_N}$} given by 
$
G(z \mid \varsigma_N) 
\coloneq  \sum_{k=0}^{\infty} \rme^{-\varsigma_N k} \, P^{\mathsf{T}_N}_{k}(\mathbf{0},z)$. 
We claim that  
\begin{align}
 G(x \mid \varsigma_N) & = [1+o_N(1)] \frac{ \log N}{\lambda_N}  \text{ uniformly in }  |x| \ge N/r^2_N, \ \text{ and } \label{eq:Green-kill-1}   \\
G(\mathbf{0} \mid \varsigma_N) &= \frac{2 \log N }{\pi } \bigl(1+ o_N(1)\bigr) . \label{eq:Green-kill-2}
\end{align}
Plugging these two estimates \eqref{eq:Green-kill-1} and \eqref{eq:Green-kill-2} back into \eqref{eq:hit-Green}, we obtain   
\begin{equation} \label{equa:lem:est_laplace-rho=0}
\lim_{N\to \infty}   \ \sup_{x :d(\mathbf{0},x) \ge N/r^{2}_N}   \Big|\lambda_N \bbE_{\mathbf{0}} \left[ \rme^{- \varsigma_N H(x) }  \right] - \frac{\pi}{2} \Big| = 0 . 
\end{equation}

\noindent\underline{\textit{Step 2.}} We now verify that two asymptotic equality \eqref{eq:Green-kill-1} and \eqref{eq:Green-kill-2}.
Let 
\[
Q(z\mid \varsigma_N)\coloneq \sum_{k=0}^{\infty}\rme^{-\varsigma_N k}q_k^N(\mathbf{0},z)
	= \frac{1+\rme^{\varsigma_N}}{2}G(z\mid \varsigma_N)
	-\frac{\rme^{\varsigma_N}}{2}\ind{z=0}.
\]
Since $\varsigma_N=o_N(1)$,  the estimates   for $Q(\cdot\mid\varsigma_N)$ imply the same estimates for $G(\cdot\mid\varsigma_N)$, with only an $O(1)$ error. Thus it is enough to prove \eqref{eq:Green-kill-1} and \eqref{eq:Green-kill-2} for $Q(\cdot \mid \varsigma_N)$.

Take a sequence $\omega_N \to \infty$ with  $\omega_N   =o(\log N/\lambda_N)$ so that $ \varsigma_N \omega_N N^2 =o_N(1)$. Using the heat kernel estimate Lemma \ref{eq:heat-kernal-torus}, we have,  uniformly in $y \in \mathsf{T}_N$
\begin{align*}
\sum_{k=  \omega_N N^2}^{\infty} \rme^{-\varsigma_N k} q_k(\mathbf{0},y) & =  [1+o_{N}(1)]  \frac{1}{|\mathsf{T}_N|} \sum_{k= \omega_N N^2}^{\infty}   \rme^{-\varsigma_N k} =   [1+o_N(1)] \frac{\log N}{\lambda_N} \, ; \ \text{ and } \\ 
\sum_{k= N^2}^{  \omega_N N^2  } \rme^{-\varsigma_N k} q_k(\mathbf{0},y) & \le   C \sum_{k=   N^2}^{  \omega_N N^2  } \bigl(  \frac{1}{k}  +  \frac{1}{N^2} \bigr)    \lesssim \log \omega_N + \omega_N   = o( \frac{\log N}{\lambda_N} ) . 
\end{align*} 
For $ |x| \geq N/r_N^2$, Lemma \ref{eq:heat-kernal-torus} yields
\begin{equation*}
\sum_{k= 0}^{  N^2} \rme^{-\varsigma_N k} q_k(\mathbf{0},x)  \leq  C      \sum_{k= 1}^{N^2 }   \frac{1}{k} e^{-c|x|^2/k} + C\le  2 C \log r_N + C' = o(\frac{\log N}{\lambda_N}) .
\end{equation*}
Furthermore
for $x= \mathbf{0}$, using Lemma \ref{eq:heat-kernal-torus} we get 
\begin{equation*}
\sum_{k= 0}^{  N^2} \rme^{-\varsigma_N k} q_k(\mathbf{0},\mathbf{0})  = [1+ o_N(1)] \sum_{k=1}^{N^2 }   \frac{1}{\pi k}  = [1+o_N(1)] \frac{2}{\pi} \log N .
\end{equation*}
All together,  the claims \eqref{eq:Green-kill-1} and \eqref{eq:Green-kill-2} are established.

\smallskip
\noindent\underline{\textit{Step 3.}} It remains to show that \eqref{equa:lem:est_laplace-rho=0} implies the required result \eqref{equa:lem:est_laplace}.    Assume $\rho\ge1$.  Clearly   $H(x) \geq H(\mathsf{B}_{\rho}(x))$. The strong Markov property yields
\begin{align*}
		\bbE_{\mathbf{0}} \bigl[\rme^{ -\varsigma_N H( \mathsf{B}_{\rho}(x) ) } \bigr] & \ge \bbE_{\mathbf{0}} \bigl[\rme^{ -\varsigma_N H(x) } \bigr] 
		\geq \bbE_{\mathbf{0}} \bigl[\rme^{ -\varsigma_N H( \mathsf{B}_{\rho}(x) ) } \bigr] \inf_{ y \in \partial \mathsf{B}_{\rho}(x) } \bbE_{y}\bigl[\rme^{ -\varsigma_N H(x) } \bigr] \\
		& \ge \bbE_{\mathbf{0}} \bigl[\rme^{ -\varsigma_N H( \mathsf{B}_{\rho}(x) ) } \bigr]  \,  \rme^{-\varsigma_N N}  \bigl[1- \,\sup_{ y \in \partial \mathsf{B}_{\rho}(x) }\bbP_y \bigl( H(x) > N \bigr) \bigr].
\end{align*}
This  reduces the problem to show  $ \,\sup_{ |y -x| \le r_N^2   }\bbP_y \bigl( H(x) > N \bigr) =o_N(1)$.
Notice that $Y_i$ does not feel the boundary conditions if $i\leq N$ and hence $\bbP_{y} \bigl( H(x) \leq N \bigr)$ is equal to the probability that the simple random walk on $\bbZ^2$ hits $\{x\}$ until time $N$. One can easily obtain from Proposition 6.4.1 in \cite{LawL2010} that  
\[
\bbP_{y} \bigl( H(x) \geq N \bigr) \lesssim \bbP_{y} \bigl( H(x) \geq H(\partial \mathsf{B}_{ N^{1/3}}(x) ) \bigr) \lesssim \frac{\log|y-x|}{\log N} \quad  \text{ for } y \in \partial \mathsf{B}_{\rho}(x).
\]
Since $\log |y-x|   \le  2\log r_N = o(\log N)$, it gives the desired lower-bound. This finishes the proof. 
\end{proof}

\begin{proof}[Proof of part (ii) in Lemma~\ref{lem:est_laplace}]  Set $\ell_N\coloneq 2N-1$, so that $|\mathsf{T}_N|=\ell_N^2$.  Let $\mu$ denote the uniform distribution on $\mathsf{T}_N$, i.e., $\mu(y) \equiv \frac{1}{|\mathsf{T}_N|}$. 
Define  
	\begin{equation*}
		Z_N(x) \coloneq  \sum_{k=0}^{\infty}[ P^{\mathsf{T}_N}_{k}(\mathbf{0},x) -  \mu(x) ]  \text{ for } x \in \mathsf{T}_N. 
 \end{equation*}
\noindent\underline{\textit{Step 1.}} We claim that  it suffices to prove  
\begin{equation}\label{hit-fund-mati-2}
	Z_N(x) =  \frac{2}{\pi} \log \frac{\ell_N}{1+|x|}  + O(1) .
\end{equation}
Indeed, it  is well-known (see e.g. \cite[Lemma 2.12]{AF14}) that 
 \begin{equation}\label{hit-fund-matix}
 \frac{1}{|\mathsf{T}_N|} \E_{\mathbf{0}} [ H(x)] = \mu(\mathbf{0})  \E_{\mathbf{0}} [ H(x)] =  Z_N(\mathbf{0})- Z_N(x) .
 \end{equation} 
Substituting \eqref{hit-fund-mati-2} into  \eqref{hit-fund-matix}, since the random walk is transitive, we get 
\begin{equation}\label{eq:rho=0}
  \sup_{x,y \in \mathsf{T}_N, x \neq y}   \big | \E_{x}[ H(y) ] - \frac{2}{\pi} |\mathsf{T}_N| \log[ 1+d(x,y)]  \big|   = O(N^2) , 
\end{equation}  
which proves \eqref{eq:mean-hitting}  for the special case $\rho=0$. If  $1 \le \rho \le d(x,y)$,   the strong Markov property implies  
\begin{equation*}
	| \E_{\mathbf{0}} [ H(x)] - 	\E_{\mathbf{0}} [ H(\mathsf{B}_{\rho}(x))] | \le \max_{y \in \mathsf{B}_{\rho}(x)}	\E_{y} [ H(x)] = \frac{2}{\pi}  |\mathsf{T}_N|   \log (1+\rho) + O(N^2)\,.
\end{equation*}
Here the last equality follows from \eqref{eq:rho=0}.
 This gives the desired  assertion (ii) in Lemma~\ref{lem:est_laplace}. 
\smallskip
 
\noindent\underline{\textit{Step 2.}} The proof of \eqref{hit-fund-mati-2} is divided into two steps. We first consider the contribution from $0 \le k \le \ell_N^2=|\mathsf{T}_N|$ in the series of \eqref{hit-fund-mati-2}.
Let us denote by $P_k^{\mathbb{Z}^2}(\mathbf{0}, x)$ the probability that simple random walk on $\mathbb{Z}^2$ is at $x$ at step $k$. Then 
\begin{align*}
  \sum_{k=0}^{ \ell_N^2}  | P^{\mathsf{T}_N}_{k}(\mathbf{0},x)  -  P_k^{\mathbb{Z}^2}(\mathbf{0}, x) | &=   \sum_{k=1}^{\ell_N^2} \sum_{y \in \mathbb{Z}^2 \setminus \{\mathbf{0}\}} P_k^{\mathbb{Z}^2}(\mathbf{0}, x+\ell_N y ) \lesssim \sum_{k=1}^{\ell_N^2} \sum_{y \in \mathbb{Z}^2 \setminus \{\mathbf{0}\}}  \frac{1}{k} e^{- c \ell_N^2/k} = O(1)\,.
\end{align*}
Moreover by local limit theorem \cite[Theorem 2.1.3]{LawL2010}, uniformly in $x \in \mathsf{T}_N$ we have  
\begin{align*}
  \sum_{k=0}^{\ell_N^2}    P_k^{\mathbb{Z}^2}(\mathbf{0}, x) &=  \sum_{k=0}^{\ell_N^2}    \frac{1}{2}[P_k^{\mathbb{Z}^2}(x)+ P_{k+1}^{\mathbb{Z}^2}(x)] +O(1)=  \sum_{k=1}^{\ell_N^2}     \frac{1}{\pi k} e^{-|x|^2/k} + O(1) \\
  &=  \frac{1}{\pi} \int_{1}^{\ell_N^2} \frac{1}{\lambda} e^{-|x|^2/\lambda} \dif \lambda  +O(1) =  \frac{1}{\pi} \int_{|x|^2/\ell_N^2}^{|x|^2} \frac{1}{\lambda} e^{-\lambda} \dif \lambda  +O(1) =  \frac{2}{\pi} \log \frac{\ell_N}{1+|x|}  + O(1). 
\end{align*}
Integrating all the previous estimates we obtain, uniformly in $x$,
\begin{equation*}
 \sum_{k=0}^{\ell_N^2}  [ P^{\mathsf{T}_N}_{k}(\mathbf{0},x) -  \mu(x) ]   =  \frac{2}{\pi} \log \frac{\ell_N}{1+|x|}  + O(1) .
\end{equation*} 

\noindent\underline{\textit{Step 3.}} Combining the first two steps, it remains to verify
\begin{equation}\label{hit-fund-mati-3}
	\sup_{x \in \mathsf{T}_N }\sum_{k= \ell_N^2}^{\infty} | P^{\mathsf{T}_N}_{k}(\mathbf{0},x) -  \mu(x) | \le   \frac{1}{|\mathsf{T}_N|} \sum_{k= \ell_N^2/2 }^{\infty} \exp(-  \Theta(k/N^2)  ) = O(1). 
\end{equation}
To this end, we denote by $(\gamma_{j})_{j=1}^{|\mathsf{T}_N|}$ the eigenvalues of $P^{\mathsf{T}_N
}$, with $\gamma_1=1$ and $1 > \gamma_j \ge -1$ for $j \ge 2$.
Using the spectral representation of reversible transition matrix (see \cite[Lemma 12.2]{LPW17}) gives the identity
\begin{equation}
	\label{eq:Spectral Representation}
	\frac{P^{\mathsf{T}_N
}_k(\mathbf{0},x)}{\mu(x)} - 1 = \sum_{j=2}^{|\mathsf{T}_N|} f_{j}(\mathbf{0}) f_{j}(x) \gamma_{j}^{k},
\end{equation} 
where $
(f_{j})$ are   eigenfunctions corresponding to $(\gamma_{j})$ satisfying $\sum_{x \in \mathsf{T}_N} f_{i}(x) f_{j}(x) \mu(x)= \ind{i=j}$.  
 Since $\ell_N=2N-1$ is odd, the walk on $\mathsf T_N$ is aperiodic.   By the example   Section 12.3.1  and Lemma 12.12 in \cite{LPW17}, $\gamma$, we have 
\begin{equation*}
\gamma_{j} = \frac{ \cos(2\pi j_1/\ell_N) + \cos(2\pi j_2/\ell_N)}{2} ;	j=(j_1,j_2), 0 \le j_1 < \ell_N,  0 \le j_2 < \ell_N\,,
\end{equation*}
and we can take $f_j$ as  product trigonometric functions and in  particular $|f_{j}(x)| \le 1$. Since $\ell_N$ is odd,  we have   $\rho_{\star}:= \max \{|\gamma_{j}| :  j \ge 2 \} =1-\Theta(N^{-2})$.  Thereby from \eqref{eq:Spectral Representation} we obtain  for all $k \ge \ell_N^2/2$, 
\begin{align*}
 \sup_{x \in \mathsf{T}_N}  \Big| \frac{P^{\mathsf{T}_N
}_k(\mathbf{0},x)}{\mu(x)}  -1 \Big| \lesssim  \sum_{j=2}^{|\mathsf{T}_N|} |\gamma_{j}|^{k} \asymp \rho_{\star}^{k} =  \exp(-  \Theta(k/N^2)  ). 
\end{align*}
This establishes \eqref{hit-fund-mati-3} and hence completes the proof.
\end{proof}

We turn to the proof of Lemma~\ref{prop:hitting-time}.
The argument will make use of the following results, the first of which asserts that the hitting time of a set $A$ is asymptotically exponential distribution if the mean hitting time of $A$ is much larger than the relaxation time of the random walk.
\begin{lem}[{\cite[Theorem 1]{AB92}}]
	\label{lem:mc-asyexp}
 Let $(Z_{t})$ be an irreducible,  reversible  continuous-time  Markov chain on a finite-state space $I$, with  generator  $Q$ and stationary distribution $\mu$.  Let  $t^{\mathrm{rel}}\coloneq  \lambda_{1}^{-1}$ be the relaxation time,  where $\lambda_1$ is the minimal non-zero eigenvalue of $-Q$.  
  Let $A$ be a  proper, non-empty subset of $I$, and let $T(A)$ be the first hitting time on $A$. Then 
	\begin{equation*}
\sup_{t>0} | \P_\mu( T(A) / \E_\mu [T(A)] > t) - e^{-t}| 
\le  
t^{\mathrm{rel}} / \E_\mu[ T(A)]\,.
	\end{equation*}
\end{lem}

A second ingredient in the proof of Lemma~\ref{prop:hitting-time} is,
\begin{lem}[Matthews' bound \cite{matthews1988}]
	\label{Matthews-lemma}
	  Using the notation in Lemma~\ref{prop:hitting-time}, for $x\in \mathsf{T}_N$, set 
\begin{equation}\label{eq:max-hitting}
	f_{\max}(x) \coloneq    \max\Bigl\{\E_{y}[H(\mathsf{B}_{\rho}(a ))]:\, a \in A , y\in (\{x\}\cup \overline{A}^{\rho} ) \setminus \mathsf{B}_{\rho}(a) \Bigr\}\,,
\end{equation} 
with $f_{\min}(x)$ defined similarly replacing the $\max$ by $\min$. Let $h_m\coloneq  \sum_{i=1}^{m} \frac{1}{i}$ be the harmonic series.  With $m:= |A|$, and for any $x\in \mathsf{T}_N \setminus \overline{A}^{\rho}$, we have 
\begin{equation*}
h_m f_{\min}(x) - h_{m-1} f_{\max}(x)	\le  \E_x\bigl[H(\overline{A}^{\rho})\bigr]  \le h_m f_{\max}(x) - h_{m-1} f_{\min}(x).
\end{equation*} 
\end{lem} 
\begin{proof}
	Following  Matthews' argument  \cite{matthews1988} (see also \cite[Theorem 11.2]{LPW17} which deal with the expectation instead of generating function like here),  it can be verified that the ``cover time"  $\max_{a \in A} H(\mathsf{B}_{\rho}(a) )$ satisfies 
	\[ f_{\min}(x) h _{m} \le    \E_{x}\bigl[ \max_{a \in A} H(\mathsf{B}_{\rho}(a) ) \bigr] \le f_{\max}(x) h _{m}  \]
 for all $x \in \mathsf{T}_N$. Notice that conditionally on  $Y_{H(\overline{A})} = y \in \mathsf{B}_{\rho}(a^{*})$, the random variable   
 $ \max_{a \in A} H(\mathsf{B}_{\rho}(a) ) -  H(\overline{A})$ has the same distribution as $\max_{a \in A \setminus \{a_*\}} H(\mathsf{B}_{\rho}(a) )$ under  $\P_{y}$.  Thus 
 \[  f_{\min}(y) h _{m-1} \le \E_{x}\bigl[ \max_{a \in A} H(\mathsf{B}_{\rho}(a) )  - H(\overline{A}) \mid  Y_{H(\overline{A})} = y \in \mathsf{B}_{\rho}(a^{*}) \bigr]\le  f_{\max}(y) h_{m-1}\,. \]
Noting that  by definition $ f_{\min}(x) \le f_{\min}(y)\le f_{\max}(y) \le f_{\max}(x) $ since $y \in  \overline{A}$ and $x \notin \overline{A}$, then the desired result follows.
\end{proof}

We can now give,
\begin{proof}[Proof of Lemma~\ref{prop:hitting-time}] 
 In the proof we abbreviate   $\overline{A} \equiv \overline{A}^{\rho}$. Let $Z_{t}$ be the continuous-time simple random walk on $\mathsf{T}_{N}$ with jumping rate $1$. Let  
 $\mu$ denote the uniform measure on $\mathsf{T}_N$, which is stationary for $(Z_{t})$.
 
 \smallskip \noindent
\underline{\emph{Step 1.}} 
We write $T(\overline{A})$  the first hitting time on $A$ by  $Z_{t}$. Clearly $\E_{\mu}[T(\overline{A})]= \E_{\mu}[H(\overline{A})]$.   
By Corollary 12.13 and Section 12.3.1. in \cite{LPW17}, for the transition matrix $P^{\mathsf{T}_N}$ of simple random walk on $\mathsf{T}_N$, the second largest eigenvalue $\gamma_2=1- \Theta(N^{-2})$. 
 The generator of $(Z_t)$ is $P^{\mathsf T_N}-I$, and hence the smallest non-zero eigenvalue of $I-P^{\mathsf{T}_N}$ is exactly $1-\gamma_2$.  Thus there is an absolute constant $C>0$ such that  
 the relaxation time $t^{\mathrm{rel}}=(1-\gamma_2)^{-1}$ of $(Z_t)$ is  bounded by $C N^2$. 
  Applying   Lemma \ref{lem:mc-asyexp} gives 
\begin{equation}\label{eq:cont-asy-exp}
	 \sup_{t\ge 0} \left|
	 \P_{\mu}\Bigl( {T(\overline{A})}>{\E_{\mu}[H(\overline{A})]}t\Bigr)-e^{-t}
	 \right|
	 \leq C \frac{ N^2}{\E_{\mu}[H(\overline{A})]}  .
\end{equation}  

We claim that uniformly in  $(x,A,\rho)  \in \mathscr{K}_{\gamma}^{N}$
\begin{equation}\label{eq:hit-mu-A}
	 \E_{\mu}[H(\overline{A})]
	 =\frac{2}{\pi}\, |\mathsf{T}_N| \frac{\log N  }{|A|} \Bigl[ 1+ O\bigl(  \frac{|A|\log |A| \log r_N}{\log N} \bigr) \Bigr].
\end{equation} 
Note that  $|A|\log |A| \log r_N \le (\log N)^{\gamma} (\log \log N)^{1+100(\gamma+1)}=o(\log N)$.  Consequently,  substituting \eqref{eq:hit-mu-A} into \eqref{eq:cont-asy-exp} yields
\begin{equation}\label{eq:cont-asy-exp-2}
	\sup_{(x,A,\rho)  \in \mathscr{K}_{\gamma}^{N}} \sup_{t\ge 0} \left|
 \P_{\mu}\Bigl( {T(\overline{A})}>{\E_{\mu}[H(\overline{A})]}t\Bigr)-e^{-t}
 \right|
  \lesssim \frac{|A|}{\log N} . 
\end{equation} 

To show \eqref{eq:hit-mu-A}, we apply Lemma \ref{Matthews-lemma}.
Combining \eqref{eq:mean-hitting} with \eqref{eq:max-hitting}, and using the separation condition in \eqref{def:K-gamma-N}, namely that uniformly over  $(x,A,\rho)  \in \mathscr{K}_{\gamma}^{N}$, 
$\min  \{d(u,v): u \neq v\in A\cup\{x\} \}\ge  {N}/{r^2_N} $,
we obtain that
\begin{equation*}
	f_{\max}(x) = \frac{2}{\pi} |\mathsf{T}_N| \log N  + O(|\mathsf{T}_N|\log r_N) \  \text{ and } \ f_{\min}(x) = \frac{2}{\pi} |\mathsf{T}_N| \log N  - \Theta(|\mathsf{T}_N| \log r_N) .
\end{equation*} 
 It then follows from   Lemma \ref{Matthews-lemma}  that  
\begin{equation}\label{eq:hit-x-A}
	  \E_{x}[H(\overline{A})]
	 =\frac{2}{\pi}\,|\mathsf{T}_N|\log N\cdot \frac{1}{|A|}+ O( |\mathsf{T}_N|\log |A|  \log r_N)\,,
\end{equation}
uniformly in  $(x,A,\rho)  \in \mathscr{K}_{\gamma}^{N}$.
Since $\mu(\{ x: d(x,A) \le N/r_N^2 \}) = \Theta(|A|/ r_{N}^{4} )$, and $\max_{x,y \in \mathsf{T}_N} \E_{x} [H(y)] \lesssim   |\mathsf{T}_N| \log N$ by Lemma \ref{lem:est_laplace},  the claim \eqref{eq:hit-mu-A} follows directly from \eqref{eq:hit-x-A}.

 \smallskip \noindent 
\underline{\emph{Step 2.}} 
 Fix $k=k_N$ such that $ e^{k_N}= (\frac{\log N}{|A|}   )^{1/2} $. By our assumption $k_N \to \infty$ as $N \to \infty$.  
Denote  $t^{\mathrm{mix}}_N$  the mixing time of  $Y_{n}$ with total variation distance   $  2^{-k_N }$; that is  
\[ 
 t_N^{\mathrm{mix}} =t_{\mathrm{mix}}(2^{-k_N}) \coloneq  \inf \Bigl\{ n \ge 0:  \text{ for any } y \in \mathsf{T}_N, \frac{1}{2}\sum_{z\in \mathsf{T}_N} \big|\P_{y}(Y_n=z)-\mu(z)\big|\le   2^{-k_N} \Bigr\}. 
\]  
 By  \cite[Example 5.3.3]{LPW17}, we have $  t_N^{\mathrm{mix}}\leq   k_{N}  N^2$. For simplicity let us write  
\[
  a_{N}(t) \coloneq \frac{2}{\pi}\,\frac{ |\mathsf{T}_N| \log N}{|A|} t  - {t_N^{\mathrm{mix}}} \ge  c k_N N^2   \ \text{ for all } \    t \ge   k_N^{2}  \frac{|A|}{\log N} . 
\] 
The Markov property then yields that  for $t \ge   k_N^{2}  \frac{|A|}{\log N} $,
\begin{align}
  &\P_x \Bigl(H(\overline{A})> \frac{2}{\pi}\,\frac{ |\mathsf{T}_N| \log N}{|A|} t\Bigr) =\sum_{y \in \mathsf{T}_{N}}    
	 \P_{y}\bigl(H(\overline{A})> a_N (t) \bigr)\P_{x}( Y_{t_N^{\mathrm{mix}} } = y , H(\overline{A}) > t_N^{\mathrm{mix}} ) \notag\\
	&\quad \in \sum_{y \in \mathsf{T}_{N}}    
	 \P_{y}\bigl(H(\overline{A})> a_N (t) \bigr)\P_{x}( Y_{t_N^{\mathrm{mix}} } = y   ) -\Bigl[ 0 \,,\, \P_{x}(H(\overline{A}) \le t_N^{\mathrm{mix}} )  \Bigr] .  \label{eq:hA-markov} 
\end{align}  
We shall show that the probability   $\P_{x}(H(\overline A) \le t_N^{\mathrm{mix}} )$   is negligible, namely
\begin{equation}\label{eq:H-is-small-2}
\sup_{ (x,A,\rho)  \in \mathscr{K}_{\gamma}^{N} }
 \P_{x}(H(\overline{A}) \le   k_N^{2} N^2 )  \lesssim  \frac{|A|\,  (\log r_N)^2 }{\log N}   \lesssim e^{-k_N} .
\end{equation}  Then replacing $\P_{x}( Y_{t_N^{\mathrm{mix}} } = y  )$ by $\mu(y)$ in \eqref{eq:hA-markov}  and using the definition of the mixing time, 
we get, 
\begin{equation}
  \sup_{(x,A,\rho)  \in \mathscr{K}_{\gamma}^{N}} \, \sup_{t \ge k^{2}_N  {|A|}/{\log N}   }  \Big|
	  \P_x \Bigl(H(\overline{A})> \frac{2}{\pi}\,\frac{ |\mathsf{T}_N| \log N}{|A|} t\Bigr) - 
	 \P_{\mu} \bigl(H(\overline{A})>a_N (t)\bigr)
	 \Big|  
	 \lesssim 2^{-k_N} . \label{eq:mix}  
\end{equation} 

 \smallskip \noindent
\underline{\emph{Step 3.}}  We now prove \eqref{eq:H-is-small-2}. By using the union bound, we obtain 
\begin{align*}
 \P_x\left(H(\overline{A})\le   k_N^2 N^2 \right)
	& \le |A| \, \max_{x,y: d(x,y) \ge N/r^{2}_N}
	 \P_x\left(H( \mathsf{B}_{\rho} (y))\le    k_N^2 N^2 \right) \\
	 &
	 \lesssim  |A|\, \max_{x,y: d(x,y) \ge N/r^{2}_N} 
	 \E_x\left[e^{-  H( \mathsf{B}_{\rho}(y)) /  (|\mathsf T_N|(\log r_N)^2)} \right] \lesssim \frac{|A|\,  (\log r_N)^2 }{\log N}\,,
\end{align*} 
as desired.
Above in the last inequality, we have  used part (i) in Lemma~\ref{lem:est_laplace} with $\lambda_N=\frac{\log N}{(\log r_N)^2} $  and the fact $k_N=o(\log r_N)$.
Furthermore,  as a  consequence of \eqref{eq:H-is-small-2}, it suffices to show \eqref{eq:hit-asy-exp} only for $t \ge   k_N^{2} \frac{|A|}{\log N} $, i.e., 
\begin{equation}
	 \lim_{N\to\infty}\ \sup_{ (x,A,\rho)  \in \mathscr{K}_{\gamma}^{N} } \ \sup_{t \ge  k_N^{2}  \frac{|A|}{\log N}}
	 \bigg|\,
	 \P_x \Bigl({H(\overline{A})}  > {\frac{2}{\pi}\,\frac{ |\mathsf{T}_N| \log N}{|A|}} t\Bigr)-e^{-t} \,
	 \bigg|
	 =0. \label{eq:small-t} 
\end{equation}
 \smallskip \noindent 
\underline{\emph{Step 4.}} By using \eqref{eq:mix} and \eqref{eq:small-t}, to prove  \eqref{eq:hit-asy-exp}, it is sufficient to show 
\begin{equation*}
   \sup_{ (x,A,\rho)  \in \mathscr{K}_{\gamma}^{N} } \, \sup_{t \ge k^{2}_N  {|A|}/{\log N}   }   \Big| 
	 \P_{\mu} \bigl(H(\overline{A})>a_N (t)\bigr) - e^{-t}
	 \Big|  = o _N(1)\,.
\end{equation*}  
Notice that given $H(\overline{A}) $,  the conditional distribution of $T(\overline{A})$ is the same as $\sum_{i=1}^{H(\overline{A})} \mathbb{e}_{i}$ where $(\mathbb{e}_{i})_{i \ge 1}$ are i.i.d. standard exponential r.v.'s.  By  
using the Chernoff bound, we can find absolute constants $c,c'>0$ such that 
\begin{equation}\label{eq:TH-deviation}
 \sup_{t \ge k_N  {|A|}/{\log N}   } \sup_{m \ge k_N N^2  }\P_{\mu} \bigl( | T(\overline{A}) -   H(\overline{A}) |  \ge H(\overline{A}) ^{2/3} \mid  H(\overline{A}) = m \bigr)  \leq   e^{- c k_N^{1/3} N^{2/3}} \le  e^{- c' N^{2/3}}  .
\end{equation}   
When the event $\{ H(\overline{A}) >  a_N(t) , |T(\overline{A}) -   H(\overline{A}) |  \le H(\overline{A}) ^{2/3}\} $ occurs, we have 
$   T(\overline{A}) >  a_N(t) - a_N(t)^{2/3}  $. This implies  
\begin{align*}
 \P_{\mu} \bigl(  H(\overline{A}) >  a_N(t)  \bigr) 
 \le \P_{\mu} \bigl( T(\overline{A}) >  a_N(t) - a_N(t)^{2/3}    \bigr) +    e^{-c'N^{2/3}}  .
\end{align*}  
On the other hand,  we have  
\begin{align*}
 \P_{\mu} \bigl(  H(\overline{A}) >  a_N(t)  \bigr) 
  & \ge \P_{\mu} \bigl( T(\overline{A}) >  a_N(t) + a_N(t)^{2/3}    \bigr)- \P_{\mu}\bigl(  T(\overline{A}) >  a_N(t) + a_N(t)^{2/3},  H(\overline{A})    \le a_N(t) \bigr) \\
  & \ge   \P_{\mu} \bigl( T(\overline{A}) >  a_N(t) + a_N(t)^{2/3}    \bigr)- e^{-c'N^{2/3}}\,,
\end{align*} 
where the second inequality follows from  \eqref{eq:TH-deviation} and the   stochastic domination again.
Therefore it suffices to verify
\begin{equation*}
 \sup_{ (x,A,\rho)  \in \mathscr{K}_{\gamma}^{N} }  \ \sup_{t \ge k_N^2  {|A|}/{\log N}   }\big|
\P_{\mu} \bigl( T(\overline{A}) >  a_N(t) \pm a_N(t)^{2/3}    \bigr)  - e^{-t} \big| = o_N(1) .
\end{equation*} 
By using \eqref{eq:cont-asy-exp-2}, it is enough to check 
\begin{equation*}
  \sup_{ (x,A,\rho)  \in \mathscr{K}_{\gamma}^{N} }  \, \sup_{t \ge k_N^2  {|A|}/{\log N}   }\Big| \exp \Bigl(- t \frac{	a_N(t) \pm a_N(t)^{2/3}  }{ t \E_{\mu}[H(\overline{A})]}  \Bigr) - e^{-t }\Big|  = o_N(1).
\end{equation*} 
This holds as uniformly in 
  $t \ge k_N^2  {|A|}/{\log N} $ and  $(x,A,\rho)  \in \mathscr{K}_{\gamma}^{N} $, we have 
\begin{equation*}
\frac{	a_N(t) \pm a_N(t)^{2/3}  }{ t \E_{\mu}[H(\overline{A})]} =  1+ O\Bigl(  \frac{|A|\log |A| \log r_N}{\log N}  \Bigr) + O\Bigl(  \frac{ t^{\mathrm{mix}}_N }{    |\mathsf{T}_N|k_N^2}   \Bigr) + O(\frac{1}{N^{2/3}}) = 1+o_N(1).
\end{equation*} 
The desired result follows from the fact $ |e^{-(1 \pm \epsilon) t} - e^{-t}| \le t e^{-t/2} \epsilon $ for all $t > 0$ and $0<\epsilon<1/2$ by using mean value theorem. This completes the proof of \eqref{eq:hit-asy-exp}.

 \smallskip \noindent
\underline{\emph{Step 5.}} Next we show \eqref{eq:hit-asy-exp-2}. Denote by $A_0=A$ and $A_2=A \setminus A_1$. Let $\theta_1=\theta$, $\theta_2=1-\theta$ and $\theta_0=1$.
 The endpoint cases $\theta=0,1$ are covered by the same argument with the convention $H(\emptyset)=\infty$; below we write the proof for the non-degenerate case. 
Denote by 
\[ H_i= {H(\overline{A}_i)}/ \Big({\frac{2}{\pi}  |\mathsf{T}_N|} \frac{\log N}{|A|}  \Big) \text{ for }  i=0,1,2. \]
Then the convergence \eqref{eq:hit-asy-exp} yields,   for $i=0,1,2$ and any $\epsilon>0$  
\begin{equation*} \limsup_{N\to\infty}\,  \sup_{  (x,A,\rho)  \in \mathscr{K}_{\gamma}^{N} }  \,
\sup_{A_{1} \in \mathrm{Sub}_{\theta, \epsilon}(A)}  \, \sup_{t >0} \,  \big| \P_{x} ( H_i > t) - \exp( - t|A_i|/|A| ) \big|  = 0 .
\end{equation*}
 Since for any positive random variable $H$ and $s >0$,
$ 1-\E[ e^{-s H} ] = \int_{0}^{\infty} \P(H > t) s e^{-s t} \dif t $, 
 together with $||A_i|/|A| -\theta_i| \le \epsilon_{N} $ 
 we derive 
\begin{equation} 
	\label{eq:laplace-conv}  \lim_{N\to\infty} \,  \sup_{  (x,A,\rho)  \in \mathscr{K}_{\gamma}^{N} }  \,
\sup_{A_{1} \in \mathrm{Sub}_{\theta, \epsilon_{N}}(A)} \, \sup_{s > 0} \Big| \, \bbE_{x} \bigl[\rme^{ - s  H_i }   \bigr] - \frac{\theta_i  }{\theta_i   +  s} \, \Big| = 0.
\end{equation}  

Since $H_0 = H_1 \wedge H_2$, by the strong Markov property we  obtain
\begin{equation*}
\E_x[e^{-s H_0}] -	\E_x[e^{- s H_1}]  = \E_x\bigl[e^{- s H_0} \ind{   H_1 >H_0 } ( 1- \E_{Y_{H(\overline{A}_2)}}[e^{- s H_1}]  ) \bigr]\,,
\end{equation*}
and hence for any $s>0$, 
\begin{equation*} 
\frac{\bbE_{x} [\rme^{ -s  H_0 }] -\bbE_{x} [ \rme^{ -s  H_1 }   ]}{1 -\inf_{y \in \overline{A}_2}  \bbE_{y} \bigl[\rme^{ -s  H_1 }   \bigr] } 
\leq 
\bbE_{x} \bigl[\rme^{ -s  H_0 } \ind{H_1 >H_0}  \bigr]  \le \P_{x} (H(\overline{A}_1) >  H(\overline{A}_2) ) . 
\end{equation*} 
 The estimate \eqref{eq:laplace-conv} applies uniformly to the terms started from $y\in\overline A_2$ as well. 
Now plugging \eqref{eq:laplace-conv} we obtain for any $s>0$
\begin{equation*}
  \liminf_{N \to \infty} \inf_{  (x,A,\rho)  \in \mathscr{K}_{\gamma}^{N} }  \,
\inf_{A_{1} \in \mathrm{Sub}_{\theta, \epsilon_{N}}(A)} \, \P_{x} (H(\overline{A}_1) >  H(\overline{A}_2) )   \ge  \frac{\frac{1}{1+s}- \frac{\theta}{\theta +s}}{1- \frac{\theta}{\theta +s}} = \frac{1-\theta}{1+s} .
\end{equation*}
Sending $s \downarrow 0$ we get  $  \liminf\limits_{N \to \infty}  \inf\limits_{  (x,A,\rho)  \in \mathscr{K}_{\gamma}^{N} }  \,
\inf\limits_{A_{1} \in \mathrm{Sub}_{\theta, \epsilon_{N}}(A)} \, \P_{x} (H(\overline{A}_1) >  H(\overline{A}_2) ) \ge 1-\theta $. Notice that   $A_1$ and $A_2$ are symmetric, the same argument also gives
 \begin{align*}
	\theta &\le    \liminf_{N \to \infty} \inf_{  (x,A,\rho)  \in \mathscr{K}_{\gamma}^{N} }  \,
\inf_{A_{1} \in \mathrm{Sub}_{\theta, \epsilon_{N}}(A)} \,\P_{x} (H(\overline{A}_2) >  H(\overline{A}_1) )   \\
	&= 1-    \limsup_{N \to \infty}  \sup_{  (x,A,\rho)  \in \mathscr{K}_{\gamma}^{N} }  \,
\sup_{A_{1} \in \mathrm{Sub}_{\theta, \epsilon_{N}}(A)} \, \P_{x} (H(\overline{A}_1) >  H(\overline{A}_2) ) .
 \end{align*}
 This completes the proof of assertion (ii). 
\end{proof}

Lastly, we have:
\begin{proof}[Proof of Lemma~\ref{l:2.13}]
We have the obvious lower bound 
\begin{equation}\label{eq:lbnd-Green}
	\mathbb{G}^{N}_{H(\overline{A}^{\rho})}(x,y) \ge \mathbb{G}^{N}_{H(\partial  \mathsf{B}_{ N/ r^3_N}   (x) )}(x,y) \ge \frac{2}{\pi} \log \frac{ N/r^3_N }{1+d(x,y)} - C   .
\end{equation} 
Above the second inequality follows from \cite[Proposition 6.3.5]{LPW17} and $C>0$ is an absolute constant. 

For the upper bound,  it suffices to bound $\mathbb{G}^{N}_{H(\overline{A}^{\rho})}(x,x)$.
We claim that upon hitting the boundary of $\mathsf{B}_{N/r^{3}_N}(x)$, the simple random walk has probability $o_N(1)$ to return back to $x$ before hitting $\overline{A}^{\rho}$. Indeed for any $z \in \partial \mathsf{B}_{N/r^{3}_N}(x)$, we have 
\begin{equation*}
 \P_{z} ( H(x) \le  H( \overline{A}^{\rho}) ) \le  \P_{z} (  H( \overline{A}^{\rho}) > \omega_N^{-1/2} N^2 \log N  ) +  \P_{z} ( H(x) \le \omega_N^{-1/2} N^2 \log N )\,.
\end{equation*}
By applying \eqref{eq:hit-asy-exp} to  the first term, and using 
 the proof of \eqref{equa:lem:est_laplace}, applied at distance $N/r_N^3$,  with Markov's inequality 
for the second, we find 
a sequence $\epsilon_N \to 0$ 
such that  $ \P_{z} ( H(x) \le  H( \overline{A}^{\rho}) ) \le \epsilon_N $ for any $(x,A,\rho)$ and any $z \in \partial \mathsf{B}_{N/r^{3}_N}(x)$.

Each return from $\partial\mathsf{B}_{N/r^{3}_N}(x)$ to $x$ before hitting
$\overline A^\rho$ starts a new independent excursion from $x$ to
$\partial\mathsf{B}_{N/r^{3}_N}(x)$, and the number of such excursions is
stochastically dominated by a geometric random variable with success
parameter $1-\epsilon_N$.
Thereby we derive
\begin{equation*}
	\mathbb{G}^{N}_{H(\overline{A}^{\rho})}(x,x) \le   (1-\epsilon_N)^{-1} \mathbb{G}^{N}_{H(\partial  \mathsf{B}_{N/r^{3}_N} (x) )}(x,x) \le (1-\epsilon_N)^{-1}  [\frac{2}{\pi} \log{N} + C ]\,.
\end{equation*} 
This completes the proof. 
\end{proof}

\notoc{\section*{Disclosure of AI Use}}

Large language models were used in the preparation of this paper.
In particular, ChatGPT (GPT-5.5) identified an issue in an earlier
version of the proof of
Proposition~\ref{lem:only:moderate:traps:visited}: contrary to our
initial claim, the random variable used in that argument did not have
the uniformly bounded density required for \eqref{eq:error-1}.
It also suggested the alternative quantity $T(\eta,\mathbb{V})$
used in the current version and identified several other minor
errors, all of which were straightforward to correct. All remaining
proof ideas were developed independently by the authors. LLMs were otherwise used only to organize and refine the
writing, assist in the preparation of figures, and to check the mathematical arguments for potential errors.
The authors have independently verified all results and take full
responsibility for the content of this paper.

\notoc{\section*{Acknowledgements}}
The research of all authors was supported by ISF grants no.~1382/2017, ~2870/21 and~3782/25. The research of A.C, H.M and A.S was supported by a fellowship from the Lady Davis Foundation.

\bibliographystyle{abbrv}
\bibliography{GFF_RWRP}

@article{Aldous1978,
  author  = {Aldous, David J.},
  title   = {Stopping Times and Tightness},
  journal = {The Annals of Probability},
  volume  = {6},
  number  = {2},
  pages   = {335--340},
  year    = {1978},
  doi     = {10.1214/aop/1176995579}
}

@article{gayrard2012convergence,
  title={Convergence of clock process in random environments and aging in {B}ouchaud’s asymmetric trap model on the complete graph},
  author={Gayrard, V{\'e}ronique},
  journal={Electron. J. Probab},
  volume={17},
  number={58},
  pages={1--33},
  year={2012}
}

@article{Geb41,
 author = {Gebelein, Hans},
 title = {Das statistische {Problem} der {Korrelation} als {Variations}- und {Eigenwertproblem} und sein {Zusammenhang} mit der {Ausgleichsrechnung}},
 fjournal = {Zeitschrift f{\"u}r Angewandte Mathematik und Mechanik (ZAMM)},
 journal = {Z. Angew. Math. Mech.},
 issn = {0044-2267},
 volume = {21},
 pages = {364--379},
 year = {1941},
 language = {German},
 doi = {10.1002/zamm.19410210604},
 zbMATH = {3042142},
 Zbl = {0026.33402}
}

@article{Royen14,
 author = {Royen, Thomas},
 title = {A simple proof of the {Gaussian} correlation conjecture extended to some multivariate gamma distributions},
 fjournal = {Far East Journal of Theoretical Statistics},
 journal = {Far East J. Theor. Stat.},
 issn = {0972-0863},
 volume = {48},
 number = {2},
 pages = {139--145},
 year = {2014},
 language = {English},
 url = {www.pphmj.com/abstract/8713.htm},
 zbMATH = {6441652},
 Zbl = {1314.60070}
}

@article{Cox89,
 author = {Cox, J. T.},
 title = {Coalescing random walks and voter model consensus times on the torus in {{\({\mathbb{Z}}^ d\)}}},
 fjournal = {The Annals of Probability},
 journal = {Ann. Probab.},
 issn = {0091-1798},
 volume = {17},
 number = {4},
 pages = {1333--1366},
 year = {1989},
 language = {English},
 doi = {10.1214/aop/1176991158},
 zbMATH = {4123027},
 Zbl = {0685.60100}
}

@article{gayrard2013convergence,
  title={Convergence of clock processes on infinite graphs and aging in {B}ouchaud's asymmetric trap model on $Z^d$},
  author={Gayrard, V{\'e}ronique and Svejda, Adela},
  journal={arXiv preprint arXiv:1309.3066},
  year={2013}
}

@article{durrett1978functional,
  title={Functional limit theorems for dependent variables},
  author={Durrett, Richard and Resnick, Sidney I},
  journal={The Annals of Probability},
  pages={829--846},
  year={1978},
  publisher={JSTOR}
}

@article{fyodorov2008freezing,
  title={Freezing and extreme-value statistics in a random energy model with logarithmically correlated potential},
  author={Fyodorov, Yan V and Bouchaud, Jean-Philippe},
  journal={Journal of Physics A: Mathematical and Theoretical},
  volume={41},
  number={37},
  pages={372001},
  year={2008}
}

@article{CHL2,
author = {Aser Cortines and Lisa Hartung and Oren Louidor},
title = {More on the structure of extreme level sets in branching {B}rownian motion},
volume = {26},
journal = {Electronic Communications in Probability},
number = {none},
publisher = {Institute of Mathematical Statistics and Bernoulli Society},
pages = {1 -- 14},
year = {2021},
doi = {10.1214/20-ECP369},
URL = {https://doi.org/10.1214/20-ECP369}
}

@article{CHL1,
  title={The structure of extreme level sets in branching {B}rownian motion},
  author={Cortines, Aser and Hartung, Lisa and Louidor, Oren},
  journal={The Annals of Probability},
  volume={47},
  number={4},
  pages={2257--2302},
  year={2019},
  publisher={JSTOR}
}

@article{mandelbrot1989multifractal,
  title={Multifractal measures, especially for the geophysicist},
  author={Mandelbrot, Beno{\^\i}t B},
  journal={Pure and applied geophysics},
  volume={131},
  number={1},
  pages={5--42},
  year={1989},
  publisher={Springer}
}

@article{cortines2018dynamical,
  title={Dynamical freezing in a spin glass system with logarithmic correlations},
  author={Cortines, Aser and Gold, Julian and Louidor, Oren},
  journal={Electronic Journal of Probability},
  volume={23},
  year={2018},
  publisher={The Institute of Mathematical Statistics and the Bernoulli Society}
}

@article{bovier13,
author = {Anton Bovier and V{\'e}ronique Gayrard},
title = {Convergence of clock processes in random environments and ageing in the $p$-spin {SK} model},
volume = {41},
journal = {The Annals of Probability},
number = {2},
publisher = {Institute of Mathematical Statistics},
pages = {817 -- 847},
year = {2013},
doi = {10.1214/11-AOP705},
URL = {https://doi.org/10.1214/11-AOP705}
}

@article{rhodes2014gaussian,
title={{G}aussian multiplicative chaos and applications: A review},
author={Rhodes, R{\'e}mi and Vargas, Vincent},
journal={Probability Surveys},
volume={11},
pages={315--392},
year={2014},
publisher={Citeseer}
}

@article{arous2008universality,
  title={Universality of the {REM} for dynamics of mean-field spin glasses},
  author={Arous, G{\'e}rard Ben and Bovier, Anton and {\v{C}}ern{\`y}, Ji{\v{r}}{\'\i}},
  journal={Communications in mathematical physics},
  volume={282},
  number={3},
  pages={663--695},
  year={2008},
  publisher={Springer}
}

@article{DRSV1,
  title={Critical {G}aussian multiplicative chaos: convergence of the derivative martingale},
  author={Duplantier, Bertrand and Rhodes, R{\'e}mi and Sheffield, Scott and Vargas, Vincent and others},
  journal={The Annals of Probability},
  volume={42},
  number={5},
  pages={1769--1808},
  year={2014},
  publisher={Institute of Mathematical Statistics}
}

@article{DRSV2,
  title={Renormalization of critical {G}aussian multiplicative chaos and {KPZ} relation},
  author={Duplantier, Bertrand and Rhodes, R{\'e}mi and Sheffield, Scott and Vargas, Vincent},
  journal={Communications in Mathematical Physics},
  volume={330},
  number={1},
  pages={283--330},
  year={2014},
  publisher={Springer}
}

@book{Ber96,
 author = {Bertoin, Jean},
 title = {L{\'e}vy processes},
 fseries = {Cambridge Tracts in Mathematics},
 series = {Camb. Tracts Math.},
 issn = {0950-6284},
 volume = {121},
 isbn = {0-521-56243-0},
 year = {1996},
 publisher = {Cambridge: Cambridge Univ. Press},
 language = {English},
 zbMATH = {918811},
 Zbl = {0861.60003}
}

@misc{AF14,
    AUTHOR = {Aldous, David and Fill, James Allen},
     TITLE = {Reversible {M}arkov Chains and Random Walks on Graphs},
      YEAR = {2002},
      NOTE = {Unfinished monograph, recompiled 2014, available 
      at \url{http://www.stat.berkeley.edu/$\sim$aldous/RWG/book.html}}
      }

@book{LPW17,
 author = {Levin, David A. and Peres, Yuval and Wilmer, Elizabeth L.},
 title = {{M}arkov chains and mixing times. {With} a chapter on ``{Coupling} from the past'' by {James} {G}. {Propp} and {David} {B}. {Wilson}.},
 edition = {2nd edition},
 isbn = {978-1-4704-2962-1; 978-1-4704-4232-3},
 year = {2017},
 publisher = {Providence, RI: American Mathematical Society (AMS)},
 language = {English},
 zbMATH = {6813269},
 Zbl = {1390.60001}
}

@incollection{AB92,
 author = {Aldous, David J. and Brown, Mark},
 title = {Inequalities for rare events in time-reversible {Markov} chains. {I}.},
 booktitle = {Stochastic inequalities. Collection of papers of conference, one of the 1991 AMS-IMS-SIAM joint summer research conferences, Seattle, WA, USA, July 1991},
 isbn = {0-940600-29-3},
 pages = {1--16},
 year = {1992},
 publisher = {Hayward, CA: IMS, Institute of Mathematical Statistics},
 language = {English},
 doi = {10.1214/lnms/1215461937},
 zbMATH = {6980216},
 Zbl = {1400.60096}
}

@misc{HLW25,
      title={A Probabilistic Proof for Stable Fluctuations in the Extremal Process of Branching {B}rownian Motion}, 
      author={Lisa Hartung and Oren Louidor and Tianqi Wu},
      year={2025},
      eprint={2504.21455},
      archivePrefix={arXiv},
      primaryClass={math.PR},
      url={https://arxiv.org/abs/2504.21455}, 
       note = {arXiv:2504.21455}
}

@misc{CHL19,
 author = {Cortines, Aser and Hartung, Lisa and Louidor, Oren},
 title = {Decorated random walk restricted to stay below a curve (supplement material)},
 year = {2019},
 howpublished = {Preprint, {arXiv}:1902.10079 [math.{PR}] (2019)},
 url = {https://arxiv.org/abs/1902.10079},
 note = {arXiv:1902.10079}
}

@article{BDZ16,
 author = {Bramson, Maury and Ding, Jian and Zeitouni, Ofer},
 title = {Convergence in law of the maximum of the two-dimensional discrete {Gaussian} free field},
 fjournal = {Communications on Pure and Applied Mathematics},
 journal = {Commun. Pure Appl. Math.},
 issn = {0010-3640},
 volume = {69},
 number = {1},
 pages = {62--123},
 year = {2016},
 language = {English},
 doi = {10.1002/cpa.21621},
 zbMATH = {6525760},
 Zbl = {1355.60046}
}

@misc{LS24,
      title={Tightness for the Cover Time of Wired Planar Domains}, 
      author={Oren Louidor and Santiago Saglietti},
      year={2024},
      eprint={2406.11034},
      archivePrefix={arXiv},
      primaryClass={math.PR},
      url={https://arxiv.org/abs/2406.11034}, 
}

@Article{arous2008,
  Title                    = {The arcsine law as a universal aging scheme for trap models.},
  Author                   = {Ben Arous, G{\'e}rard and {\v{C}}ern{\`y}, Ji{\v{r}}{\'\i}},
  Journal                  = {Comm. Pure Appl. Math.},
  Year                     = {2008},
  Number                   = {3},
  Pages                    = {289-329},
  Volume                   = {61}
}

@Article{arous2006course,
  Title                    = {Course 8 Dynamics of trap models},
  Author                   = {Ben Arous, G{\'e}rard and {\v{C}}ern{\`y}, Ji{\v{r}}{\'\i}},
  Journal                  = {Les Houches},
  Year                     = {2006},
  Pages                    = {331--394},
  Volume                   = {83},

  Publisher                = {Elsevier}
}

@Article{ding2013exponential,
  Title                    = {Exponential and double exponential tails for maximum of two-dimensional discrete {G}aussian free field},
  Author                   = {Ding, Jian},
  Journal                  = {Probability Theory and Related Fields},
  Year                     = {2013},
  Number                   = {1-2},
  Pages                    = {285--299},
  Volume                   = {157},

  Publisher                = {Springer}
}

@Book{LawL2010,
  Title                    = {Random Walk: A Modern Introduction},
  Author                   = {Lawler, G.F. and Limic, V.},
  Publisher                = {Cambridge University Press},
  Year                     = {2010},
  Series                   = {Cambridge Studies in Advanced Mathematics}
}

@Article{matthews1988,
  Title                    = {Covering problems for {M}arkov chains},
  Author                   = {Matthews, P.},
  Journal                  = {Ann. Probab.},
  Year                     = {1988},
  Number                   = {3},
  Pages                    = {1215-1228},
  Volume                   = {16}
}

@article {BL1,
    AUTHOR = {Biskup, Marek and Louidor, Oren},
     TITLE = {Extreme local extrema of two-dimensional discrete {G}aussian
              free field},
   JOURNAL = {Comm. Math. Phys.},
  FJOURNAL = {Communications in Mathematical Physics},
    VOLUME = {345},
      YEAR = {2016},
    NUMBER = {1},
     PAGES = {271--304},
      ISSN = {0010-3616,1432-0916},
   MRCLASS = {60G15 (60G70)},
  MRNUMBER = {3509015},
MRREVIEWER = {Yoshihiro\ Abe},
       DOI = {10.1007/s00220-015-2565-8},
       URL = {https://doi.org/10.1007/s00220-015-2565-8},
}

@article {BL2,
    AUTHOR = {Biskup, Marek and Louidor, Oren},
     TITLE = {Conformal symmetries in the extremal process of
              two-dimensional discrete {G}aussian free field},
   JOURNAL = {Comm. Math. Phys.},
  FJOURNAL = {Communications in Mathematical Physics},
    VOLUME = {375},
      YEAR = {2020},
    NUMBER = {1},
     PAGES = {175--235},
      ISSN = {0010-3616,1432-0916},
   MRCLASS = {60G15 (60G70 83C45)},
  MRNUMBER = {4082182},
       DOI = {10.1007/s00220-020-03698-0},
       URL = {https://doi.org/10.1007/s00220-020-03698-0},
}

@article {BL3,
    AUTHOR = {Biskup, Marek and Louidor, Oren},
     TITLE = {Full extremal process, cluster law and freezing for the
              two-dimensional discrete {G}aussian free field},
   JOURNAL = {Adv. Math.},
  FJOURNAL = {Advances in Mathematics},
    VOLUME = {330},
      YEAR = {2018},
     PAGES = {589--687},
      ISSN = {0001-8708},
   MRCLASS = {60G15 (60G55 60G57 60G70 62G30 82B41)},
  MRNUMBER = {3787554},
       DOI = {10.1016/j.aim.2018.02.018},
       URL = {https://doi.org/10.1016/j.aim.2018.02.018},
}

@incollection {biskuppims,
    AUTHOR = {Biskup, Marek},
     TITLE = {Extrema of the two-dimensional discrete {G}aussian free field},
 BOOKTITLE = {Random graphs, phase transitions, and the {G}aussian free
              field},
    SERIES = {Springer Proc. Math. Stat.},
    VOLUME = {304},
     PAGES = {163--407},
 PUBLISHER = {Springer, Cham},
      YEAR = {[2020] \copyright 2020},
      ISBN = {978-3-030-32011-9; 978-3-030-32010-2},
   MRCLASS = {60G15 (60G70 60K37)},
  MRNUMBER = {4043225},
       DOI = {10.1007/978-3-030-32011-9\{_}3}

@article{biskup2024near,
  title={Near-maxima of the two-dimensional discrete {G}aussian free field},
  author={Biskup, Marek and Gufler, Stephan and Louidor, Oren},
  journal={Annales de l'Institut Henri Poincar{\'e} (B) Probabilit{\'e}s et Statistiques},
  volume={60},
  number={1},
  pages={281--311},
  year={2024}
}

@article{carpentier2001glass,
  title={Glass transition of a particle in a random potential, front selection in nonlinear renormalization group, and entropic phenomena in {L}iouville and sinh-{G}ordon models},
  author={Carpentier, David and Le Doussal, Pierre},
  journal={Physical review E},
  volume={63},
  number={2},
  pages={026110},
  year={2001},
  publisher={APS}
}

@article{castillo2001freezing,
  title={Freezing of dynamical exponents in low dimensional random media},
  author={Castillo, Horacio E and Le Doussal, Pierre},
  journal={Physical review letters},
  volume={86},
  number={21},
  pages={4859},
  year={2001},
  publisher={APS}
}

@article{derrida1981random,
  title={Random-energy model: An exactly solvable model of disordered systems},
  author={Derrida, Bernard},
  journal={Physical Review B},
  volume={24},
  number={5},
  pages={2613},
  year={1981},
  publisher={APS}
}

@article{arous2003glauberII,
  title={{G}lauber Dynamics of the Random Energy Model {II}},
  author={Arous, G{\'e}rard Ben and Bovier, Anton and Gayrard, V{\'e}ronique},
  journal={Communications in mathematical physics},
  volume={236},
  number={1},
  pages={1--54},
  year={2003},
  publisher={Springer}
}

@article{arous2003glauberI,
  title={{G}lauber Dynamics of the Random Energy Model {I}},
  author={Arous, G{\'e}rard Ben and Bovier, Anton and Gayrard, V{\'e}ronique},
  journal={Communications in mathematical physics},
  volume={235},
  number={1},
  pages={379--425},
  year={2003},
  publisher={Springer}
}

@article{arous2006aging,
  title={Aging in two-dimensional {B}ouchaud's model},
  author={Arous, G{\'e}rard Ben and {\v{C}}ern{\`y}, Ji{\v{r}}{\'\i} and Mountford, Thomas},
  journal={Probability theory and related fields},
  volume={134},
  number={1},
  pages={1--43},
  year={2006},
  publisher={Springer}
}

@article{arous2008arcsine,
  title={The arcsine law as a universal aging scheme for trap models},
  author={Arous, G{\'e}rard Ben and {\v{C}}ern{\`y}, Ji{\v{r}}{\'\i}},
  journal={Communications on Pure and Applied Mathematics},
  volume={61},
  number={3},
  pages={289--329},
  year={2008},
  publisher={Wiley Online Library}
}

@article{arguin2016extrema,
  title={Extrema of Log-correlated Random Variables},
  author={Arguin, Louis-Pierre},
  journal={Advances in Disordered Systems, Random Processes and Some Applications},
  pages={166},
  year={2016},
  publisher={Cambridge University Press}
}

\end{document}